\documentclass[a4paper, 10pt, reqno, DIV=calc]{amsart}

\usepackage[utf8]{inputenc}
\usepackage[T1]{fontenc}
\usepackage{lmodern}
\usepackage{microtype}
\usepackage{longtable}
\usepackage{amsmath, amsthm, amssymb, thmtools}
\usepackage{relsize}
\usepackage{mathrsfs}								
\usepackage{esint}									
\usepackage{enumitem}								
\usepackage{listings}								
\usepackage{stmaryrd}								
\usepackage[mathscr]{eucal}
\newcommand\mathcall[1]{\text{\usefont{U}{BOONDOX-cal}{m}{n}#1}}

\usepackage[dvipsnames]{xcolor}
\usepackage[linktocpage]{hyperref}
\definecolor{colorred}{HTML}{B00000}
\definecolor{colorgreen}{HTML}{258300}
\definecolor{colorblue}{HTML}{2e32fa}
\hypersetup{colorlinks=true, linkcolor=colorred, citecolor=colorgreen, urlcolor=black}								
\usepackage{xifthen}								
\usepackage{url}
\usepackage{booktabs}								
\usepackage{todonotes}
\usepackage{caption}
\usepackage{bbm}									
\usepackage{stmaryrd}								
\makeatletter
\newcommand\MyAutoefPhrasecolorGroup[1]{%
  \color@begingroup\color{MyCurrentcolor}#1\endgroup
}%
\def\HyRef@testreftype#1.#2\\{%
 \colorlet{MyCurrentcolor}{.}%
 \ltx@IfUndefined{#1autorefname}{%
   \ltx@IfUndefined{#1name}{%
     \HyRef@StripStar#1\\*\\\@nil{#1}%
     \ltx@IfUndefined{\HyRef@name autorefname}{%
       \ltx@IfUndefined{\HyRef@name name}{%
         \def\HyRef@currentHtag{}%
         \Hy@Warning{No autoref name for `#1'}%
       }{%
         \edef\HyRef@currentHtag{%
           \noexpand\MyAutoefPhrasecolorGroup{%
             \expandafter\noexpand\csname\HyRef@name name\endcsname
           }%
           \noexpand~%
         }%
       }%
     }{%
       \edef\HyRef@currentHtag{%
         \noexpand\MyAutoefPhrasecolorGroup{%
           \expandafter\noexpand
           \csname\HyRef@name autorefname\endcsname
         }%
         \noexpand~%
       }%
     }%
   }{%
     \edef\HyRef@currentHtag{%
       \noexpand\MyAutoefPhrasecolorGroup{%
         \expandafter\noexpand\csname#1name\endcsname
       }%
       \noexpand~%
     }%
   }%
 }{%
   \edef\HyRef@currentHtag{%
     \noexpand\MyAutoefPhrasecolorGroup{%
       \expandafter\noexpand\csname#1autorefname\endcsname
     }%
     \noexpand~%
   }%
 }%
}%
\makeatother

\numberwithin{equation}{section}

\newcommand{\nlb}{{\ensuremath{\textnormal{b}}}}
\newcommand{\nlc}{{\ensuremath{\textnormal{c}}}}
\newcommand{\nld}{{\ensuremath{\textnormal{d}}}}

\newcommand{\nlC}{{\ensuremath{\textnormal{C}}}}

\newcommand{\nlS}{{\ensuremath{\textnormal{S}}}}

\newcommand{\rmc}{{\ensuremath{\mathrm{c}}}}
\newcommand{\rmd}{{\ensuremath{\mathrm{d}}}}
\newcommand{\rme}{{\ensuremath{\mathrm{e}}}}

\newcommand{\rmn}{{\ensuremath{\mathrm{n}}}}

\newcommand{\rmp}{{\ensuremath{\mathrm{p}}}}
\newcommand{\rmq}{{\ensuremath{\mathrm{q}}}}

\newcommand{\rmt}{{\ensuremath{\mathrm{t}}}}

\newcommand{\rmA}{{\ensuremath{\mathrm{A}}}}
\newcommand{\rmB}{{\ensuremath{\mathrm{B}}}}
\newcommand{\rmC}{{\ensuremath{\mathrm{C}}}}
\newcommand{\rmD}{{\ensuremath{\mathrm{D}}}}
\newcommand{\rmE}{{\ensuremath{\mathrm{E}}}}

\newcommand{\rmH}{{\ensuremath{\mathrm{H}}}}

\newcommand{\rmR}{{\ensuremath{\mathrm{R}}}}

\newcommand{\rmT}{{\ensuremath{\mathrm{T}}}}

\newcommand{\sfa}{{\ensuremath{\mathsf{a}}}}
\newcommand{\sfb}{{\ensuremath{\mathsf{b}}}}

\newcommand{\sfd}{{\ensuremath{\mathsf{d}}}}

\newcommand{\sfh}{{\ensuremath{\mathsf{h}}}}

\newcommand{\sfn}{{\ensuremath{\mathsf{n}}}}

\newcommand{\sfp}{{\ensuremath{\mathsf{p}}}}

\newcommand{\sft}{{\ensuremath{\mathsf{t}}}}

\newcommand{\sfx}{{\ensuremath{\mathsf{x}}}}

\newcommand{\sfL}{{\ensuremath{\mathsf{L}}}}

\newcommand{\calA}{{\ensuremath{\mathcall{A}}}}

\newcommand{\calD}{{\ensuremath{\mathcall{D}}}}
\newcommand{\calE}{{\ensuremath{\mathcall{E}}}}
\newcommand{\calF}{{\ensuremath{\mathcall{F}}}}
\newcommand{\calG}{{\ensuremath{\mathcall{G}}}}
\newcommand{\calH}{{\ensuremath{\mathcall{H}}}}

\newcommand{\calJ}{{\ensuremath{\mathcall{J}}}}

\newcommand{\calL}{{\ensuremath{\mathcall{L}}}}
\newcommand{\calM}{{\ensuremath{\mathcall{M}}}}
\newcommand{\calN}{{\ensuremath{\mathcall{N}}}}

\newcommand{\calQ}{{\ensuremath{\mathcall{Q}}}}

\newcommand{\calS}{{\ensuremath{\mathcall{S}}}}

\newcommand{\calV}{{\ensuremath{\mathcall{V}}}}
\newcommand{\calW}{{\ensuremath{\mathcall{W}}}}

\newcommand{\bbI}{{\ensuremath{\mathbb{I}}}}

\newcommand{\bdmu}{{\ensuremath{\boldsymbol{\mu}}}}
\newcommand{\bdnu}{{\ensuremath{\boldsymbol{\nu}}}}

\newcommand{\bdsigma}{{\ensuremath{\boldsymbol{\sigma}}}}

\newcommand{\bdGamma}{{\ensuremath{\boldsymbol{\Gamma}}}}
\newcommand{\bdDelta}{{\ensuremath{\boldsymbol{\Delta}}}}

\newcommand{\B}{\sfb}

\newcommand{\N}{\boldsymbol{\mathrm{N}}}						
\newcommand{\Q}{\boldsymbol{\mathrm{Q}}}						
\newcommand{\R}{\boldsymbol{\mathrm{R}}}						
\renewcommand{\S}{\boldsymbol{\mathrm{S}}}						
\renewcommand{\d}{\,\mathrm{d}}				
\newcommand{\M}{\boldsymbol{\mathrm{M}}}

\let\limsup\undefined
\let\liminf\undefined
\let\div\undefined
\DeclareMathOperator*{\limsup}{limsup}		
\DeclareMathOperator*{\liminf}{liminf}		
\DeclareMathOperator*{\esssup}{esssup}		
\DeclareMathOperator*{\essinf}{essinf}		
\DeclareMathOperator{\supp}{spt}			
\DeclareMathOperator{\div}{div}				
\DeclareMathOperator{\Span}{span}			
\DeclareMathOperator{\tr}{tr}				
\DeclareMathOperator{\DIV}{\boldsymbol{\mathrm{div}}}	
\DeclareMathOperator{\norm}{\boldsymbol{\mathrm{n}}}	

\let\originalleft\left			
\let\originalright\right
\renewcommand{\left}{\mathopen{}\mathclose\bgroup\originalleft}
\renewcommand{\right}{\aftergroup\egroup\originalright}

\newcommand{\mapdef}[3][]{\ifthenelse{\isempty{#1}}{#2\quad\longmapsto\quad #3}{#1\colon\quad #2\quad\longmapsto\quad #3}}		
\newcommand{\der}[2][]{\ifthenelse{\isempty{#1}}{\frac{\nld}{\nld #2}}{\left.\frac{\nld}{\nld #2}\right\vert_{#1}}}				
\makeatletter
\newcommand{\checknarg}{\@ifnextchar\bgroup{\gobblenarg}{}}
\newcommand{\gobblenarg}[1]{\@ifnextchar\bgroup{,\ \! #1\gobblenarg}{,\ \! #1}}

\theoremstyle{definition}
\newtheorem{bump}{Bump}[section]

\theoremstyle{plain}
\newtheorem{theorem}[bump]{Theorem}
\newtheorem{proposition}[bump]{Proposition}
\newtheorem{definition}[bump]{Definition}
\newtheorem{lemma}[bump]{Lemma}
\newtheorem{corollary}[bump]{Corollary}
\newtheorem{assumption}[bump]{Assumption}

\theoremstyle{remark}
\newtheorem{remark}[bump]{Remark}
\newtheorem{example}[bump]{Example}

\newtheoremstyle{cited}
{\topsep}		
{\topsep}		
{\itshape}		
{}				
{\bfseries}		
{\textbf{.}}	
{.5em}			
{\thmname{#1} \thmnumber{#2} \thmnote{\normalfont#3}}		

\theoremstyle{cited}			

\makeatletter
\let\@fnsymbol\@arabic	 		
\makeatother

\makeatletter
\def\nonumberfootnote{\xdef\@thefnmark{}\@footnotetext}			
\makeatother

\newcommand{\mms}{\mathit{M}}				
\newcommand{\met}{\sfd}						
\newcommand{\W}{\mathit{W}}					
\newcommand{\meas}{\mathfrak{m}}			
\newcommand{\Leb}{\calL}					
\newcommand{\Haus}{\calH}					
\newcommand{\Borel}{\mathcall{B}}			
\newcommand{\vol}{\mathfrak{v}}				
\newcommand{\Prob}{\boldsymbol{\mathrm{P}}}					
\newcommand{\Exp}{\boldsymbol{\mathrm{E}}}	
\newcommand{\Id}{\mathrm{Id}}				
\newcommand{\F}{\calF}						
\newcommand{\Kato}{\boldsymbol{\mathscr{K}}}
\newcommand{\cem}{\dagger}

\newcommand{\RCD}{\mathrm{RCD}}				
\newcommand{\BE}{\mathrm{BE}}				
\newcommand{\eb}{{\mathrm{eb}}}
\newcommand{\ec}{{\mathrm{ec}}}
\newcommand{\ebc}{{\mathrm{ebc}}}

\newcommand{\bounded}{\nlb}					
\newcommand{\comp}{\nlc}					
\newcommand{\bs}{\mathrm{bs}}				
\newcommand{\loc}{\mathrm{loc}}				
\newcommand{\qloc}{\mathrm{qloc}}			
\newcommand{\TV}{\mathrm{TV}}				
\newcommand{\HS}{{\mathrm{HS}}}				
\newcommand{\Ric}{\mathrm{Ric}}				
\newcommand{\Cont}{\nlC}					
\newcommand{\Ell}{\mathit{L}}				
\newcommand{\Lip}{\mathrm{Lip}}				
\newcommand{\Sobo}{\nlS}					
\newcommand{\BV}{\mathrm{BV}}				
\newcommand{\Test}{\mathrm{Test}}			
\newcommand{\PCM}{\mathrm{Pcm}}				
\newcommand{\SF}{\mathrm{SF}}				

\newcommand{\Ch}{\calE}						
\newcommand{\Dom}{\mathcall{D}}					
\newcommand{\reg}{{\mathrm{reg}}}

\DeclareMathOperator{\Hess}{Hess}			
\newcommand{\ChHeat}{\sfp}					
\newcommand{\HHeat}{\sfh}					
\newcommand{\CHeat}{\sft}					
\newcommand{\Schr}[1]{\sfp^{#1}}			
\newcommand{\push}{\sharp}					

\newcommand{\One}{1}				
\newcommand{\Harm}{\mathcall{H}}			
\newcommand{\RIC}{\boldsymbol{\mathrm{Ric}}}

\newcommand{\Hodge}{\smash{\vec{\Delta}}}		
\newcommand{\Bochner}{\smash{\square}} 
\newcommand{\DELTA}{\bdDelta}		
\newcommand{\Meas}{\mathfrak{M}}				
\newcommand{\Hom}{\mathrm{Hom}}				
\newcommand{\cl}{\mathrm{cl}}				
\newcommand{\CAP}{\mathrm{cap}}				

\newcommand{\fin}{{\mathrm{f}}}
\newcommand{\sigmafin}{\sigma}
\newcommand{\finR}{{\mathrm{fR}}}
\newcommand{\sigmafinR}{{\sigma\mathrm{R}}}
\newcommand{\bR}{{\mathrm{fR}}}
\newcommand{\sigmaR}{{\sigma\mathrm{R}}}
\newcommand{\bc}{{\mathrm{bc}}}
\newcommand{\II}{\boldsymbol{\mathrm{I\!I}}}
\newcommand{\E}{\calE}
\newcommand{\ric}{\mathrm{ric}}
\newcommand{\RMA}{\boldsymbol{\rmA}}

\newcommand{\cov}{{\mathrm{cov}}}
\newcommand{\con}{{\mathrm{con}}}
\newcommand{\dR}{{\mathrm{dR}}}
\newcommand{\surf}{\mathfrak{s}}
\newcommand{\sym}{{\mathrm{sym}}}
\newcommand{\asym}{{\mathrm{asym}}}
\DeclareMathOperator{\sgn}{sgn}
\newcommand{\Reg}{{\mathrm{Reg}}}

\usepackage{blindtext}

\allowdisplaybreaks

\providecommand{\bysame}{\leavevmode\hbox to3em{\hrulefill}\thinspace}

\let\oldtocsection=\tocsection

\let\oldtocsubsection=\tocsubsection

\let\oldtocsubsubsection=\tocsubsubsection

\renewcommand{\tocsection}[2]{\hspace{0em}\oldtocsection{#1}{#2}}
\renewcommand{\tocsubsection}[2]{\hspace{1em}\oldtocsubsection{#1}{#2}}
\renewcommand{\tocsubsubsection}[2]{\hspace{2em}\oldtocsubsubsection{#1}{#2}}

\newcommand{\nocontentsline}[3]{}
\newcommand{\tocless}[2]{\bgroup\let\addcontentsline=\nocontentsline#1{#2}\egroup}

\begin{document}

\title[Vector calculus for tamed Dirichlet spaces]{Vector calculus for tamed Dirichlet spaces}
\author{Mathias Braun}
\address{Department of Mathematics, University of Toronto Bahen Centre, 40 St. George Street Room 6290, Toronto, Ontario M5S 2E4, Canada}
\email{braun@math.toronto.edu}
\date{\today}
\subjclass[2010]{Primary: 53C21, 58J35; Secondary: 31C25, 35J10, 35K20, 58J32}
\keywords{Heat flow; Kato class; Ricci curvature; Schrödinger semigroup}

\begin{abstract} In the language of $\Ell^\infty$-modules proposed by Gigli, we introduce a first order calculus on a topological Lusin measure space $(\mms,\meas)$ carrying a quasi-regular, strongly local Dirichlet form $\Ch$. Furthermore, we develop a second order calculus if $(\mms,\Ch,\meas)$ is tamed by a signed measure in the extended Kato class in the sense of Erbar, Rigoni, Sturm and Tamanini. This allows us to define e.g.~Hessians, covariant and exterior derivatives, Ricci curvature, and second fundamental form.
\end{abstract}

\maketitle
\thispagestyle{empty}

\tableofcontents

\setcounter{section}{-1}
\section{Introduction}

\subsubsection*{Background} The so-called $\RCD(K,\infty)$ condition, $K\in\R$, introduced in \cite{ambrosio2015a,ambrosio2014b, lott2009, sturm2006a, sturm2006b}, gives a meaning to the Ricci curvature of an infinitesimally Hilbertian  metric measure space $(\mms,\met,\meas)$ being bounded from below by $K$. It can be defined in at least two equivalent ways \cite{ambrosio2014b, ambrosio2015}: the \emph{Lagrangian} one by geodesic $K$-convexity of the relative entropy on the $2$-Wasserstein space, as well as the \emph{Eulerian} one phrased as a weak version of the \emph{Bakry--\smash{Émery} condition} $\BE_2(K,\infty)$ \cite{bakry1985a,bakry1985b} coupled with the Sobolev-to-Lipschitz property. In the latter picture, Gigli \cite{gigli2015,gigli2018} developed a powerful first and second order calculus. It leads to natural nonsmooth analogues to the notions of Hessian, covariant and exterior derivative, Ricci curvature,  Hodge's theorem, etc. This machinery has already provided deep structural and geometric results \cite{brue2019, brue2020}. In the abstract diffusion operator setting, a pointwise definition of a Ricci tensor is due to \cite{sturm2018}.

Recently, spaces with \emph{nonconstant}, even \emph{nonuniform} lower Ricci bounds have attracted  high attention. Using Schrödinger operator theory, the condition $\BE_2(K,\infty)$ can be given a meaning --- which also works  perfectly in the framework of \emph{Dirichlet spaces} $(\mms,\Ch,\meas)$ --- even if $K$ is replaced by a function, a measure, or a distribution  \cite{erbar2020,sturm2020}. (For functions, an equivalent Lagrangian counterpart still exists  \cite{braun2021,sturm2015,sturm2020}.) In fact, there is important evidence in the measure-valued case once \emph{boundaries} come into play. First, albeit, say, compact  Riemannian manifolds with \emph{convex} boundary are still covered by the $\RCD$ theory \cite{han2020}, already the appearance of a small boundary concavity makes it impossible for the relative entropy to be $K$-convex for any $K\in\R$ \cite{sturm2020, wang2014}. Second, on a compact Riemannian manifold $\mms$  with not necessarily convex boundary $\partial\mms$, the signed Borel measure
\begin{align}\label{Eq:Kappa meas intro}
\kappa  := \mathcall{k}\,\vol + \mathcall{l}\,\surf
\end{align}
plays the natural role of a lower ``Ricci'' bound \cite{erbar2020,hsu2002,sturm2020}. Here $\vol$ and $\surf$ are the volume measure and surface measure of $\mms$ and $\partial\mms$, and $\mathcall{k}$ and $\mathcall{l}$ are the pointwise lowest eigenvalues of the Ricci tensor $\Ric$ and the second fundamental form $\mathbb{I}$.

Of particular interest in the outlined business of singular  Ricci bounds is the \emph{extended Kato class} $\Kato_{1-}(\mms)$ of signed measures on $\mms$, already for Riemannian manifolds without boundary \cite{braun2020a, carron2019, gueneysu2015, gueneysu2017, gueneysu2020, magniez2020,rose2019, rose2020} or their Ricci limits \cite{carron2021}. This is just the right class of measure-valued potentials for which the associated Feynman--Kac semigroup has good properties \cite{stollmann1996}. In a recent work, Erbar, Rigoni, Sturm and Tamanini \cite{erbar2020} introduced the notion of \emph{tamed spaces}, i.e.~Dirichlet spaces supporting distribution-valued synthetic lower Ricci bounds in terms of \emph{$\Ch$-quasi-local distributions} $\smash{\kappa\in\F_\qloc^{-1}}$. These include the extended Kato class of $\mms$. 

\subsubsection*{Objective} Inspired by and following \cite{gigli2018}, our goal is to construct a functional first and second order calculus over Dirichlet spaces that are tamed by signed extended Kato class measures. (There are various reasons for working with $\smash{\kappa\in \Kato_{1-}(\mms)}$ rather than with general $\smash{\kappa\in\F_\qloc^{-1}}$, which are summarized in an own paragraph below. Still, already in the former case, many arguments become technically more challenging compared to \cite{gigli2018}.) In turn, this will induce a first order calculus on vector-valued objects. A functional first order structure for Dirichlet spaces is, of course, well-known to exist \cite{baudoin2019, cipriani2003, eberle1999, hinz2013, hinz2015, ionescu2012}. Here, we introduce it by the approach through $\Ell^\infty$-modules  \cite{gigli2018} and show its compatibility with the previous works. On the other hand, besides \cite{gigli2018} higher order objects are only studied in one-dimensional cases \cite{baudoin2019,hinz2015} or under restrictive structural assumptions \cite{honda2014,liu2002}. In our general approach, the two most important  quantities will be
\begin{itemize}
\item the \emph{Hessian} operator on appropriate functions,  along with proving that sufficiently many of these do exist,  and
\item a measure-valued \emph{Ricci curvature}.
\end{itemize}
In addition, we concisely incorporate the tamed analogue of the finite-dimensional $\BE_2(K,N)$ condition \cite{erbar2020}, $K\in\R$ and $N\in [1,\infty)$,  following the $\RCD^*(K,N)$-treatise \cite{han2018} which is not essentially different from \cite{gigli2018}.  More details  are outlined below. Before, we summarize our motivations in extending \cite{gigli2018} to tamed spaces.

The first, evident, reason is the larger setting. Dirichlet spaces are more general than metric measure spaces: they cover e.g.~certain noncomplete spaces, extended metric measure spaces such as configuration spaces \cite{albeverio1998,erbar2015b}, etc. From many perspectives, they seem to be the correct framework in which elements of a vector calculus should be studied \cite{baudoin2019, hinz2013}. Also, the considered lower Ricci bounds, examples of which are due to \cite{braunrigoni2021,erbar2020, gueneysu2020}, may be highly irregular. In fact \cite{honda2018}, already for uniform lower bounds, the Bakry--\smash{Émery} setting is strictly larger than the $\RCD$ one if the Sobolev-to-Lipschitz property is dropped. See \autoref{Sub:Tam} below.

Second, we want to pursue a thorough discussion of how the appearance of a ``boundary'' --- or more precisely, an $\meas$-negligible, non-$\Ch$-polar set --- in $\mms$ affects the calculus objects that are introduced similarly as in \cite{gigli2018}. Besides the need of measure-valued Ricci bounds to describe curvature of $\meas$-singular sets as by \eqref{Eq:Kappa meas intro}, boundaries play an increasing role in recent research  \cite{brue2020, sturm2020}. This motivated us to make sense, in all generality and apart from extrinsic structures, of measure-valued boundary objects such as \emph{normal components} of vector fields or, reminiscent of \cite{han2019}, a \emph{second fundamental form}. In fact, our guiding example is the case of compact Riemannian manifolds with boundary which, unlike only partly in the $\RCD$ setting, is fully covered by tamed spaces. Returning to this setting from time to time also provides us with a negative insight on an open question in \cite{gigli2018}, namely whether ``$H = W$'', see e.g.~\autoref{Sub:Calc rules hess}.  (This does not conflict with the smooth ``$H=W$'' results \cite{schwarz1995} as our ``$H$-spaces'' are different from the smooth ones.)

Lastly, we hope that our toolbox becomes helpful in further investigations of tamed spaces. Possible directions  could include
\begin{itemize}
\item the study of covariant Schrödinger operators \cite{braun2020,gueneysu2017},
\item rigidity results for and properties of finite-dimensional tamed spaces \cite{brue2020,bruesemola2020},
\item the study of bounded variation functions under Kato conditions \cite{brue2019,  buffa2019,gueneysu2015},
\item super-Ricci flows \cite{kopfer2018,sturm2018b}, noting that the Kato condition, in contrast to $\Ell^p$-conditions, on the Ricci curvature along Kähler--Ricci flows is stable   \cite{tian2016},
\item a structure theory for Kato Ricci limit or tamed spaces \cite{carron2021,mondino2019}, or 
\item the proof of a Bismut--Elworthy--Li formula \cite{bismut1984, elworthy1994}.
\end{itemize}

\subsubsection*{First order calculus} Let $(\mms,\Ch,\meas)$ be a quasi-regular, strongly local Dirichlet space,  see \autoref{Sub:Dirichlet forms} for basics on these. To simplify the presentation in this introduction, we assume that $\Ch$ admits a carré du champ $\smash{\Gamma\colon\F_\rme\to\Ell^1(\mms)}$, i.e.
\begin{align*}
\Ch(f) = \big\Vert \Gamma(f)\big\Vert_{\Ell^1(\mms)}
\end{align*}
for every $f\in\F$, where $\F$ (or $\F_\rme$) is the (extended) domain of $\Ch$. This will be our setting  from \autoref{Ch:Tangent module} on, see \autoref{As:Gamma-operator}. However, the space of $1$-forms in \autoref{Sec:Cotangent module} can  even be constructed in a  ``universal'' sense, see \autoref{Th:Universal}, by relying on the more general concepts of \emph{energy measures} for $\Ch$ and \emph{$\Ch$-dominant measures} \cite{delloschiavo2020,hino2010}, cf.~\autoref{Sub:Energy} and \autoref{Sub:Carré}.

To speak about vector-valued objects, we employ the theory of \emph{$\Ell^p$-normed $\Ell^\infty$-modules}, $p\in[1,\infty]$, w.r.t.~a given measure --- here $\meas$ --- introduced in \cite{gigli2018}, see \autoref{Sub:Linfty modles}. This is a Banach space $\calM$ endowed with a group action by $\Ell^\infty(\mms)$ and a map $\vert\cdot\vert \colon \calM\to\Ell^p(\mms)$, the \emph{pointwise norm}, such that
\begin{align*}
\Vert\cdot\Vert_\calM = \big\Vert\vert\cdot\vert\big\Vert_{\Ell^p(\mms)}.
\end{align*}
In terms of $\vert\cdot\vert$, all relevant $\meas$-a.e.~properties of elements of $\calM$, e.g.~their $\meas$-a.e. vanishing outside some given Borel set $A\subset\mms$, can be rigorously made sense of. $\Ell^\infty(\mms)$ is chosen as acting group given that multiplying vector-valued objects by functions should preserve the initial object's $\meas$-integrability. Thus, to some extent $\Ell^\infty$-modules allow us to speak of generalized \emph{sections} without any vector bundle (which we will also not define). We believe that  this interpretation is more straight\-forward and better suited for analytic purposes than the \emph{fiber} one by measurable Hilbert fields from \cite{baudoin2019, cipriani2003, eberle1999, hinz2013, hinz2015, ionescu2012}  --- albeit the approaches are equivalent, see \autoref{Re:Link MHB} --- where such a bundle is actually constructed.

The space $\Ell^2(T^*\mms)$ of $\Ell^2$-$1$-forms w.r.t.~$\meas$, termed \emph{cotangent module}  \cite{gigli2018}, is explicitly constructed in \autoref{Sub:Construction}. By duality, the \emph{tangent module} $\Ell^2(T\mms)$ of $\Ell^2$-vector fields w.r.t.~$\meas$ is then defined in  \autoref{Sec:TMod}. All in all, the discussion from \autoref{Sec:Cotangent module} and \autoref{Ch:Tangent module} leads to the following.

\begin{theorem}\label{Th:Module structure intro} $\Ell^2(T^*\mms)$ and $\Ell^2(T\mms)$ are  $\Ell^2$-normed $\Ell^\infty$-modules with pointwise norms  both denoted by $\vert\cdot\vert$. They  come with  a linear  \emph{differential} $\rmd\colon \F_\rme \to\smash{\Ell^2(T^*\mms)}$ and a linear \emph{gradient} $\nabla \colon\F_\rme\to\smash{\Ell^2(T\mms)}$ such that for every $f\in\F_\rme$,
\begin{align*}
\vert\rmd f\vert = \vert\nabla f\vert = \Gamma(f)^{1/2}\quad\meas\text{-a.e.}
\end{align*}
\end{theorem}

Both $\rmd$ and $\nabla$ obey all expected locality and calculus rules, cf.~\autoref{Cor:Calculus rules d}. Moreover, polarization of $\vert\cdot\vert$ induces a \emph{pointwise scalar product} $\langle\cdot,\cdot\rangle$ on $\smash{\Ell^2(T^*\mms)^2}$ and $\smash{\Ell^2(T\mms)^2}$ which, by integration w.r.t.~$\meas$, turns the latter into Hilbert spaces, respectively, see \autoref{Th:Module structure}.

\subsubsection*{Measure-valued divergence} Recall the  \emph{Gauß--Green  formula}
\begin{align}\label{Eq:Green intro}
-\int_\mms \rmd h(X)\d\vol = \int_\mms h\div_\vol X\d\vol - \int_{\partial\mms} h\,\langle X,\sfn\rangle\d\surf,
\end{align}
valid for every compact Riemannian manifold $\mms$ with boundary $\partial\mms$, every $X\in\Gamma_\comp(T\mms)$ and every $h\in\Cont_\comp^\infty(\mms)$. Here, $\sfn$ is the outward-pointing unit normal vector field at $\partial\mms$. This motivates our first key differential object, the \emph{measure-valued divergence} of appropriate vector fields, which in turn  is suitable to define the \emph{normal component} of the latter. Indeed, the point of introducing these, and the essence of our ``boundary discussion'', is that the second fundamental form of $\partial \mms$ at gradients is --- loosely speaking, see \autoref{Sub:Riem mflds} and \autoref{Ex:Sec fund form smooth} --- given  by
\begin{align}\label{Eq:2nd fund form intro}
\mathbb{I}(\nabla f,\nabla f) = -\frac{1}{2}\big\langle\nabla \vert \nabla f\vert^2,\sfn\big\rangle
\end{align}
for every $f\in \Cont^\infty(\mms)$ for which $\nabla f\in\Gamma(T\mms)$ is \emph{tangential} to $\partial\mms$, i.e.
\begin{align}\label{Eq:Vanishing normal cpnts intro}
\langle \nabla f,\sfn\rangle = 0\quad\text{at }\partial\mms.
\end{align}

We thus propose the following definition leaned on  \cite{buffa2019}, see \autoref{Def:Measure-valued divergence}.

\begin{definition} We say that $X\in\Ell^2(T\mms)$ has a \emph{measure-valued divergence}, briefly $X\in\Dom(\DIV)$, if there exists a $\sigma$-finite signed Borel measure $\DIV X$ charging no $\Ch$-polar sets such that for sufficiently many  $h\in\F$,
\begin{align*}
-\int_M \rmd h(X) \d\meas = \int_\mms \widetilde{h}\d\!\DIV X.
\end{align*}
\end{definition}

In turn, keeping in mind \eqref{Eq:Green intro} and using Lebesgue's decomposition
\begin{align*}
\DIV X = \DIV_\ll X + \DIV_\perp X
\end{align*}
of $\DIV X$ w.r.t.~$\meas$, we define the normal component of $X\in\Dom(\DIV)$ by
\begin{align*}
\norm X := -\DIV_\perp X,
\end{align*}
see \autoref{Def:Normal component}. Calculus rules for $\DIV X$ and $\norm X$, $X\in\Dom(\DIV)$, are listed in  \autoref{Sec:Divergences}. In our generality, we do not know more about the support of $\norm X$ than its $\meas$-singularity.  Nevertheless, these notions are satisfactorily compatible with other recent \emph{extrinsic} approaches to Gauß--Green's formula and boundary components on (subsets of) $\RCD$ spaces  \cite{brue2019,buffa2019, sturm2020} as outlined in \autoref{Ch:Appendix}. 

The advantage of this measure point of view compared  to the $\Ell^2$-one from \cite{gigli2018}, see \autoref{Def:L2 div}, is its ability to ``see'' the normal component of  $X\in\Dom(\DIV)$ rather than the latter being  left out in the relevant integration by parts formulas and interpreted as zero. This distinction does mostly not matter: matching with the interpretation of the generator $\Delta$ of  $\Ch$ as \emph{Neumann Laplacian}, on tamed spaces, for many $g \in \F\cap\Ell^\infty(\mms)$ and $f \in\Dom(\Delta)$ --- e.g.~for $g,f\in\Test(\mms)$, cf.~\autoref{Le:Div g nabla f} and \eqref{Eq:Test fcts intro} below ---  the vector field $X := g\,\nabla f\in\Ell^2(T\mms)$ belongs to $\Dom(\DIV)$ with 
\begin{align}\label{Eq:Div ids intro}
\begin{split}
\DIV_\ll X  &=  \big[\rmd g(\nabla f) + g\,\Delta f\big]\,\meas,\\
\norm X &= 0.
\end{split}
\end{align}
In particular, reminiscent of \eqref{Eq:2nd fund form intro} and \eqref{Eq:Vanishing normal cpnts intro} one would  desire  to have a large class of vector fields with vanishing normal component to define a second fundamental form. (In fact, many relevant spaces will be defined in terms of such vector fields, hence all Laplace-type operators considered in this work, see \autoref{Def:Bochner Laplacian} and \autoref{Def:Hodge Lapl}, implicitly obey Neumann boundary conditions in certain senses.) By now, it is however not even clear if there exist (m)any $f\in\F$ with 
\begin{enumerate}[label=\alph*.]
\item\label{La:RPF1} $\vert\nabla f\vert^2\in\F$, not to say with
\item\label{La:RPF2} $\smash{\nabla \vert \nabla f\vert^2\in\Dom(\DIV)}$. 
\end{enumerate}
These issues  appear similarly when initially trying to define higher order differential operators, as briefly illustrated now along with addressing \ref{La:RPF1} and \ref{La:RPF2}

\subsubsection*{Second order calculus} The subsequent pointwise formulas hold on the interior $\mms^\circ$ of any Riemannian manifold $\mms$ with boundary, for every $f,g_1,g_2\in\Cont^\infty(\mms)$, every $X,X_1,X_2\in\Gamma(T\mms)$ and every $\omega\in\Gamma(T^*\mms)$ \cite{lee2018,petersen2006}:
\begin{align}\label{Eq:Formulas intro}
\begin{split}
2\Hess f(\nabla g_1,\nabla g_2) &= \big\langle\nabla\langle\nabla f,\nabla g_1\rangle,\nabla g_2\big\rangle + \big\langle\nabla\langle \nabla f,\nabla g_2\rangle,\nabla g_1\big\rangle\\
&\qquad\qquad - \big\langle \nabla\langle \nabla g_1,\nabla g_2\rangle,\nabla f\big\rangle,\\
\big\langle \nabla_{\nabla g_1}X,\nabla g_2\big\rangle &= \big\langle\nabla\langle X,\nabla g_1\rangle,\nabla g_2\big\rangle - \Hess g_2(X,\nabla g_1),\\
\rmd\omega(X_1,X_2) &= \rmd\big[\omega(X_2)\big](X_1) -\rmd\big[\omega(X_1)\big](X_2)\\
&\qquad\qquad - \omega(\nabla_{X_1}X_2-\nabla_{X_2}X_1).\textcolor{white}{\big\vert}
\end{split}
\end{align}
The first identity characterizes the Hessian $\Hess f$ of $f$, the second is a definition of the covariant derivative $\nabla X$ of $X$ in terms of that Hessian, and in turn, the exterior derivative $\rmd\omega$ of $\omega$ can be defined with the help of $\nabla$. (A similar formula is true for the exterior  differential acting on forms of any degree, see \autoref{Ex:Ext der smooth}.) Hence, we may and will axiomatize these three differential operators in the previous order. In the sequel, we only outline how we paraphrase the first identity in \eqref{Eq:Formulas intro} nonsmoothly. The operators $\nabla$ and $\rmd$ can then be defined by similar (integration by parts) procedures and, as for the Hessian, satisfy a great diversity of expected calculus rules, see \autoref{Sub:Calc rules hess},  \autoref{Ch:Covariant der} and \autoref{Ch:Ext derivative} for details.

Up to the small point of defining the two-fold tensor product $\Ell^2((T^*)^{\otimes 2}\mms)$ of $\Ell^2(T^*\mms)$, see \autoref{Sub:Tensor products}, and keeping in mind \eqref{Eq:Div ids intro}, the following, stated in \autoref{Def:Hess}, is naturally motivated by  \eqref{Eq:Formulas intro}.

\begin{definition}\label{Def:Intodef Hess} The space $\Dom(\Hess)$ consists of all $f\in\F$ such that there exists some tensor $\Hess f\in\Ell^2((T^*)^{\otimes 2}\mms)$ such that for every $g_1,g_2\in\Test(\mms)$,
\begin{align*}
&2\int_\mms h\Hess f(\nabla g_1,\nabla g_2)\d\meas\\
&\qquad\qquad = -\int_\mms \langle\nabla f,\nabla g_1\rangle\d\!\DIV_\ll(h\,\nabla g_2) - \int_\mms \langle\nabla f,\nabla g_2\rangle\d\!\DIV_\ll(h\,\nabla g_1)\\
&\qquad\qquad\qquad\qquad - \int_\mms h\,\big\langle \nabla f,\nabla\langle\nabla g_1,\nabla g_2\rangle\big\rangle\d\meas.
\end{align*}
\end{definition}

The advantage of  \autoref{Def:Intodef Hess} is that the r.h.s.~of the defining property only contains one derivative of $f$.  All terms make sense if, as stated, $g_1$ and $g_2$ are in 
\begin{align}\label{Eq:Test fcts intro}
\Test(\mms) := \big\lbrace f \in \Dom(\Delta)\cap\Ell^\infty(\mms) : \vert\nabla f\vert\in\Ell^\infty(\mms),\ \Delta f\in\F\big\rbrace,
\end{align}
cf.~\autoref{Sub:Test fcts} and \autoref{Sub:Dom Hess}.  $\Test(\mms)$ is dense in $\F$,  and is a cornerstone of our discussion, playing the role of \emph{smooth} functions. For instance,  $\vert\nabla f\vert^2\in\F$ for $f\in\Test(\mms)$  by \autoref{Pr:Bakry Emery measures}, which also  addresses \ref{La:RPF1}~above. (In fact, $\vert \nabla f\vert^2$ is in the domain of the \emph{measure-valued Schrödinger operator} $\smash{\DELTA^{2\kappa}}$,  \autoref{Def:Measure valued Schr}. For possible later extensions, $\kappa$ will mostly not be separated from the considered operators. Hence, \ref{La:RPF2}~will be answered quite late, but positively, in \autoref{Le:Fin tot var}.) The latter technical grounds have been laid in \cite{erbar2020} following \cite{savare2014}, are summarized in \autoref{Sub:Self-imp}, and are one key place where taming  by $\smash{\kappa\in\Kato_{1-}(\mms)}$ is needed.

In \autoref{Th:Hess}, we show that $\Dom(\Hess)$ is nonempty, in fact, dense in $\Ell^2(\mms)$.

\begin{theorem}\label{Th:Hess intro} Every $f\in\Test(\mms)$ belongs to $\Dom(\Hess)$ with
\begin{align*}
\int_\mms \big\vert\!\Hess f\big\vert_\HS^2\d\meas \leq \int_\mms (\Delta f)^2\d\meas  - \big\langle \kappa\,\big\vert\,\vert\nabla f\vert^2\big\rangle.
\end{align*}
\end{theorem}

Here $\smash{\vert\cdot\vert_\HS}$ is the pointwise Hilbert--Schmidt-type norm on $\Ell^2((T^*)^{\otimes 2}\mms)$ --- as well as the two-fold tensor product $\Ell^2(T^{\otimes 2}\mms)$ of $\Ell^2(T\mms)$ --- see \autoref{Sub:Tensor products}.

The key ingredient for the proof of \autoref{Th:Hess intro} is \autoref{Le:Extremely key lemma}. It results from a variant of  the famous \emph{self-improvement} technique \cite{bakry1985a}. Here we follow \cite{gigli2018}, see also \cite{erbar2020,  savare2014, sturm2018}. The idea is to replace $f\in\Test(\mms)$ in the taming condition
\begin{align*}
\DELTA^{2\kappa}\frac{\vert \nabla f\vert^2}{2}  - \big\langle\nabla f,\nabla\Delta f\big\rangle\,\meas\geq 0
\end{align*}
from \autoref{Pr:Bakry Emery measures} by a  polynomial in appropriate test functions. By  optimizing over the coefficients, \autoref{Th:Hess intro} follows by integrating the resulting inequality
\begin{align*}
\DELTA^{2\kappa}\frac{\vert\nabla f\vert^2}{2} - \big\langle\nabla f,\nabla\Delta f\big\rangle\,\meas \geq \big\vert\!\Hess f\big\vert_\HS^2\,\meas.
\end{align*}

\subsubsection*{Ricci curvature and second fundamental form} The second main result of our work is the  existence of the named measure-valued \emph{curvature tensors}. Both are defined by \emph{Bochner's identity}.  The latter requires some work to be made sense of at least for the large class $\Reg(T\mms)$ of \emph{regular vector fields}, i.e.~all linear combinations of elements of the form $X := g\,\nabla f\in\Ell^2(T\mms)$,  $g\in\Test(\mms)\cup\R\,\One_\mms$ and $f\in\Test(\mms)$, see \autoref{Sub:Test objects}. (It is generally larger than the one of \emph{test vector fields} $\Test(T\mms)$ considered in \cite{gigli2018}.) Such $X$, first, obey $\smash{\vert X\vert^2\in\Dom(\DELTA^{2\kappa})}$ by \autoref{Le:Pre.Bochner}, second,  have a covariant derivative $\nabla X\in\Ell^2(T^{\otimes 2}\mms)$ by \autoref{Th:Properties W12 TM}, and third, have a $1$-form counterpart $\smash{X^\flat\in\Ell^2(T^*\mms)}$ in the domain of the Hodge Laplacian $\smash{\Hodge}$ by \autoref{Le:Hodge test}. Therefore, for $X\in\Reg(T\mms)$ the  definition
\begin{align*}
\RIC^\kappa(X,X) := \DELTA^{2\kappa}\frac{\vert X\vert^2}{2} + \Hodge X^\flat(X)\,\meas - \big\vert\nabla X\big\vert_\HS^2\,\meas
\end{align*}
makes sense. In fact, a variant of which is  \autoref{Th:Ricci measure}, we have the following. 

\begin{theorem}\label{Th:RIC intro} The previous map $\RIC^\kappa$ extends continuously to the closure $\smash{H_\sharp^{1,2}(T\mms)}$ of $\Reg(T\mms)$ w.r.t.~an appropriate $H^{1,2}$-norm, see \autoref{Def:Hsharp}, with values in the space of Borel measures on $\mms$ with finite total variation charging no $\Ch$-polar sets.
\end{theorem}

The nonnegativity implicitly asserted therein comes precisely from the taming condition. Abusing terminology, the map $\smash{\RIC^\kappa}$ will be called \emph{$\kappa$-Ricci measure}. 

Finally, in \autoref{Sub:Curv tensors from RIC} we separate the measure $\kappa$ from $\RIC^\kappa$. To this aim, in \autoref{Le:Pre.Bochner} and \autoref{Le:Fin tot var} we  discover that $\nabla\vert X\vert^2\in\Dom(\DIV)$ together with the relation $\DIV \nabla \vert X\vert^2 = \DELTA^{2\kappa}\vert X\vert + 2\,\vert X\vert_\sim^2\,\kappa$ ---  for an $\Ch$-quasi-continuous $\meas$-version $\vert X\vert_\sim^2$ of $\vert X\vert^2$ --- for every $X\in\Reg(T\mms)$, linking the operator $\DIV$ to the $\kappa$-Ricci measure $\RIC^\kappa$ (recall  \ref{La:RPF2} above). Based on this observation we then set
\begin{align}\label{Eq:RIC formula intro}
\begin{split}
\RIC(X,X) &:= \RIC^\kappa(X,X) + \vert X\vert^2_\sim\,\kappa\\
&\textcolor{white}{:}= \DIV\nabla \frac{\vert X\vert^2}{2}  + \Hodge X^\flat(X)\,\meas - \big\vert \nabla X\big\vert_\HS^2\,\meas
\end{split}
\end{align}
for $X\in\Reg(T\mms)$, which in turn allows us to define, even for $\smash{X\in H_\sharp^{1,2}(T\mms)}$,
\begin{itemize}
\item the \emph{Ricci curvature} of $\mms$, see \autoref{Def:Ric}, by $\RIC_\ll(X,X)$, and
\item the \emph{second fundamental form} of $\mms$, see \autoref{Def:Second fund form}, by 
\begin{align*}
\II(X,X) := \RIC_\perp(X,X).
\end{align*}
\end{itemize}

These definitions have serious smooth evidence thanks to --- and in fact have been partly inspired by --- the work \cite{han2019}. Therein, it is shown that on any  Riemannian manifold $\mms$ with boundary, the map $\RIC$  similarly defined according to \cite{gigli2018}, hence to our work as well, satisfies
\begin{align*}
\RIC(\nabla f,\nabla f) = \Ric(\nabla f,\nabla f)\,\vol + \mathbb{I}(\nabla f,\nabla f)\,\surf
\end{align*}
for every $\smash{f\in\Cont_\comp^\infty(\mms)}$ subject to \eqref{Eq:Vanishing normal cpnts intro}. Moreover, on $\RCD(K,\infty)$ spaces, $K\in\R$, according to \cite{gigli2018} we particularly have $\II(X,X)\geq 0$ for every $\smash{X\in H_\sharp^{1,2}(T\mms)}$, which is a way of analytically stating \emph{convexity} of $\mms$, see \autoref{Re:Conv RCD c}.  In general, it should be noted that $\II$ may be also concentrated on interior singularities though, cf.~\autoref{Re:Also conc on interior sing}. A further novel, but natural suggestion of our treatise is to interpret the  Ricci curvature of tamed spaces as something $\meas$-absolutely continuous.

\subsubsection*{Other interesting results} Our treatise comes with further beautiful results that are worth mentioning here and hold in great generality. Examples are
\begin{itemize}
\item metric compatibility of the covariant derivative $\nabla$ w.r.t.~the ``Riemannian metric'' $\langle\cdot,\cdot\rangle$,  \autoref{Pr:Compatibility}, and
\item a nonsmooth analogue of the Hodge theorem,  \autoref{Th:Hodge thm}.
\end{itemize}
Moreover, we address various points that have not been treated in \cite{gigli2018}, but rather initiated in \cite{braun2020, han2018}, among others
\begin{itemize}
\item semigroup domination of the heat flow on vector fields  w.r.t.~the functional one, \autoref{Th:HSU Bochner}, as well as of the heat flow on $1$-forms w.r.t.~the Schrödinger semigroup with potential $\kappa$, \autoref{Th:HSU forms},
\item spectral bottom estimates for the Bochner Laplacian,  \autoref{Cor:Spectra cov}, and the Hodge Laplacian,  \autoref{Cor:Spec bds},
\item a vector version of the measure-valued $q$-Bochner inequality \cite{braun2021}, $q\in [1,2]$,  \autoref{Th:Vector Bochner}, and
\item boundedness of the ``local dimension'' of $\Ell^2(T\mms)$ by $\lfloor N\rfloor$, \autoref{Pr:Upper bound local dimension}, and the existence of an extension of $\RIC_\ll$ to all of $\Ell^2(T\mms)$ on tamed spaces with upper dimension bound $N\in [1,\infty)$,  \autoref{Pr:Ric extension}.
\end{itemize}

\subsubsection*{Comments on the extended Kato condition} Finally, we comment on the assumption $\kappa\in\smash{\Kato_{1-}(\mms)}$ and technical issues, compared to \cite{gigli2018}, which arise later. 

In \cite[Cor.~3.3.9, Cor.~3.6.4]{gigli2018}, the following ``integrated Bochner inequality'' for $\RCD(K,\infty)$ spaces, $K\in\R$, is derived for suitable $X\in \Ell^2(T\mms)$:
\begin{align}\label{Eq:RCD control intro}
\int_\mms \big\vert\nabla X\big\vert_\HS^2\d\meas \leq \int_\mms \vert\rmd X^\flat\vert^2 \d\meas + \int_\mms \vert\delta X^\flat\vert^2\d\meas - K\int_\mms \vert X\vert^2\d\meas.
\end{align}
Here $\delta$ is the codifferential operator. The interpretation of \eqref{Eq:RCD control intro} is that \emph{an appropriate first order norm on $1$-forms controls the first order topology on vector fields qualitatively and quantitatively}. Indeed, first, for gradient vector fields, by heat flow regularization \eqref{Eq:RCD control intro} implies that $\Dom(\Delta)\subset\Dom(\Hess)$, and \eqref{Eq:RCD control intro} is stable under this procedure. Second, \eqref{Eq:RCD control intro} is crucial in the $\RCD$ version of \autoref{Th:RIC intro} \cite[Thm.~3.6.7]{gigli2018}, for  extending \eqref{Eq:RIC formula intro} beyond $\Test(T\mms)$  requires continuous dependency of the covariant term w.r.t.~a contravariant norm. In both cases, the curvature term  is clearly continuous, even in $\F$ or $\smash{\Ell^2(T\mms)}$, respectively.

The latter is  wrong in our situation: already on a compact Riemannian manifold $\mms$ with boundary and $\mathcall{l} \neq 0$, the pairing $\smash{\big\langle \kappa\,\big\vert\,\vert X\vert^2\big\rangle}$ according to \eqref{Eq:Kappa meas intro} does not even make sense for general $X\in\Ell^2(T\mms)$. Hence, we will have to deal with two correlated problems: controlling our calculus by stronger continuity properties of $X\mapsto \big\langle \kappa\,\big\vert\,\vert X\vert^2\big\rangle$, but also vice versa. (The fact that certain first order norms on $1$-forms bound covariant ones on compact Riemannian manifolds with boundary, a classical result by Gaffney \cite{schwarz1995}, see \autoref{Re:Gaffney}, is already nontrivial.) 

The key property of $\kappa\in\Kato_{1-}(\mms)$ in this direction is that $\smash{f \mapsto \big\langle \kappa\,\big\vert\,f^2\big\rangle}$ is (well-defined and) $\Ch$-form bounded on $\F$ with form bound  smaller than $1$, see \autoref{Le:Form boundedness}. That is, there exist $\rho'\in [0,1)$ and $\alpha'\in\R$ such that for every $f\in\F$,
\begin{align}\label{Eq:Form boundedness intro}
\big\vert\big\langle\kappa\,\big\vert\,f^2\big\rangle\big\vert \leq \rho'\int_\mms \vert\nabla f\vert^2\d\meas + \alpha'\int_\mms f^2\d\meas.
\end{align}

Now, from \eqref{Eq:Form boundedness intro}, we first note that the pairing $\smash{\big\langle\kappa\,\big\vert\,\vert X\vert^2\big\rangle}$ is well-defined for all $X$ in a covariant first order space termed $\smash{H^{1,2}(T\mms)}$, see \autoref{Def:H12 vfs}, since for every $X\in H^{1,2}(T\mms)$ we have $\vert X\vert\in \F$ by \emph{Kato's inequality}
\begin{align}\label{Eq:Kato inequ intro}
\big\vert\nabla \vert X\vert\big\vert\leq \big\vert \nabla X\big\vert_\HS\quad\meas\text{-a.e.},
\end{align}
as proven in \autoref{Le:Kato inequality}. The latter is essentially a consequence of metric compatibility of $\nabla$, cf.~\autoref{Pr:Compatibility}, and Cauchy--Schwarz's inequality. In particular, combining \eqref{Eq:Form boundedness intro} with \eqref{Eq:Kato inequ intro} will imply that $\smash{X\mapsto \big\langle \kappa\,\big\vert\,\vert X\vert^2\big\rangle}$ is even continuous in $H^{1,2}(T\mms)$, see \autoref{Cor:Kappa cont}. For completeness, we also mention here that Kato's inequality is  useful at other places as well, e.g.~in proving the above mentioned semigroup domination results. On $\RCD(K,\infty)$ spaces, $K\in\R$ --- on which \eqref{Eq:Kato inequ intro} has been proven in \cite{debin2021} in order to find ``quasi-continuous representatives'' of vector fields --- this has been observed in \cite{braun2020}.

However, extending the inequality 
\begin{align}\label{Eq:Cov control intro}
\int_\mms \big\vert\nabla X\big\vert_\HS^2\d\meas \leq \int_\mms\vert\rmd X^\flat\vert^2\d\meas + \int_\mms \vert\delta X^\flat\vert^2\d\meas - \big\langle\kappa\,\big\vert \,\vert X\vert^2\big\rangle
\end{align}
similar to \eqref{Eq:RCD control intro}, see \autoref{Le:Inclusion}, from $X\in \Reg(T\mms)$ --- for which it is valid by many careful computations, see \autoref{Le:Extremely key lemma} and \autoref{Le:Pre.Bochner},  and the $\BE_1(\kappa,\infty)$ condition, see \autoref{Pr:Fin tot var} and \autoref{Cor:Kappa X^2 identity} --- continuously to more general $\smash{X\in H_\sharp^{1,2}(T\mms)}$ requires better control on the curvature term. Here is where the form bound $\rho' \in [0, 1)$ comes into play. Indeed, using \eqref{Eq:Form boundedness intro} and \eqref{Eq:Kato inequ intro}, 
\begin{align*}
-\big\langle\kappa\,\big\vert\,\vert X\vert^2\big\rangle  \leq \rho'\int_\mms \big\vert\nabla X\big\vert_\HS^2\d\meas + \alpha'\int_\mms \vert X\vert^2\d\meas,
\end{align*}
and this can be merged with \eqref{Eq:Cov control intro} to obtain the desired continuous control of the covariant by a contravariant first order norm. In fact, this kind of argumentation, without already having Kato's inequality at our disposal, will also be pursued in our proof that $\Dom(\Delta)\subset \Dom(\Hess)$, see \autoref{Cor:Dom(Delta) subset W22}.

In view of this key argument, we believe that the extended Kato framework is somewhat maximal possible for which a second order calculus, at least with the presented diversity of higher order differential operators, as below can be developed.

Lastly, it is worth to spend few words on a different technical issue. Namely, to continuously extend $\RIC^\kappa$ in \autoref{Th:RIC intro} w.r.t.~a meaningful target topology, we need to know in advance that $\smash{\DELTA^{2\kappa}\vert X\vert^2}$ has finite total variation for $X\in\Reg(T\mms)$. Even for gradient vector fields, this is not discussed in \cite{erbar2020}. On the other hand, the corresponding $\RCD$ space result \cite[Lem.~2.6]{savare2014} uses their stochastic completeness  \cite{ambrosio2014a}. In our work, the latter is neither assumed nor generally known to be a consequence of the condition $\kappa\in\Kato_{1-}(\mms)$. Compare with \autoref{Subsub:Intr compl}. In \autoref{Pr:Fin tot var}, we give an alternative, seemingly new proof of the above finiteness which relies instead on the $\BE_1(\kappa,\infty)$ condition.

\subsubsection*{Organization} \autoref{Pt:I} deals with the first order differential calculus over $\mms$. We first recapitulate basic notions about Dirichlet forms and $\Ell^\infty$-modules in \autoref{Sec:Prel Dir...}. Then appropriate analogs of differential $1$-forms and vector fields are introduced in \autoref{Sec:Cotangent module} and  \autoref{Ch:Tangent module}, respectively. The latter contains a thorough discussion on functional and measure-valued divergences in \autoref{Sec:Divergences}.

In \autoref{Pt:II}, we study elements of a second order calculus on tamed spaces, whose properties are recorded in \autoref{Sub:Tamed spaces}. We go on with giving a meaning to the Hessian in \autoref{Sec:Hessian}, to the covariant derivative in \autoref{Ch:Covariant der}, and exterior differential and Hodge theory in \autoref{Ch:Ext derivative}. \autoref{Sec:Curvature} contains the appropriate existence of a Ricci curvature and a second fundamental form.

\subsubsection*{Acknowledgments} The author warmly thanks Lorenzo Dello Schiavo, Chiara Rigoni, Christian Rose, Karl-Theodor Sturm and Luca Tamanini for many fruitful and enlightening discussions as well as helpful comments. 

Funding by the European Research Council through the ERC-AdG ``RicciBounds'' is gratefully acknowledged.

\section*{List of main vector spaces}\label{List of notations}

The page numbers refer to the places in the text where the respective objects are introduced first. 

\begingroup
\centering
\begin{longtable}{rl}
$\F$, $\Dom(\Ch)$ & domain of $\Ch$, p.~\pageref{Not:Form domain} \\
$\F_\rme$ & extended domain, p.~\pageref{Def:Extended domain} \\
$\Cont_0(\mms)$ & continuous functions vanishing at $\infty$, p.~\pageref{Not:C_0} \\
$\Dom(\Delta)$ & domain of the Laplacian, p.~\pageref{Sub:Neumann Laplacian} \\
$\Dom(\Delta^{q\kappa})$ & domain of the Schrödinger operator with\\
& potential $q\kappa$, p.~\pageref{Not:Schrödinger operator} \\
$\Dom(\DELTA)$ & domain of the measure-valued Laplacian,   p.~\pageref{Not:Meas val Lapl} \\
$\Dom(\DELTA^{q\kappa})$ & domain of the measure-valued Schrödinger\\
& operator with potential $q\kappa$, p.~\pageref{Not:Measure valued Schr I}, p.~\pageref{Not:Meas valued Schr Ii} \\
$\Dom(\div)$ & domain of the $\Ell^2$-divergence, p.~\pageref{Def:L2 div} \\
$\Dom(\DIV)$ & domain of the measure-valued divergence, p.~\pageref{Def:Measure-valued divergence} \\
$\Dom_\TV(\DIV)$ & domain of the measure-valued divergence with\\
& finite total variation, p.~\pageref{Eq:Dom TV} \\
$\Test(\mms)$ & test functions, p.~\pageref{Not:Test fcts} \\
$\Test_{\Ell^\infty}(\mms)$ & test functions with bounded Laplacian, p.~\pageref{Not:Test fcts Linfty} \\
$\calG$ & $\Ell^1$-functions with $\Ell^1$-differential, p.~\pageref{Def:W11} \\
$\calG_\reg$ & closure of $\calG \cap \Test(\mms)$ in $\calG$, p.~\pageref{Def:H11 functions} \\
$\Dom(\Hess)$ & domain of the Hessian, p.~\pageref{Def:Hess} \\
$\Dom_\reg(\Hess)$ & closure of $\Test(\mms)$ in $\Dom(\Hess)$, p.~\pageref{Def:H22} \\
$\Dom_{2,2,1}(\Hess)$ & $\F$-functions with $\Ell^1$-Hessian, p.~\pageref{Def:W221} \\
$\Ell^2(T^*\mms)$ & cotangent module, p.~\pageref{Def:Cotangent module} \\
$\Ell^2(T\mms)$ & tangent module, p.~\pageref{Def:Tangent module} \\
$\Ell^0(T^*\mms)$ & $\Ell^0$-module induced by $\Ell^2(T^*\mms)$, p.~\pageref{Not:L0's} \\
$\Ell^0(T\mms)$ & $\Ell^0$-module induced by $\Ell^2(T\mms)$, p.~\pageref{Not:L0's} \\
$\Ell^p(T^*\mms)$ & $\Ell^p$-space induced by $\Ell^0(T^*\mms)$, p.~\pageref{Not:Lp sp.} \\
$\Ell^p(T\mms)$ & $\Ell^p$-space induced by $\Ell^0(T\mms)$, p.~\pageref{Not:Lp sp.} \\
$\Ell^2((T^*)^{\otimes 2}\mms)$ & two-fold tensor product of $\Ell^2(T^*\mms)$, p.~\pageref{Not:L2 tensor pr} \\
$\Ell^2(T^{\otimes 2}\mms)$ & two-fold tensor product of $\Ell^2(T\mms)$, p.~\pageref{Not:L2 tensor pr} \\
$\Ell^2(\Lambda^kT^*\mms)$ & $k$-fold exterior product of $\Ell^2(T^*\mms)$, p.~\pageref{Not:Lp sp ext p} \\
$\Ell^2(T_r^s\mms)$ &  $\Ell^2$-tensor fields of type $(r,s)$, p.~\pageref{Not:L2 tfs rs} \\
$\Reg(T^*\mms)$ & regular $1$-forms, p.~\pageref{Not:reg 1 forms} \\ 
$\Reg(T\mms)$ & regular vector fields, p.~\pageref{Not:reg 1 vfs} \\
$\Test(T^*\mms)$  & test $1$-forms, p.~\pageref{Not:reg 1 forms} \\
$\Test(T\mms)$ & test vector fields, p.~\pageref{Not:reg 1 vfs} \\
$\Reg(\Lambda^kT^*\mms)$ & $k$-fold exterior product of $\Reg(T^*\mms)$, p.~\pageref{Not:reg lambda k} \\
$\Reg(T^{\otimes 2}\mms)$ & two-fold tensor product of $\Reg(T\mms)$, p.~\pageref{Not:Two fold tps} \\
$\Reg(T_s^r\mms)$ & regular $(r,s)$-tensor fields, p.~\pageref{Not:Reg tfields} \\
$\Test(\Lambda^kT^*\mms)$ & $k$-fold exterior product of $\Test(T^*\mms)$, p.~\pageref{Not:test lambda k} \\
$\Test(T^{\otimes 2}\mms)$ & two-fold tensor product of $\Test(T\mms)$, p.~\pageref{Not:Two fold tps} \\
$W^{1,2}(T\mms)$ & $\Ell^2$-vector fields with $\Ell^2$-covariant derivative,  p.~\pageref{Def:Cov der} \\
$H^{1,2}(T\mms)$ & closure of $\Reg(T\mms)$ in $W^{1,2}(T\mms)$, p.~\pageref{Def:H12 vfs} \\
$W^{(2,1)}(T\mms)$ & $\Ell^2$-vector fields with $\Ell^1$-covariant derivative, p.~\pageref{Def:W21} \\
$\Dom(\Bochner)$ & domain of the Bochner Laplacian, p.~\pageref{Def:Bochner Laplacian} \\
$W^{1,2}(T_r^s\mms)$ & $\Ell^2$-tensor fields of type $(r,s)$ with $\Ell^2$-covariant\\
&  derivative, p.~\pageref{Def:Tensor cov der} \\
$\Dom(\rmd^k)$ & domain of the differential in $\Ell^2(\Lambda^kT^*\mms)$, p.~\pageref{Def:Exterior derivative} \\
$\Dom(\delta^k)$ & domain of the codifferential in $\Ell^2(\Lambda^kT^*\mms)$, p.~\pageref{Def:Codifferential} \\
$W^{1,2}(\Lambda^kT^*\mms)$ & intersection of $\Dom(\rmd^k)$ and $\Dom(\delta^k)$, p.~\pageref{Def:W12 forms space} \\
$H^{1,2}(\Lambda^kT^*\mms)$ & closure of $\Reg(\Lambda^kT^*\mms)$ in $W^{1,2}(\Lambda^kT^*\mms)$,  p.~\pageref{Def:H Hodge space} \\
$\Dom(\Hodge)$ & domain of the Hodge Laplacian, p.~\pageref{Def:Hodge Lapl} \\
$\Harm(\Lambda^kT^*\mms)$ & harmonic $k$-forms, p.~\pageref{Def:Hodge Lapl} \\
$\smash{H_\sharp^{1,2}(T\mms)}$ & image of $H^{1,2}(T^*\mms)$ under $\sharp$, p.~\pageref{Def:Hsharp}
\end{longtable}
\endgroup

\part{First order differential structure}\label{Pt:I}

\section{Preliminaries. Dirichlet spaces and module theory}\label{Sec:Prel Dir...}

\subsection{Notations}\label{Sub:Notation} We start with basic  terminologies used all over the paper.

For the remainder of this article, we make the following topological assumption. Compare with \autoref{Re:Lusin}.

\begin{assumption} $(\mms,\tau)$ is a topological Lusin space, i.e.~a continuous injective image of a Polish space, endowed with a $\sigma$-finite Borel measure $\meas$ on $\mms$,   according to \autoref{Subsub:Measure spaces} below, with full topological support.
\end{assumption}

The given topology $\tau$ is considered as understood and does, up to few exceptions, usually not appear in our subsequent notation.

\subsubsection{Measures}\label{Subsub:Measure spaces} All very elementary measure-theoretic  terminologies are agreed upon \cite{bogachev2007a, bogachev2007b, halmos1950}. More specific points are shortly addressed now.

The Borel $\sigma$-algebra induced by $\tau$ is denoted by $\Borel(\mms)$, while its Carathéodory completion w.r.t.~a Borel measure $\mu$ on $\mms$ is denoted by $\Borel^\mu(\mms)$. (If not explicitly stated otherwise, we identify certain subsets of $\mms$ with their equivalence classes in $\Borel^\mu(\mms)$.) The \emph{support} of every Borel measure $\mu$ on $\mms$ is defined \cite[Sec.~V.1]{ma1992} and denoted by $\supp\mu$. By $\smash{\Meas_\fin^+(\mms)}$, $\smash{\Meas^+_\sigmafin(\mms)}$, $\smash{\Meas^{\pm}_\fin(\mms)}$ and $\smash{\Meas_\sigmafin^{\pm}(\mms)}$, we intend the spaces of Borel measures on $\mms$ which are finite, $\sigma$-finite, signed and finite, as well as signed and $\sigma$-finite, respectively. Here, $\sigma$-finiteness of $\smash{\mu\in\Meas_\sigma^\pm(\mms)}$ refers to the existence of an increasing sequence of \emph{open} subsets of $\mms$ on whose elements $\mu$ is finite. The subscripts $\rmR$, such as in $\smash{\Meas_\finR^+(\mms)}$, or $\Ch$, such as in $\smash{\Meas_\fin^+(\mms)_\Ch}$, indicate the respective subclass of (signed) measures which are Radon or do not charge $\Ch$-polar sets according to \autoref{Sub:Quasi-notions}. 

Given any $\mu\in \Meas_\sigmafin^\pm(\mms)$, there exist unique $\mu^+,\mu^-\in \Meas_\sigmafin^+(\mms)$ such that
\begin{align*}
\mu = \mu^+ - \mu^-.
\end{align*}
This is the \emph{Jordan decomposition} of $\mu$ into its positive part $\smash{\mu^+}$ and its negative part $\smash{\mu^-}$. At least one of the measures $\mu^+$ or $\mu^-$ is finite --- hence $\mu^+-\mu^-$ is well-defined ---  while they are both finite if $\smash{\mu\in \Meas_\fin^\pm(\mms)}$ \cite[Thm.~29.B]{halmos1950}. 

The \emph{total variation} $\vert\mu\vert\in \smash{\Meas_\sigmafin^+(\mms)}$ of $\mu\in \smash{\Meas_\sigmafin^\pm(\mms)}$ is defined by
\begin{align*}
	\vert \mu\vert := \mu^+ + \mu^-.
\end{align*}
This gives rise to the \emph{total variation norm}
\begin{align*}
\Vert\mu\Vert_\TV := \vert\mu\vert[\mms]
\end{align*}
of $\mu$. Note that $(\Meas_\fin^\pm(\mms),\Vert \cdot\Vert_\TV)$ is indeed a normed vector space. 

In this notation, given a not necessarily nonnegative $\nu \in\Meas_\sigmafin^\pm(\mms)$, we  write $\nu\ll\mu$ if $\nu \ll \vert \mu\vert$, or equivalently $\vert\nu \vert \ll \vert\mu\vert$ \cite[Thm.~30.A]{halmos1950}, for \emph{absolute continuity}. Note that if $\mu,\nu\in \Meas_\fin^+(\mms)$ are \emph{singular} to each other, written $\mu\perp\nu$, then
\begin{align}\label{Eq:TV under singularity}
\vert \mu + \nu\vert = \vert\mu \vert + \vert\nu\vert.
\end{align}

Lastly, recall that given any $\smash{\mu,\nu\in \Meas_\sigma^{\pm}(\mms)}$, $\mu$ admits a \emph{Lebesgue decomposition} w.r.t.~$\nu$ \cite[Thm.~32.C]{halmos1950} --- that is, there exist unique $\smash{\mu_\ll,\mu_\perp\in\Meas_\sigma^\pm(\mms)}$ whose sum is well-defined and with the property that 
\begin{align*}
\mu_\ll &\ll \nu,\\
\mu_\perp &\perp\nu,\\
\mu &= \mu_\ll + \mu_\perp.
\end{align*}

\subsubsection{Functions}\label{Subsub:Fcts} Let $\Ell_0(\mms)$ and $\Ell_\infty(\mms)$ be the spaces of real-valued and bounded real-valued $\Borel(\mms)$-measurable functions defined everywhere on $\mms$. Write $\SF(\mms)$ for the space of \emph{simple functions}, i.e.~of those $f\in\Ell_\infty(\mms)$ with finite range.

Let $\mu$ be a Borel measure on $\mms$. Let $\smash{\Ell^0(\mms,\mu)}$ be the real vector space of equivalence classes of elements in $\smash{\Ell_0(\mms)}$ w.r.t.~$\mu$-a.e.~equality. We mostly make neither notational nor descriptional distinction between ($\mu$-a.e.~properties of) functions $f\in\Ell_0(\mms,\mu)$ and (properties of) its equivalence class $[f]_\mu \in\Ell^0(\mms,\mu)$. In the only case where this difference matters, see \autoref{Sub:Quasi-notions} below, we use the distinguished notations $\smash{\widetilde{f}}$ or $f_\sim$.

The \emph{support} of $f\in\Ell^0(\mms,\mu)$ is defined as
\begin{align*}
\supp_\mu f := \supp(\vert f\vert\,\mu)
\end{align*}
and we briefly write $\supp$ for $\supp_\meas$. Let $\smash{\Ell^0_\comp(\mms,\mu)}$ be the class of $f\in \Ell^0(\mms,\mu)$ such that $\supp_\mu f$ is compact. We write $\Ell^0(\mms)$ for $\Ell^0(\mms,\meas)$.

Given $p\in [1,\infty]$ we denote the (local) $p$-th order Lebesgue spaces w.r.t.~$\mu$ by $\Ell^p(\mms,\mu)$, with the usual norm $\smash{\Vert\cdot \Vert_{\Ell^p(\mms,\mu)}}$, and $\smash{\Ell^p_\loc(\mms,\mu)}$. We always abbreviate $\smash{\Ell^p(\mms,\meas)}$ and $\smash{\Ell^p_\loc(\mms,\meas)}$ by $\smash{\Ell^p(\mms)}$ and $\smash{\Ell^p_\loc(\mms)}$, respectively. 

If $(\mms,\met)$ is a metric space --- in which case we always assume that $\tau$ coincides with the topology induced by $\met$ --- we write $\Lip(\mms)$ for the space of real-valued Lipschitz functions w.r.t.~$\met$. $\Lip_\bs(\mms)$ is the class of boundedly supported elements in $\Lip(\mms)$,  i.e.~which vanish identically outside a ball in $\mms$.

Lastly, the $\mu$\emph{-essential supremum} of a family $(f_i)_{i\in I}$ in $\Ell^0(\mms,\mu)$ with arbitrary, not necessarily countable index set $I$ is the minimal function $f\in\Ell^0(\mms,\mu)$ such that $f\geq f_i$ $\mu$-a.e.~for every $i\in I$. It exists and is unique. We write 
\begin{align*}
\mu\text{-}\!\esssup\{ f_i : i\in I\} &:= f,\\
\mu\text{-}\!\essinf\{f_i : i \in I\} &:= - \mu\text{-}\!\esssup\{- f_i : i\in I\}.
\end{align*}
while the prefixes ``$\mu$-'' are dropped from these notations  if $\mu = \meas$.

\subsection{Riemannian manifolds with boundary}\label{Sub:Riem mflds} One family of guiding examples for our constructions pursued from \autoref{Sec:Cotangent module} on are Riemannian manifolds, possibly noncompact and possibly with boundary. Here we collect basic terminologies on these. See \cite{lee1997, lee2018, petersen2006, schwarz1995} for further reading.

\subsubsection{Setting}\label{Sub:Setting} Unless explicitly stated otherwise, any Riemannian manifold $\mms$ is understood to have topological dimension $d\geq 2$ and to be smooth, i.e.~to be locally homeomorphic to $\smash{\R^d}$ or $\smash{\R^{d-1}\times [0,\infty)}$ depending on whether $\partial\mms = \emptyset$ or $\partial\mms \neq \emptyset$ and with smooth transition functions. Recall that a function $f\colon \mms\to\R$ is smooth \cite[Ch.~1]{lee2018} if $\smash{f\circ\sfx^{-1}\colon \sfx(U) \to \R}$ is smooth in the ordinary Euclidean sense for every  chart $(U,\sfx)$ on $\mms$ --- if $(U,\sfx)$ is a boundary chart, i.e.~$U\cap\partial\mms \neq \emptyset$, this means that $\smash{f\circ\sfx^{-1}}$ has a smooth extension to an open subset of $\smash{\R^d}$. For simplicity, any Riemannian manifold is assumed to be connected and (metrically) complete. Set
\begin{align*}
\mms^\circ := \mms \setminus \partial\mms.
\end{align*}
Denote by $\vol \in \smash{\Meas_{\sigmaR}^+(\mms)}$ the Riemannian volume measure induced by the metric tensor $\langle\cdot,\cdot\rangle$. Let $\smash{\surf\in\Meas_\sigmaR^+(\partial\mms)}$ be the surface measure on $\partial\mms$. 

If $\partial\mms \neq \emptyset$, then $\partial\mms$ is a smooth codimension $1$ submanifold of $\mms$. It  naturally becomes Riemannian when endowed with the pullback metric
\begin{align*}
\langle\cdot,\cdot\rangle_\jmath := \jmath^*\langle\cdot,\cdot\rangle
\end{align*}
under the natural inclusion $\jmath\colon \partial\mms\to\mms$. The map $\jmath$ induces a natural inclusion $\smash{\rmd\jmath\colon T\partial\mms\to T\mms\big\vert_{\partial\mms}}$ which is \emph{not} surjective. In particular, the vector bundles $T\partial\mms$ and $\smash{T\mms\big\vert_{\partial \mms}}$ do \emph{not} coincide. Rather, $T\partial\mms$ is identifiable with the codimension $1$ subbundle $\rmd\jmath(T\partial\mms)$ of $\smash{T\mms\big\vert_{\partial\mms}}$.

\subsubsection{Sobolev spaces on vector bundles} Denote the space of smooth sections of a real vector bundle $\boldsymbol{\mathrm{F}}$ over $\mms$ (or $\partial\mms$) by $\Gamma(\boldsymbol{\mathrm{F}})$. (This is a slight abuse of notation since we also denote carré du champs associated to Dirichlet energies by $\Gamma$, see \autoref{Sub:Carré}. However, it will always be clear from the context which meaning is intended.) With a connection $\nabla$ on $\boldsymbol{\mathrm{F}}$ --- always chosen to be the Levi-Civita one if $\boldsymbol{\mathrm{F}} := T\mms$ --- one can define Sobolev spaces $W^{k,p}(\boldsymbol{\mathrm{F}})$, $k\in\N_0$ and $p\in [1,\infty)$, in various ways, e.g.~by completing $\Gamma_\comp(\boldsymbol{\mathrm{F}})$ w.r.t.~an appropriate norm w.r.t.~$\vol$ (or $\surf$), or in a weak sense \cite[Sec.~1.3]{schwarz1995}. The ``natural'' approaches all coincide if $\mms$ is compact \cite[Thm.~1.3.6]{schwarz1995}, but for noncompact $\mms$, without further geometrical restrictions ambiguities may occur \cite{eichhorn1988}. In our work, it is always either  clear from the context which definition is intended, or precise meanings are simply irrelevant when only the compact setting is considered.

Throughout, it is useful to keep in mind the following \emph{trace theorem} for compact $\mms$ \cite[Thm.~1.3.7]{schwarz1995}. If $\mms$ is noncompact, it only holds true locally \cite[p.~39]{schwarz1995}. 

\begin{proposition}\label{Pr:Trace thm} For every $k\in\N_0$ and every $p\in [1,\infty)$, the natural restriction map $\smash{\cdot\big\vert_{\partial\mms}}\colon \Gamma(\boldsymbol{\mathrm{F}})\to\smash{\Gamma(\boldsymbol{\mathrm{F}}\big\vert_{\partial\mms})}$ in the spirit of \autoref{Sub:Nml cpnts} below extends to a continuous --- in fact, compact --- map from $W^{k+1,p}(\boldsymbol{\mathrm{F}})$ to $\smash{W^{k,p}(\boldsymbol{\mathrm{F}}\big\vert_{\partial\mms})}$.
\end{proposition}

\subsubsection{Normal and tangential components}\label{Sub:Nml cpnts} Denote by $\smash{\sfn\in\Gamma(T\mms\big\vert_{\partial\mms})}$ the outward pointing unit normal vector field at $\partial\mms$. It can be smoothly extended to a (non-relabeled) vector field $\sfn$ on an open neighborhood of $\partial\mms$ \cite[Thm.~1.1.7]{schwarz1995}.

The restriction $\smash{X\big\vert_{\partial\mms}}$ of a given $X\in\Gamma(T\mms)$ to $\partial\mms$ decomposes into a \emph{normal part} $\smash{X^\perp\in\Gamma(T\mms\big\vert_{\partial\mms})}$ and a \emph{tangential part} $\smash{X^{\Vert}\in\Gamma(T\mms\big\vert_{\partial\mms})}$ defined by
\begin{align}\label{Eq:Normal parts vfs smooth world}
\begin{split}
X^\perp &:= \langle X,\sfn\rangle\,\sfn,\\
X^\Vert &:= X\big\vert_{\partial\mms} - X^\perp.
\end{split}
\end{align}
Unlike the end of \autoref{Sub:Setting}, under a slight abuse of  notation, in a unique way every $\smash{X^\Vert \in \Gamma(T\partial\mms)}$ can be identified with a tangential element $\smash{X\in \Gamma(T\mms\big\vert_{\partial\mms})}$, i.e.~$\smash{X^\perp=0}$ \cite[p.~16]{schwarz1995}. In turn, such an $X$ can be smoothly extended to an open neighborhood of $\partial\mms$ \cite[Lem.~8.6]{lee2018}. That is, knowing the restrictions to $\partial\mms$ of all $X\in\Gamma(T\mms)$ with purely tangential boundary components suffices to recover the entire intrinsic covariant structure of $\partial\mms$.

Similarly, the tangential part $\smash{\rmt\,\omega\in \Gamma(\Lambda^kT^*\mms\big\vert_{\partial\mms})}$ and the normal part $\rmn\,\omega\in\smash{\Gamma(\Lambda^kT^*\mms\big\vert_{\partial\mms})}$, $k\in\N$, of a $k$-form $\omega\in\Gamma(\Lambda^kT^*\mms)$ at $\partial\mms$ are defined by
\begin{align}\label{Eq:Normal parts forms smooth world}
\begin{split}
\rmt\,\omega(X_1,\dots,X_k) &:= \omega(X_1^\Vert,\dots,X_k^\Vert),\\
\rmn\,\omega &:= \omega\big\vert_{\partial\mms} - \rmt\,\omega
\end{split}
\end{align}
for every $\smash{X_1,\dots,X_k\in\Gamma(T\mms\big\vert_{\partial\mms})}$. Taking tangential and normal parts of differential forms is dual to each other through the Hodge $\star$-operator \cite[Prop.~1.2.6]{schwarz1995}.

\subsection{Dirichlet forms}\label{Sub:Dirichlet forms} In this section, we summarize various important notions of Dirichlet spaces. This survey is enclosed by two examples in \autoref{Sub:Guiding} we frequently use for illustrative reasons in the sequel. 

For further, more detailed accounts, we refer to the books \cite{bouleau1991, chen2012, fukushima2011, ma1992}. 

\subsubsection{Basic definitions}\label{Sub:Forms Sobolev} We always fix  a symmetric, quasi-regular \cite[Def.~III.3.1]{ma1992} and strongly local Dirichlet form $(\Ch,\F)$ with linear domain $\F := \Dom(\Ch)$\label{Not:Form domain} which is dense in $\Ell^2(\mms)$. (In our work, symmetry of $\Ch$ may and will be  assumed without restriction, see \autoref{Re:Non-symm}.)  Here, strong locality \cite[Def.~1.3.17]{chen2012} means that for every $f,g\in\F\cap \Ell^0_\comp(\mms)$ such that $f$ is constant on a neighborhood of $\supp g$,
\begin{align*}
\Ch(f,g)=0.
\end{align*}

\begin{remark}\label{Re:Lusin} As every topological space which carries a quasi-regular Dirichlet form $(\Ch,\F)$ is, up to removing an $\Ch$-polar set, a Lusin space \cite[Rem.~IV.3.2]{ma1992}, our assumption on $\mms$ stated at the beginning of  \autoref{Sub:Notation} is in fact not restrictive.
\end{remark}

\begin{remark}\label{Re:Jump} Strong locality is not strictly necessary to run the construction of the cotangent module in \autoref{Sub:Construction}. One could more generally assume $\Ch$ to have trivial killing part in its Beurling--Deny decomposition, see \cite[Thm.~4.3.4]{chen2012}, \cite[Thm.~3.2.1]{fukushima2011} and \cite[Thm.~5.1]{kuwae1998}. However, if $\Ch$ has nontrivial jump part --- which precisely distinguishes it from being strongly local \cite[Prop.~4.3.1]{chen2012} --- the important calculus rules from \autoref{Th:Properties energy measure} below typically fail.
\end{remark}

The triple $(\mms,\Ch,\meas)$ is called \emph{Dirichlet space}. If we say that a property holds for $\Ch$, we usually mean that it is satisfied by the pair $(\Ch,\F)$. We abbreviate
\begin{align*}
\F_\bounded &:= \F\cap \Ell^\infty(\mms),\\
\F_\comp &:= \F\cap \Ell^0_\rmc(\mms),\\
\F_\bc &:= \F_\bounded \cap \F_\comp.\textcolor{white}{L^0_\rmc}
\end{align*}
By definition, 
$\Ch$ is \emph{closed} \cite[Def.~I.2.3]{ma1992}, i.e.~$(\F,\Vert\cdot\Vert_\F)$ is complete, where 
\begin{align*}
\big\Vert f\big\Vert_{\F}^2 := \big\Vert f\big\Vert_{\Ell^2(\mms)}^2 + \Ch(f),
\end{align*}
and \emph{Markovian} \cite[Def.~I.4.5]{ma1992}, i.e.~$\smash{\min\{f^+,1\}\in \F}$ as well as
\begin{align*}
\Ch\big(\!\min\{f^+,1\}\big) \leq \Ch(f)
\end{align*}
for every $f\in\F$. Here, we abbreviate
\begin{align*}
	\Ch(f) := \Ch(f,f)
\end{align*}
and we do so analogously for the diagonal values of any other bilinear form  encountered in the sequel without further notice. 

A densely defined, quadratic form on $\Ell^2(\mms)$ --- for notational convenience, we concentrate on $\Ch$ --- is called \emph{closable} if for every $\Ch$-Cauchy sequence $(f_n)_{n\in\N}$ in $\Dom(\Ch)$ with $\Vert f_n\Vert_{\Ell^2(\mms)} \to 0$ as $n\to\infty$, we have $\Ch(f_n)\to 0$ as $n\to\infty$. $\Ch$ is closable if and only if it has a closed extension \cite[p.~4]{fukushima2011}. Here, we term a sequence $(f_n)_{n\in\N}$ in $\F$ \emph{$\Ch$-Cauchy} if $\Ch(f_n-f_m) \to 0$ as $n,m\to\infty$, and \emph{$\Ch$-bounded} if
\begin{align*}
\sup_{n\in\N}\Ch(f_n)< \infty.
\end{align*}

\subsubsection{Basic properties}\label{Sub:Quasi-notions} For all relevant quasi-notions evolving around the definition of quasi-regularity, we refer to \cite[Ch.~1]{chen2012} or  \cite[Ch.~III]{ma1992}. Here, we solely state the following useful properties \cite[Prop.~IV.3.3]{ma1992} frequently used in our work. (Here and in the sequel, ``$\Ch$-q.e.'' and ``$\Ch$-q.c.''~abbreviate \emph{$\Ch$-quasi-everywhere} \cite[Def.~III.2.1]{ma1992} and \emph{$\Ch$-quasi-continuous}  \cite[Def.~III.3.2]{ma1992}, respectively.)

\begin{proposition}\label{Le:Props q.r. s.l. Dirichlet form} For every quasi-regular Dirichlet form $\Ch$, the following  hold.
\begin{enumerate}[label=\textnormal{\textcolor{black}{(}\roman*\textcolor{black}{)}}]
\item $\F$ is a separable Hilbert space w.r.t.~$\smash{\Vert\cdot\Vert_{\F}}$.
\item Every $f\in \F$ has an $\Ch$-q.c.~$\meas$-version  $\smash{\widetilde{f}}$.
\item If a function $f$ is $\Ch$-q.c.~and is nonnegative $\meas$-a.e.~on an open set $U\subset\mms$, then $f$ is nonnegative $\Ch$-q.e.~on $U$. In particular, an $\Ch$-q.c.~$\meas$-representative $\smash{\widetilde{f}}$ of any $f\in \F$ is $\Ch$-q.e.~unique.
\item If $f\in \F \cap\Ell^\infty(\mms)$, then any $\Ch$-q.c.~$\meas$-version $\smash{\widetilde{f}}$ of $f$ obeys 
\begin{align*}
	\vert \widetilde{f}\vert\leq \Vert f\Vert_{\Ell^\infty(\mms)}\quad\Ch\text{-q.e.}
\end{align*}
\end{enumerate}
\end{proposition}

\subsubsection{Extended domain}\label{Sub:Extended domain} We now define the main space  around which  \autoref{Sec:Cotangent module} is built. In view of \autoref{Def:Differential}, the  point is our goal to speak about $\Ell^2$-differentials $\rmd f$ of appropriate $f\in\Ell^0(\mms)$ without imposing any  integrability assumption on $f$.

The following definition can be found in \cite[p.~690]{kuwae1998}.

\begin{definition}\label{Def:Extended domain}
The \emph{extended domain} $\F_\rme$ of $\Ch$ is defined to consist of all $f\in \Ell^0(\mms)$ for which there exists an $\Ch$-Cauchy sequence $(f_n)_{n\in\N}$ of elements of $\F$ such that $f_n \to f$ pointwise $\meas$-a.e.~as $n\to\infty$.
\end{definition}

\begin{example} In the metric measure terminology of \cite{ambrosio2014a,gigli2018}, up to a possible ``non-Riemannian structure'' of $\mms$, see \autoref{Re:Non-symm} below, $\F_\rme$ is contained in the \emph{Sobolev class} $\smash{\Sobo^2(\mms)}$ \cite[Def.~2.1.4]{gigli2018}, which in turn is contained in the class of functions in $\dot{\F}_\loc$ (cf.~\autoref{Re:FLOC DEF:}) with integrable carré du champ, see \cite[Thm.~6.2]{ambrosio2014a}. These inclusions, which may be strict in general, cause no ambiguity of our approach compared to \cite{gigli2018}, cf.~\autoref{Re:Gigli comp}.
\end{example}

We say that a sequence $(f_n)_{n\in\N}$ converges to $f$ in $\F_\rme$ if $f_n\in \F$ for every $n\in\N$, $(f_n)_{n\in\N}$ is $\Ch$-Cauchy, and $f_n \to f$ $\meas$-a.e.~as $n\to\infty$. The following \autoref{Pr:Extended domain props} is provided thanks to  \cite[Prop.~3.1, Prop.~3.2]{kuwae1998}. Let us also set
\begin{align*}
\F_\eb &:= \F_\rme\cap\Ell^\infty(\mms),\\
\F_\ec &:= \F \cap \Ell^0_\comp(\mms),\\
\F_\ebc &:= \F_\eb\cap\F_\ec.\textcolor{white}{L_\rmc^0}
\end{align*}

\begin{proposition}\label{Pr:Extended domain props} The extended domain $\F_\rme$ has the following properties.
\begin{enumerate}[label=\textnormal{\textcolor{black}{(}\roman*\textcolor{black}{)}}]
\item  $\Ch$ uniquely extends to a \textnormal{(}non-relabeled\textnormal{)} real-valued  bilinear form on $\smash{\F_\rme^2}$ in such a way that for every $f\in \smash{\F_\rme}$,
\begin{align*}
\Ch(f) = \lim_{n\to\infty}\Ch(f_n)
\end{align*}
for every sequence $(f_n)_{n\in\N}$ that converges to $f$ in $\smash{\F_\rme}$.
\item $\smash{\Ch^{1/2}}$ is a seminorm on $\F_\rme$.
\item\label{La:E bounded} A function $f\in \Ell^0(\mms)$ belongs to $\F_\rme$ if and only if there exists an $\Ch$-bounded sequence $(f_n)_{n\in\N}$ in $\F_\rme$ such that $f_n \to f$ $\meas$-a.e.~as $n\to\infty$. In other words, the functional $\Ch_1\colon \Ell^0(\mms)\to [0,\infty]$ defined by
\begin{align*}
\Ch_1(f) := \begin{cases}
\Ch(f) & \text{if }f\in \F_\rme,\\
\infty & \text{otherwise}
\end{cases}
\end{align*}
is lower semicontinuous w.r.t.~pointwise $\meas$-a.e.~convergence. In particular, $\Ch_1$ viewed as a functional from $\smash{\Ell^2(\mms)}$ with values in $[0,\infty]$ is convex and $\Ell^2$-lower semicontinuous.
\item If $f,g\in\F_\rme$, then $\min\{f,g\},\min\{f^+,1\}\in\F_\rme$ with
\begin{align*}
\Ch\big(\!\min\{f,g\}\big) &\leq \Ch(f) + \Ch(g),\\
\Ch\big(\!\min\{f^+,1\}\big) &\leq \Ch\big(f, \min\{f^+,1\}\big).
\end{align*}
\item\label{La:W12 = S2 cap L2} We have
\begin{align*}
\F = \F_\rme\cap\Ell^2(\mms).
\end{align*}
\item Every $f\in\F_\rme$ has an $\Ch$-q.c.~$\meas$-version $\smash{\widetilde{f}}$ which is $\Ch$-q.e.~unique.
\end{enumerate}
\end{proposition}

In general, unfortunately, the constant function $\One_\mms$ does not belong to $\F_\rme$. Since we nevertheless need approximations of $\One_\mms$ in $\F$ at various instances, we record the following result due to \cite[Thm.~4.1]{kuwae1998}.

\begin{lemma}\label{Le:Approx to id} There exists a sequence $(G_n)_{n\in\N}$ of $\Ch$-quasi-open Borel subsets of $\mms$ such that $G_n\subset G_{n+1}$ $\Ch$-q.e.~for every $n\in \N$, $\smash{\bigcup_{n\in\N} G_n}$ covers $\mms$ up to an $\Ch$-polar set, and for every $n\in\N$ there exists $\smash{g_n \in \F_\bounded}$ such that
\begin{align*}
g_n = 1\quad\meas\text{-a.e.}\quad\text{on }G_n.
\end{align*}
\end{lemma}

\begin{remark}\label{Re:FLOC DEF:} In the terminology of \cite{kuwae1998}, \autoref{Le:Approx to id} asserts that $\One_\mms$ belongs to the \emph{local space} $\smash{\dot{\F}_\loc}$. It contains $\F_\rme$ \cite[Thm.~4.1]{kuwae1998}, but this inclusion may be strict. For instance, the energy measure $\bdmu$ from \autoref{Sub:Energy} below does in general not extend to a bilinear form on $\smash{\dot{\F}_\loc^2}$ in any reasonable sense \cite[Rem.~2.13]{delloschiavo2020}.
\end{remark}

\begin{remark}\label{Re:No global control} Unlike the setting of \autoref{Ex:mms} which is mostly worked upon in \cite{gigli2018}, \autoref{Le:Approx to id} does \emph{not} provide a global control on $(\Ch(g_n))_{n\in\N}$. In particular, we do not know in general whether $\Ch(g_n)\to 0$ as $n\to\infty$, a property which is closely related to recurrence of $\Ch$ \cite[Thm.~1.6.3]{fukushima2011}. 
\end{remark}

\subsubsection{Energy measures}\label{Sub:Energy}  $\Ch$ can be represented in terms of so-called \emph{energy measures}  --- a fact which also holds if the form admits no killing but a nontrivial jump part as  described in \autoref{Re:Jump}, see \cite[Thm.~3.2.1]{fukushima2011} and its proof --- a concept which is recorded now. This leads to the notion of carré du champ discussed in \autoref{Sub:Carré}. In \autoref{Sub:Differential}, it is furthermore used to give a precise meaning to the ``differential of a function $f \in \F_\rme$ on a Borel set $A\subset\mms$''.

The class $\smash{\F_\bounded}$ is an algebra w.r.t.~pointwise multiplication \cite[Prop.~I.2.3.2]{bouleau1991} which is dense in $\F$ \cite[Cor. 2.1]{kuwae1998}. In fact,
\begin{align*}
\Ch(f\,g) \leq 2\,\big\Vert f\big\Vert_{\Ell^\infty(\mms)}^2\,\Ch(g) + 2\,\big\Vert g\big\Vert_{\Ell^\infty(\mms)}^2\,\Ch(f)
\end{align*} 
for every $\smash{f,g\in\F_\bounded}$. Together with \cite[Thm.~5.2]{kuwae1998}, there exists a symmetric bilinear map $\smash{\bdmu\colon \F_\bounded^2} \to \smash{\Meas_{\bR}^{\pm}(\mms)_\Ch}$ such that for every $\smash{f,g,h\in \F_\bounded}$,
\begin{align*}
2\int_\mms \widetilde{h} \d\bdmu_{f,g} = \Ch(f\,h, g) + \Ch(g\,h, f) - \Ch(f\,g, h).
\end{align*}
The map $\bdmu$ is uniquely determined by this identity. It uniquely extends to a (non-relabeled) symmetric bilinear form $\smash{\bdmu \colon \F_\rme^2\to \Meas_\bR^{\pm}(\mms)_\Ch}$ by approximation in $\F_\rme$  \cite[Ch.~5]{kuwae1998}. The diagonal of $\bdmu$  takes values in $\smash{\Meas_\finR^+(\mms)_\Ch}$. Depending on the context, the map $\bdmu$ or, given any $f\in \smash{\F_\rme}$, the measure $\smash{\bdmu_f}$ are called the \emph{energy measure} (associated to $f$).

The following facts about $\bdmu$ with $\F_\rme$ and $\smash{\F_\eb}$  replaced by $\F$ and $\smash{\F_\bounded}$, respectively, have been proven in \cite[Thm.~2.8]{delloschiavo2020} and \cite[Lem.~5.2]{kuwae1998}. By approximation in $\F_\rme$, any of these properties  extend to the former classes. See also \autoref{Re:CDC} below for slightly more refined statements.

\begin{proposition}\label{Th:Properties energy measure}  The map $\bdmu$ satisfies the following obstructions.
	\begin{enumerate}[label=\textnormal{\textcolor{black}{(}\roman*\textcolor{black}{)}}]
	\item \emph{Representation formula.} For every $f,g\in\smash{\F_\rme}$,
	\begin{align*}
	\Ch(f,g) = \bdmu_{f,g}[\mms].
	\end{align*}
	\item \emph{Cauchy--Schwarz inequality.} For every $f,g\in \smash{\F_\rme}$ and every $u,v\in\Ell_\infty(\mms)$,
	\begin{align*}
	\Big\vert\!\int_\mms u\,v\d\bdmu_{f,g}\Big\vert \leq \Big[\!\int_\mms u^2\d\bdmu_f\Big]^{1/2}\,\Big[\!\int_\mms v^2\d\bdmu_g\Big]^{1/2}.
	\end{align*}
	\item\label{La:Minimum mu} \emph{Truncation.} For every $f,g,h\in\smash{\F_\rme}$, $\min\{f,g\}\in \F_\rme$ and
\begin{align*}
	\bdmu_{\min\{f,g\}, h} &= \One_{\{\widetilde{f}\leq \widetilde{g}\}}\,\bdmu_{f,h} + \One_{\{\widetilde{f} > \widetilde{g}\}}\,\bdmu_{g,h},\\
	\One_{\{\widetilde{f}=0\}}\,\bdmu_f&=0.
\end{align*}	
	\item\label{La:Chain rule} \emph{Chain rule.} For every $k,l\in\N$, every $f\in \smash{\F_\eb^k}$ and $g\in \smash{\F_\eb^l}$ as well as every $\smash{\varphi \in\Cont^1(\R^k)}$ and $\smash{\psi\in\Cont^1(\R^l)}$ --- with $\varphi(0) = \psi(0) = 0$ if $\meas[\mms] = \infty$ --- we have $\varphi\circ f,\psi\circ g\in \smash{\F_\rme}$ with
	\begin{align*}
	\bdmu_{\varphi\circ f,\psi\circ g} = \sum_{i=1}^k\sum_{j=1}^l \big[\partial_i \varphi\circ \widetilde{f}\big]\, \big[\partial_j\psi\circ \widetilde{g}\big]\,\bdmu_{f_i,g_j}.
\end{align*}
\item \emph{Leibniz rule.} For every $f,g,h\in\smash{\F_\eb}$, we have $f\,g\in \F_\eb$ and
\begin{align*}
\bdmu_{f\,g,h} = \widetilde{f}\,\bdmu_{g,h} + \widetilde{g}\,\bdmu_{f,h}.
\end{align*}
\item\label{La:Strong locality} \emph{Strong locality.} For every $\Ch$-quasi-open $G\subset \mms$, every $f\in \smash{\F_\rme}$ such that $f$ is constant $\meas$-a.e.~on $G$ and every $g\in \smash{\F_\rme}$,
\begin{align*}
\One_G\,\bdmu_{f,g} = 0.
\end{align*}
	\end{enumerate}
\end{proposition}

\subsubsection{Carré du champ and $\Ch$-dominance}\label{Sub:Carré} Following \cite{delloschiavo2020,hino2010},  we now record the concept of carré du champs and (minimal) $\Ch$-dominant measures. The second notion will especially be relevant in studying the ``universality'' of the cotangent module $\Ell^2(T^*\mms)$ and its compatibility to \cite{baudoin2019, cipriani2003, eberle1999, hinz2013, hinz2015, ionescu2012}, see \autoref{Th:Universal} and \autoref{Re:Comp Dir sp}, but also in \autoref{Re:Hino} below.

Towards the aim of representing $\Ch$ in terms of a ``scalar product of gradients'', the following definition from \cite[Def.~2.16]{delloschiavo2020} and \cite[Def.~2.1]{hino2010} is useful.

\begin{definition}\label{Def:Dominance} A Borel measure $\mu$ on $\mms$ is called 
\begin{enumerate}[label=\textnormal{\alph*.}]
\item \emph{$\Ch$-dominant} if it is $\sigma$-finite and $\bdmu_f \ll \mu$ for every $f\in\smash{\F}$, and 
\item \emph{minimal $\Ch$-dominant} if it is $\Ch$-dominant and $\mu\ll \nu$ for every $\Ch$-dominant Borel measure $\nu$ on $\mms$.
\end{enumerate}
\end{definition}

Any two minimal $\Ch$-dominant measures are mutually equivalent. The class of $f\in\F$ for which $\smash{\bdmu_f}$ is minimal $\Ch$-dominant is dense in $\F$, and every minimal $\Ch$-dominant measure does not charge $\Ch$-polar sets \cite[Prop.~2.18]{delloschiavo2020}. 

Given any $\Ch$-dominant $\mu$, there exists a unique symmetric bilinear $\smash{\Gamma_\mu\colon \F_\rme^2}\to\Ell^1(\mms,\mu)$ such that for every $\smash{f,g\in \F_\rme}$,
\begin{align*}
\Ch(f,g) = \int_\mms\Gamma_\mu(f,g)\d\mu.
\end{align*}
In particular, $\bdmu_f\ll \mu$ for every $f\in\F_\rme$ \cite[Lem.~2.2]{hino2010}, and the calculus rules from \autoref{Th:Properties energy measure} transfer accordingly to $\smash{\Gamma_\mu}$ at the $\mu$-a.e.~level. 

We say that $(\mms,\Ch,\meas)$ or simply $\Ch$ admits a carré du champ if $\meas$ is $\Ch$-dominant, in which case we abbreviate $\Gamma_\meas$ by $\Gamma$ and term it the \emph{carré du champ \textnormal{(}operator\textnormal{)}}  associated with $\Ch$. In fact, if $\Ch$ admits a carré du champ, then $\meas$ is already \emph{minimal} $\Ch$-dominant according to the remark after \cite[Def.~2.16]{delloschiavo2020}.

\begin{remark} In general, $\Ch$ might not always admit a carré du champ \cite[Ex.~2.17]{delloschiavo2020}, which actually motivated \autoref{Def:Dominance} in \cite{delloschiavo2020,hino2010}. This lack of $\Gamma$, however, is excluded later in \autoref{As:Gamma-operator}.
\end{remark}

\begin{remark}\label{Re:CDC} Following \cite[Thm.~I.7.1.1]{bouleau1991} or \cite[Thm.~4.3.8]{chen2012}, we have
\begin{align}\label{La:ll}
	f_\push\bdmu_f \ll \Leb^1
	\end{align}
	for every $f\in \F_\rme$. 
	In other words, given an $\Ch$-dominant $\mu$, $\smash{\Gamma_\mu(f) = 0}$ holds $\mu$-a.e.~on $f^{-1}(C)$ for every $\Leb^1$-negligible Borel set $C\subset\R$. In particular, by approximation of Lipschitz by $\smash{\Cont^1}$-functions, see e.g.~the proof of \cite[Thm.~2.2.6]{gigli2018}, the same conclusion as in \ref{La:Chain rule} in \autoref{Th:Properties energy measure} holds for $f\in\smash{\F_\rme^k}$ and $g\in \smash{\F_\rme^l}$ under the hypotheses that $\varphi\in \smash{\Lip(\R^k)}$ and $\smash{\psi\in \smash{\Lip(\R^l)}}$, where the partial derivatives $\smash{\partial_i\varphi}$ and $\smash{\partial_j\psi}$ are defined arbitrarily on their respective sets of non-differentiability points.

If $\Ch$ admits a carré du champ, \eqref{La:ll} translates into
	\begin{align*}
	f_\push\big[\Gamma(f)\,\meas\big]\ll\Leb^1.
	\end{align*}
	for every $f\in \F_\rme$. In this framework all $\Ch$-q.c.~representatives in \autoref{Th:Properties energy measure} can equivalently be replaced by their genuine $\meas$-versions. This fact will be used in the sequel when referring to \autoref{Th:Properties energy measure} without further mention.
	\end{remark}

\subsubsection{Neumann Laplacian}\label{Sub:Neumann Laplacian} By \cite[Thm.~1.3.1]{fukushima2011}, $\Ch$ is uniquely associated to a nonpositive, self-adjoint --- hence closed and densely defined --- operator $\Delta$ on $\Ell^2(\mms)$ with domain $\Dom(\Delta)$ as follows.\label{Not:Laplacian} A function $f\in \F$ belongs to $\Dom(\Delta)$ if and only if there exists $h \in \Ell^2(\mms)$ such that for every $g\in\F$,
\begin{align}\label{Eq:IBP Laplacian}
-\int_\mms g\,h \d\meas = \Ch(g,f).
\end{align}
If such a function $h$ exists, the element
\begin{align*}
\Delta f := h
\end{align*}
is unique and termed \emph{\textnormal{(}Neumann\textnormal{)} Laplacian} of $f$. See also \autoref{Ex:Mflds} below.

If $\Ch$ admits a carré du champ, from \autoref{Th:Properties energy measure} the following is straightforward to derive, see e.g.~\cite[Sec.~I.6]{bouleau1991}.

\begin{lemma}\label{Le:Delta Leibniz rule} If $\Ch$ admits a carré du champ, the following hold.
\begin{enumerate}[label=\textnormal{\textcolor{black}{(}\roman*\textcolor{black}{)}}]
\item \emph{Leibniz rule.} For every $f,g\in\Dom(\Delta)\cap \Ell^\infty(\mms)$ with $\Gamma(f),\Gamma(g)\in\Ell^\infty(\mms)$,
\begin{align*}
\Delta(f\,g) = f\,\Delta g + 2\,\Gamma(f,g) + g\,\Delta f\quad\meas\text{-a.e.}
\end{align*}
\item \emph{Chain rule.} For every $f\in\Dom(\Delta)\cap\Ell^\infty(\mms)$ with $\Gamma(f)\in\Ell^\infty(\mms)$ and every $\varphi\in\Cont^\infty(I)$ for some interval $I\subset\R$ which contains $0$ and the image of $f$ --- with $\varphi(0) = 0$ if $\meas[\mms] = \infty$ --- we have
\begin{align*}
\Delta(\varphi\circ f) = \big[\varphi'\circ f\big]\,\Delta f + \big[\varphi''\circ f\big]\,\Gamma(f)\quad\meas\text{-a.e.}
\end{align*}
\end{enumerate}
\end{lemma}

\subsubsection{Neumann heat flow}\label{Subsub:Neumann heat flow} By the spectral theorem \cite[Lem.~1.3.2]{fukushima2011}, $\Delta$ induces a strongly continuous semigroup $(\ChHeat_t)_{t\geq 0}$ ---  the so-called \emph{\textnormal{(}Neumann\textnormal{)} heat semigroup} or \emph{\textnormal{(}Neumann\textnormal{)} heat flow}  --- of linear operators on $\Ell^2(\mms)$ by
\begin{align*}
\ChHeat_t := \rme^{\Delta t}.
\end{align*}
The curve $t\mapsto \ChHeat_t f$ belongs to $\Cont^1((0,\infty);\Ell^2(\mms))$ with
\begin{align*}
\frac{\rmd}{\rmd t}\ChHeat_t f = \Delta\ChHeat_t f
\end{align*}
for every $f\in \Ell^2(\mms)$ and every $t>0$, and for such $f$, the function $t\mapsto \Ch(\ChHeat_tf)$ is nonincreasing on $(0,\infty)$. The heat flow commutes with $\Delta$, i.e.~
\begin{align*}
\ChHeat_t\,\Delta = \Delta\,\ChHeat_t\quad\text{on }\Dom(\Delta)
\end{align*}
for every $t\geq 0$. Moreover, if $f\in\F$ then $\ChHeat_t f\to f$ in $\F$ as $t\to 0$.

For every $t\geq 0$, $\ChHeat_t$ is a bounded, self-adjoint operator on $\Ell^2(\mms)$ with norm no larger than $1$. By the first Beurling--Deny criterion \cite[Thm.~1.3.2]{davies1989}, $(\ChHeat_t)_{t\geq 0}$ is \emph{positivity preserving}, i.e.~for every nonnegative $f\in \Ell^2(\mms)$, we have $\ChHeat_tf \geq 0$ $\meas$-a.e. Moreover, by \cite[Thm.~1.3.3]{davies1989}, $(\ChHeat_t)_{t\geq 0}$ is \emph{sub-Markovian}, i.e.~if $f\leq 1_\mms$ $\meas$-a.e.~then $\ChHeat_tf\leq \One_\mms$ $\meas$-a.e.~as well for every $t\geq 0$. Hence, $(\ChHeat_t)_{t\geq 0}$ extends to a semigroup of bounded contraction operators from $\Ell^p(\mms)$ to $\Ell^p(\mms)$ for every $p\in [1,\infty]$ which is strongly continuous if $p<\infty$ and weakly$^*$ continuous if $p=\infty$. Finally, at various occasions we will need the subsequent standard a priori estimates, see e.g.~\cite{brezis1973} or the arguments in \cite[Subsec.~3.4.4]{gigli2018}.

\begin{lemma}\label{Th:Heat flow properties} For every $f\in\Ell^2(\mms)$ and every $t>0$,
\begin{align*}
\Ch(\ChHeat_t f) &\leq \frac{1}{2t}\,\big\Vert f\big\Vert_{\Ell^2(\mms)}^2,\\
\big\Vert\Delta \ChHeat_tf\big\Vert_{\Ell^2(\mms)}^2 &\leq \frac{1}{2t^2}\big\Vert f\big\Vert_{\Ell^2(\mms)}^2.
\end{align*}
\end{lemma}

\subsubsection{Associated Markov process}\label{Sub:Markov} By quasi-regularity \cite[Thm.~IV.3.5]{ma1992}, $\Ch$ is properly associated to an $\meas$-reversible Markov process $\smash{\M := (\Omega,\calA,(\B_t)_{t\geq 0}, (\Prob^x)_{x\in\mms_{\cem}})}$ consisting of a process $(\B_t)_{t\geq 0}$ with values in a Hausdorff topological space $\mms_{\cem} := \mms \sqcup \{\cem\}$ with \emph{cemetery} $\cem$ and associated lifetime $\zeta \colon \Omega\to [0,\infty]$ defined on a measurable space $(\Omega,\calA)$, a family $(\Prob^x)_{x\in\mms_{\cem}}$ of probability measures on $(\Omega,\calA)$, and an (implicitly given) filtration $(\calA_t)_{t \geq 0}$ w.r.t.~which $\smash{((\B_t)_{t\geq 0},(\Prob^x)_{x\in \mms_\cem})}$ satisfies the Markov property. See \cite[Sec.~IV.1]{ma1992} for precise definitions. The indicated association is that for every $f\in \Ell^2(\mms)$ and every $t>0$,
\begin{align*}
\ChHeat_t f = \Exp^{\,\cdot}\big[f(\B_{2t})\,\One_{\{t < \zeta/2\}}\big]\quad\meas\text{-a.e.},
\end{align*}
and the map $\smash{x \mapsto \Exp^x[f(\B_{2t})\,\One_{\{t < \zeta/2\}}]}$ defined on $\mms$ is $\Ch$-q.c.~\cite[Def.~IV.2.5]{ma1992}. These expressions make sense regardless of the chosen $\meas$-version of $f$ \cite[Rem.~IV.2.6]{ma1992} and respect the convention $f(\cem) := 0$ for $f$ initially defined on $\mms$.  $\M$ can be constructed to be   \emph{right} \cite[Def.~IV.1.8]{ma1992}  and \emph{$\meas$-tight special standard} \cite[Def.~IV.1.13]{ma1992}. Among other things, this means that~$(\B_0)_\push\Prob^x = \delta_x$ for every $x\in\mms_{\cem}$ and that $(\B_t)_{t\geq 0}$ obeys the strong Markov property w.r.t.~$(\calA_t)_{t\geq 0}$.

By strong locality of $\Ch$, see e.g.~\cite[Thm.~1.5]{ma1992} or  \cite[Thm.~4.3.4]{chen2012}, $\M$ can be even chosen to  satisfy the a.s.~sample path continuity property
\begin{align*}
\Prob^{\,\cdot}\big[t\mapsto \sfb_t \text{ is continuous on }[0,\zeta/2)\big] = 1\quad\Ch\text{-q.e.}
\end{align*}

\subsubsection{Two guiding examples}\label{Sub:Guiding}

The following frameworks frequently serve as guiding examples  throughout our treatise and are mainly listed to fix notation.

\begin{example}[Riemannian manifolds with boundary]\label{Ex:Mflds} Let $\mms$ be a Riemannian manifold with boundary as in \autoref{Sub:Riem mflds}, and let $\meas$ be a Borel measure on $\mms$ which is locally equivalent to $\vol$. Let $W^{1,2}(\mms^\circ)$ be the Sobolev space w.r.t.~$\meas$  defined in the usual sense on $\mms^\circ$. Define  $\Ch\colon W^{1,2}(\mms^\circ)\to [0,\infty)$ through
\begin{align*}
\Ch(f) := \int_{\mms^\circ} \vert\nabla f\vert^2\d\meas
\end{align*}
and the quantity $\Ch(f,g)$, $f,g\in W^{1,2}(\mms^\circ)$, by polarization. 
Then $(\Ch, W^{1,2}(\mms))$ is a Dirichlet form which is strongly local and regular, since $\smash{\Cont_\comp^\infty(\mms)}$ is a dense set of $W^{1,2}(\mms)$ which is also uniformly dense in $\Cont_0(\mms)$.\label{Not:C_0} (The latter is the space of continuous functions on $\mms$ vanishing at $\infty$.) $\Ch$ admits a carré du champ which is precisely given by $\vert\nabla \cdot\vert^2$.  See  \cite{chen2012, davies1989, fukushima2011, sturm2021} for details.

Furthermore, suppose that $\smash{\meas := \rme^{-2w}\,\vol}$ for some $\smash{w\in \Cont^2(\mms)}$. By Green's formula, see e.g.~\cite[p.~44]{lee1997}, $\Delta$ is the self-adjoint realization of the drift Laplacian 
\begin{align*}
\Delta_0 - 2\,\langle\nabla w,\nabla\cdot\rangle
\end{align*}
w.r.t.~\emph{Neumann boundary conditions}, where $\Delta_0$ is the Laplace--Beltrami operator on $\mms$, initially defined on functions $f\in \Cont_\comp^\infty(\mms)$ with
\begin{align*}
\rmd f(\sfn)=0\quad\text{on }\partial\mms.
\end{align*}
This equation makes sense $\surf$-a.e.~for every $f\in\Dom(\Delta)$ by the local trace theorem.
\end{example}

\begin{example}[Infinitesimally Hilbertian metric measure spaces]\label{Ex:mms} Let $(\mms,\met,\meas)$ be an infinitesimally Hilbertian metric measure space, according to (4.19) in \cite{ambrosio2014b}, for which $\meas$ satisfies the growth condition (4.2) in \cite{ambrosio2014a}. In this case $\Ch$  is the \emph{Cheeger energy} introduced in \cite[Thm.~4.5]{ambrosio2014a}, with domain denoted by $W^{1,2}(\mms)$. By \cite[Thm.~4.1]{savare2014}, $\Ch$ is  quasi-regular and strongly local, and it admits a carré du champ which $\meas$-a.e.~coincides with the \emph{minimal relaxed gradient} from \cite[Def.~4.2]{ambrosio2014a}.
\end{example}

\begin{remark}[Subsets]\label{Re:Subsets} Let $(\mms,\met,\meas)$ be as in \autoref{Ex:mms} and $E\subset \mms$ be a closed subset. Assume that $E= \smash{\overline{E}^\circ}$, $\meas[E] >0$, $\meas[\partial E] = 0$, and that the length distance $\met_E$ on $E^2$ induced by $\met$ is nondegenerate. Then $(E,\met_E,\meas_E)$, where $\meas_E := \meas[\,\cdot\cap E]$, induces a quasi-regular, strongly local Dirichlet space $(E,\Ch_E,\meas_E)$ \cite{sturm2020, sturm2021} whose carré du champs coincide $\meas_E$-a.e.~on $E^\circ$. Moreover, in the sense of restrictions of functions, see e.g.~(47) in  \cite{sturm2020}, we have
\begin{align}\label{Eq:W12 inclusions subsets}
W^{1,2}(\mms)\big\vert_E \subset W^{1,2}(E).
\end{align}
\end{remark}

Whenever $\mms$ fits into \autoref{Ex:Mflds} or \autoref{Ex:mms}, we intend the canonically induced Dirichlet space $(\mms,\Ch,\meas)$ without further notice.

\subsection{$L^\infty$-modules}\label{Sub:Linfty modles} Throughout this section, we assume that $\mu$ is a $\sigma$-finite Borel measure on $\mms$. The theory of \emph{$\Ell^p(\mms,\mu)$-normed $\Ell^\infty(\mms,\mu)$-modules}, $p\in [1,\infty]$, over general measure spaces has been introduced in \cite{gigli2018} and is recorded now.

As in \autoref{Subsub:Fcts}, when $\mu = \meas$ --- which will be the relevent case in most of our work unless in \autoref{Sec:Cotangent module} --- we drop the measure from the notation. 

\subsubsection{Definition and basic properties} 

\begin{definition}\label{Def:Modules} Given $p\in [1,\infty]$, a real Banach space $(\calM,\Vert \cdot\Vert_\calM)$ or simply $\calM$ is termed an \emph{$\Ell^p$-normed $\Ell^\infty$-module \textnormal{(}over $\mms$\textnormal{)}} if it comes with
\begin{enumerate}[label=\textnormal{\alph*.}]
\item\label{La:Hut} a bilinear map $\cdot\colon \Ell^\infty(\mms,\mu)\times\calM\to\calM$ satisfying
\begin{align*}
(f\,g)\cdot v &= f\cdot(g\cdot v),\\
\One_\mms\cdot v &= v,
\end{align*}
\item a nonnegatively valued map $\vert\cdot\vert_\mu \colon \calM\to \Ell^p(\mms,\mu)$ such that
\begin{align*}
\vert f\cdot v\vert_\mu &= \vert f\vert\,\vert v\vert_\mu\quad\mu\text{-a.e.},\\
\Vert v\Vert_\calM &= \big\Vert \vert v\vert_\mu\big\Vert_{\Ell^p(\mms,\mu)},
\end{align*}
\end{enumerate}
for every $f,g\in\Ell^\infty(\mms,\mu)$ and every $v\in\calM$. If only \autoref{La:Hut} is satisfied, we call $(\calM,\Vert\cdot\Vert_{\calM})$ or simply $\calM$ an \emph{$\Ell^\infty$-premodule}.
\end{definition}

We neither express the space $\mms$ nor the reference measure $\mu$ in the terminology of an $\Ell^p$-normed $\Ell^\infty$-module for brevity. In this section, every module is considered over the same $\mms$ and w.r.t.~the same $\mu$. Later, it will always either be clear from the context or explicitly indicated which $\mms$ and $\mu$ are intended.

\begin{remark} In \cite[Def.~1.2.10]{gigli2018}, spaces obeying \autoref{Def:Modules} are called \emph{$\Ell^p$-normed premodules} and are as such a priori more general than $\Ell^p$-normed $\Ell^\infty$-modules. However, these notions coincide for $p\in [1,\infty)$ \cite[Prop.~1.2.12]{gigli2018}. What only might be missing in the case $p=\infty$ is the gluing property \cite[Def.~1.2.1, Ex.~1.2.5]{gigli2018}, whose lack will never occur in our work,  hence the minor change of terminology.
\end{remark}

\begin{example} $\Ell^p(\mms,\mu)$ is an $\Ell^p$-normed $\Ell^\infty$-module w.r.t.~$\mu$, $p\in[1,\infty]$.
\end{example}

We call $\calM$ an \emph{$\Ell^\infty$-module} if it is $\Ell^p$-normed for some $p\in [1,\infty]$ --- which is assumed throughout the rest of this section --- and \emph{separable} if it is a separable Banach space. We term $v\in\calM$ \emph{\textnormal{(}$\mu$-essentially\textnormal{)} bounded} if $\vert v\vert_\mu\in\Ell^\infty(\mms,\mu)$. If $\calM$ is separable, it admits a countable dense subset of bounded elements. We drop the $\cdot$ sign if the multiplication on $\calM$ is understood. By \cite[Prop.~1.2.12]{gigli2018}, $\vert \cdot\vert_\mu$ is local in the sense that $\vert v\vert_\mu = 0$ $\mu$-a.e.~on $E$ if and only if $1_E\,v = 0$ for every $v\in\calM$ and every $E\in\Borel^\mu(\mms)$. We write 
\begin{align*}
\{v = 0\} &:= \{\vert v\vert_\mu = 0\},\\
\{v\neq 0\} &:= \{v=0\}^\rmc
\end{align*}
Lastly, for every $v,w\in\calM$ we have the $\mu$-a.e.~triangle inequality
\begin{align*}
\vert v+ w\vert_\mu \leq \vert v\vert_\mu + \vert w\vert_\mu\quad\mu\text{-a.e.},
\end{align*}
which shows that the map $\vert\cdot\vert_\mu\colon \calM\to \Ell^p(\mms,\mu)$ is continuous.

$\calM$ is called \emph{Hilbert module} if it is an $\Ell^2$-normed $\Ell^\infty$-module and a Hilbert space \cite[Def.~1.2.20, Prop.~1.2.21]{gigli2018}. Its pointwise norm $\vert\cdot\vert_\mu$ satisfies a pointwise $\mu$-a.e.~parallelogram identity. In particular, it induces a \emph{pointwise scalar product} $\langle\cdot,\cdot\rangle_\mu \colon \calM^2\to\Ell^1(\mms,\mu)$ which is $\Ell^\infty(\mms,\mu)$-bilinear, $\mu$-a.e.~nonnegative definite, local in both components, satisfies the pointwise $\mu$-a.e.~Cauchy--Schwarz inequality, and reproduces the Hilbertian scalar product on $\calM$ by integration w.r.t.~$\mu$.

Let $\calM$ and $\calN$ be $\Ell^p$-normed $\Ell^\infty$-modules, $p\in [1,\infty]$, such  that $\calN$ is a closed subspace of $M$. Then the quotient $\calM/\calN$ is an $\Ell^p$-normed $\Ell^\infty$-module as well \cite[Prop.~1.2.14]{gigli2018}  with pointwise norm  given by
\begin{align*}
\vert[v]\vert_{\mu} := \mu\text{-}\!\essinf\!\big\lbrace \vert v+w\vert_\mu : w\in\calN\big\rbrace.
\end{align*}
For instance, given $E\in\Borel^\mu(\mms)$, the $\Ell^\infty$-module $\smash{\calM\big\vert_E}$ consisting of all $v\in\calM$ such that $\{v\neq 0\}\subset E$ can be canonically identified with $\smash{\calM/\calM\big\vert_{E^\rmc}}$.

\subsubsection{Duality}\label{Sub:Duality} Let $\calM$ and $\calN$ be $\Ell^\infty$-normed modules. Slightly abusing notation, denote both pointwise norms by $\vert \cdot\vert_\mu$. A map $T\colon \calM\to\calN$ is called \emph{module morphism} if it is a bounded linear map in the sense of functional analysis and
\begin{align}\label{Eq:BRABBL}
T(f\,v) = f\,T(v)
\end{align}
for every $v\in\calM$ and every $f\in\Ell^\infty(\mms,\mu)$.
The set of all such module morphisms is written $\Hom(\calM;\calN)$ and is equipped with the usual operator norm $\Vert \cdot\Vert_{\calM;\calN}$. We term $\calM$ and $\calN$ \emph{isomorphic \textnormal{(}as $\Ell^\infty$-modules\textnormal{)}} if there exist $T\in\Hom(\calM;\calN)$ and $S\in\Hom(\calN;\calM)$ such that $T\circ S = \Id_\calN$ and $S\circ T=\Id_\calM$. Any such $T$ is called \emph{module isomorphism}. If in addition, such a $T$ is a norm isometry, it is called \emph{module isometric isomorphism}. In fact, by \eqref{Eq:BRABBL} every module isometric isomorphism $T$ preserves pointwise norms $\mu$-a.e., i.e.~for every $v\in\calM$,
\begin{align*}
\vert T(v)\vert_\mu = \vert v\vert_\mu\quad\mu\text{-a.e.}
\end{align*}

The \emph{dual module} to $\calM$ is defined by
\begin{align*}
\calM^* := \Hom(\calM;\Ell^1(\mms,\mu))
\end{align*}
and will be endowed with the usual operator norm. The pointwise pairing between $v\in\calM$ and $L\in\calM^*$ is denoted by $L(v)\in\Ell^1(\mms,\mu)$. If $\calM$ is $\Ell^p$-normed, then $\calM^*$ is an $\Ell^q$-normed $\Ell^\infty$-normed module, where $p,q\in [1,\infty]$ with $1/p+1/q=1$ \cite[Prop.~1.2.14]{gigli2018} with naturally defined multiplication and, by a slight abuse of notation, pointwise norm given by
\begin{align}\label{Eq:dual norm}
\vert L\vert_\mu := \mu\text{-}\!\esssup\!\big\lbrace \vert L(v)\vert : v\in\calM,\ \vert v\vert_\mu \leq 1\ \mu\text{-a.e.}\big\rbrace.
\end{align}
By  \cite[Cor.~1.2.16]{gigli2018}, if $p<\infty$,
\begin{align*}
\vert v\vert_\mu = \mu\text{-}\!\esssup\!\big\lbrace \vert L(v)\vert : L\in \calM^*,\ \vert L\vert_\mu \leq 1\ \mu\text{-a.e.}\big\rbrace
\end{align*}
for every $v\in\calM$. Moreover, if $p<\infty$, in the sense of functional analysis $\calM^*$ and the Banach space dual $\calM'$ of $\calM$ are isometrically isomorphic \cite[Prop.~1.2.13]{gigli2018}. In this case, the natural  pointwise pairing map $\calJ\colon \calM \to\calM^{**}$, where  $\calM^{**} := \Hom(\calM^*;\Ell^1(\mms,\mu))$, belongs to $\Hom(\calM;\calM^{**})$ and constitutes a norm isometry \cite[Prop.~1.2.15]{gigli2018}. We term $\calM$ \emph{reflexive \textnormal{(}as $\Ell^\infty$-module\textnormal{)}} if $\calJ$ is surjective. If $\calM$ is $\Ell^p$-normed for $p\in (1,\infty)$, this is equivalent to $\calM$ being reflexive as Banach space \cite[Cor.~1.2.18]{gigli2018}, while for $p=1$,  the implication from ``reflexive as Banach space'' to ``reflexive as $\Ell^\infty$-module'' still holds \cite[Prop.~1.2.13, Prop.~1.2.17]{gigli2018}. In particular, all Hilbert modules are reflexive in both senses.

If $\calM$ is a Hilbert module, we have the following analogue of the Riesz representation theorem \cite[Thm.~1.2.24]{gigli2018}. For $v\in\calM$, let $L_v\in\calM^*$ be given by
\begin{align*}
L_v(w) := \langle v,w\rangle.
\end{align*}

\begin{proposition}\label{Th:Riesz theorem modules} Let $\calM$ be a Hilbert module. Then the map which sends $v\in\calM$ to $L_v\in\calM^*$ is a module isometric  isomorphism, and in particular a norm isometry. Moreover, for every $l\in\calM'$ there exists a unique $v\in\calM$ with
\begin{align*}
l = \int_\mms \langle v,\cdot\rangle_\mu\d\mu.
\end{align*}
\end{proposition}

\subsubsection{$L^0$-modules}\label{Sub:L0 modules} Let $\calM$ be an  $\Ell^\infty$-module. Following \cite[Sec.~1.3]{gigli2018} we now recall a natural concept of building a topological vector space $\calM^0$ of ``measurable elements of $\mms$ without integrability restrictions'' containing $\calM$ with continuous inclusion as well as a (non-relabeled) extension of the pointwise norm $\vert \cdot\vert_\mu\colon\calM^0\to\Ell^0(\mms,\mu)$ such that for every $v\in\calM^0$, $v\in \calM$ if and only if $\vert v\vert_\mu\in\Ell^p(\mms,\mu)$.

Let $(B_i)_{i\in\N}$ a Borel partition of $\mms$ such that $\meas[B_i] \in (0,\infty)$ for every $i\in\N$. Denote by $\calM^0$ the completion of $\calM$ w.r.t.~the distance $\met_{\calM^0}\colon \calM^2\to[0,\infty)$ with
\begin{align*}
\met_{\calM^0}(v,w) := \sum_{i\in \N} \frac{2^{-i}}{\meas[B_i]}\int_{B_i}\min\{\vert v-w\vert,1\}\d\mu.
\end{align*}
We refer to $\calM^0$ as the \emph{$\Ell^0$-module} associated to $\calM$. The induced topology on $\calM^0$ does not depend on the choice of $(B_i)_{i\in\N}$ \cite[p.~31]{gigli2018}. Additionally, scalar and functional multiplication, and the pointwise norm $\vert\cdot\vert_\mu$ extend continuously to $\calM^0$, so that all $\mu$-a.e.~properties mentioned in \autoref{Sub:Linfty modles} hold for general elements in $\calM^0$ and $\Ell^0(\mms,\mu)$ in place of $\calM$ and $\Ell^\infty(\mms,\mu)$. The pointwise pairing of $\calM$ and $\calM^*$ extends uniquely and continuously to a bilinear map on $\calM^0\times(\calM^*)^0$ with values in $\Ell^0(\mms,\mu)$ such that for every $v\in\calM^0$ and every $L\in (\calM^*)^0$,
\begin{align*}
\vert L(v)\vert \leq \vert L\vert_\mu\,\vert v\vert_\mu\quad\mu\text{-a.e.},
\end{align*}
and we have the following characterization of elements in $(\calM^*)^0$ \cite[Prop.~1.3.2]{gigli2018}.

\begin{proposition}\label{Pr:Dual of L0} Let $T\colon \calM^0\to \Ell^0(\mms,\mu)$ be a linear map for which there exists $f\in\Ell^0(\mms,\mu)$ such that for every $v\in\calM$,
\begin{align*}
\vert T(v)\vert\leq f\,\vert v\vert_\mu\quad\mu\text{-a.e.}
\end{align*}
Then there exists a unique $L\in(\calM^*)^0$ such that 
for every $v\in\calM$, 
\begin{align*}
L(v) = T(v)\quad\mu\text{-a.e.},
\end{align*}
and we furthermore have
\begin{align*}
\vert L\vert_\mu \leq f\quad\mu\text{-a.e.}
\end{align*}
\end{proposition}

\begin{remark} To some extent, one can make sense of $\Ell^0$-normed modules w.r.t.~a \emph{submodular outer measure} $\mu^*$ on $\mms$ \cite[Def.~2.4]{debin2021}. A prominent example of such a $\mu^*$ is the \emph{$\Ch$-capacity} $\CAP_\Ch$ \cite[Def.~2.6, Prop.~2.8]{debin2021}, which --- towards the aim of defining \emph{quasi-continuity} of vector fields over metric measure spaces --- motivated the authors of \cite{debin2021} to study this kind of modules.  However, what lacks for such $\mu^*$ is a working definition of dual modules, and in particular \autoref{Pr:Dual of L0} seems unavailable \cite[Rem.~3.3]{debin2021}.
\end{remark}

\begin{remark} The concept of $\Ell^0$-modules is tightly linked to the one of \emph{measurable fields of Hilbert spaces} \cite{eberle1999}. See \cite[Rem.~1.4.12]{gigli2018} and \cite[Ch.~2]{hinz2013} for details.
\end{remark}

\subsubsection{Local dimension and dimensional decomposition}\label{Sub:Local dimension} Given an $\Ell^\infty$-module $\calM$ and $E\in \Borel^\mu(\mms)$, we say that $v_1,\dots,v_n\in\calM$, $n\in\N$, are \emph{independent \textnormal{(}on $E$\textnormal{)}} if all  functions $f_1,\dots,f_n\in\Ell^\infty(\mms,\mu)$ obeying
\begin{align*}
f_1\,v_1 + \dots + f_n\,v_n = 0\quad \mu\text{-a.e.}\quad\text{on }E
\end{align*}
vanish $\mu$-a.e.~on $E$. This notion of local independence is well-behaved under passage to subsets and under module isomorphisms, see \cite[p.~34]{gigli2018} for details. The \emph{span} $\Span_E\,\calV$ of a subset $\calV\subset\calM$ on $E\subset \mms$ --- briefly $\Span\,\calV$ if $E=\mms$ --- is the space consisting of all $\smash{v\in \calM\big\vert_E}$ which possess the following property: there exists a disjoint  partition $(E_k)_{k\in\N}$ of $E$ in $\Borel^\mu(\mms)$ such that for every $k\in\N$, we can find $m_k\in\N$, $\smash{v_1^k,\dots,v_{m_k}^k\in \calV}$ and $\smash{f_1^k,\dots,f_{m_k}^k}\in\Ell^\infty(\mms,\mu)$ such that
\begin{align*}
\One_{E_k}\,v = f_1^k\,v_1^k + \dots + f_{m_k}^k\,v_{m_k}^k.
\end{align*}
Its closure $\smash{\cl_{\Vert\cdot\Vert_\calM}\Span_E\,\calV}$ is usually referred to as the space \emph{generated} by $\calV$ on $E$, or simply by $\calV$ if $E=\mms$ \cite[Def.~1.4.2]{gigli2018}. 

If $\calV$ is a finite set, $\Span_E\,\calV$ is closed \cite[Prop.~1.4.6]{gigli2018}, a fact which gives additional strength to the following notions \cite[Def.~1.4.3]{gigli2018}.

\begin{definition} A family $\{v_1,\dots,v_n\}\subset \calM$, $n\in\N$, is said to be a \emph{\textnormal{(}local\textnormal{)} basis \textnormal{(}on $E$\textnormal{)}} if $v_1,\dots,v_n$ are independent on $E$, and
\begin{align*}
\Span_E\,\{v_1,\dots,v_n\} = \calM\big\vert_E.
\end{align*}
We say that $\calM$ has \emph{\textnormal{(}local\textnormal{)} dimension $n$ \textnormal{(}on $E$\textnormal{)}} if there exists a local basis of $\calM$ on $E$. We say that $\calM$ has \emph{infinite dimension} on $E$ if it does not have finite dimension on any subset of $E$ with positive $\mu$-measure.
\end{definition}

For the well-posedness of this definition and its link to local independence, we refer to \cite[Prop.~1.4.4]{gigli2018}. If $\calM$ is $\Ell^p$-normed w.r.t.~$\mu$, $p<\infty$, and has dimension $n\in\N$ on $E$, then $\calM^*$ has dimension $n$ on $E$ as well \cite[Thm.~1.4.7.]{gigli2018}.

The following important structural result is due to \cite[Prop.~1.4.5]{gigli2018}.

\begin{proposition}\label{Th:Dimensional decomposition} For every $\Ell^\infty$-module $\calM$ w.r.t.~$\mu$, there exists a unique Borel partition $(E_n)_{n\in\N\cup\{\infty\}}$ of $\mms$ such that
\begin{enumerate}[label=\textnormal{\alph*.}]
\item for every $n\in\N$ with $\mu[E_n] >0$, $\calM$ has dimension $n$ on $E_n$, and
\item for every $E\in \Borel^\mu(\mms)$ with $E\subset E_\infty$, $\calM$ has infinite dimension on $E$.
\end{enumerate}
\end{proposition} 

\begin{remark}\label{Re:Link MHB} Using \autoref{Th:Dimensional decomposition}, it is possible to establish a one-to-one-cor\-respondence between separable Hilbert modules and \emph{direct integrals} of separable Hilbert spaces \cite[Thm.~1.4.11, Rem.~1.4.12]{gigli2018}. Albeit at a structural level, this provides the link of \autoref{Sec:Cotangent module} and \autoref{Ch:Tangent module} to earlier axiomatizations of spaces of $1$-forms and vector fields on Dirichlet spaces \cite{baudoin2019, cipriani2003, eberle1999, hinz2013, hinz2015, ionescu2012}, and in view of the universal property of \autoref{Th:Universal} below, we do not enter into details here and leave these to the interested reader. We only point out the remark at \cite[p.~42]{gigli2018} that the interpretation of this link should be treated with some care.
\end{remark}

\begin{remark}[Hino index]\label{Re:Hino} Unlike \autoref{Th:Dimensional decomposition},  in Dirichlet form theory there already exists a natural notion of ``pointwise tangent space dimension'' in terms of the \emph{\textnormal{(}pointwise\textnormal{)} Hino index} introduced in \cite{hino2010}, which is quickly recorded now. For the modules under our consideration, these two notions of local dimension turn out to coincide, see \autoref{Cor:Hino} and also \autoref{Sec:Structural consequences}.

Let $\mu$ be minimal $\Ch$-dominant. Let $(f_i)_{i\in\N}$ be a sequence in $\F$ whose linear span is dense in $\F$. The \emph{pointwise index} of $(\Ch, \F)$ or simply $\Ch$ \cite[Def.~2.9, Prop.~2.10]{hino2010} is the function $\rmp\colon \mms \to \N_0\cup\{\infty\}$ given by
\begin{align*}
\rmp := \sup_{n\in\N} \mathrm{rank}\, \big[\Gamma_\mu(f_i,f_j)\big]_{i,j\in \{1,\dots,n\}}.
\end{align*}
See \cite[Ch.~2]{hino2010} for a thorough discussion on the well-definedness of $\rmp$. In particular, by  \cite[Prop.~2.11]{hino2010} we know that $\rmp > 0$ $\mu$-a.e.~unless $\Ch$ is trivial. The \emph{index} $\rmp^*\in \N\cup\{\infty\}$ of $(\Ch, \F)$ or simply $\Ch$ is then
\begin{align*}
\rmp^* := \Vert \rmp\Vert_{\Ell^\infty(\mms,\mu)}.
\end{align*}
These definitions are  independent of the choice of minimal $\Ch$-dominant $\mu$. By \cite[Prop.~2.10]{hino2010}, we have the following result. (Another  probabilistic aspect of it not treated in our work is that $\rmp^*$ coincides with the so-called \emph{martingale dimension} \cite{hino2008, kusuoka1989} w.r.t.~the Markov process on $\mms$ from \autoref{Sub:Markov} \cite[Thm.~3.4]{hino2010}.)
\end{remark}

\begin{lemma} For every $n\in\N$ and every $f_1,\dots,f_n \in \F$,
\begin{align*}
\mathrm{rank}\,\big[\Gamma_\mu(f_i,f_j)\big]_{i,j\in \{1,\dots,n\}} \leq \rmp \leq \rmp^*\quad\mu\text{-a.e.}
\end{align*}
Moreover, $\rmp$ is the $\mu$-a.e.~smallest function satisfying the first $\mu$-a.e.~inequality for every $n\in\N$ and every $f_1,\dots,f_n \in \F$.
\end{lemma}

\subsubsection{Tensor products}\label{Sub:Tensor products} Let $\calM_1$ and $\calM_2$ be two \emph{Hilbert} modules. Again, by a slight abuse of notation, we denote both pointwise scalar products by $\langle\cdot,\cdot\rangle_\mu$. 

Let $\calM_1^0\odot\calM_2^0$ be the ``tensor product''  consisting of all finite linear combinations of formal elements $v\otimes w$, $\smash{v\in\calM_1^0}$ and $\smash{w\in\calM_2^0}$, obtained by factorizing appropriate vector spaces  \cite[Sec.~1.5]{gigli2018}. It naturally comes with a multiplication $\cdot\colon \Ell^0(\mms,\mu)\times (\calM_1^0\odot\calM_2^0)\to\Ell^0(\mms,\mu)$ defined through
\begin{align*}
f\,(v\otimes w) := (f\,v)\otimes w = v\otimes(f\,w)
\end{align*}
 and a pointwise scalar product $:_\mu\colon (\calM_1^0\odot\calM_2^0)^2\to \Ell^0(\mms,\mu)$ given by
\begin{align}\label{Eq:Tensor product pointwise sc prod}
(v_1\otimes w_1) :_\mu (v_2\otimes w_2) &:= \langle v_1, v_2\rangle_\mu\,\langle w_1,w_2\rangle_\mu,
\end{align}
both extended to $\calM_1^0\odot \calM_2^0$ by (bi-)linearity. Then $:_\mu$ is bilinear, $\mu$-a.e.~nonnegative definite, symmetric, and local in both components \cite[Lem.~3.2.19]{gigli2020}. 

The pointwise \emph{Hilbert--Schmidt norm} $\vert\cdot\vert_{\HS,\mu}\colon \calM_1^0\odot \calM_2^0\to \Ell^0(\mms,\mu)$ is given by
\begin{align}\label{Eq:Pointwise norm HS}
\vert A\vert_{\HS,\mu} := \sqrt{A:_\mu A}.
\end{align}
This map satisfies the $\mu$-a.e.~triangle inequality and is $1$-homogeneous w.r.t.~multiplication with $\Ell^0(\mms,\mu)$-functions \cite[p.~44]{gigli2018}. 

Consequently, the map $\Vert \cdot\Vert_{\calM_1\otimes\calM_2}\colon\calM_1^0\odot \calM_2^0\to [0,\infty]$ defined through
\begin{align*}
\Vert A\Vert_{\calM_1\otimes\calM_2} := \big\Vert \vert A\vert_{\HS,\mu}\big\Vert_{\Ell^2(\mms,\mu)}
\end{align*}
has all properties of a norm except that it might take the value $\infty$.

\begin{definition} The \emph{tensor product}  $\calM_1\otimes\calM_2$ is the completion w.r.t.~$\Vert\cdot\Vert_{\calM_1\otimes\calM_2}$ of the subspace that consists of all $A\in\calM_1^0\odot\calM_2^0$ such that $\smash{\Vert A \Vert_{\calM_1\otimes\calM_2} < \infty}$.
\end{definition}

Inductively, for $\calM := \calM_1$ and $k\in\N$, up to unique identification we set
\begin{align*}
\calM^{\otimes k} := \calM^{\otimes (k-1)}\otimes \calM = \calM\otimes\calM^{\otimes (k-1)},
\end{align*}
where we conventionally set $\calM^{\otimes 0} := \Ell^2(\mms)$ as well. 

Through \eqref{Eq:Tensor product pointwise sc prod}, $\calM_1\otimes\calM_2$ naturally becomes a Hilbert module \cite[p.~45]{gigli2018}. If $\calM_1$ and $\calM_2$ are separable, then so is $\calM_1\otimes\calM_2$. Indeed, if $D_i\subset \calM_i$ are countable dense subsets consisting of \emph{bounded} elements, $i\in \{1,2\}$, then the linear span of elements of the form $v\otimes w$, $v\in D_1$ and $w\in D_2$, is dense in $\calM_1\otimes\calM_2$. Here, boundedness is essential as underlined by the next remark.

\begin{remark} The space $\calM_1\otimes \calM_2$ should not be confused with the tensor product $\calM_1\otimes_\rmH\calM_2$ in the Hilbert space sense \cite{kadison1983}. Indeed, in general these do not coincide \cite[Rem.~1.5.2]{gigli2018}. For instance, for $v\in\calM_1$ and $w\in\calM_2$ we always have $v\otimes_\rmH w \in \calM_1\otimes_\rmH\calM_2$ since the corresponding norm is
\begin{align*}
\Vert v\otimes_\rmH w\Vert_{\calM_1\otimes_\rmH\calM_2} = \Vert v\Vert_{\calM_1}\,\Vert w\Vert_{\calM_2},
\end{align*}
but according to \eqref{Eq:Tensor product pointwise sc prod}, the norm
\begin{align*}
\Vert v\otimes w \Vert_{\calM_1\otimes \calM_2} = \Big[\!\int_\mms \vert v\vert_\mu^2\,\vert w\vert_\mu^2\d\mu\Big]^{1/2}
\end{align*}
in $\smash{\calM_1^0\odot \calM_2^0}$ might well be infinite unless, for instance, $v$ or $w$ is bounded.

More intuitively,  we should think about $v\otimes w \in\smash{\calM_1^0\odot\calM_2^0}$ as section $x\mapsto v(x)\otimes w(x)$ over $\mms$, while $v\otimes_\rmH w\in\calM_1\otimes_\rmH\calM_2$ is interpreted as section $(x,y) \mapsto v(x)\otimes w(y)$ over $\mms^2$. The latter point of view is not relevant in this work, but is crucial e.g.~in obtaining a spectral representation of the heat kernel on $1$-forms on compact $\RCD^*(K,N)$ spaces \cite[Thm.~6.11]{braun2020}, $K\in\R$ and $N\in [1,\infty)$.
\end{remark}

Lastly, we introduce the concept of \emph{symmetric} and \emph{antisymmetric parts} in the case $\calM := \calM_1 = \calM_2$. Denote by $\smash{A^\top\in \calM^{\otimes 2}}$ the \emph{transpose} of $A\in\calM^{\otimes 2}$ as defined in \cite[Sec.~1.5]{gigli2018}. For instance, for bounded $v,w\in\calM$ we have
\begin{align}\label{Eq:Transpose}
(v\otimes w)^\top = w\otimes v.
\end{align}
It is an involutive module isometric isomorphism. We shall call $A\in \calM^{\otimes 2}$ \emph{symmetric} if $\smash{A=A^\top}$ and \emph{antisymmetric} if $\smash{A = -A^\top}$. We write $\smash{\calM_\sym^{\otimes 2}}$ and $\smash{\calM_\asym^{\otimes 2}}$ for the subspaces of symmetric and antisymmetric elements in $\calM^{\otimes 2}$, respectively. These are closed and pointwise $\mu$-a.e.~orthogonal w.r.t.~$:_\mu$. As usual, for every $A\in\calM^{\otimes 2}$ there exist a unique $\smash{A_\sym\in\calM^{\otimes 2}_\sym}$, the \emph{symmetric part} of $A$, and a unique $\smash{A_\asym\in\calM_\asym^{\otimes 2}}$, the \emph{antisymmetric part} of $A$, such that
\begin{align*}
A = A_\sym + A_\asym.
\end{align*}
In particular, we have
\begin{align}\label{Eq:sym plus asym}
\big\vert A\big\vert_{\HS,\mu}^2 = \big\vert A_\sym\big\vert_{\HS,\mu}^2 + \big\vert A_\asym\big\vert_{\HS,\mu}^2\quad\mu\text{-a.e.}
\end{align}
Next, we present a duality formula for symmetric parts crucially exploited later in \autoref{Le:Pre.Bochner}. If $D\subset\calM$ is a set of bounded elements generating $\calM$ in the sense of \autoref{Sub:Local dimension}, then $\{v\otimes v : v\in D\}$ generates $\smash{\calM_\sym^{\otimes 2}}$, and after \cite[Prop.~1.4.9]{gigli2018} for every $A\in\calM^{\otimes 2}$ we have the duality formula
\begin{align}\label{Eq:Duality formula symm part I}
\begin{split}
\big\vert A_\sym\big\vert_{\HS,\mu}^2 &= \mu\text{-}\!\esssup\!\Big\lbrace 2A: \sum_{j=1}^m v_j\otimes v_j - \Big\vert\!\sum_{j=1}^m v_j\otimes v_j\Big\vert_{\HS,\mu}^2:\\
&\qquad\qquad m\in\N,\ v_1,\dots,v_m\in D\Big\rbrace.
\end{split}
\end{align}

\subsubsection{Traces} Let $\calM$ be a Hilbert module over $\mms$. In terms of local bases outlined in \autoref{Sub:Local dimension}, it is possible to define the \emph{trace} of an element $\smash{A\in\calM_\sym^{\otimes 2}}$. 

As usual, Gram--Schmidt orthonormalization combined with \cite[Thm.~1.4.11]{gigli2018} entails the following. Denoting by $(E_n)_{n\in\N\cup\{\infty\}}$ the dimensional decomposition of $\calM$ according to \autoref{Th:Dimensional decomposition},  for every $n\in\N$ and every Borelian $E\subset E_n$ with $\mu[E] \in (0,\infty)$, there exists a basis $\smash{\{e_1^n,\dots,e_n^n\}\subset\calM\big\vert_E}$ of $\calM$ on $E$ with
\begin{align}\label{Eq:Obstr}
\big\langle e_i^n,e_j^n\big\rangle_\mu = \delta_{ij}\quad\mu\text{-a.e.}\quad\text{on }E
\end{align}
for every $i,j\in\{1,\dots,n\}$. Moreover, for every Borelian $E\subset E_\infty$ with $\mu[E] \in (0,\infty)$ there exists a sequence $\smash{(e_i^\infty)_{i\in\N}}$ in $\smash{\calM\big\vert_E}$ which generates $\calM$ on $E$ and satisfies \eqref{Eq:Obstr} for every $i,j\in\N$ and $n:=\infty$. Any such $\smash{\{e_1^n,\dots,e_n^n\}}$, $n\in\N$, or $\smash{(e_i^\infty)_{i\in\N}}$ is called a \emph{pointwise orthonormal basis} of $\calM$ on $E$. In particular, for every $E\subset E_n$ as above, $n\in\N\cup\{\infty\}$ and every $v\in\calM$ we can write
\begin{align*}
\One_E\,v = \sum_{i=1}^n \big\langle v, e_i^n\big\rangle_\mu\,e_i^n.
\end{align*}

Lastly, if $E\subset E_n$ is such a Borel set, $n\in\N\cup\{\infty\}$, we define
\begin{align}\label{Eq:Pointwise trace}
\One_E\,\tr A := \sum_{i=1}^n A : (e_i^n\otimes e_i^n).
\end{align}
This does not depend on the choice of the pointwise orthonormal basis.

\subsubsection{Exterior products}\label{Sub:Exterior products} Let $\calM$ be a Hilbert module and $k\in\N_0$. Set $\Lambda^0\calM^0 := \Ell^0(\mms,\mu)$ and, for $k\geq 1$, let $\Lambda^k\calM^0$ be the ``exterior product'' constructed by suitably factorizing $\smash{(\calM^0)^{\odot k}}$ \cite[Sec.~1.5]{gigli2018}. The representative of $v_1\odot\dots\odot v_k$, $v_1,\dots,v_k\in\calM^0$, in $\Lambda^k\calM^0$ is written $v_1\wedge\dots\wedge v_k$. $\smash{\Lambda^k\calM^0}$ naturally comes with a multiplication $\cdot\colon \Ell^0(\mms,\mu)\times \Lambda^k\calM^0\to\Lambda^k\calM^0$ via
\begin{align*}
f\,(v_1\wedge\dots v_k) := (f\,v_1) \wedge\dots \wedge v_k = \dots = v_1\wedge\dots\wedge (f\,v_k)
\end{align*}
and a pointwise scalar product $\langle\cdot,\cdot\rangle_\mu\colon (\Lambda^k\calM^0)^2\to\Ell^0(\mms,\mu)$ defined by
\begin{align}\label{Eq:Ptw scalar prod ext}
\langle v_1\wedge\dots\wedge v_k, w_1\wedge\dots \wedge w_k\rangle_\mu := \det\big[\langle v_i,w_j\rangle_\mu\big]_{i,j\in\{1,\dots,k\}}
\end{align}
 up to a factor $k!$, both extended to $\Lambda^k\calM^0$ by (bi-)linearity. Then $\smash{\langle\cdot,\cdot\rangle_\mu}$ is bilinear, $\mu$-a.e.~nonnegative definite, symmetric, and local in both components.

\begin{remark} Given any $k,k'\in\N_0$, the map assigning to $v_1\wedge\dots\wedge v_k\in\Lambda^k\calM^0$ and $\smash{w_1\wedge\dots\wedge w_{k'}\in\Lambda^{k'}\calM^0}$ the element $\smash{v_1\wedge\dots\wedge v_k \wedge w_1\wedge\dots\wedge w_{k'}\in\Lambda^{k+k'}\calM^0}$ can and will be uniquely extended by bilinearity and continuity to a bilinear map $\wedge \colon \smash{\Lambda^k\calM^0}\times\smash{\Lambda^{k'}\calM^0}\to \smash{\Lambda^{k+k'}\calM^0}$ termed \emph{wedge product} \cite[p.~47]{gigli2018}. If $k=0$ or $k'=0$, it simply corresponds to multiplication of elements of $\smash{\Lambda^{k'}\calM^0}$ or $\smash{\Lambda^k\calM^0}$, respectively, with functions in $\Ell^0(\mms,\mu)$ according to \eqref{Eq:Ptw scalar prod ext}.
\end{remark}

By a slight abuse of notation, define the map $\smash{\vert\cdot\vert_\mu\colon\Lambda^k\calM^0}\to\Ell^0(\mms,\mu)$ by
\begin{align*}
\vert\omega\vert_\mu := \sqrt{\langle \omega,\omega\rangle_\mu}.
\end{align*}
It obeys the $\mu$-a.e.~triangle inequality and is homogeneous w.r.t.~multiplication with $\Ell^0(\mms,\mu)$-functions \cite[p.~47]{gigli2018}.

It follows that the map $\Vert \cdot\Vert_{\Lambda^k\calM}\colon \Lambda^k\calM^0\to [0,\infty]$ defined by
\begin{align*}
\Vert \omega\Vert_{\Lambda^k\calM} := \big\Vert \vert \omega\vert_\mu\big\Vert_{\Ell^2(\mms,\mu)}
\end{align*}
has all properties of a norm except that $\Vert \omega\Vert_{\Lambda^k\calM}$ might be infinite.

\begin{definition} The \emph{\textnormal{(}$k$-fold\textnormal{)} exterior product} $\smash{\Lambda^k\calM}$ is defined as the completion w.r.t.~$\Vert\cdot\Vert_{\Lambda^k\calM}$ of the subspace consisting of all $\omega\in\Lambda^k\calM^0$ such that $\Vert \omega\Vert_{\Lambda^k\calM} < \infty$.
\end{definition}

The space $\Lambda^k\calM$ naturally becomes a Hilbert module and, if $\calM$ is separable, is separable as well \cite[p.~47]{gigli2018}.

\section{Cotangent module}\label{Sec:Cotangent module}

In this chapter, following \cite[Sec.~2.2]{gigli2018} we discuss a key object of our treatise, namely the \emph{cotangent module} $\Ell^2(T^*\mms)$,  i.e.~the space of differential $1$-forms that are square-integrable in a certain ``universal'' sense made precise in \autoref{Th:Module structure}.

\subsection{The construction}\label{Sub:Construction}

Define the \emph{pre-cotangent module} $\PCM$ by
\begin{align*}
\begin{split}
\PCM &:= \Big\lbrace (f_i,A_i)_{i\in\N} : (A_i)_{i\in\N}\text{ partition of }\mms \text{ in }\Borel(\mms),\\
&\qquad\qquad (f_i)_{i\in\N} \text{ in } \F_\rme,\  \sum_{i\in\N} \bdmu_{f_i}[A_i] < \infty\Big\rbrace.
\end{split}
\end{align*}
Moreover, define a relation $\rmR$ on $\PCM$ by declaring that $(f_i,A_i)_{i\in\N} \,\rmR\, (g_j, B_j)_{j\in\N}$ if and only if $\smash{\bdmu_{f_i-g_j}[A_i\cap B_j] = 0}$ for every $i,j\in\N$. $\rmR$ is in fact an equivalence relation by  \autoref{Th:Properties energy measure}. The equivalence class of an element $(f_i,A_i)_{i\in\N} \in \PCM$ w.r.t.~$\rmR$ is shortly denoted by $[f_i,A_i]$. As made precise in \autoref{Th:Module structure} and  \autoref{Def:Differential}, we think of $[f_i,A_i]\in \PCM/\rmR$ as the $1$-form which equals $\rmd f_i$ on $A_i$ for every $i\in\N$ in a certain ``universal'' a.e.~sense. 

$\PCM/\rmR$ becomes a vector space via the well-defined operations
\begin{align*}
	[f_i,A_i] + [g_j,B_j] &:= [f_i + g_j, A_i\cap B_j],\\
	\lambda\,[f_i,A_i] &:= [\lambda\, f_i, A_i]
\end{align*}
for every $[f_i,A_i], [g_j,B_j]\in\PCM/\rmR$ and every $\lambda\in\R$. 

In an analogous way, we define an action of $\SF(\mms)$ on $\PCM/\rmR$ as follows. Let $h\in\SF(\mms)$ and $[f_i,A_i]\in\PCM/\rmR$. Write
\begin{align}\label{Eq:simple function h}
h = \One_{B_1}\,h_1 + \dots + \One_{B_k}\,h_k
\end{align}
where $k\in\N$, $B_1,\dots,B_k\in\Borel(\mms)$ are disjoint, and $h_1,\dots,h_k\in\R$. Define
\begin{align}\label{Eq:SF mult}
h\,[f_i,A_i] := [h_j\,f_i, A_i\cap B_j],
\end{align}
where we set $B_j := \emptyset$ and $h_j := 0$ for every integer $j > k$. It is straightforward to verify that this definition is well-defined, independent of the particular way of writing $h$, and gives rise to a bilinear map $\SF(\mms)\times\PCM/\rmR \to \PCM/\rmR$ such that for every $[f_i,A_i]\in \PCM/\rmR$ and every $h,k\in \SF(\mms)$, 
\begin{align*}
(h\,k)\,[f_i,A_i] &= h\,\big(k\,[f_i,A_i]\big),\\
\One_\mms\,[f_i,A_i] &= [f_i,A_i].
\end{align*}

Lastly, by \autoref{Th:Properties energy measure}, the map $\Vert\cdot\Vert_{\Ell^2(T^*\mms)} \colon \PCM/\rmR\ \to [0,\infty)$ given by
\begin{align*}
\big\Vert [f_i,A_i]\big\Vert_{\Ell^2(T^*\mms)}^2 := \sum_{i\in\N} \bdmu_{f_i}[A_i]
\end{align*}
constitutes a norm on $\PCM/\rmR$.

\begin{definition}\label{Def:Cotangent module} We define the Banach space $\smash{(\Ell^2(T^*\mms), \Vert\cdot\Vert_{\Ell^2(T^*\mms)})}$ as the completion of $\smash{(\PCM/\rmR, \Vert \cdot \Vert_{\Ell^2(T^*\mms)})}$. The pair $\smash{(\Ell^2(T^*\mms), \Vert\cdot\Vert_{\Ell^2(T^*\mms)})}$ or simply $\Ell^2(T^*\mms)$ is henceforth called \emph{cotangent module}, and the elements of $L^2(T^*\mms)$ are called \emph{cotangent vector fields} or \emph{\textnormal{(}differential\textnormal{)} $1$-forms}.
\end{definition}

\begin{remark} The name ``cotangent \emph{module}'' for $\Ell^2(T^*\mms)$ is justified by \autoref{Th:Module structure} below. We point out for now that the notation $\Ell^2(T^*\mms)$ is purely formal since we did and do not define any kind of cotangent bundle $T^*\mms$. It rather originates in the analogy of $\Ell^2(T^*\mms)$ with the space of $\Ell^2$-sections of the cotangent bundle $T^*\mms$ in the smooth setting described in \autoref{Re:Comp smooth} below. On the other hand, by the structural characterization of Hilbert modules as direct integral of measurable fields of certain Hilbert spaces $(\calH_x)_{x\in\mms}$, see e.g.~\cite[p.~4381]{hinz2013} and \autoref{Re:Comp Dir sp} below, one could think of a fictive cotangent bundle as something ``a.e.~defined''. This point of view has been taken in the approaches \cite{baudoin2019, cipriani2003, eberle1999, hinz2013, hinz2015, ionescu2012}.
\end{remark}

A main ingredient to establish \autoref{Th:Module structure} is the following lemma, which is an immediate consequence of the above construction 
and the density of $\SF(\mms)$ in $\Ell_\infty(\mms)$ w.r.t.~the uniform norm.

\begin{lemma}\label{Pr:Group action} The map from $\SF(\mms)\times \PCM/\rmR$ into $\PCM/\rmR$ defined in \eqref{Eq:SF mult} extends continuously and uniquely to a bilinear map from  $\smash{\Ell_\infty(\mms)\times\Ell^2(T^*\mms)}$ into $\smash{\Ell^2(T^*\mms)}$ satisfying, for every $f,g\in\Ell_\infty(\mms)$ and every $\omega\in \Ell^2(T^*\mms)$, 
	\begin{align*}
		(f\,g)\,\omega &= f\,(g\,\omega),\\
		\One_\mms\,\omega &= \omega,\\
		\Vert f\,\omega\Vert_{\Ell^2(T^*\mms)} &\leq \sup\vert f\vert(\mms)\,\Vert \omega\Vert_{\Ell^2(T^*\mms)}.
	\end{align*}
\end{lemma}

\begin{theorem}[Module property]\label{Th:Module structure} For every $\Ch$-dominant Borel measure $\mu$ on $\mms$, the cotangent module $L^2(T^*\mms)$ is an $L^2$-normed $L^\infty$-module over $\mms$ w.r.t.~$\mu$ whose pointwise norm $\vert\cdot\vert_\mu$ satisfies, for every $[f_i,A_i]\in \PCM/\rmR$,
\begin{align}\label{Eq:ptw norm mu}
\big\vert [f_i,A_i]\big\vert_\mu = \sum_{i\in\N} \One_{A_i}\,\Gamma_\mu(f_i)^{1/2}\quad\mu\text{-a.e.}
\end{align}	
In particular, $\Ell^2(T^*\mms)$ is a Hilbert module w.r.t.~$\mu$.
\end{theorem}

\begin{proof} By $\Ch$-dominance, passing to $\mu$-versions in $\Ell_\infty(\mms)$ of any given $f\in\Ell^\infty(\mms,\mu)$ in the action of $\Ell_\infty(\mms)$ on $\Ell^2(T^*\mms)$ induces a well-defined bilinear map from $\Ell^\infty(\mms,\mu)\times \Ell^2(T^*\mms)$ to $\Ell^2(T^*\mms)$ which, thanks to \autoref{Pr:Group action}, turns $\Ell^2(T^*\mms)$ into an $\Ell^\infty$-premodule w.r.t.~$\mu$.
	
Now define the pointwise norm $\smash{\vert\cdot\vert_\mu\colon \PCM/\rmR\to\Ell^2(\mms,\mu)}$ by \eqref{Eq:ptw norm mu}. This map is clearly an isometry from $\PCM/\rmR$ to $\Ell^2(\mms,\mu)$, whence by continuous and unique extension to  $\Ell^2(T^*\mms)$, it will only be necessary to prove the required properties for $\smash{\vert\cdot\vert_\mu}$ from \autoref{Def:Modules} for elements of $\PCM/\rmR$ and $\SF(\mms)$, respectively. Indeed, for $[f_i,A_i]\in\PCM/\rmR$,  by Fubini's theorem,
\begin{align*}
\int_\mms \big\vert [f_i,A_i]\big\vert_\mu^2\d\mu = \sum_{i\in\N} \int_{A_i}\Gamma_\mu(f_i)\d\mu = \sum_{i\in\N} \bdmu_{f_i}[A_i]= \big\Vert [f_i,A_i]\big\Vert_{\Ell^2(T^*\mms)}^2.
\end{align*}
On the other hand, writing a given $h\in\SF(\mms)$ according to \eqref{Eq:simple function h},
\begin{align*}
\big\vert h\,[f_i,A_i]\big\vert_\mu  = \sum_{i,j\in\N} \One_{A_i}\,\One_{B_j}\,\vert h_j\vert\,\Gamma_\mu(f_i)^{1/2}= \vert h\vert\,\big\vert [f_i,A_i]\big\vert_\mu\quad\mu\text{-a.e.},
\end{align*}
which establishes the desired properties for $\vert\cdot\vert_\mu$.

The last statement follows since  $\Vert\cdot\Vert_{\Ell^2(T^*\mms)}$ satisfies the parallelogram identity, which is a  consequence of the bilinearity and symmetry of $\bdmu$.
\end{proof}

\begin{remark} Conceptually, one could alternatively construct $\Ell^2(T^*\mms)$ as follows. Given an $\Ch$-dominant $\mu$, define the ``$\Ell^0$-module'' --- put in quotes since it is not a priori induced by an $\Ell^\infty$-module --- $\smash{\Ell^0(T^*\mms)_\mu}$ w.r.t.~$\mu$ as completion w.r.t.~an appropriate distance constructed from $\smash{\Gamma_\mu^{1/2}}$, compare with \autoref{Sub:L0 modules}. Then restrict to the subspace of elements whose induced pointwise norm belongs to $\Ell^2(\mms,\mu)$. See \cite{eberle1999,gigli2020,hinz2013} for details. The advantage of our above approach is that it is clearer in advance that the resulting space does not depend on $\mu$.
\end{remark}

\subsection{Differential of a function in the extended domain}\label{Sub:Differential} $\Ell^2(T^*\mms)$ directly provides a  notion of a \emph{differential} acting on functions in $\F_\rme$. The behavior of a given element of $\Ell^2(T^*\mms)$ is completely determined by its interaction with differentials of functions in $\smash{\F_\rme}$, see \autoref{Th:Universal}. In turn, this will be used to phrase the calculus rules from \autoref{Th:Properties energy measure} at the $\Ell^\infty$-module level, see \autoref{Cor:Calculus rules d}.

We start with the following definition.

\begin{definition}\label{Def:Differential} The \emph{differential} of any function $f\in \F_\rme$ is defined by 
\begin{align*}
\rmd f:= [f,\mms],
\end{align*}
 where $[f,\mms]\in\PCM/\rmR$  is the representative of the sequence $(f_i,A_i)_{i\in\N}$ given by $f_1 := f$, $A_1 := \mms$, $f_i := 0$ and $A_i := \emptyset$ for every $i \geq 2$.
\end{definition}

As usual, we call a $1$-form $\omega\in\Ell^2(T^*\mms)$ \emph{exact} if, for some $f\in\F_\rme$,
\begin{align*}
\omega = \rmd f.
\end{align*}

The differential $\rmd$ is a linear operator on $\F_\rme$. By \eqref{Eq:ptw norm mu}, w.r.t.~the $\Ell^\infty$-module structure induced by \emph{any} $\Ch$-dominant $\mu$ according to \autoref{Th:Module structure},
\begin{align}\label{Eq:TRE}
\vert \rmd f\vert_\mu = \Gamma_\mu(f)^{1/2}\quad\mu\text{-a.e.}
\end{align}
holds for every $f\in \F_\rme$.

\subsubsection{Universality} To derive the  calculus rules from \autoref{Cor:Calculus rules d} relying on  \autoref{Th:Properties energy measure}, we need to prove the following density property. It is independent of the particular choice of the $\Ch$-dominant Borel measure $\mu$ on $\mms$ that induces the background $\Ell^\infty$-module structure on $\Ell^2(T^*\mms)$ according to \autoref{Th:Module structure}. A related  version involving the local density of ``regular vector fields'' which relies on the second order calculus developed in \autoref{Pt:II} is stated in \autoref{Le:Regular vfs generation} below. 

\begin{lemma}\label{Pr:Generators cotangent module} The cotangent module $\Ell^2(T^*\mms)$ is generated, in the sense of $\Ell^\infty$-modules, by $\smash{\rmd\,\F_\rme}$ and by $\smash{\rmd\, \F}$. In particular, $\Ell^2(T^*\mms)$ is separable.
\end{lemma}

\begin{proof} By the definition of the norm $\smash{\Vert \cdot \Vert_{\Ell^2(T^*\mms)}}$ on $\PCM/\rmR$ and after passing to the limit, we see that the family $\calS$ of finite linear combinations of objects $\One_A\d f$, $A\in \Borel(\mms)$ and $f\in \F_\rme$, is dense in $\PCM/\rmR$. Given that $\PCM/\rmR$ is dense in $\Ell^2(T^*\mms)$ by construction of the latter space, the first claim follows.

Next, given any $A\in\Borel(\mms)$ and $\smash{f\in \F_\rme}$, by definition of $\smash{\F_\rme}$ and \autoref{Th:Properties energy measure} there exists a sequence $(f_n)_{n\in\N}$ in $\F$ such that
\begin{align*}
\lim_{n\to\infty}\big\Vert \One_A\,\rmd(f_n - f)\big\Vert_{\Ell^2(T^*\mms)}^2 &= \lim_{n\to\infty} \bdmu_{f_n-f}[A] \leq \lim_{n\to\infty}\Ch(f_n-f) =0.
\end{align*}
Thus, $\rmd\,\F$ generates $\calS$, and by the argument from the first part of the proof and a diagonal procedure, $\smash{\rmd\,\F}$ generates $\smash{\Ell^2(T^*\mms)}$.

To see the separability of $\Ell^2(T^*\mms)$, note that $\rmd\,\F$ is a separable subset since $\smash{\Vert\rmd f\Vert_{\Ell^2(T^*\mms)}\leq \Vert f\Vert_{\F}}$ for every $f\in \F$ and by  the separability of $\F$ granted by \autoref{Pr:Extended domain props}. In particular, $\Ell^2(T^*\mms)$ is separable by \cite[Prop.~1.4.10]{gigli2018}.
\end{proof}

\begin{corollary}\label{Cor:Hino} Let $(E_n)_{n\in{\N\cup\{\infty\}}}$ be the dimensional decomposition of $\Ell^2(T^*\mms)$, seen as an $\Ell^2$-normed $\Ell^\infty$-module w.r.t.~a given $\Ch$-dominant $\mu$. Denote by $\rmp$ the pointwise index from \autoref{Re:Hino}. Then for every $n\in \N\cup\{\infty\}$,
\begin{align*}
\rmp = n\quad\mu\text{-a.e.}\quad\text{on }E_n.
\end{align*}
\end{corollary}

\begin{theorem}[Universal property]\label{Th:Universal} Let $\mu$ be an $\Ch$-dominant Borel measure on $\mms$. Let $(\calM,\mathcall{d})$ be a tuple consisting of an $\Ell^2$-normed $\Ell^\infty$-module over $\mms$ w.r.t.~$\mu$ with pointwise norm denoted by $_\mu\vert\cdot\vert$ and a linear map $\smash{\mathcall{d}\colon \F_\rme \to \calM}$ such that
\begin{enumerate}[label=\textnormal{\alph*.}]
\item\label{La:gen} $\calM$ is generated by $\smash{\rmd\,\F_\rme}$, and
\item\label{La:Carre} for every $f\in\smash{\F_\rme}$,
\begin{align*}
_\mu\vert \mathcall{d}f\vert = \Gamma_\mu(f)^{1/2}\quad\mu\text{-a.e.}
\end{align*}
\end{enumerate}
Then there exists a unique module isomorphism $\Phi\colon \Ell^2(T^*\mms) \to \calM$ such that
\begin{align}\label{Eq:Commutation d}
\Phi\circ \rmd = \mathcall{d}.
\end{align}
\end{theorem}

\begin{proof} First, we observe that for every $f,g\in \F_\rme$ and every Borel subset $A\subset\mms$, we have $\One_A\d f = \One_A\d g$ w.r.t.~the $\Ell^\infty$-module structure of $\Ell^2(T^*\mms)$ from \autoref{Th:Module structure} if and only if $\One_A\,\mathcall{d}f = \One_A\,\mathcall{d}g$ w.r.t.~the given $\Ell^\infty$-module structure of $\calM$ (both w.r.t.~$\mu$). Indeed, this follows from  combining \autoref{Th:Module structure} with \ref{La:Carre}:
\begin{align}\label{Eq:Gluck}
\One_A\,\big\vert\rmd (f-g)\big\vert_\mu = \One_A\,\Gamma_\mu(f-g)^{1/2} = \One_A\textcolor{white}{\big\vert_{\textcolor{black}{\mu}}}\big\vert\mathcall{d}(f-g)\big\vert\quad\mu\text{-a.e.}
\end{align}

Define the real vector space $\calQ\subset\calM$ by
\begin{align*}
\calQ &:= \Big\lbrace \sum_{i\in\N} \One_{A_i}\,\mathcall{d} f_i : (A_i)_{i\in\N}\text{ partition of }\mms\text{ in } \Borel(\mms),\\
&\qquad\qquad (f_i)_{i\in\N} \text{ in } \F_\rme,\ \sum_{i\in\N} \big\Vert\One_{A_i}\,_\mu\vert \mathcall{d} f_i\vert\big\Vert_{\Ell^2(\mms,\mu)}^2 < \infty \Big\rbrace,
\end{align*}
as well as the map $\Phi\colon \PCM\to \calQ$ by
\begin{align}\label{Eq:Initial def}
\Phi \sum_{i\in\N} \One_{A_i}\d f_i := \sum_{i\in\N}\One_{A_i}\,\mathcall{d}f_i.
\end{align}
By \eqref{Eq:Gluck}, $\Phi$ is well-defined, it is linear, and by definition \eqref{Eq:Commutation d} holds. Moreover, the relation \eqref{Eq:Initial def} will entail the claimed uniqueness of the continuous extension of $\Phi$ to $\Ell^2(T^*\mms)$ as a byproduct. To prove the actual existence of such an extension, we first observe that by \eqref{Eq:Gluck},
\begin{align*}
\textcolor{white}{\Big\vert_{\textcolor{Black}{\mu}}}\Big\vert\!\sum_{i\in\N}\One_{A_i}\,\mathcall{d}f_i\Big\vert = \Big\vert\!\sum_{i\in\N}\One_{A_i}\d f_i\Big\vert_\mu\quad\mu\text{-a.e.}  
\end{align*}
In particular, $\Phi$ is a norm isometry, i.e.
\begin{align*}
\big\Vert\Phi\,\omega\big\Vert_\calM = \Vert \omega\Vert_{\Ell^2(T^*\mms)}
\end{align*}
for every $\omega\in \PCM$. By the density of $\PCM/\rmR$ in $\Ell^2(T^*\mms)$, $\Phi$ extends uniquely and continuously to a linear isometry $\smash{\Phi\colon \Ell^2(T^*\mms)\to\calM}$ which also preserves the respective pointwise norms. The latter map is injective, has closed image and is actually surjective by the trivial identity $\Phi(\PCM)=\calQ$ and item \ref{La:gen} above. Thus, $\Phi$ is a Banach space isomorphism. The $\Ell^\infty$-linearity of $\Phi$ finally follows by continuity and uniqueness of the extension of $\Phi$ and the density of $\SF(\mms)$ in $\Ell^\infty(\mms,\mu)$ w.r.t.~the uniform norm after deriving the  elementary identity
\begin{align*}
\Phi\, \One_B\,\omega = \One_B\,\Phi\,\omega
\end{align*} 
for every $\omega\in\PCM$ and every $B\in\Borel(\mms)$ from \eqref{Eq:Initial def}.
\end{proof}

\subsubsection{Properties of the differential} Now we establish elementary calculus rules associated with the differential $\rmd$ from \autoref{Def:Differential}. 

We start with the following ``closedness'' property of it w.r.t.~$\meas$ [sic].

\begin{lemma}\label{Le:d closed} For every sequence $(f_n)_{n\in\N}$ in $\F_\rme$ which converges to $f\in\Ell^0(\mms)$ pointwise $\meas$-a.e.~in such a way that the sequence $(\rmd f_n)_{n\in\N}$ converges to $\omega\in \Ell^2(T^*\mms)$ in $\Ell^2(T^*\mms)$, we have $\smash{f\in \F_\rme}$ as well as
\begin{align}\label{Eq:df=omega}
\rmd f =\omega.
\end{align}
In particular, if $f_n \in \F$ for every $n\in\N$, $f_n \rightharpoonup f$ in $\Ell^2(\mms)$ and $\rmd f_n \rightharpoonup \omega$ in $\Ell^2(T^*\mms)$ as $n\to\infty$ for some $f\in \Ell^2(\mms)$ and $\omega\in \Ell^2(T^*\mms)$, then $f\in \F$, and the identity \eqref{Eq:df=omega} holds accordingly.
\end{lemma}

\begin{proof} Since $(f_n)_{n\in\N}$ is $\Ch$-bounded by assumption, by \autoref{Pr:Extended domain props} we directly obtain that $\smash{f\in \F_\rme}$. Therefore, $\smash{f-f_m\in \F_\rme}$ for every $m\in\N$, and again by \autoref{Pr:Extended domain props} and using that $(f_n)_{n\in\N}$ is $\Ch$-Cauchy by \eqref{Eq:TRE},
\begin{align*}
\limsup_{m\to\infty} \Vert \rmd(f - f_m)\Vert_{\Ell^2(T^*\mms)} \leq \limsup_{m\to\infty}\liminf_{n\to\infty} \Vert \rmd(f_n-f_m)\Vert_{\Ell^2(T^*\mms)}=0.
\end{align*}
It follows that $\rmd f_m \to \rmd f$ in $\Ell^2(T^*\mms)$ as $m\to\infty$, whence $\rmd f =\omega$.

The second claim is now due to Mazur's lemma, up to possibly passing to suitable pointwise $\meas$-a.e.~converging subsequences.
\end{proof}

``Expected'' calculus rules for $\rmd$ hold if the measure $\mu$ under consideration is minimal $\Ch$-dominant (recall from \autoref{Sub:Carré} that such $\mu$ does not charge $\Ch$-polar sets). In particular, all identities in \autoref{Cor:Calculus rules d} below make sense by the definition of $\rmd$ in terms of $\bdmu$ and   \autoref{Th:Properties energy measure}.

\begin{proposition}\label{Cor:Calculus rules d} Let $\mu$ be a minimal $\Ch$-dominant Borel measure on $\mms$. Then  w.r.t.~the $\Ell^\infty$-module structure of $\Ell^2(T^*\mms)$ from \autoref{Th:Module structure} induced by $\mu$, the  following properties hold.
\begin{enumerate}[label=\textnormal{\textcolor{black}{(}\roman*\textcolor{black}{)}}]
\item\label{La:Uno} \emph{Locality.} For every $f,g\in \F_\rme$,
\begin{align*}
\One_{\{\widetilde{f} = \widetilde{g}\}}\d f = \One_{\{\widetilde{f}=\widetilde{g}\}}\d g.
\end{align*}
\item\label{La:Due} \emph{Chain rule.} For every $f\in \F_\rme$ and every $\Leb^1$-negligible Borel set $C\subset\R$, 
\begin{align*}
\One_{\widetilde{f}^{-1}(C)}\d f=0.
\end{align*}
In particular, for every $\varphi\in\Lip(\R)$,
\begin{align*}
\rmd(\varphi\circ f) = \big[\varphi'\circ \widetilde{f}\big]\d f,
\end{align*}
where the derivative $\smash{\varphi'}$ is defined arbitrarily on the intersection of the set of non-differentiability points of $\varphi$ with the image of $\smash{\widetilde{f}}$.
\item\label{La:Quattro} \emph{Leibniz rule.} For every $\smash{f,g\in \F_{\eb}(\mms)}$,
\begin{align*}
\rmd(f\,g) = \widetilde{f}\d g + \widetilde{g}\d f.
\end{align*}
\end{enumerate}
\end{proposition}

\begin{proof} By the linearity of $\rmd$, it  is sufficient to consider the case where $g$ vanishes identically in point \ref{La:Uno}. Up to passing to a Borelian $\Ch$-quasi-closed representative of the $\Ch$-quasi-closed set $\smash{\{\widetilde{f}=0\}}$ \cite[Lem.~2.5]{delloschiavo2020}, by \autoref{Th:Module structure}  we have $\smash{\One_{\{\widetilde{f}=0\}}\d f = 0}$ if and only if $\Gamma_\mu(f) = 0$ $\mu$-a.e.~on $\smash{\{\widetilde{f}=0\}}$. By  \autoref{Th:Properties energy measure}, the assertion thus follows from the identities
\begin{align*}
\big\Vert \One_{\{\widetilde{f}=0\}}\d f\big\Vert_{\Ell^2(T^*\mms)} =\bdmu_f\big[\{\widetilde{f}=0\}\big]^{1/2} = 0.
\end{align*}

The first point of \ref{La:Due} follows similar lines, taking \autoref{Re:CDC} into account.

To prove the second claim in \ref{La:Due}, using  \autoref{Th:Properties energy measure} we note that
\begin{align*}
\big\langle\rmd(\varphi\circ f),\rmd g\big\rangle_\mu = \big\langle\big[\varphi'\circ \widetilde{f}\big]\d f,\rmd g\big\rangle_\mu\quad\mu\text{-a.e.}
\end{align*}
for every $\smash{g\in \F_\rme}$. \autoref{Pr:Generators cotangent module} allows us to extend this property to arbitrary $\omega\in\Ell^2(T^*\mms)$ in place of $\rmd g$. The Hilbert space structure of $\Ell^2(T^*\mms)$ from \autoref{Th:Module structure}, the fact that $\langle\cdot,\cdot\rangle_\mu$ induces the scalar product on $\Ell^2(T^*\mms)$ by integration w.r.t.~$\mu$ as well as the arbitrariness of $\omega$ terminate the proof of \ref{La:Due}.

Item \ref{La:Quattro} follows analogously to the previous argument.
\end{proof}

\begin{remark}\label{Re:Min E dom} If $\meas$ is minimal $\Ch$-dominant, all identities in \autoref{Cor:Calculus rules d} can equivalently be phrased with $f$ and $g$ in place of their $\Ch$-q.c.~$\meas$-versions $\smash{\widetilde{f}}$ and $\widetilde{g}$, respectively. In particular, we shall later refer to the corresponding modified version of  \autoref{Cor:Calculus rules d}  without further comment.
\end{remark}

\subsection{Some remarks on the axiomatization} 

\begin{remark}[Compatibility with the smooth case]\label{Re:Comp smooth} If $\mms$ is a Riemannian manifold with boundary, in the setting of \autoref{Ex:Mflds}, $\Ell^2(T^*\mms)$ coincides with the space of $\Ell^2$-sections of the cotangent bundle $T^*\mms$ w.r.t.~$\meas$, and $\rmd$ is the usual $\meas$-a.e.~defined differential for, say, boundedly supported Lipschitz functions on $\mms$.
\end{remark}

\begin{remark}[Compatibility with {\cite{gigli2018}}]\label{Re:Gigli comp} By the $\meas$-a.e.~equality between carré du champ and minimal relaxed gradient outlined in \autoref{Ex:mms}, our approach is fully compatible with the one from \cite[Sec.~2.2]{gigli2018} for \emph{infinitesimally Hilbertian} metric measure spaces $(\mms,\met,\meas)$. 

Indeed, let $\Ell^2(T^*\mms)_\calD$ be the $\Ell^2$-cotangent module constructed analogously to \autoref{Sub:Construction} w.r.t.~a given reference domain $\calD \supset \F$ whose elements all  belong to $\smash{\dot{\F}_\loc}$ and have finite energy measure (note that $\calD := \F_\rme$ in \autoref{Sub:Construction}). Then by \autoref{Th:Universal} and the locality properties from \autoref{Cor:Calculus rules d}, $\Ell^2(T^*\mms)_\calD$ and $\Ell^2(T^*\mms)$ coincide; compare with \cite[Prop.~4.1.6, Prop.~4.1.8]{gigli2020}.
\end{remark}

\begin{remark}[Compatibility with {\cite{baudoin2019, cipriani2003, eberle1999, hinz2013, hinz2015, ionescu2012}}]\label{Re:Comp Dir sp} The space $\calH$ of $1$-forms constructed in \cite[Ch.~2]{hinz2013} agrees with $\Ell^2(T^*\mms)$. For $f\in\F\cap\Cont_0(\mms)$, the differential $\rmd f \in\Ell^2(T^*\mms)$ corresponds to the element $f\otimes \One_\mms \in \calH$ in \cite{hinz2013}. For instance, this follows by first proving that $\calH$ is an $\Ell^2$-normed $\Ell^\infty$-module w.r.t.~an $\Ch$-dominant $\mu$ \cite[Ass.~2.1]{hinz2013} which is generated by $\mathcall{d}(\F\cap \Cont_0(\mms))$, $\mathcall{d} := \cdot\otimes\One_\mms$, and applying \autoref{Th:Universal}. The former is already done somewhat implicitly in \cite[Ch.~2]{hinz2013}.
\end{remark}

\begin{remark}[Non-symmetric Dirichlet forms]\label{Re:Non-symm} The concepts presented above could be generalized to the case when $\Ch$ is non-symmetric in the sense of \cite[Def.~I.2.4]{ma1992} and satisfies the so-called \emph{strong sector condition} from (I.2.4) in \cite{ma1992}. (The latter always holds if $\Ch$ is symmetric.) See \cite{kuwae1998,ma1992} for further reading. In particular, by \cite[Thm.~3.5]{mataloni1999} and owing to the transfer method \cite{chen1994}, there exists a bilinear map $\smash{\bdnu\colon \F^2\to \Meas_{\bR}^\pm(\mms)_\Ch}$ such that
\begin{align*}
\Ch(f,g) = \bdnu_{f,g}[\mms]
\end{align*}
for every $f,g\in \F$. Under the strong sector condition, $\bdnu$ can be extended to a non-relabeled bilinear map on $\smash{\F_\rme^2}$ by approximation. However,  
\begin{align*}
\bdnu = \bdmu
\end{align*}
holds on the diagonal of $\smash{\F_\rme^2}$, where solely in this remark, $\bdmu$ designates the energy measure of the symmetric part $\smash{\widetilde{\Ch}}$ of $\Ch$, see (I.2.1) in \cite{ma1992}. In particular, the construction of $\Ell^2(T^*\mms)$ --- and hence also the one from \autoref{Ch:Tangent module} below by \autoref{Th:Riesz theorem modules} --- only depends on $\smash{\widetilde{\Ch}}$. In other words, even when starting with a non-symmetric form $\Ch$, the cotangent module $\Ell^2(T^*\mms)$ will be a Hilbert space. 

This does not conflict with the not necessary infinitesimally Hilbertian setting for metric measure spaces in \cite[Ch.~2]{gigli2018}, although our approach is quite similar to \cite{gigli2018}, see \autoref{Re:Gigli comp}. The reason is that the underlying energy form in \cite{gigli2018} is neither a priori induced by, nor a posteriori can be turned into a bilinear form in a reasonable way. It is rather given in terms of the ``diagonal energy measure'' $\vert\rmD \cdot\vert^2\,\meas$, where $\vert\rmD\cdot\vert^2$ is the so-called \emph{minimal relaxed gradient} introduced in \cite[Def.~4.2]{ambrosio2014a}. Said differently, the Hilbertianity of $\Ell^2(T^*\mms)$ in our setting simply comes since we have started with a bilinear rather than merely a $2$-homogeneous  form.
\end{remark}

\section{Tangent module}\label{Ch:Tangent module}

\subsection{Tangent vector fields and gradients}\label{Sec:TMod} Now we dualize the concept of cotangent module introduced in \autoref{Sec:Cotangent module} to define the so-called \emph{tangent module} $\Ell^2(T\mms)$, i.e.~an appropriate space of vector fields on $\mms$. 

At the $\Ell^\infty$-module level, different structures may arise depending on the choice of $\Ch$-dominant Borel measure $\mu$ in \autoref{Th:Module structure}. In fact, although the tangent module $\Ell^2(T\mms)_\mu$ with induced gradient $\nabla_\mu$ as in \autoref{Def:Tangent module} and \autoref{Def:Gradient} make perfect sense for every such $\mu$, from now on we assume the following.

\begin{assumption}\label{As:Gamma-operator} $\Ch$ admits a carré du champ w.r.t.~$\meas$.
\end{assumption}

Unless explicitly stated otherwise, all $\Ell^\infty$-modules as well as their respective properties in the sequel are understood w.r.t.~$\meas$.

\begin{remark} 
Besides \autoref{As:Gamma-operator} being worked under throughout \autoref{Pt:II}, at the current stage it has two further advantages. First, by \autoref{Re:Min E dom} the expected calculus rules from \autoref{Cor:Calculus rules d} hold and transfer accordingly to its induced gradient $\nabla$. Second, \autoref{As:Gamma-operator} avoids  incompatibilities between the divergence operator induced by different choices of (minimal) $\Ch$-dominant reference measures and the Laplacian $\Delta$ associated to $\Ch$, see \autoref{Re:Incomp}.
\end{remark}

\begin{definition}\label{Def:Tangent module} The \emph{tangent module} $(\Ell^2(T\mms),\Vert\cdot\Vert_{\Ell^2(T\mms)})$ or simply $\Ell^2(T\mms)$ is
\begin{align*}
\Ell^2(T\mms) :=\Ell^2(T^*\mms)^*
\end{align*}
and it is endowed with the norm $\smash{\Vert\cdot\Vert_{\Ell^2(T\mms)}}$ induced by \eqref{Eq:dual norm}. The elements of $\Ell^2(T\mms)$ will be called \emph{vector fields}.
\end{definition}

As in \autoref{Sub:Linfty modles}, the pointwise pairing between $\omega\in \Ell^2(T^*\mms)$ and $X\in \Ell^2(T\mms)$ is denoted by $\omega(X)\in\Ell^1(\mms)$, and, by a slight abuse of notation, $\vert X\vert\in\Ell^2(\mms)$ denotes the pointwise norm of such an $X$. By \autoref{Pr:Generators cotangent module} and \autoref{Th:Riesz theorem modules}, $\Ell^2(T\mms)$ is a separable Hilbert module. 

Furthermore, in terms of the pointwise scalar product $\langle\cdot,\cdot\rangle$ on $\Ell^2(T^*\mms)$ and $\Ell^2(T\mms)$, respectively, \autoref{Th:Riesz theorem modules} allows us to introduce the  \emph{musical isomorphisms} $\sharp\colon \Ell^2(T^*\mms)\to\Ell^2(T\mms)$ and $\flat := \sharp^{-1}$ defined by
\begin{align}\label{Eq:Musical isos}
\big\langle \omega^\sharp, X\big\rangle := \omega(X) =: \big\langle X^\flat,\omega\big\rangle\quad\meas\text{-a.e.}
\end{align}

\begin{definition}\label{Def:Gradient} The \emph{gradient} of a  function $f\in \F_\rme$ is defined by
\begin{align*}
\nabla f := (\rmd f)^\sharp.
\end{align*}
\end{definition}

Observe from \eqref{Eq:Musical isos} that $\nabla f$, where $f\in\F_\rme$, is characterized as the unique element  $X\in \Ell^2(T\mms)$ which satisfies
\begin{align*}
\rmd f(X) = \vert \rmd f\vert^2 = \vert X\vert^2\quad\meas\text{-a.e.}
\end{align*}
This uniqueness may fail on ``non-Riemannian'' spaces --- that we do not consider here --- such as Finsler manifolds. Compare with \cite[Subsec.~2.3.1]{gigli2018}.

The gradient operator is clearly linear on $\F_\rme$ and closed in the sense of \autoref{Le:d closed}.  By \eqref{Eq:Musical isos} and \autoref{Cor:Calculus rules d}, all calculus rules from \autoref{Cor:Calculus rules d} transfer accordingly to the gradient. Moreover, \autoref{Pr:Generators cotangent module} ensures that $\Ell^2(T\mms)$ is generated, in the sense of $\Ell^\infty$-modules, by $\smash{\nabla\,\F_\rme}$  and by $\smash{\nabla\,\F}$.

\subsection{Divergences}\label{Sec:Divergences} Now we introduce and study two notions of \emph{divergence} of suitable elements of $\Ell^2(T\mms)$. The first in \autoref{Def:L2 div} is an $\Ell^2$-approach similar to \cite[Subsec.~2.3.3]{gigli2018}. The second in \autoref{Def:Measure-valued divergence} axiomatizes a measure-valued divergence $\DIV$. Both approaches are compatible in the sense of \autoref{Pr:Div comp}.

\subsubsection{$\Ell^2$-divergence} The following is similar to \cite[Def.~2.3.11]{gigli2018}. See also \cite[Ch.~3]{hinz2013} for a similar construction for regular Dirichlet spaces.

\begin{definition}\label{Def:L2 div} We define the space $\Dom(\div)$ to consist of all $X\in \Ell^2(T\mms)$ for which there exists a function $f\in \Ell^2(\mms)$ such that for every $h\in \F$,
\begin{align*}
-\int_\mms h\,f\d \meas = \int_\mms \d h(X)\d\meas.
\end{align*}
In case of existence, $f$ is unique, called the \emph{divergence} of $X$ and denoted by $\div X$.
\end{definition}

The uniqueness comes from the density of $\F$ in $\Ell^2(\mms)$. Note that $\div$ is a linear operator on $\Dom(\div)$, which thus turns $\Dom(\div)$ into a vector space.

By \eqref{Eq:IBP Laplacian}, we have $\nabla\,\Dom(\Delta)\subset\Dom(\div)$ and
\begin{align}\label{Eq:div nabla = Delta}
\div \nabla f=\Delta f\quad\meas\text{-a.e.}
\end{align}
for every $f\in\Dom(\Delta)$. Moreover, employing the Leibniz rule in \autoref{Th:Properties energy measure}, one easily can verify that for every $X\in\Dom(\div)$ and every $\smash{f\in \F_\eb}$ with $\vert\rmd f\vert\in\Ell^\infty(\mms)$, we have $f\,X\in\Dom(\div)$ and
\begin{align}\label{Eq:Div lbnz rle}
\div(f\,X) = f\div X + \rmd f(X)\quad\meas\text{-a.e.}
\end{align}

\begin{remark}\label{Re:Incomp} In fact, \eqref{Eq:div nabla = Delta} and \eqref{Eq:Div lbnz rle} are the key reasons for  \autoref{As:Gamma-operator} at this first order level. Indeed, given a different $\Ch$-dominant $\mu$, the pairing $\langle \nabla_\mu f, X\rangle_\mu$ in \eqref{Eq:Div lbnz rle} would require a possible divergence $\div_\mu X$ for $X\in\Ell^2(T\mms)_\mu$ to belong to $\Ell^2(\mms,\mu)$, while  \eqref{Eq:div nabla = Delta} somewhat forces the identity $\mu =\meas$. 

One way to bypass this if \autoref{As:Gamma-operator} does not hold is to regard the $\Ell^2$-divergence as simply acting on $1$-forms instead of vector fields  \cite[Ch.~3]{hinz2013}.
\end{remark}

\subsubsection{Measure-valued divergence} The next definition has partly been inspired by the work  \cite[Def.~4.1]{buffa2019}. 

\begin{definition}\label{Def:Measure-valued divergence} We define the space $\Dom(\DIV)$ to consist of all $X \in \Ell^2(T\mms)$ for which there exists  $\smash{\nu\in \Meas_\sigmafinR^{\pm}(\mms)_\Ch}$ such that for every $h\in\smash{\F_\bc}$, 
	\begin{align}\label{Eq:Def prop Div}
		-\int_\mms \widetilde{h}\d\nu = \int_\mms \rmd h(X)\d\meas.
	\end{align}
	In case of existence, $\nu$ is unique, termed the \emph{measure-valued divergence} of $X$ and denoted by $\DIV X$.
\end{definition}

The uniqueness statement follows by density of $\F_\bc$ in $\F$ by quasi-regularity of $\Ch$. The divergence $\DIV$ is clearly a linear operator on\label{Eq:Dom TV}
\begin{align*}
\Dom_\TV(\DIV) := \big\lbrace X \in \Ell^2(T\mms) : \Vert\!\DIV X\Vert_\TV < \infty\big\rbrace,
\end{align*}
and the latter is a vector space. 

\subsubsection{Normal components and Gauß--Green's formula}\label{Sub:normal cpts} Before we proceed with various properties of the notions from \autoref{Def:L2 div} and \autoref{Def:Measure-valued divergence}, it is convenient to introduce some further notation.

\begin{definition}\label{Def:Normal component} Given $X\in\Dom(\DIV)$, its \emph{divergence} $\smash{\div_1 X\in\Ell^1_\loc(\mms)}$ and its \emph{normal component} $\smash{\norm X\in\Meas_\sigmafinR^\pm(\mms)_\Ch}$ are  defined by
\begin{align*}
\div_1 X &:=\frac{\rmd\!\DIV_\ll X}{\rmd \meas},\\
\norm X &:= -\DIV_\perp X.
\end{align*}
We define $\Dom_{L^2}(\DIV)$ as the space of all $X\in\Dom(\DIV)$ such that $\div_1 X\in\Ell^2(\mms)$.
\end{definition}

\autoref{Def:Normal component} is justified in the smooth setting according to the subsequent \autoref{Ex:Mflds with boundary} and \autoref{Re:Green's formula}.

\begin{example}\label{Ex:Mflds with boundary} Let $\mms$ be a Riemannian manifold with boundary $\partial\mms$. Recall that we denote by $\sfn$ the outward pointing unit normal vector field at $\partial\mms$. Then for every smooth vector field $X$ on $\mms$ with compact support and every $h\in\Cont_\comp^\infty(\mms)$, by Green's formula, see e.g.~\cite[p.~44]{lee1997}, we have
\begin{align*}
-\int_\mms \rmd h(X)\d\vol= \int_\mms h\div_\vol X\d\vol - \int_{\partial\mms} h\,\langle X,\sfn\rangle\d\surf.
\end{align*}
Here, $\div_\vol X$ is the usual divergence w.r.t.~$\vol$. Thus, since $\smash{\Cont_\comp^\infty(\mms)}$ is a core for $\F$ (recall \autoref{Ex:Mflds}), $X\in\Dom_\TV(\DIV)$ with
\begin{align*}
\div X &= \div_\vol X\quad\vol\text{-a.e.},\\
\norm X &= \langle X,\sfn\rangle\,\surf\textcolor{white}{X^\perp}\\
&= \langle X^\perp,\sfn\rangle\,\surf.
\end{align*}
Hence $\norm X$ is a measure which is supported on $\partial\mms$.

If $\vol$ is instead replaced by $\meas := \rme^{-2w}\,\vol$, $w\in\Cont^2(\mms)$, then for every $X$ as above, in terms of the metric tensor $\langle\cdot,\cdot\rangle$ we  have $X\in\Dom_\TV(\DIV)$ with
\begin{align*}
\div X &= \div_\vol X - 2\,\langle\nabla w,X\rangle\quad \meas\text{-a.e.},\\
\norm X &= \rme^{-2w}\,\langle X,\sfn\rangle\,\surf\textcolor{white}{X^\perp}\\
&= \rme^{-2w}\,\langle X^\perp,\sfn\rangle\,\surf.
\end{align*}
\end{example}

In fact, this smooth framework is the main motivation for the terminology of \autoref{Def:Normal component}, since in general we a priori do not know anything about the support of $\norm X$, $X\in\Dom(\DIV)$. As suggested by \autoref{Ex:Mflds with boundary} and the following \autoref{Re:Green's formula}, we may and will interpret $\norm X$ as the normal component of $X$ w.r.t.~a fictive \emph{outward pointing} unit normal vector field --- or more generally, as the ``normal'' part of $X$ w.r.t.~a certain $\meas$-singular set. Examples which link our \emph{intrinsic} approach to normal components with existing \emph{extrinsic} ones from \cite{brue2019,buffa2019,sturm2020} are discussed in \autoref{Ch:Appendix} below.

As it turns out in \autoref{Le:Div g nabla f} and \autoref{Sub:Test objects}, most vector fields of interest have vanishing normal component. Hence, normal components mostly do not appear in our subsequent treatise. However, our notion of second fundamental form in \autoref{Sub:Second fund form} really needs to make sense of a  \emph{nonvanishing} normal component of certain vector fields, see \eqref{Eq:Sec fund form normal cpt} and \autoref{Ex:Sec fund form smooth}. This is precisely the motivation behind \autoref{Def:Measure-valued divergence} and \autoref{Def:Normal component}, which have not been introduced in \cite{gigli2018}. In any case, in possible further applications it might be useful to see the normal component of vector fields in $\Dom(\DIV)$.

\begin{remark}[Gauß--Green formula]\label{Re:Green's formula} In terms of \autoref{Def:Measure-valued divergence}, given any $X\in\Dom(\DIV)$ and $h\in \smash{\F_\bc}$, in analogy to \autoref{Ex:Mflds with boundary} we have
\begin{align*}
-\int_\mms \rmd h(X)\d\meas = \int_\mms h\div X\d\meas - \int_\mms \widetilde{h}\,\d\!\norm X.
\end{align*}
\end{remark}

\subsubsection{Calculus rules} The proof of the following simple result directly follows from the respective definitions and is thus omitted. 

\begin{lemma}\label{Pr:Div comp} If $X\in\Dom(\div)$ satisfies $(\div X)^+\in\Ell^1(\mms)$ or $(\div X)^-\in\Ell^1(\mms)$, then $X\in\Dom_{\Ell^2}(\DIV)$ and
\begin{align*}
\div_1 X &= \div X\quad\meas\text{-a.e.},\\
\norm X &= 0.
\end{align*}
Conversely, if $X\in \Dom_{\Ell^2}(\DIV)$ with $\norm X = 0$, then $X\in \Dom(\div)$ with
\begin{align*}
\div X = \div_1 X\quad\meas\text{-a.e.}
\end{align*}
\end{lemma}

\begin{remark}\label{Re:vbl} The additional integrability condition in the first part of \autoref{Pr:Div comp} ensures  that $\div X\,\meas$ is a well-defined signed Borel measure for $X\in\Dom(\div)$.  In fact, unlike \eqref{Eq:div nabla = Delta} this small technical issue prevents us from saying that in general, if $f\in\Dom(\Delta)$ then $\nabla f\in\Dom_{\Ell^2}(\DIV)$ with $\div_1 \nabla f = \Delta f$ $\meas$-a.e.~and $\norm \nabla f=0$. This would correspond to the role of $\Delta$ as Neumann Laplacian from \autoref{Sub:Neumann Laplacian}. 

However, by the integration by parts \autoref{Def:L2 div} as well as \autoref{Le:Div identities} and \autoref{Le:Div g nabla f} below  there is still formal evidence in keeping this link in mind. In particular, we may and will interpret every $X\in\Dom(\div)$ as having vanishing normal component although the natural assignment $\DIV X := \div X\,\meas$ --- which entails $\norm X=0$ --- might not be well-defined in $\smash{\Meas_{\sigma}^\pm(\mms)_\Ch}$. 
\end{remark}

Based upon \autoref{Pr:Div comp}, as long as confusion is excluded we make no further notational distinction and identify, for suitable $X\in\Ell^2(T\mms)$,
\begin{align*}
\div X = \div_1 X.
\end{align*}

\begin{lemma}\label{Le:Div identities} For every $X\in\Dom(\DIV)$ and every $f\in \smash{\F_\eb}$ such that $\vert X\vert\in \Ell^\infty(\mms)$ or $\vert\rmd f\vert \in\Ell^\infty(\mms)$, we have $f\,X\in\Dom(\DIV)$ with
\begin{align*}
\DIV(f\,X) &= \widetilde{f}\DIV X + \rmd f(X)\,\meas,\\
\!\textcolor{white}{\widetilde{f}}\!\div(f\,X) &= f\div X + \rmd f(X)\quad\meas\text{-a.e.},\\
\norm(f\,X) &= \widetilde{f}\norm X.
\end{align*}
In particular, if $X\in\Dom_{L^2}(\DIV)$ then also $f\,X\in\Dom_{L^2}(\DIV)$.
\end{lemma}

\begin{proof} The first identity, from which the second follows, is straightforward to deduce from  \autoref{Th:Properties energy measure}. The last claim on $\Dom_{L^2}(\DIV)$ is a direct consequence of these two identities. To prove the remaining claim $\smash{\norm(f\,X) = \widetilde{f}\norm X}$, we compute
\begin{align*}
\norm(f\,X) &= -\DIV_\perp(f\,X)\\
&=  f\div X\,\meas - \widetilde{f}\,\DIV X\\
&= -\widetilde{f}\,\DIV_\perp X\\
&= \widetilde{f}\,\norm X.\qedhere
\end{align*}
\end{proof}

\begin{example} Every identity stated in \autoref{Le:Div identities} is fully justified in the smooth context of \autoref{Ex:Mflds with boundary}. In particular, for every smooth vector field $X$ on $\mms$ with compact support and every $f\in\Cont_\comp^\infty(\mms)$,
\begin{align*}
\norm(f\,X) = \langle f\,X, \sfn\rangle\,\surf = f\,\langle X,\sfn\rangle\,\surf = f\norm X.
\end{align*}
\end{example}

Finally, we show that \autoref{Pr:Div comp} is not void, i.e.~that there will exist an $\Ell^2$-dense set of vector fields --- see \autoref{Sub:Test objects} below --- which satisfy both hypotheses of \autoref{Pr:Div comp} even in a slightly stronger version.

\begin{lemma}\label{Le:Div g nabla f} Suppose that $f\in \Dom(\Delta)$ and $\smash{g\in \F_\bounded}$. Moreover, suppose that $\vert\nabla f \vert\in\Ell^\infty(\mms)$ or $\vert\rmd g\vert\in\Ell^\infty(\mms)$. Then $g\,\nabla f\in \Dom_\TV(\DIV)\cap\Dom(\div)$ with
\begin{align*}
\div(g\,\nabla f) &= \rmd g(\nabla f) + g\,\Delta f\quad\meas\text{-a.e.},\\
\norm(g\,\nabla f) &=0.
\end{align*}
\end{lemma}

\begin{proof} By \eqref{Eq:Div lbnz rle} and \eqref{Eq:div nabla = Delta}, we already know that $g\,\nabla f\in\Dom(\div)$.

To show that $g\,\nabla f\in\Dom_\TV(\DIV)$, note that under the given assumptions, the r.h.s.~of the identity for $\div(g\,\nabla f)$ belongs to $\Ell^1(\mms)$. In particular, by what we have already proved before, $\nu := \div X\,\meas \in \Meas_{\fin}^\pm(\mms)_\Ch$ satisfies the defining property \eqref{Eq:Def prop Div}  for $\DIV X$, yielding the claim.
\end{proof}

We conclude this survey with a duality formula. Recall \autoref{Re:vbl} to see why it might be hard to verify in general that the second inequality is an equality.

\begin{proposition} For every $f\in \F$,
\begin{align*}
\Ch(f) &= \sup\!\Big\lbrace\!-\! 2\int_\mms f\div X\d\meas - \int_\mms\vert X\vert^2\d\meas : X\in\Dom(\div)\Big\rbrace\\
&\geq \sup\!\Big\lbrace\!-\! 2\int_\mms f\div X\d\meas - \int_\mms\vert X\vert^2\d\meas : X\in\Dom_{L^2}(\DIV),\ \norm X =0\Big\rbrace.
\end{align*}
Moreover, if $\mms$ has finite $\meas$-measure, the second inequality is an equality.
\end{proposition}

\begin{proof} The second inequality is a trivial consequence of \autoref{Pr:Div comp}. 

Moreover, if $\meas[\mms] < \infty$, by \autoref{Pr:Div comp} every $X\in\Dom(\div)$ belongs to $\Dom_{\Ell^2}(\DIV)$ with $\DIV X = \div X\,\meas$, whence $\norm  X=0$. This yields the last statement.

Let us finally turn to the first equality. 

To prove ``$\geq$'', given  any $X\in\Dom(\div)$, it follows from Young's inequality and \autoref{Def:L2 div} that
\begin{align*}
\Ch(f) + \int_\mms\vert X\vert^2\d\meas\geq 2\int_\mms \rmd f(X)\d\meas = -2\int_\mms f\div X\d\meas.
\end{align*}
Rearranging terms and taking the supremum over all $X$ as above directly implies the desired inequality.

To prove the converse inequality, given any $f\in \F$ and any $t>0$, we set $X_t := \nabla\ChHeat_tf\in\Ell^2(T\mms)$. Since $\ChHeat_tf\in\Dom(\Delta)$, it follows from \eqref{Eq:div nabla = Delta} that $X_t\in \Dom(\div)$ with $\div X_t = \Delta \ChHeat_t f$ $\meas$-a.e. Using the semigroup property of $(\ChHeat_t)_{t\geq 0}$, \eqref{Eq:IBP Laplacian} and the nonincreasingness of the function $t\mapsto \Ch(\ChHeat_tf)$ on $[0,\infty)$, we arrive at
\begin{align*}
&\limsup_{t\to 0} \Big[\!- 2\int_\mms f\div X_t\d\meas - \int_\mms \vert X_t\vert^2\d\meas\Big]\\
&\qquad\qquad = \limsup_{t\to 0}\Big[\Ch(\ChHeat_tf) -\int_\mms f\Delta\ChHeat_tf\d\meas - \int_\mms\vert\nabla\ChHeat_tf\vert^2\d\meas\Big]\\
&\qquad\qquad = \limsup_{t\to 0}\Big[\Ch(\ChHeat_tf) + \int_\mms \vert\nabla\ChHeat_{t/2}f\vert^2\d\meas - \int_\mms \vert\nabla \ChHeat_tf\vert^2\d\meas\Big]\\
&\qquad\qquad \geq \limsup_{t\to 0}\Ch(\ChHeat_tf) = \Ch(f).\textcolor{white}{\int_\mms}
\end{align*}
The desired inequality readily follows.
\end{proof}

\part{Second order differential structure}\label{Pt:II}

Throughout this part, let \autoref{As:Gamma-operator} hold. We employ the following convention: given $\smash{f\in\Ell^0(\mms)}$, integrals of possibly degenerate terms such as $1/f$ are (consistently) understood as restricted on $\{f\neq 0\}$ without further notice.

\section{Preliminaries. The taming condition}

\subsection{Basic notions of tamed spaces}\label{Sub:Tamed spaces} In this section, we outline the theory of \emph{tamed spaces} introduced in \cite{erbar2020} that will be needed below throughout.
 
\subsubsection{Quasi-local distributions} Given an $\Ch$-quasi-open set $G\subset\mms$, let $\smash{\F_G^{-1}}$ denote the dual space of the closed \cite[p.~84]{chen2012} subspace
\begin{align*}
\F_G := \big\lbrace f\in \F: \widetilde{f} = 0\ \Ch\text{-q.e.~on }G^\rmc\big\rbrace
\end{align*}
of $\F$. Note that if $G'\subset\mms$ is $\Ch$-quasi-open as well with $G\subset G'$, then $\smash{\F_G^{-1}\supset\F_{G'}^{-1}}$. Let $\smash{\F_\qloc^{-1}}$ denote the space of \emph{$\Ch$-quasi-local distributions} on $\F$, i.e.~the space of all objects $\kappa$ for which there exists an $\Ch$-quasi-open $\Ch$-nest $(G_n)_{n\in\N}$ such that $\smash{\kappa\in \bigcap_{n\in\N} \F_{G_n}^{-1}}$. Every distribution $\smash{\kappa\in\F_\qloc^{-1}}$ is uniquely associated with an \emph{additive functional}, or briefly AF, $(\sfa^\kappa_t)_{t\geq 0}$ in the sense of \cite[Lem.~2.7, Lem.~2.9]{erbar2020}. Here, uniqueness means up to $\meas$-equivalence of AF's \cite[p.~423]{fukushima2011}.  (We refer to \cite[Ch.~II]{fukushima2011} or \cite[Ch.~4, App.~A]{chen2012} for a concise overview over basic notions about AF's.) The AF $\smash{(\sfa_t^\kappa)_{t\geq 0}}$ is independent of the chosen $\Ch$-nest \cite[Lem.~2.11]{erbar2020}. All AF's will be understood as being zero beyond the explosion time $\zeta$.

\begin{remark} (The defining properties of) AF's   are linked to the Markov process $\M$ from \autoref{Sub:Markov} \cite[Def.~A.3.1]{chen2012}. 
\end{remark}

\begin{example}\label{Ex:Nearly Borel}
If $\smash{\kappa\in\F_\qloc^{-1}}$ is induced in the evident way through a nearly Borel \cite[Def.~A.1.28]{chen2012} function $f\in\Ell^2(\mms)$, by \cite[Rem.~2.8]{erbar2020}, for every $t\in [0,\zeta)$,
\begin{align*}
\sfa_t^\kappa = \frac{1}{2}\int_0^{2t}f(\sfb_s)\d s.
\end{align*}
\end{example}

\subsubsection{Extended Kato class} The relevant distributions $\smash{\kappa\in\F_\qloc^{-1}}$ in our work, see \autoref{Re:Why Kato}, will be induced by signed measures in the \emph{extended Kato class} $\Kato_{1-}(\mms)$, which is introduced now. In fact, Feynman--Kac semigroups and energy forms induced by more general distributions $\smash{\kappa\in\F_\qloc^{-1}}$ are studied in \cite{erbar2020}. 

Given any Borelian $f\colon\mms\to\R$, set
\begin{align*}
\rmq\text{-}\!\sup f := \inf\!\Big\lbrace\!\sup_{x\in (M')^\rmc} f(x) : M'\subset\mms \text{ is } \Ch\text{-polar}\Big\rbrace.
\end{align*}
Recall from \cite[p.~84]{chen2012} that a measure $\nu\in\Meas^+(\mms)$ is \emph{$\Ch$-smooth} if it does not charge $\Ch$-polar sets and $\nu[F_n] < \infty$ for every $n\in\N$, for some $\Ch$-nest $\smash{(F_n)_{n\in\N}}$ of closed sets. Every Radon measure charging no $\Ch$-polar set is $\Ch$-smooth, but the converse does not hold  \cite[Ex.~2.33]{delloschiavo2020}. By the Revuz correspondence \cite[Thm.~A.3.5]{chen2012}, any $\Ch$-smooth $\smash{\nu\in\Meas^+(\mms)}$ is uniquely associated to a positive continuous AF, or briefly PCAF, $\smash{(\sfa_t^\nu)_{t\geq 0}}$ by the subsequent identity, valid for every nonnegative $f\in\Ell_0(\mms)$:
\begin{align*}
\int_\mms f\d\nu = \lim_{t\to 0}\frac{1}{t}\int_\mms \Exp^{\,\cdot}\Big[\!\int_0^t f(\sfb_{2s})\d\sfa^\nu_s\Big]\d\meas.
\end{align*}

Of course, the existence of a \emph{compact} $\Ch$-nest, by quasi-regularity of $\Ch$, ensures that every $\sigma$-finite $\nu\in\Meas^+(\mms)$ charging  no $\Ch$-polar set is  $\Ch$-smooth. Hence the following definition \cite[Def.~2.21]{erbar2020} is meaningful.

\begin{definition}\label{Def:Kato class} Given $\rho\geq 0$, the \emph{$\rho$-Kato class} $\Kato_\rho(\mms)$ of $\mms$ is defined to consist of all $\smash{\mu\in\Meas_{\sigma}^\pm(\mms)}$ which do not charge $\Ch$-polar sets with
\begin{align*}
\lim_{t\to 0}\rmq\text{-}\!\sup \Exp^{\,\cdot}\big[\sfa_t^{2\vert\mu\vert}\big]\leq \rho.
\end{align*}
$\Kato_0(\mms)$ is called \emph{Kato class} of $\mms$, while the \emph{extended Kato class} of $\mms$ is
\begin{align*}
\Kato_{1-}(\mms) := \bigcup_{\rho\in [0,1)}\Kato_\rho(\mms).
\end{align*}
\end{definition}

\begin{remark}\label{Re:iffs} By definition, we have $\mu \in \Kato_\rho(\mms)$ if and only if $\vert\mu\vert\in\Kato_\rho(\mms)$ if and only if $\smash{\mu^+,\mu^-\in\Kato_\rho(\mms)}$, $\rho \geq 0$. 
\end{remark}

The following important lemma is due to \cite[Cor.~2.25]{erbar2020}. The immediate corollary that one can always choose $\rho'< 1$ for $\mu\in\Kato_{1-}(\mms)$ therein is used  crucially in our work (and also in \cite{erbar2020}), see e.g.~\autoref{Cor:Dom(Delta) subset W22} and \autoref{Le:Inclusion}.

\begin{lemma}\label{Le:Form boundedness} For every $\rho,\rho'\geq 0$ with $\rho < \rho'$ and every $\mu\in\Kato_\rho(\mms)$ there exists $\alpha'\in\R$ such that for every $f\in \F$,
\begin{align*}
\int_\mms \widetilde{f}^2\d\mu \leq \rho'\,\Ch(f) + \alpha'\int_\mms f^2\d\meas.
\end{align*}
\end{lemma}

In particular, w.r.t.~a sequence $(G_n)_{n\in\N}$ of open subsets of $\mms$ witnessing the $\sigma$-finiteness of $\nu := \vert\mu\vert$, by Cauchy--Schwarz's inequality every $\mu\in\Kato_\rho(\mms)$, $\rho\geq 0$, induces a (non-relabeled)  element $\smash{\mu\in\F_\qloc^{-1}(\mms)}$ by setting, for $\smash{f\in\bigcup_{n\in\N} \F_{G_n}}$,
\begin{align*}
\big\langle \mu\,\big\vert\, f \big\rangle := \int_\mms \widetilde{f}\d\mu.
\end{align*}
If $\mu \in \Kato_\rho(\mms)$ is nonnegative, then the AF's coming from the Revuz correspondence w.r.t.~$\mu$ and from its property as inducing an element in $\smash{\F_\qloc^{-1}}$ are $\meas$-equivalent. The key feature about general $\mu\in \Kato_{1-}(\mms)$ is that by Khasminskii's lemma \cite[Lem.~2.24]{erbar2020}, the induced distribution is \emph{moderate} \cite[Def.~2.13]{erbar2020}, i.e.~
\begin{align*}
\sup_{t\in [0,1]}\rmq\text{-}\!\sup\Exp^{\,\cdot}\big[\rme^{-\sfa_t^{2\mu}}\big] < \infty.
\end{align*}

\subsubsection{Feynman--Kac semigroup} Let $q\in [1,2]$, and note that $q\kappa/2 \in\Kato_{1-}(\mms)$ for every $\kappa\in \Kato_{1-}(\mms)$. Given such $\kappa$, we define a family $(\Schr{q\kappa}_t)_{t\geq 0}$ of operators acting on nonnegative nearly Borel functions $f\in \Ell_0(\mms)$ by
\begin{align*}
\Schr{q\kappa}_tf := \Exp^{\,\cdot}\big[\rme^{-\sfa_t^{q\kappa}}f(\sfb_{2t})\,\One_{\{t < \zeta/2\}}\big].
\end{align*}
It naturally extends to nearly Borelian $f\in\Ell_0(\mms)$ for which the latter expectation for $\vert f\vert$ in place of $f$ is finite, see \cite[Def.~2.10]{erbar2020}. (For later  convenience, we have to change the notation from \cite{erbar2020} a bit, also at later times, see \autoref{Re:Notation change} below.) It is $\meas$-symmetric and maps $\meas$-equivalence classes to $\meas$-equivalence classes. Since $q\kappa/2$ is moderate, it extends to an exponentially bounded semigroup of linear operators on $\Ell^p(\mms)$ for every $p\in[1,\infty]$ \cite[Lem.~2.11, Rem.~2.14]{erbar2020}. That is, there exists a finite constant $C>0$ such that for every $p\in [1,\infty]$ and every $t\geq 0$,
\begin{align*}
\big\Vert\Schr{q\kappa}_t\big\Vert_{\Ell^p(\mms);\Ell^p(\mms)}\leq C\,\rme^{Ct}.
\end{align*}

\subsubsection{The perturbed energy form} One of the main results from \cite{erbar2020} is that for $\kappa\in\Kato_{1-}(\mms)$ --- in fact, for more general $\smash{\kappa\in\F_\qloc^{-1}}$, cf.~\cite[Thm.~2.49]{erbar2020} --- $(\Schr{q\kappa}_t)_{t\geq 0}$ is properly associated to an  energy form $\Ch^{q\kappa}$, $q\in [1,2]$. Indeed, by \cite[Thm.~2.47, Thm.~2.49, Cor.~2.51]{erbar2020}, the quadratic form
\begin{align}\label{Eq::E^k identity}
\Ch^{q\kappa}(f) := \Ch(f) + q\,\big\langle \kappa\,\big\vert\, f^2\big\rangle
\end{align}
with finiteness domain $\Dom(\Ch^{q\kappa}) = \F$ is closed, lower semibounded and associated with $(\Schr{q\kappa}_t)_{t\geq 0}$ \cite[Thm.~1.3.1, Lem.~1.3.2]{fukushima2011}. 

The corresponding generator, henceforth termed $\Delta^{q\kappa}$ with domain $\Dom(\Delta^{q\kappa})$,\label{Not:Schrödinger operator} is called \emph{Schrödinger operator} with potential $q\kappa$.

\begin{remark}\label{Re:Why Kato} One reason for considering perturbations of $\Ch$ by $\kappa\in\Kato_{1-}(\mms)$ is the following. By \autoref{Le:Form boundedness}, the map $f\mapsto\big\langle\kappa^-\,\big\vert \,f^2\big\rangle$ on $\F$ is form bounded w.r.t.~$\Ch$, hence w.r.t.~$\smash{\Ch^{q\kappa^+}}$, with some form bound $\rho' < 1$. Hence, by \cite[Thm.~2.49]{erbar2020}, $\Ch^{q\kappa}$ is closed with domain $\smash{\Dom(\Ch^{q\kappa}) = \big\lbrace f\in\F : \big\langle\kappa^+\,\big\vert\, f^2\big\rangle< \infty\big\rbrace}$. Again by \autoref{Le:Form boundedness}, the latter is all of $\F$, which is technically required in the setting of \autoref{Sub:Self-imp}, compare with \cite[Ch.~6]{erbar2020}. See also \autoref{Re:No restriction} below.
\end{remark}

\begin{remark}\label{Re:Notation change} For our analytic and geometric purposes, we use differently scaled forms, operators and semigroups than \cite{erbar2020}. Let us list the relations of the main objects in \cite{erbar2020}, on the  l.h.s.'s, with our notation, on the respective  r.h.s.'s:
\begin{align*}
\mathscr{E} &= \Ch/2,\textcolor{white}{\big\vert}\\
\sfL &= \Delta/2,\textcolor{white}{\big\vert}\\
P_t &= \ChHeat_{t/2},\textcolor{white}{\big\vert}\\
B_t &= \sfb_t\ \text{[sic]},\textcolor{white}{\big\vert}\\
\mathscr{E}^{q\kappa/2} &= \Ch^{q\kappa}/2\textcolor{white}{\big\vert}\\
\sfL^{q\kappa/2} &= \Delta^{q\kappa}/2,\\
P_t^{q\kappa/2} &= \Schr{q\kappa}_{t/2}.
\end{align*}
\end{remark}

\subsubsection{Tamed spaces}\label{Sub:Tam}

\begin{definition}\label{Def:BE cond} Suppose that $q\in \{1,2\}$, $\kappa\in \Kato_{1-}(\mms)$ and $N\in [1,\infty]$. We say that $(\mms,\Ch,\meas)$ or simply $\mms$ satisfies the \emph{$q$-Bakry--Émery condition}, briefly $\BE_q(\kappa,N)$, if for every $f\in \Dom(\Delta)$ with $\Delta f\in\F$ and every nonnegative $\phi\in \Dom(\Delta^{q\kappa})$ with $\Delta^{q\kappa}\phi\in\Ell^\infty(\mms)$, we have
\begin{align*}
&\frac{1}{q}\int_\mms \Delta^{q\kappa}\phi\,\vert\nabla f\vert^q\d\meas - \int_\mms\phi\,\vert\nabla f\vert^{q-1}\,\big\langle\nabla f,\nabla \Delta f \big\rangle\d\meas\\
&\qquad\qquad \geq \frac{1}{N}\int_\mms\phi\,\vert\nabla f\vert^{q-1}\,(\Delta f)^2\d\meas.
\end{align*}
The latter term is understood as $0$ if $N:= \infty$.
\end{definition}

\begin{remark} Through a mollification argument using $(\Schr{q\kappa}_t)_{t\geq 0}$ \cite[Lem.~6.2]{erbar2020}, the class of elements $\phi\in \Dom(\Delta^{q\kappa})$ with $\Delta^{q\kappa}\phi\in\Ell^\infty(\mms)$ is dense in $\Ell^2(\mms)$.
\end{remark}

\begin{assumption}\label{As:Bakry Emery} We henceforth assume that $\mms$ satisfies the $\BE_2(\kappa,N)$ condition for given $\kappa\in\Kato_{1-}(\mms)$ and $N\in [1,\infty]$.
\end{assumption}

For certain $k\colon\mms\to\R$ and $N\in [1,\infty]$, write $\BE_2(k,N)$ instead of $\BE_2(k\,\meas, N)$. Of course, \autoref{Def:BE cond} \cite[Def.~3.1, Def.~3.5]{erbar2020} generalizes the well-known Bakry--\smash{Émery} condition for uniform lower Ricci bounds \cite{ambrosio2015, bakry1985a, bakry1985b, erbar2015, gigli2015}. Variable lower Ricci bounds have been first studied by \cite{braun2021, sturm2015} in a synthetic context.

In the framework of \autoref{As:Bakry Emery}, we say that $(\mms,\Ch,\meas)$ or simply $\mms$ is \emph{tamed}. Although this is not the original definition of taming from \cite[Def.~3.2]{erbar2020}, for $\smash{\kappa\in\Kato_{1-}(\mms)}$ they are in fact equivalent, see \autoref{Sub:Self-imp} below.

\begin{remark}\label{Re:No restriction} From the taming point of view, one might regard the implicit assumption that $\kappa^+\in\Kato_{1-}(\mms)$ (recall \autoref{Re:iffs}) as unnatural. However, for the \emph{qualitative} message of this article --- the existence of a rich second order calculus --- one can simply ignore $\kappa^+$ by setting it to zero. To obtain \emph{quantitative} results in applications, to bypass the assumption $\kappa^+\in\Kato_{1-}(\mms)$, a useful tool could be appropriate ``cutoffs'' and monotone approximations by elements in $\smash{\Kato_{1-}(\mms)}$, see e.g.~\autoref{Sub:Funct inequ hf 1-forms} below and \cite[Lem.~2.1]{braun2021}.
\end{remark}

\begin{example}[Manifolds] Any compact Riemannian manifold $\mms$ is tamed by 
\begin{align*}
\kappa := \mathcall{k}\,\vol + \mathcall{l}\,\surf
\end{align*}
which in fact belongs to $\Kato_0(\mms)$ \cite[Thm.~4.4]{erbar2020} (recall \autoref{Ex:Mflds}). More generally \cite{braun2020a,braunrigoni2021}, let $\mms$ be a ``regular'' Lipschitz Riemannian manifold, in the sense of \cite{braunrigoni2021}, that is quasi-isometric to a Riemannian manifold with uniformly lower bounded Ricci curvature. Suppose that the Ricci curvature of $\mms$, where defined, is bounded from below by a function $\mathcall{k}\in\Ell^p(\mms,\Xi\,\vol)$, $p > d/2$, where $\Xi\colon\mms\to\R$ is given by $\Xi(x) := \vol[B_1(x)]^{-1}$. Then $\mms$ is tamed by $\kappa := \mathcall{k}\,\vol \in\Kato_0(\mms)$.
\end{example}

\begin{example}[$\RCD$ spaces]\label{Ex:The RCD ex} Every $\RCD(K,\infty)$ space $(\mms,\met,\meas)$, $K\in\R$, according to \autoref{Ex:mms} is tamed by $\kappa := K\,\meas\in\Kato_0(\mms)$ \cite{ambrosio2014b} (recall \autoref{Ex:Nearly Borel}). 
\end{example}

\begin{example}[Almost smooth spaces] Let $(\mms,\met,\meas)$ be a $d$-dimensional almost smooth metric measure space, $d\in\N$, in the sense of \cite[Def.~3.1, Def.~3.16]{honda2018}, examples of which include the gluing of two pointed, compact Riemannian manifolds (\emph{not} necessarily  of the same dimension) at their base points. Then, under few further assumptions, Honda proved that if the ``generalized Ricci curvature'' of $\mms$ is bounded from below by $K(d-1)$, $K\in\R$, then the $\BE_2(K(d-1),d)$ condition holds for $\mms$ \cite[Thm.~3.7, Thm.~3.17]{honda2018}. The $\RCD^*(K(d-1),N)$ condition, however, does not hold in general \cite[Rem.~3.9]{honda2018}.
\end{example}

\begin{example}[Configuration spaces] Further important nonsmooth examples are configuration spaces $\smash{\mathcal{Y}}$ over Riemannian manifolds $\mms$ \cite{albeverio1998, erbar2015b}. The  Dirichlet form $\smash{\Ch^\mathcal{Y}}$ on $\smash{\mathcal{Y}}$ constructed in \cite{albeverio1998} is quasi-regular and strongly local, cf.~the proof of  \cite[Thm.~6.1]{albeverio1998}. If $\Ric \geq K$ on $\mms$, $K\in\R$, then $\smash{(\mathcal{Y},\Ch^\mathcal{Y},\pi)}$ is tamed by $\kappa := K\,\pi\in\Kato_0(\mms)$ as well by \cite[Thm.~4.7]{erbar2015} and \cite[Thm.~3.6]{erbar2020}. Here $\pi$ is the Poisson (probability) measure on $\smash{\mathcal{Y}}$, up to intensity. 

A similar result even over more general spaces is announced in \cite{delloschiavo2020}.
\end{example}

Spaces that are tamed by some measures in $\Kato_{1-}(\mms)$ or even $\Kato_0(\mms)$ may also have cusp-like singularities \cite[Thm.~4.6]{erbar2020}, have singular \cite[Thm.~2.36]{erbar2020} or not semiconvex boundary \cite[Thm.~4.7]{erbar2020} or be of Harnack-type \cite{braunrigoni2021,erbar2020}.

\subsubsection{Intrinsically complete Dirichlet spaces}\label{Subsub:Intr compl}

An interesting, but not exhaustive, class of tamed spaces we sometimes consider is the one of intrinsically complete $\mms$ as introduced in \cite[Def.~3.8]{erbar2020}.

\begin{definition}\label{Def:Intr compl} We call $(\mms,\Ch,\meas)$ or simply $\mms$ \emph{intrinsically complete} if there exists a sequence $(\phi_n)_{n\in\N}$ in $\F$ such that $\smash{\meas\big[\{\phi_n > 0\}\big] < \infty}$, $0\leq \phi_n\leq 1$ $\meas$-a.e.~and $\vert\nabla \phi_n\vert \leq 1$ $\meas$-a.e.~for every $n\in\N$ as well as $\phi_n \to \One_\mms$ and $\vert\nabla \phi_n\vert\to 0$ pointwise $\meas$-a.e.~as $n\to\infty$.
\end{definition}

Intrinsically complete tamed spaces are \emph{stochastically complete}, i.e.
\begin{align*}
\ChHeat_t\One_\mms = \One_\mms\quad\meas\text{-a.e.}
\end{align*}
for every $t\geq 0$ \cite[Thm.~3.11]{erbar2020}. In our work, intrinsically complete spaces provide a somewhat better version of \autoref{Le:Approx to id}, sometimes used to get rid of differentials in certain expressions. See e.g.~\autoref{Re:Test closure} and \autoref{Re:Wd120 in W12}. However, none of our results will severely rely on intrinsic completeness.

\subsection{Self-improvement and singular $\Gamma_2$-calculus}\label{Sub:Self-imp} In fact, under the above \autoref{As:Bakry Emery}, the condition $\BE_2(\kappa,N)$ is \emph{equivalent} to $\BE_1(\kappa,N)$ \cite[Thm.~6.9]{erbar2020}, $N\in [1,\infty]$, albeit the latter is a priori stronger \cite[Thm.~3.4, Prop.~3.7]{erbar2020}. In particular, the heat flow $(\ChHeat_t)_{t\geq 0}$ satisfies the important contraction estimate
\begin{align}\label{Eq:1BE grad est}
\vert\nabla\ChHeat_tf\vert\leq \Schr{\kappa}_t\vert\nabla f\vert\quad\meas\text{-a.e.}
\end{align}
for every $f\in\F$ and every $t\geq 0$. See \cite{bakry1985a, bakry1985b, savare2014} for corresponding results for constant $\kappa$ and \cite[Thm.~3.6]{braun2021} for the first nonconstant result in this direction. 

We briefly recapitulate the singular $\Gamma_2$-calculus developed in \cite{erbar2020, savare2014}, since the involved calculus objects, in particular those from \autoref{Sub:Test fcts} below, are crucial in our treatise as well, see e.g.~\autoref{Th:Hess} and \autoref{Th:Ricci measure}.

\subsubsection{Test functions}\label{Sub:Test fcts} Define the set of \emph{test functions} by\label{Not:Test fcts}
\begin{align*}
\Test(\mms) := \big\lbrace f\in \Dom(\Delta)\cap\Ell^\infty(\mms) : \vert\nabla f\vert\in \Ell^\infty(\mms),\ \Delta f\in \F\big\rbrace
\end{align*}
It is an algebra w.r.t.~pointwise multiplication, and if $f\in \Test(\mms)^n$ and $\varphi\in\Cont^\infty(\R^n)$ with $\varphi(0)=0$, $n\in\N$, then $\varphi\circ f\in \Test(\mms)$  as well \cite[Lem.~3.2]{savare2014}. 

Since $\BE_2(\kappa,N)$ implies $\BE_2(-\kappa^-,N)$ \cite[Prop.~6.7]{erbar2020}, a variant of the reverse Poincaré inequality states that for every $f\in \Ell^2(\mms)\cap\Ell^\infty(\mms)$ and every $t>0$,
\begin{align*}
\vert\nabla\ChHeat_tf\vert^2 \leq \frac{1}{2t}\,\big\Vert \Schr{-2\kappa^-}_t\big\Vert_{\Ell^\infty(\mms);\Ell^\infty(\mms)}\,\Vert f\Vert_{\Ell^\infty(\mms)}\quad\meas\text{-a.e.},
\end{align*}
and hence $\ChHeat_t f\in \Test(\mms)$ \cite[Cor.~6.8]{erbar2020}. In particular, $\Test(\mms)$ is dense in $\F$.

We use the following two approximation results henceforth exploited at various instances. For convenience, we outline the proof of \autoref{Le:Product convergence}. \autoref{Le:Mollified heat flow}, which yields a useful density result for the slightly smaller set\label{Not:Test fcts Linfty}
\begin{align*}
\Test_{\Ell^\infty}(\mms) := \big\lbrace f \in\Test(\mms) : \Delta f\in \Ell^\infty(\mms)\big\rbrace,
\end{align*}
results from \eqref{Eq:1BE grad est} and an approximation by a mollified heat flow \cite{erbar2020, savare2014}.

\begin{lemma}\label{Le:Product convergence} For every $f\in\Test(\mms)$, there exist sequences $(g_n)_{n\in\N}$ and $(h_n)_{n\in\N}$ in $\Test(\mms)$ which are bounded in $\Ell^\infty(\mms)$ with $g_n\,h_n\to f$ in $\F$ as $n\to\infty$.
\end{lemma}

\begin{proof} Given any $k,m\in\N$, we define $g_k := 2\arctan(k\,f)/\pi\in\Test(\mms)$ and $h_m := (f^2 +  2^{-m})^{1/2} - 2^{-m/2}\in\Test(\mms)$. Then $g_k\,h_m \to g_k\,\vert f\vert$ in $\Ell^2(\mms)$ as $m\to\infty$ for every $k\in\N$, and $g_k\,\vert f\vert \to f$ in $\Ell^2(\mms)$ as $k \to\infty$. Using \autoref{Cor:Calculus rules d} and Lebesgue's theorem, it follows that $g_k\,h_m\to g_k\,\vert f\vert$ in $\F$ as $m\to\infty$ for every $k\in\N$, and $g_k\,\vert f\vert \to f$ in $\F$ as $k\to\infty$. We conclude by a diagonal argument.
\end{proof}

\begin{lemma}\label{Le:Mollified heat flow} For every $f\in \F$ such that $a\leq f\leq b$ $\meas$-a.e., $a,b\in [-\infty,\infty]$, there exists a sequence $(f_n)_{n\in\N}$ in $\Test_{\Ell^\infty}(\mms)$ which converges to $f$ in $\F$ such that $a\leq f_n\leq b$ $\meas$-a.e.~for every $n\in\N$. If moreover $\vert\nabla f\vert\in\Ell^\infty(\mms)$, then $(f_n)_{n\in\N}$ can be constructed such that $(\vert \nabla f_n\vert)_{n\in\N}$ is bounded in $\Ell^\infty(\mms)$.
\end{lemma}

\subsubsection{Measure-valued Schrödinger operator} A further regularity property of functions $f\in\Test(\mms)$ that will be crucial in defining the $\kappa$-Ricci measure in \autoref{Sub:Ricci measure} is that their carré du champs $\vert\nabla f\vert^2$ have $\F$-regularity under $\BE_2(\kappa,N)$. In fact, $\vert\nabla f\vert^2$ admits a measure-valued Schrödinger operator in the sense of \autoref{Def:Measure valued Schr}. This is recorded in \autoref{Pr:Bakry Emery measures} and is due to \cite[Lem.~6.4]{erbar2020}.

\begin{definition}\label{Def:Measure valued Schr} We define $\Dom(\DELTA^{2\kappa})$ to consist of all $u\in \F$ for which there exists  $\smash{\iota\in\Meas_{\sigmafinR}^\pm(\mms)_\Ch}$ such that for every $h\in \F$, we have $\smash{\widetilde{h}\in \Ell^1(\mms,\iota)}$ and
\begin{align*}
\int_\mms \widetilde{h}\d\iota = -\Ch^{2\kappa}(h,u).
\end{align*}
In case of existence, $\iota$ is unique, denoted by $\DELTA^{2\kappa} u$ and shall be called the \emph{measure-valued Schrödinger operator} with potential $2\kappa$.\label{Not:Measure valued Schr I}
\end{definition}

\begin{proposition}\label{Pr:Bakry Emery measures} For every $f\in\Test(\mms)$ we have $\vert\nabla f\vert^2\in \F$ and even $\vert\nabla f\vert^2\in \Dom(\DELTA^{2\kappa})$. Moreover,
\begin{align*}
\frac{1}{2}\DELTA^{2\kappa}\vert\nabla f\vert^2 -  \big\langle\nabla f,\nabla\Delta f\big\rangle\,\meas \geq  \frac{1}{N}\,(\Delta f)^2\,\meas.
\end{align*}
\end{proposition}

An advantage of interpreting the Schrödinger operator associated to $\Ch^{2\kappa}$ as a measure is that the potential $2\kappa$ can be separated from $\DELTA^{2\kappa}$ to give the \emph{measure-valued Laplacian} $\smash{\DELTA := \DELTA^{2\kappa} + 2\kappa}$, which fits well with the divergence objects from \autoref{Sec:Divergences}, see \autoref{Sub:Measure valued Lapl} below. This is technically convenient in defining the drift-free Ricci measure $\RIC$ in \autoref{Def:Ric} and the second fundamental form $\II$ in \autoref{Def:Second fund form} without $\kappa$-dependency. However, for possible later extensions, e.g.~when $\kappa$ is not a (signed) measure, we decided not to separate the distribution $\kappa$ from the other calculus objects under consideration until \autoref{Sub:Measure valued Lapl}.

\subsubsection{Singular $\Gamma_2$-operator} Given \autoref{Pr:Bakry Emery measures}, following \cite{erbar2020,savare2014} we introduce the map $\bdGamma_2^{2\kappa}\colon \Test(\mms)\to \Meas_\sigmafinR^+(\mms)_\Ch$ by
\begin{align}\label{Eq:Gamma_2 opertor}
\bdGamma_2^{2\kappa}(f) := \frac{1}{2}\DELTA^{2\kappa}\vert\nabla f\vert^2 - \big\langle\nabla f,\nabla \Delta f\big\rangle\,\meas.
\end{align}
According to the Lebesgue decomposition in \autoref{Subsub:Measure spaces}, we decompose
\begin{align*}
\bdGamma_2^{2\kappa}(f) = \bdGamma_2^{2\kappa}(f)_\ll + \bdGamma_2^{2\kappa}(f)_\perp
\end{align*}
w.r.t.~$\meas$, $f\in\Test(\mms)$. A consequence of \autoref{Pr:Bakry Emery measures} is that
\begin{align*}
\bdGamma_2^{2\kappa}(f)_\perp \geq 0
\end{align*}
and, defining $\smash{\gamma_2^{2\kappa}(f) := \rmd \bdGamma_2^{2\kappa}(f)_\ll/\rmd\meas\in \Ell^1(\mms)}$,
\begin{align*}
\gamma_2^{2\kappa}(f) \geq \frac{1}{N}\,(\Delta f)^2 \quad\meas\text{-a.e.}
\end{align*}
Further calculus rules of $\bdGamma_2^{2\kappa}$ are summarized in the next \autoref{Le:Calculus rules}. To this aim, note that $\langle\nabla u,\nabla v\rangle\in\F$ for every $u,v\in\Test(\mms)$ by \autoref{Pr:Bakry Emery measures}, whence it makes sense define the ``pre-Hessian'' $\rmH[\cdot]\colon \Test(\mms)^3 \to \Ell^2(\mms)$ by
\begin{align}\label{Eq:Pre-Hessian}
\begin{split}
2\,\rmH[f](g_1,g_2) &:= \big\langle\nabla g_1,\nabla\langle\nabla f,\nabla g_2\rangle\big\rangle +\big\langle\nabla g_2,\nabla\langle\nabla f,\nabla g_1\rangle\big\rangle\\
 &\qquad\qquad - \big\langle\nabla f,\nabla\langle\nabla g_1,\nabla g_2\rangle\big\rangle.
\end{split}
\end{align}

\begin{lemma}\label{Le:Calculus rules} Let $\alpha\in\N$, $q\in\Test(\mms)^\alpha$ and $\varphi\in\Cont^\infty(\R^\alpha)$ with $\varphi(0) = 0$. Moreover, given any $i,j\in \{1,\dots,\alpha\}$, set $\varphi_i := \partial_i\varphi$ and $\varphi_{ij} := \partial_i\partial_j\varphi$. Define $\RMA^{2\kappa}[\varphi\circ q]\in\Meas_\fin^\pm(\mms)_\Ch$ and $\rmB[\varphi\circ q], \rmC[\varphi\circ q], \rmD[\varphi\circ q]\in\Ell^1(\mms)$ by
	\begin{align*}
	\RMA^{2\kappa}[\varphi\circ q] &:= \sum_{i,j=1}^\alpha \big[\varphi_i\circ\widetilde{q}\big]\,\big[\varphi_j\circ \widetilde{q}\big]\,\bdGamma_2^{2\kappa}(q_i,q_j),\\
	\rmB[\varphi\circ q] &:= 2\sum_{i,j,k=1}^\alpha \big[\varphi_i\circ q\big]\,\big[\varphi_{jk}\circ q]\,\rmH[q_i](q_j, q_k),\\
	\rmC[\varphi\circ q] &:= \sum_{i,j,k,l =1}^\alpha \big[\varphi_{ik}\circ q\big]\,\big[\varphi_{jl}\circ q\big]\,\langle\nabla q_i,\nabla q_j\rangle\,\langle\nabla q_k,\nabla q_l\rangle,\\
	\rmD[\varphi\circ q] &:= \Big[\!\sum_{i=1}^\alpha \big[\varphi_i\circ q\big]\,\Delta q_i + \sum_{i,j=1}^\alpha\big[\varphi_{ij}\circ q]\,\langle\nabla q_i,\nabla q_j\rangle\Big]^2.
	\end{align*}
	Then we have the identities 
	\begin{align*}
		\bdGamma_2^{2\kappa}(\varphi\circ q) &= \RMA^{2\kappa}[\varphi\circ q] + \big[\rmB[\varphi\circ q] + \rmC[\varphi\circ q]\big]\,\meas,\\
	\big[\Delta(\varphi\circ q)\big]^2\,\meas &= \rmD[\varphi\circ q]\,\meas.
	\end{align*}
\end{lemma}

\begin{remark}\label{Re:No a priori fin tot var} Note that all measures in \autoref{Le:Calculus rules} are identified as \emph{finite}. This technical point is shown in \autoref{Pr:Fin tot var} below. Albeit the latter is an \emph{a posteriori} consequence of $\BE_1(\kappa,\infty)$, see also \autoref{Re:RCD cons}, the results of \cite[Lem.~6.5, Thm.~6.6, Thm.~6.9]{erbar2020} --- in particular \autoref{Le:Calculus rules} --- are still deducible  \emph{a priori} from \autoref{As:Bakry Emery} by restriction of the identities from \cite[Lem.~6.5, Thm.~6.6]{erbar2020} to subsets of finite measure, or interpreting the asserted identities in a suitable weak sense.
\end{remark}

\subsubsection{Finiteness of total variations}\label{Sub:Schr props} The final goal of this subsection is to prove in \autoref{Pr:Fin tot var} that $\Vert\DELTA^{2\kappa}\vert\nabla f\vert^2\Vert_\TV < \infty$, $f\in\Test(\mms)$. Besides the technical \autoref{Re:No a priori fin tot var}, this fact will be of decisive help in \emph{continuously} extending the $\kappa$-Ricci measure beyond regular vector fields, see \autoref{Th:Ricci measure}.

The following is a minor variant of \cite[Lem.~6.2]{erbar2020} with potential $\kappa$ instead of $2\kappa$, proven in a completely analogous way.

\begin{lemma}\label{Le:BE1 lemma} Let $u\in \Ell^2(\mms)\cap\Ell^\infty(\mms)$ be nonnegative, and let $g\in \Ell^2(\mms)$. Suppose that for every nonnegative $\phi\in\Dom(\Delta^\kappa)\cap\Ell^\infty(\mms)$ with $\Delta^\kappa\phi\in\Ell^\infty(\mms)$,
\begin{align*}
\int_\mms u\,\Delta^\kappa\phi\d\meas \geq -\int_\mms g\,\phi\d\meas.
\end{align*}
Then $u\in \F$ as well as
\begin{align*}
\Ch^\kappa(u) \leq \int_\mms u\,g\d\meas.
\end{align*}
Moreover, there exists a unique measure $\bdsigma\in \smash{\Meas_\sigma^+(\mms)_\Ch}$ such that for every $h\in \F$, we have $\smash{\widetilde{h}\in \Ell^1(\mms,\bdsigma)}$ and
\begin{align*}
\int_\mms \widetilde{h}\d\bdsigma = -\Ch^\kappa(h,u) +\int_\mms h\,g\d\meas.
\end{align*}
\end{lemma}

\begin{proposition}\label{Pr:Fin tot var} For every $f\in\Test(\mms)$, $\vert \nabla f\vert$ belongs to $\smash{\F_\bounded}$, and the signed Borel measure $\DELTA^{2\kappa} \vert\nabla f\vert^2$ has finite total variation. Moreover,
\begin{align}\label{Eq:Identities Gamma2 DELTA}
\begin{split}
\bdGamma_2^{2\kappa}(f)[\mms] &=\int_\mms (\Delta f)^2\d\meas - \big\langle \kappa\,\big\vert\,\vert\nabla f\vert^2\big\rangle,\\
\DELTA^{2\kappa} \vert\nabla f\vert^2[\mms] &= -2\,\big\langle\kappa\,\big\vert\,\vert\nabla f\vert^2\big\rangle.
\end{split}
\end{align}
\end{proposition}

\begin{proof} By \autoref{Pr:Bakry Emery measures}, we already know that $\vert\nabla f\vert^2\in\Dom(\DELTA^{2\kappa})$. Now recall that by the self-improvement property of $\BE_2(\kappa,N)$, $(\mms,\Ch,\meas)$ obeys $\BE_1(\kappa,\infty)$ according to \autoref{Def:BE cond} \cite[Thm.~6.9]{erbar2020}. By \autoref{Le:BE1 lemma} applied to $u := \vert\nabla f\vert \in \Ell^2(\mms)\cap\Ell^\infty(\mms)$ and $g := \smash{-\One_{\{\vert \nabla f\vert > 0\}}\,\langle\nabla f,\nabla \Delta f\rangle\,\vert\nabla f\vert^{-1}}\in\Ell^2(\mms)$ as well as by \eqref{Eq::E^k identity}, we obtain $\smash{\vert \nabla f\vert\in \F_\bounded}$ and the unique existence of an element $\smash{\bdsigma\in \Meas_\sigmafinR^+(\mms)_\Ch}$ such that for every $h\in \F$, we have $\smash{\widetilde{h}\in \Ell^1(\mms,\bdsigma)}$ and
\begin{align}\label{Eq:Tuuut}
\int_\mms \widetilde{h}\d\bdsigma &= -\Ch^\kappa\big(h, \vert\nabla f\vert\big) -\int_{\{\vert\nabla f\vert > 0\}} h\,\big\langle\nabla f,\nabla \Delta f\big\rangle\,\vert\nabla f\vert^{-1}\d\meas.
\end{align}
Inserting $h := \phi\,\vert\nabla f\vert$ for $\phi\in\F_\bounded$  in \eqref{Eq:Tuuut} yields
\begin{align*}
\int_\mms \widetilde{\phi}\,\vert\nabla f\vert_\sim\d\bdsigma &= -\Ch^\kappa\big(\phi\,\vert\nabla f\vert, \vert\nabla f\vert\big) -\int_\mms \phi\,\big\langle\nabla f,\nabla \Delta f\big\rangle\d\meas.
\end{align*}
Hence for such $\phi$ and using the \autoref{Def:Measure valued Schr} of $\DELTA^{2\kappa}$,
\begin{align*}
\int_\mms \widetilde{\phi}\d\DELTA^{2\kappa}\vert\nabla f\vert^2 &= - \Ch^{2\kappa}\big(\phi,\vert\nabla f\vert^2\big)\\
&= - 2\int_\mms \vert\nabla f\vert\,\big\langle\nabla\phi,\nabla\vert\nabla f\vert\big\rangle\d\meas - 2\,\big\langle\kappa\,\big\vert\,\phi\, \vert\nabla f\vert^2\big\rangle\\
&= -2\,\Ch^{\kappa}\big(\phi\,\vert\nabla f\vert,\vert\nabla f\vert\big) +2\int_\mms \phi\,\big\vert\nabla\vert\nabla f\vert\big\vert^2\d\meas\\
&= 2\int_\mms \widetilde{\phi}\,\vert\nabla f\vert_\sim\d\bdsigma + 2\int_\mms\phi\,\big\langle\nabla f,\nabla\Delta f\big\rangle\d\meas\\
&\qquad\qquad + 2\int_\mms\phi\,\big\vert\nabla\vert\nabla f\vert\big\vert^2\d\meas.
\end{align*}
Since $\smash{\phi\in\F_\bounded}$ is arbitrary, we get
\begin{align}\label{Eq:Meas id}
\DELTA^{2\kappa}\vert\nabla f\vert^2 = 2\,\vert\nabla f\vert_\sim\,\bdsigma + 2\,\big\langle\nabla f,\nabla \Delta f \big\rangle\,\meas + 2\,\big\vert\nabla\vert\nabla f\vert\big\vert^2\,\meas.
\end{align}
Indeed, the r.h.s.~is well-defined since $\langle\nabla f,\nabla\Delta f\rangle\,\meas$ and $\smash{\big\vert\nabla \vert\nabla f\vert\big\vert^2\,\meas}$ define (signed) Borel measures of finite total variation. Setting $h := \vert\nabla f\vert$ in \eqref{Eq:Tuuut} implies that $\vert\nabla f\vert_\sim\,\bdsigma$ is finite as well, whence $\DELTA^{2\kappa} \vert\nabla f\vert^2$ is of finite total variation.

Finally, the second identity from \eqref{Eq:Identities Gamma2 DELTA} follows from combining \eqref{Eq:Meas id} with \eqref{Eq:Tuuut} for $h := \vert \nabla f\vert$, which in turn gives the first identity by the definition \eqref{Eq:Gamma_2 opertor}.
\end{proof}

\begin{remark}\label{Re:RCD cons} On $\RCD(K,\infty)$  spaces $(\mms,\met,\meas)$, $K\in\R$, according to \autoref{Ex:mms} the argument for the finiteness of $\Vert\DELTA^{2K}\vert\nabla f\vert^2\Vert_\TV$ is more straightforward: it follows by conservativeness of the heat flow $(\ChHeat_t)_{t\geq 0}$ \cite{ambrosio2014b,ambrosio2015} and does not require the detour over the $\BE_1(K,\infty)$ condition \cite[Lem.~2.6]{savare2014}. In our work, conservativeness is \emph{neither} assumed nor generally a \emph{consequence} of \autoref{Def:BE cond} (recall \autoref{Re:No global control}). 
\end{remark}

\begin{remark}[Caveat] The relation \eqref{Eq:Tuuut} suggests to derive that $\vert\nabla f\vert\in\Dom(\DELTA^\kappa)$, with $\DELTA^\kappa$ defined appropriately as in \autoref{Def:Measure valued Schr},  with
\begin{align*}
\DELTA^\kappa \vert\nabla f\vert = \bdsigma  + \big\langle \nabla f,\nabla\Delta f\big\rangle\,\vert \nabla f\vert^{-1}\,\meas.
\end{align*}
However, it is not clear if the r.h.s.~defines an element of $\smash{\Meas_\sigmafinR^\pm(\mms)}$ since neither the summands on the r.h.s.~typically define  finite measures, nor we really know whether the last summand is a signed measure (it might take both the value $\infty$ and $-\infty$).

The situation changes when treating the $\Ell^1$-Bochner inequality for test vector fields, see \autoref{Th:Vector Bochner} below. 
\end{remark}

\subsection{Lebesgue spaces and test objects}\label{Sub:Test objects} This section is a survey over the definition of the spaces $\Ell^p(T^*\mms)$ and $\Ell^p(T\mms)$, $p\in [1,\infty]$, of $p$-integrable (co-)tangent vector fields w.r.t.~$\meas$.

\subsubsection{The $\Ell^0$-modules $\Ell^0(T^*\mms)$ and $\Ell^0(T\mms)$}\label{Sub:L0 modules assoc to cotangent} Let $\Ell^0(T^*\mms)$ and $\Ell^0(T\mms)$ be the $\Ell^0$-modules as in \autoref{Sub:L0 modules}  associated to $\Ell^2(T^*\mms)$ and $\Ell^2(T\mms)$, i.e.\label{Not:L0's}
\begin{align*}
\Ell^0(T^*\mms) &:= \Ell^2(T^*\mms)^0,\\
\Ell^0(T\mms) &:= \Ell^2(T\mms)^0.
\end{align*}
The characterization of Cauchy sequences in these spaces \cite[p.~31]{gigli2018}  grants that the pointwise norms $\vert\cdot\vert\colon \Ell^2(T^*\mms)\to\Ell^2(\mms)$ and $\vert\cdot\vert\colon\Ell^2(T\mms)\to\Ell^2(\mms)$ as well as the musical isomorphisms $\flat\colon \Ell^2(T\mms)\to\Ell^2(T^*\mms)$ and $\sharp\colon \Ell^2(T^*\mms)\to\Ell^2(T\mms)$ uniquely extend to (non-relabeled) continuous maps $\vert\cdot \vert\colon \Ell^0(T^*\mms)\to\Ell^0(\mms)$, $\vert\cdot \vert\colon\Ell^0(T\mms)\to\Ell^0(\mms)$, $\flat\colon \Ell^0(T\mms)\to\Ell^0(T^*\mms)$ and $\sharp\colon \Ell^0(T^*\mms)\to\Ell^0(T\mms)$. (And the latter two will restrict to pointwise isometric module isomorphisms between the respective $\Ell^p$-spaces, $p\in[1,\infty]$, from \autoref{Sub:Lebesgue sp}.)

\subsubsection{The Lebesgue spaces $\Ell^p(T^*\mms)$ and $\Ell^p(T\mms)$}\label{Sub:Lebesgue sp} For $p\in [1,\infty]$, let $\Ell^p(T^*\mms)$ and $\Ell^p(T\mms)$ be the Banach spaces consisting of all $\omega\in\Ell^0(T^*\mms)$  and $X\in\Ell^0(T\mms)$ such that $\vert\omega\vert\in \Ell^p(\mms)$ and $\vert X\vert\in\Ell^p(\mms)$, respectively, endowed with the norms\label{Not:Lp sp.}
\begin{align*}
\Vert \omega\Vert_{\Ell^p(T^*\mms)} &:= \big\Vert \vert\omega\vert\big\Vert_{\Ell^p(\mms)},\\
\Vert X\Vert_{\Ell^p(T\mms)} &:= \big\Vert \vert X\vert\big\Vert_{\Ell^p(\mms)}.
\end{align*}
Since by \autoref{Pr:Generators cotangent module}, $\Ell^2(T^*\mms)$ is separable --- and so is $\Ell^2(T\mms)$ by \autoref{Th:Riesz theorem modules} --- one easily derives that if $p<\infty$, the spaces $\Ell^p(T^*\mms)$ and $\Ell^p(T\mms)$ are separable as well. Since $\Ell^2(T^*\mms)$ and $\Ell^2(T\mms)$ are reflexive as Hilbert spaces, by the discussion from \autoref{Sub:Duality} it follows that $\Ell^p(T^*\mms)$ and $\Ell^p(T\mms)$ are reflexive for every $p\in [1,\infty]$, and that for $q\in [1,\infty]$ such that $1/p+1/q=1$, in the sense of $\Ell^\infty$-modules we have the duality
\begin{align*}
 \Ell^p(T^*\mms)^* = \Ell^q(T\mms).
\end{align*}

\subsubsection{Test and regular objects}\label{Sub:Test reg} As in \cite[p.~102]{gigli2018}, using \autoref{Le:Mollified heat flow} we see that the linear span of all elements of the form $h\,\nabla g$, $g\in\Test_{\Ell^\infty}(\mms)$ and $h\in\Test(\mms)$, is weakly$^*$ dense in $\Ell^\infty(T\mms)$. This motivates to consider the subsequent subclasses of $\Ell^2(T\mms)$  consisting of \emph{test vector fields} or \emph{regular vector fields}, respectively:\label{Not:reg 1 vfs}
\begin{align*}
\Test_{\Ell^\infty}(T\mms) &:= \Big\lbrace\! \sum_{i=1}^n g_i\,\nabla f_i : n\in\N,\ f_i,g_i\in\Test_{\Ell^\infty}(\mms)\Big\rbrace,\\
\Test(T\mms) &:= \Big\lbrace\! \sum_{i=1}^n g_i\,\nabla f_i : n\in\N,\ f_i,g_i\in\Test(\mms)\Big\rbrace,\\
\Reg(T\mms) &:= \Big\lbrace\!\sum_{i=1}^n g_i\,\nabla f_i : n\in\N,\ f_i\in\Test(\mms),\ g_i\in \Test(\mms)\cup \R\,\One_\mms\Big\rbrace.
\end{align*}

\begin{remark} These three classes play different roles in the sequel. $\Test_{\Ell^\infty}(T\mms)$ is just needed for technical reasons when some second order $\Ell^\infty$-control is required. $\Test(T\mms)$ is usually the class of vector fields w.r.t.~which certain objects are defined by testing against, while $\Reg(T\mms)$ is the typical class of vector fields for which such objects \emph{are} defined. We make this distinction between $\Test(T\mms)$ and $\Reg(T\mms)$, which has not been done in \cite{gigli2018}, for the  reason that we want to include both vector fields with ``regular'' zeroth order part as well as pure gradient vector fields for differential objects such as the covariant derivative, see \autoref{Th:Properties W12 TM}, or the exterior differential, see \autoref{Th:Wd12 properties}. However, under the usual closures that we take below, it is not clear if gradient vector fields belong to those w.r.t.~test rather than regular objects. Compare with \autoref{Re:Test closure}.
\end{remark}

Of course $\Test_{\Ell^\infty}(T\mms)\subset \Test(T\mms),\Test_{\Ell^\infty}(T\mms)\subset \Ell^1(T\mms)\cap\Ell^\infty(T\mms)$, while merely $\Reg(T\mms)\subset \Ell^2(T\mms)\cap\Ell^\infty(T\mms)$. By \autoref{Le:Div g nabla f}, we have $\Test(T\mms)\subset\Dom_\TV(\DIV)\cap\Dom(\div)$ --- as well as $\norm X = 0$ for every $X\in\Test(T\mms)$ --- while only $\Reg(T\mms)\subset\Dom(\div)$. By \autoref{Pr:Generators cotangent module} and \autoref{Th:Riesz theorem modules}, all classes are dense in $\Ell^p(T\mms)$, $p\in [1,\infty)$, and weakly$^*$ dense in $\Ell^\infty(T\mms)$.  Furthermore, we set\label{Not:reg 1 forms}
\begin{align*}
\Test_{\Ell^\infty}(T^*\mms) &:= \Test_{\Ell^\infty}(T\mms)^\flat,\\
\Test(T^*\mms) &:= \Test(T\mms)^\flat,\\
\Reg(T^*\mms) &:= \Reg(T\mms)^\flat.
\end{align*}

\subsubsection{Lebesgue spaces on tensor products}\label{Sub:Leb sp tp} Denote the two-fold tensor products of $\Ell^2(T^*\mms)$ and $\Ell^2(T\mms)$, respectively, in the sense of \autoref{Sub:Tensor products} by\label{Not:L2 tensor pr}
\begin{align*}
\Ell^2((T^*)^{\otimes 2}\mms) &:= \Ell^2(T^*\mms)^{\otimes 2},\\
\Ell^2(T^{\otimes 2}\mms) &:= \Ell^2(T\mms)^{\otimes 2}.
\end{align*}
By the discussion from \autoref{Sub:Tensor products}, \autoref{Th:Module structure} and \autoref{Th:Riesz theorem modules}, both are separable Hilbert modules. They are pointwise isometrically module isomorphic: the respective  pairing is initially defined by
\begin{align*}
(\omega_1\otimes\omega_2)(X_1\otimes X_2) := \omega_1(X_1)\,\omega_2(X_2)\quad\meas\text{-a.e.}
\end{align*} 
for $\omega_1,\omega_2\in\Ell^2(T^*\mms)\cap\Ell^\infty(T^*\mms)$ and $X_1,X_2\in\Ell^2(T^*\mms)\cap\Ell^\infty(T^*\mms)$, and is ex\-ten\-ded by linearity and continuity to $\Ell^2((T^*)^{\otimes 2}\mms)$ and $\Ell^2(T^{\otimes 2}\mms)$, respectively. By a slight abuse of notation, this pairing,  with \autoref{Th:Riesz theorem modules}, induces the musical isomorphisms $\smash{\flat\colon \Ell^2(T^{\otimes 2}\mms) \to \Ell^2((T^*)^{\otimes 2}\mms)}$ and $\smash{\sharp := \flat^{-1}}$ given by
\begin{align*}
A^\sharp : T := A(T) =: A : T^\flat\quad\meas\text{-a.e.}
\end{align*}

We let $\Ell^p((T^*)^{\otimes 2}\mms)$ and $\Ell^p(T^{\otimes 2}\mms)$, $p\in\{0\}\cup [1,\infty]$, be defined similarly to \autoref{Sub:L0 modules assoc to cotangent} and \autoref{Sub:Lebesgue sp}. For $p\in[1,\infty]$, these spaces  naturally become Banach which, if $p<\infty$, are separable.

Lastly, we define the subsequent $\Ell^p$-dense sets, $p\in [1,\infty]$, intended strongly if $p<\infty$ and weakly$^*$ if $p = \infty$, reminiscent of \autoref{Sub:Tensor products}:\label{Not:Two fold tps}
\begin{align*}
\Test_{\Ell^\infty}((T^*)^{\otimes 2}\mms) &:= \Test_{\Ell^\infty}(T^*\mms)^{\odot 2},\\
\Test_{\Ell^\infty}(T^{\otimes 2}\mms) &:= \Test_{\Ell^\infty}(T\mms)^{\odot 2},\\
\Test((T^*)^{\otimes 2}\mms) &:= \Test(T^*\mms)^{\odot 2},\\
\Test(T^{\otimes 2}\mms) &:= \Test(T\mms)^{\odot 2},\\
\Reg((T^*)^{\otimes 2}\mms) &:= \Reg(T^*\mms)^{\odot 2},\\
\Reg(T^{\otimes 2}\mms) &:= \Reg(T\mms)^{\odot 2}.
\end{align*}

\subsubsection{Lebesgue spaces on exterior products} Given any $k\in\N_0$, we set\label{Not:Lp sp ext p}
\begin{align*}
\Ell^2(\Lambda^kT^*\mms) &:= \Lambda^k\Ell^2(T^*\mms),\\
\Ell^2(\Lambda^kT\mms) &:= \Lambda^k\Ell^2(T\mms),
\end{align*}
where the exterior products are intended as in \autoref{Sub:Exterior products}. For $k\in\{0,1\}$, we employ the consistent interpretations
\begin{align*}
\Ell^2(\Lambda^1T^*\mms) &:= \Ell^2(T^*\mms),\\
\Ell^2(\Lambda^1T\mms) &:= \Ell^2(T\mms),\\
\Ell^2(\Lambda^0T^*\mms) &:= \Ell^2(\Lambda^0T\mms) := \Ell^2(\mms).
\end{align*}
By \autoref{Sub:Exterior products}, these are naturally Hilbert modules. As in \autoref{Sub:Leb sp tp}, $\smash{\Ell^2(\Lambda^kT^*\mms)}$ and $\smash{\Ell^2(\Lambda^kT\mms)}$ are pointwise isometrically module isomorphic. For brevity, the induced pointwise pairing between $\omega\in\Ell^2(\Lambda^kT^*\mms)$ and $X_1\wedge\dots X_k\in\Ell^2(\Lambda^kT\mms)$, $X_1,\dots,X_k\in\Ell^2(T\mms)\cap\Ell^\infty(T\mms)$, is written
\begin{align*}
\omega(X_1,\dots,X_k) := \omega(X_1\wedge\dots\wedge X_k).
\end{align*}

We let $\Ell^p(\Lambda^kT^*\mms)$ and $\Ell^p(\Lambda^kT\mms)$, $p\in\{0\}\cup [1,\infty]$, be  as in \autoref{Sub:L0 modules assoc to cotangent} and \autoref{Sub:Lebesgue sp}. For $p\in[1,\infty]$, these spaces  are Banach and, if $p<\infty$, additionally separable.

Let the formal $k$-th exterior products, $k\in\N_0$, of the classes from \autoref{Sub:Test reg} be defined through\label{Not:test lambda k}
\begin{align*}
\Test_{\Ell^\infty}(\Lambda^kT^*\mms) &:= \Big\lbrace\!\sum_{i=1}^n f_i^0\d f_i^1\wedge\dots\wedge \rmd f_i^k : n\in\N,\ f_i^j \in\Test_{\Ell^\infty}(\mms)\Big\rbrace,\\
\Test_{\Ell^\infty}(\Lambda^kT\mms) &:= \Big\lbrace\!\sum_{i=1}^n f_i^0\,\nabla f_i^1\wedge\dots\wedge \nabla f_i^k : n\in\N,\ f_i^j \in\Test_{\Ell^\infty}(\mms)\Big\rbrace,\\
\Test(\Lambda^kT^*\mms) &:= \Big\lbrace\!\sum_{i=1}^n f_i^0\d f_i^1\wedge\dots\wedge \rmd f_i^k : n\in\N,\ f_i^j \in\Test(\mms)\Big\rbrace,\\
\Test(\Lambda^kT\mms) &:= \Big\lbrace\!\sum_{i=1}^n f_i^0\,\nabla  f_i^1\wedge\dots\wedge \nabla f_i^k : n\in\N,\ f_i^j \in\Test(\mms)\Big\rbrace\\
\Reg(\Lambda^kT^*\mms) &:= \Big\lbrace\!\sum_{i=1}^n f_i^0\d f_i^1\wedge\dots\wedge \rmd f_i^k : n\in\N,\ f_i^j \in\Test(\mms),\\
&\qquad\qquad f_i^0\in\Test(\mms)\cup\R\,\One_\mms\Big\rbrace,\textcolor{white}{\sum_i^n}\\
\Reg(\Lambda^kT\mms) &:= \Big\lbrace\!\sum_{i=1}^n f_i^0\,\nabla f_i^1\wedge\dots\wedge \nabla f_i^k : n\in\N,\ f_i^j \in\Test(\mms),\\
&\qquad\qquad f_i^0\in\Test(\mms)\cup\R\,\One_\mms\Big\rbrace.\textcolor{white}{\sum_i^n}
\end{align*}
We\label{Not:reg lambda k} employ the evident interpretations for $k=1$, while the respective spaces for $k=0$ are identified with those spaces  to which their generic elements's zeroth order terms belong to. These classes are dense in their respective $\Ell^p$-spaces, $p\in[1,\infty]$ --- strongly if $p<\infty$, and weakly$^*$ if $p=\infty$.

\section{Hessian}\label{Sec:Hessian}

\subsection{The Sobolev space $\Dom(\Hess)$}\label{Sub:Dom Hess} Now we define the key object of our second order differential structure, namely the \emph{Hessian} of suitable functions $f\in \F$. We choose an integration by parts procedure as in \cite[Subsec.~3.3.1]{gigli2018}, motivated by the subsequent Riemannian example.

\begin{example}\label{Ex:Hess} Let $\mms$ be a  Riemannian manifold with boundary. The metric compatibility of $\nabla$ allows us to rephrase the definition of the Hessian $\Hess f \in \Gamma((T^*)^{\otimes 2}\mms)$ of a function $f\in \Cont^\infty(\mms)$ pointwise as
\begin{align}\label{La:Smth hess}
\begin{split}
	2\Hess f(\nabla g_1,\nabla g_2) &= 2\,\big\langle\nabla_{\nabla g_1}\nabla f,\nabla g_2\big\rangle\\
	&= \big\langle \nabla g_1, \nabla \langle\nabla f,\nabla g_2\rangle\big\rangle + \big\langle \nabla g_2, \nabla \langle\nabla f,\nabla g_1\rangle\big\rangle\\
	&\qquad\qquad - \big\langle \nabla f, \nabla \langle\nabla g_1,\nabla g_2\rangle\big\rangle
	\end{split}
\end{align}
for every $g_1,g_2\in \Cont_\comp^\infty(\mms)$, see e.g.~\cite[p.~28]{petersen2006}. The first equality ensures that $\Hess f$ is $\Cont^\infty$-linear in both components. Thus, since smooth gradient vector fields  locally generate $TM$, the second equality characterizes the Hessian of $f$. 

We now restrict our attention to those $g_1$ and $g_2$ whose derivatives constitute \emph{Neumann vector fields}, i.e.~
\begin{align}\label{Eq:Vanishing normal}
\langle\nabla g_1,\sfn\rangle = \langle\nabla g_2,\sfn\rangle = 0\quad\text{on }\partial\mms,
\end{align}
e.g.~to $\smash{g_1,g_2\in\Cont_\comp^\infty(\mms^\circ)}$. Multiply \eqref{La:Smth hess} by a function $h\in\Cont_\comp^\infty(\mms)$ and integrate (by parts). In this case, recall that $h\,\nabla g_1,h\,\nabla g_2\in \Dom_{\TV}(\DIV)\cap\Dom(\div)$ with
\begin{align}\label{Eq:Smooth div identities}
\begin{split}
\DIV(h\, \nabla g_1) = \div_\vol(h\,\nabla g_1)\,\vol,\\
\DIV(h\, \nabla g_2)  = \div_\vol(h\,\nabla g_2)\,\vol
\end{split}
\end{align}
by \autoref{Ex:Mflds with boundary}. The resulting integral identity reads
\begin{align*}
&2\int_\mms h\Hess f(\nabla g_1,\nabla g_2)\d \vol\\
&\qquad\qquad = -\int_\mms \langle \nabla f,\nabla g_2\rangle\div_\vol(h\,\nabla g_1) \d\vol  - \int_\mms \langle\nabla f,\nabla g_1\rangle\div_\vol(h\,\nabla g_2)\d\vol\\
&\qquad\qquad\qquad\qquad  -\int_\mms h\,\big\langle\nabla f,\nabla\langle\nabla g_1,\nabla g_2\rangle\big\rangle\d\vol.
\end{align*}
In turn, this integral identity characterizes $\Hess f$ on $\mms^\circ$ by the arbitrariness of $g_1$, $g_2$ and $h$, and hence on $\mms$ by the existence of a smooth extension to all of $\mms$.
\end{example}

Observe that on the r.h.s.~of the previous integral identity, no second order expression in $f$ is present. Moreover, as we have already noted in \autoref{Le:Div g nabla f}, $\Test(T\mms)$ is a large class of vector fields obeying a nonsmooth version of \eqref{Eq:Vanishing normal}. Next, note that all volume integrals on the r.h.s.~--- with $\vol$ replaced by $\meas$ --- are well-defined for every $f\in \F$ and $g_1,g_2,h\in\Test(\mms)$. Indeed, the first one exists since $\nabla g_2\in\Ell^\infty(\mms)$, whence $\langle\nabla f,\nabla g_2\rangle\in \Ell^2(\mms)$, and since $h\,\nabla g_1\in \Dom_{\TV}(\DIV)\cap\Dom(\div)$ with $\div(h\,\nabla g_1)  \in \Ell^2(\mms)$ by \autoref{Le:Div g nabla f}. An analogous argument applies for the second volume integral. The last one is well-defined since $\langle \nabla g_1,\nabla g_2\rangle \in \F$ thanks to \autoref{Pr:Bakry Emery measures} and polarization, and since $h\in\Ell^\infty(\mms)$.

These observations lead to the following definition. 

\begin{definition}\label{Def:Hess}
	We define the space $\Dom(\Hess)$ to consist of all $f\in \F$ for which there exists $A\in\Ell^2((T^*)^{\otimes 2}\mms)$ such that for every $g_1,g_2, h\in \Test(\mms)$,
	\begin{align*}
		&2\int_{\mms} h\,A(\nabla g_1,\nabla g_2)\d\meas\\
		&\qquad\qquad = -\int_{\mms} \langle \nabla f,\nabla g_1\rangle\div(h\,\nabla g_2) \d\meas - \int_{\mms} \langle\nabla f,\nabla g_2\rangle\div(h\,\nabla g_1)\d\meas\\
		&\qquad\qquad\qquad\qquad -\int_{\mms}h\,\big\langle \nabla f,\nabla\langle \nabla g_1,\nabla g_2\rangle\big\rangle\d\meas.
	\end{align*}
	If such an $A$ exists, it is unique, denoted by $\Hess f$ and termed the \emph{Hessian} of $f$.
\end{definition}

Indeed, given any $f\in \Dom(\Hess)$ there is at most one $A$ as in \autoref{Def:Hess} by density of $\Test(T^{\otimes 2}\mms)$ in $\Ell^2(T^{\otimes 2}\mms)$, since $\Test(\mms)$ is an algebra. In particular, $\Dom(\Hess)$ is a vector space and $\Hess$ is a linear operator on it. Further elementary properties are collected in \autoref{Th:Hess properties}.

The space $\Dom(\Hess)$ is endowed with the norm $\Vert \cdot\Vert_{\Dom(\Hess)}$ given by
\begin{align*}
\big\Vert f\big\Vert_{\Dom(\Hess)}^2 := \big\Vert f\big\Vert_{\Ell^2(\mms)}^2 + \big\Vert \rmd f\big\Vert_{\Ell^2(T^*\mms)}^2 + \big\Vert \!\Hess f\big\Vert_{\Ell^2((T^*)^{\otimes 2}\mms)}^2.
\end{align*}
Furthermore we define the energy functional $\E_2\colon \F\to [0,\infty]$ by
\begin{align*}
		\E_2(f) := \begin{cases}\displaystyle \int_{\mms} \big\vert\!\Hess f\big\vert_\HS^2\d\meas & \text{if } f\in \Dom(\Hess),\\
		\infty & \text{otherwise}.
	\end{cases}
\end{align*}

\begin{theorem}\label{Th:Hess properties} The space $\Dom(\Hess)$, the Hessian $\Hess$ and the functional $\Ch_2$ have the following properties.
	\begin{enumerate}[label=\textnormal{\textcolor{black}{(}\roman*\textcolor{black}{)}}]
		\item\label{La:H1} $\Dom(\Hess)$ is a separable Hilbert space w.r.t.~$\smash{\Vert\cdot\Vert_{\Dom(\Hess)}}$.
		\item\label{La:H2} The Hessian is a closed operator on $\Dom(\Hess)$, i.e.~the image of the map $\Id\times \Hess\colon \Dom(\Hess)\to \F\times\Ell^2((T^*)^{\otimes 2}\mms)$ is a closed subspace of $\F\times\Ell^2((T^*)^{\otimes 2}\mms)$.
		\item\label{La:H3} For every $f\in \Dom(\Hess)$, the tensor $\Hess f$ is symmetric, i.e.
\begin{align*}
\Hess f = (\Hess f)^\top
\end{align*}		
		according to the definition of the transpose from  \eqref{Eq:Transpose}.
		\item\label{La:H4} $\E_2$ is $\F$-lower semicontinuous, and for every $f\in \F$, 
		\begin{align*}
			\E_2(f) &= \sup\!\Big\lbrace\! -\! \sum_{l=1}^r \int_{\mms} \langle \nabla f,\nabla g_l\rangle\div(h_l\,h'_l\,\nabla g'_l) \d\meas\\
			&\qquad\qquad - \sum_{l=1}^r \int_{\mms} \langle \nabla f,\nabla g_l'\rangle\div(h_l\,h'_l\,\nabla g_l) \d\meas\\
			&\qquad\qquad - \sum_{l=1}^r \int_{\mms} h_l\,h_l'\,\big\langle \nabla f,\nabla\langle\nabla g_l,\nabla g_l'\rangle\big\rangle\d\meas\\
			&\qquad\qquad - \int_{\mms} \Big\vert\! \sum_{l=1}^r h_l\,h'_l\,\nabla g_l\otimes \nabla g'_l\Big\vert^2\d\meas :\\
			&\qquad\qquad\qquad\qquad r\in \N,\ g_l, g'_l, h_l, h_l'\in\Test(\mms)\Big\rbrace.
		\end{align*}
	\end{enumerate}
\end{theorem}

\begin{proof} Item \ref{La:H2} follows since the r.h.s.~of the defining property of the Hessian in \autoref{Def:Hess} is continuous in $f$ and $A$ w.r.t.~weak convergence in $\F$ and $\Ell^2((T^*)^{\otimes 2}\mms)$, respectively, for fixed $g_1,g_2,h\in\Test(\mms)$.

The Hilbert space property of $\smash{\Dom(\Hess)}$ in \ref{La:H1} is a direct consequence of the completeness of $\F$, \ref{La:H2} and since $\smash{\Vert\cdot\Vert_{\Dom(\Hess)}}$ trivially satisfies the parallelogram identity. Hence, we are left with the separability of $\Dom(\Hess)$. Since $\F$ and $\Ell^2((T^*)^{\otimes 2}\mms)$ are separable by \autoref{Le:Props q.r. s.l. Dirichlet form},  \autoref{Pr:Generators cotangent module} and the discussion from \autoref{Sub:Tensor products}, their Cartesian product is a separable Hilbert space w.r.t.~the norm $\Vert\cdot\Vert$, where
\begin{align*}
\Vert (f,A) \Vert^2 := \big\Vert f\big\Vert_{\F}^2 + \big\Vert A\big\Vert_{\Ell^2((T^*)^{\otimes 2}\mms)}^2.
\end{align*}
In particular, $\smash{\Id\times\Hess\colon \Dom(\Hess)\to \F\times\Ell^2((T^*)^{\otimes 2}\mms)}$ is a bijective isometry onto its image, whence the claim follows from \ref{La:H2}.

Concerning \ref{La:H3}, setting $h := h_1\,h_2$ with $h_1,h_2\in\Test(\mms)$, we easily see that for every $g_1,g_2\in\Test(\mms)$ the r.h.s.~of the defining property of $\Hess f$, $f\in \Dom(\Hess)$, in \autoref{Def:Hess} is symmetric in $h_1$ and $h_2$ as well as $g_1$ and $g_2$, respectively, and it is furthermore bilinear in $h_1\,\nabla g_1$ and $h_2\,\nabla g_2$. Hence, using \eqref{Eq:Transpose}  we deduce the symmetry of $\Hess f$, $f\in \Dom(\Hess)$, on $\Test(T^{\otimes 2}\mms)$ and hence on all of $\Ell^2(T^{\otimes 2}\mms)$ by a density argument.

The $\F$-lower semicontinuity of $\Ch_2$ in \ref{La:H4}  directly follows since bounded subsets of the Hilbert space $\smash{\Ell^2((T^*)^{\otimes 2}\mms)}$ are weakly relatively compact, combined with Mazur's lemma and \ref{La:H2}.

We finally turn to the duality formula in \ref{La:H4}. 

Let us first prove the inequality ``$\geq$'', for which we assume without restriction that $f\in \Dom(\Hess)$. By duality of $\Dom(\Hess)$ and its Hilbert space dual $\Dom(\Hess)'$ as well as the density of $\Test(T^{\otimes 2}\mms)$ in $\Ell^2(T^{\otimes 2}\mms)$,
\begin{align*}
\Ch_2(f) &= \sup\!\Big\lbrace 2\sum_{i=1}^n\int_\mms \Hess f(X_i,X_i')\d\meas\\
&\qquad\qquad - \int_\mms \Big\vert\! \sum_{i=1}^n X_i\otimes X_i'\Big\vert^2\d\meas : n\in\N,\ X_i,X_i'\in\Test(T\mms)\Big\rbrace.
\end{align*}
Let us write $\smash{X_i \otimes X_i' := h_{i1}\,h_{i1}'\,\nabla g_{i1}\otimes \nabla g_{i1}' + \dots + h_{im}\,h_{im}'\,\nabla g_{im}\otimes g_{im}'}$ for certain elements $g_{ij}, g_{ij}', h_{ij}, h_{ij}'\in\Test(\mms)$, $i\in \{1,\dots,n\}$ and $j\in \{1,\dots,m\}$ with $n,m\in\N$. Then by \autoref{Def:Hess} and since $\Test(\mms)$ is an algebra,
\begin{align*}
&2\sum_{i=1}^n\int_\mms \Hess f(X_i,X_i')\d\meas\\
&\qquad\qquad = 2\sum_{i=1}^n\sum_{j=1}^m \int_\mms h_{ij}\,h_{ij}' \Hess f(\nabla g_{ij},\nabla g_{ij}')\d\meas\\
&\qquad\qquad = - \sum_{i=1}^n\sum_{j=1}^m\int_\mms \langle \nabla f,\nabla g_{ij}\rangle\div(h_{ij}\,h_{ij}'\,\nabla g_{ij}')\d\meas\\
&\qquad\qquad\qquad\qquad - \sum_{i=1}^n\sum_{j=1}^m\int_\mms \langle\nabla f,\nabla g_{ij}'\rangle\div(h_{ij}\,h_{ij}'\,\nabla g_{ij})\d\meas\\
&\qquad\qquad\qquad\qquad - \sum_{i=1}^n\sum_{j=1}^m \int_\mms h_{ij}\,h_{ij}'\,\big\langle \nabla f,\nabla\langle\nabla g_{ij},\nabla g_{ij}'\rangle\big\rangle\d\meas,
\end{align*}
which terminates the proof of ``$\geq$''. 

Turning to ``$\leq$'' in \ref{La:H4}, we may and will assume without loss of generality that the supremum, henceforth denoted by $C$, on the r.h.s.~of the claimed formula is finite. Consider the operator $\Phi\colon\Test(T^{\otimes 2}\mms)\to \R$ given by
\begin{align}\label{Eq:Blurr}
\begin{split}
&2\,\Phi\sum_{i=1}^n\sum_{j=1}^m h_{ij}\,h_{ij}'\,\nabla g_{ij}\otimes \nabla g_{ij}'\\
&\qquad\qquad := -\sum_{i=1}^n\sum_{j=1}^m \int_\mms \langle\nabla f,\nabla g_{ij}\rangle\div(h_{ij}\,h_{ij}'\,\nabla g_{ij}')\d\meas\\
&\qquad\qquad\qquad\qquad - \sum_{i=1}^n \sum_{j=1}^m \int_\mms \langle\nabla f,\nabla g_{ij}'\rangle\div(h_{ij}\,h_{ij}'\,\nabla g_{ij})\d\meas\\
&\qquad\qquad\qquad\qquad - \sum_{i=1}^n\sum_{j=1}^m\int_\mms h_{ij}\,h_{ij}'\,\big\langle\nabla f,\nabla\langle\nabla g_{ij},\nabla g_{ij}'\rangle\big\rangle\d\meas.
\end{split}
\end{align}
The value of $\Phi(T)$ is independent of the particular way of writing $T\in\Test(T^{\otimes 2}\mms)$. Indeed, if $T= 0$ but $\Phi(T)\neq 0$, letting $\lambda \to \infty$ in the identity $\Phi(\lambda\,T)\,\sgn\Phi(T) = \Phi(T)\,\lambda\,\sgn\Phi(T)$ implied by \eqref{Eq:Blurr} would contradict the assumption that $C < \infty$. The map $\Phi$ is thus well-defined, it is  linear, and for every $T\in\Test(T^{\otimes 2}\mms)$,
\begin{align*}
2\,\Phi(T) \leq C + \big\Vert T\big\Vert_{\Ell^2(T^{\otimes 2}\mms)}^2.
\end{align*}
Replacing $T$ by $\lambda\, T$ and optimizing over $\lambda\in\R$ gives
\begin{align}\label{Eq:Phi T norm}
\vert \Phi(T)\vert \leq \sqrt{C}\,\Vert T\Vert_{\Ell^2(T^{\otimes 2}\mms)}
\end{align}
for every $T\in\Test(T^{\otimes 2}\mms)$. Hence, $\Phi$ uniquely induces a (non-relabeled) element of the Hilbert space dual $\Ell^2(T^{\otimes 2}\mms)'$ of $\Ell^2(T^{\otimes 2}\mms)$. By  \autoref{Th:Riesz theorem modules}, we find a unique element $\smash{A'\in \Ell^2((T^*)^{\otimes 2}\mms)}$ such that
\begin{align*}
\Phi(T) = \int_\mms A'(T)\d\meas
\end{align*}
for every $T\in \Ell^2(T^{\otimes 2}\mms)$. Now \autoref{Le:Product convergence}, \autoref{Le:Div g nabla f} as well as the continuity of $\Phi$ allow us to replace the terms $\smash{h_{ij}\,h_{ij}'}$ by arbitrary elements $\smash{k_{ij}\in\Test(\mms)}$, still retaining the identity \eqref{Eq:Blurr} with $\smash{k_{ij}}$ in place of $\smash{h_{ij}\,h_{ij}'}$, $i\in\{1,\dots,n\}$ and $j\in \{1,\dots,m\}$. In particular, by  \autoref{Def:Hess}, we deduce that $f\in \Dom(\Hess)$ and $A' = \Hess f$. By \autoref{Th:Riesz theorem modules} again and \eqref{Eq:Phi T norm}, we obtain
\begin{align*}
\Vert \!\Hess f\Vert_{\Ell^2((T^*)^{\otimes 2}\mms)} = \Vert \Phi\Vert_{\Ell^2(T^{\otimes 2}\mms)'} \leq \sqrt{C},
\end{align*}
which is precisely what was left to prove.
\end{proof}

\begin{remark} If $\Ch_2$ is extended to $\Ell^2(\mms)$ by $\Ch_2(f) := \infty$ for $f\in \Ell^2(\mms)\setminus \F$, it is unclear if the resulting functional is $\Ell^2$-lower semicontinuous. To bypass this issue in applications, one might instead use that by \autoref{Th:Hess properties}, the functional $\smash{\Ch_2^\varepsilon\colon \Ell^2(\mms)\to [0,\infty]}$ given by
\begin{align*}
\Ch_2^\varepsilon(f) := \begin{cases}
\displaystyle\varepsilon\int_\mms\vert\nabla f\vert^2\d\meas + \int_\mms\big\vert\!\Hess f\big\vert_\HS^2\d\meas & \text{if }f\in \Dom(\Hess),\\
\infty & \text{otherwise}
\end{cases}
\end{align*} 
is $\Ell^2$-lower semicontinuous for every $\varepsilon > 0$. 

If $\mms$ is, say, a compact Riemannian manifold without boundary, one can easily prove using the Bochner identity  that the (nonpositive)  generator associated with $\Ch_2^\varepsilon$ \cite[Thm.~1.3.1]{fukushima2011} is the Paneitz-type operator
\begin{align*}
-\Delta ^2 +\varepsilon\,\Delta f + \div(\Ric^\flat\,\nabla\cdot).
\end{align*}
\end{remark}

\begin{remark}\label{Re:Hess tr Delta} In general, the Hessian is not the trace of the Laplacian in the sense of \eqref{Eq:Pointwise trace}. This already happens on weighted Riemannian manifolds without boundary: of course, the associated Laplacian $\Delta$ is defined by partial integration w.r.t.~the reference measure \cite[Sec.~3.6]{grigoryan2009}, while the definition of Hessian only depends on the metric tensor. See also the second part of \autoref{Ex:Mflds with boundary}. Examples of abstract spaces for which this \emph{is} the case --- and which currently enjoy high research interest \cite{brue2020, dephilippis2018, honda2020} --- are \emph{noncollapsed} $\RCD(K,N)$ spaces, $K\in\R$ and $N\in [1,\infty)$ \cite[Thm.~1.12]{dephilippis2018}. See also \autoref{Re:Dim RCD}  below.
\end{remark}

\begin{remark}\label{Re:Conf trafos}
In line with \autoref{Re:Hess tr Delta}, although a priori $\meas$ plays a role in Definition \ref{Def:Hess}, we expect the Hessian to only depend on conformal transformations of $\langle\cdot,\cdot\rangle$, but not on drift transformations of $\meas$. For instance, this is known on $\RCD^*(K,N)$ spaces, $K\in\R$ and $N\in [1,\infty)$, see e.g.~\cite[Prop.~3.11]{han2019} or \cite[Lem.~2.16]{hansturm2019}, and it does not seem hard to adapt the arguments from \cite{han2019} to more general settings.
\end{remark}

\begin{remark} As an alternative to \autoref{Def:Hess}, one can define $\Dom(\Hess)$ as the finiteness domain of the r.h.s.~of the duality formula in \ref{La:H4} in \autoref{Th:Hess}. The Hessian of $f\in \Dom(\Hess)$ is then well-defined by the same duality arguments as in the proof of \autoref{Th:Hess}.
\end{remark}

\subsection{Existence of many functions in $\Dom(\Hess)$} Up to now, we still do not know whether $\Dom(\Hess)$ is nontrivial. The ultimate goal of this section is to prove that $\Test(\mms)\subset \Dom(\Hess)$ in \autoref{Th:Hess}, whence $\Dom(\Hess)$ is even dense in $\Ell^2(\mms)$.

The strategy is reliant on \cite[Subsec.~3.3.2]{gigli2018}, which has itself been inspired by the ``self-improvement'' works \cite{bakry1985a,bakry1985b}, see \cite[Rem.~3.3.10]{gigli2018} and also \cite{erbar2020, savare2014, sturm2018}. The key technical part (not only for \autoref{Th:Hess}, but in fact  for \autoref{Th:Ricci measure} below as well) is contained in \autoref{Le:Extremely key lemma}, where --- loosely speaking and up to introducing the relevant objects later --- we show that
\begin{align*}
\vert\nabla X : T\vert^2 \leq \Big[\Delta^{2\kappa} \frac{\vert X\vert^2}{2} + \big\langle X,(\Hodge X^\flat)^\sharp\big\rangle - \big\vert(\nabla X)_\asym\big\vert_\HS^2 \Big]\,\big\vert T\big\vert_\HS^2\quad\meas\text{-a.e.}
\end{align*}
for  $X,T\in\Test(T\mms)$. Of course, neither we introduced the covariant derivative $\nabla$, \autoref{Def:Cov der} or the Hodge Laplacian $\Hodge$, \autoref{Def:Hodge Lapl}, yet,  nor in general we have $\vert X\vert^2\in\Dom(\Delta^{2\kappa})$ for  $X\in\Test(T\mms)$. Reminiscent of \autoref{Pr:Bakry Emery measures} and \cite[Cor.~6.3]{erbar2020}, we instead rephrase the above inequality in terms of measures, and the involved objects  $\nabla X$ and $\smash{\Hodge X^\flat}$ therein as the ``r.h.s.'s of the identities one would expect for $\nabla X$ and $\smash{\Hodge X^\flat}$ for $X\in\Test(\mms)$'', rigorously proven in \autoref{Th:Properties W12 TM} and \autoref{Le:Hodge test} below. In particular, by optimization over $T\in\Test(T\mms)$,
\begin{align*}
\big\vert (\nabla X)_\sym\big\vert_\HS^2 \leq \Delta^{2\kappa}\frac{\vert X\vert^2}{2} + \big\langle X, (\Hodge X^\flat)^\sharp\big\rangle - \big\vert(\nabla X)_\asym\big\vert_\HS^2\quad\meas\text{-a.e.},
\end{align*}
which is the Bochner inequality for vector fields according to \eqref{Eq:sym plus asym}. For $X := \nabla f$, $f\in\Test(\mms)$, this essentially provides \autoref{Th:Hess}. Details about this inequality for general $X\in\Reg(T\mms)$, leading to \autoref{Th:Ricci measure}, are due to  \autoref{Le:Pre.Bochner}.

We start with a technical preparation. Given $\smash{\mu,\nu\in \Meas_\fin^+(\mms)}$, we define the Borel measure  $\smash{\sqrt{\mu\,\nu}\in \Meas_\fin^+(\mms)}$ as follows. Let $\smash{\iota\in \Meas_\fin^+(\mms)}$ with $\mu\ll\iota$ and $\nu\ll\iota$ be arbitrary, denote the respective densities w.r.t.~$\iota$ by $\smash{f,g\in \Ell^1(\mms,\iota)}$, and set
\begin{align*}
	\sqrt{\mu\,\nu} := \sqrt{f\,g}\,\iota.
\end{align*}
For instance, one can choose $\iota:=\vert \mu\vert + \vert\nu\vert$ \cite[Thm.~30.A]{halmos1950} --- in fact, the previous definition is independent of the choice of $\iota$, whence $\sqrt{\mu\,\nu}$ is well-defined.

The following important measure theoretic lemma is due to \cite[Lem.~3.3.6]{gigli2018}.

\begin{lemma}\label{Le:Measure lemma} Let $\mu_1,\mu_2,\mu_3\in\Meas_\fin^\pm(\mms)$ satisfy the inequality
\begin{align*}
\lambda^2\,\mu_1 + 2\lambda\,\mu_2 + \mu_3 \geq 0
\end{align*}	
for every $\lambda\in\R$. Then the following properties hold. 
\begin{enumerate}[label=\textnormal{\textcolor{black}{(}\roman*\textcolor{black}{)}}]
	\item\label{La:Le1} The elements $\mu_1$ and $\mu_3$ are nonnegative, and
	\begin{align*}
		\vert\mu_2\vert\leq \sqrt{\mu_1\,\mu_3}.
	\end{align*}
	\item\label{La:Le2} We have $\mu_2\ll\mu_1$, $\mu_2\ll\mu_3$ and
	\begin{align*}
		\Vert\mu_2\Vert_\TV \leq \sqrt{\Vert \mu_1\Vert_\TV\,\Vert \mu_3\Vert_\TV}.
	\end{align*}
	\item\label{La:Le3} The $\meas$-singular parts $(\mu_1)_\perp$ and $(\mu_3)_\perp$ of $\mu_1$ and $\mu_3$ are nonnegative. Moreover, expressing the densities of the $\meas$-absolutely continuous parts of $\mu_i$ by $\rho_i := \rmd(\mu_i)_\ll/\rmd\meas \in \Ell^1(\mms)$, $i\in\{1,2,3\}$, we have
	\begin{align*}
		\vert\rho_2\vert^2\leq \rho_1\,\rho_3\quad\meas\text{-a.e.}
	\end{align*}
\end{enumerate}
\end{lemma}

In the subsequent lemma, all terms where $N'$ is infinite are interpreted as being zero. Similar proofs can be found in \cite{braun2020, gigli2018, han2018}.

\begin{lemma}\label{Le:Extremely key lemma} Let $N'\in [N,\infty]$, $n,m\in\N$, $f,g \in\Test(\mms)^n$ and $h\in\Test(\mms)^m$. Define $\mu_1[f,g] \in\Meas_\fin^\pm(\mms)$ as
	\begin{align*}
	\mu_1[f,g] &:= \sum_{i,i'=1}^n \widetilde{g}_i\,\widetilde{g}_{i'} \,\bdGamma_2^{2\kappa}(f_i, f_{i'}) +2 \sum_{i,i'=1}^n g_i\,\rmH[f_i](f_{i'},g_{i'})\,\meas\\
	&\qquad\qquad +\frac{1}{2} \sum_{i,i'=1}^n \big[\langle\nabla f_i,\nabla f_{i'}\rangle\,\langle\nabla g_i,\nabla g_{i'}\rangle + \langle\nabla f_i,\nabla g_{i'}\rangle\,\langle \nabla g_i,\nabla f_{i'}\rangle\big]\,\meas\\
	&\qquad \qquad - \frac{1}{N'}\,\Big[\! \sum_{i=1}^n \big[g_i\,\Delta f_i + \langle \nabla f_i,\nabla g_i\rangle\big]\Big]^2\,\meas
	\end{align*}
	As in \autoref{Le:Measure lemma}, we denote the density of the $\meas$-absolutely continuous part of $\mu_1[f,g]$ by $\smash{\rho_1[f,g] := \rmd\mu_1[f,g]_\ll/\rmd \meas\in \Ell^1(\mms)}$. Then the $\meas$-singular part $\smash{\mu_1[f,g]_\perp}$ of $\mu_1[f,g]$ as well as $\rho_1[f,g]$ are nonnegative, and
	\begin{align*}
	&\Big[\!\sum_{i=1}^n\sum_{j=1}^m \big[\langle\nabla f_i,\nabla h_j\rangle\,\langle\nabla g_i,\nabla h_j\rangle + g_i\,\rmH[f_i](h_j,h_j)\big]\\
	&\qquad\qquad\qquad\qquad - \frac{1}{N'} \sum_{i=1}^n\sum_{j=1}^m \big[g_i\,\Delta f_i + \langle \nabla f_i,\nabla g_i\rangle\big]\,\vert\nabla h_j\vert^2 \Big]^2\\
	&\qquad\qquad \leq \rho_1[f,g]\,\Big[\!\sum_{j,j'=1}^m \langle\nabla h_j,\nabla h_{j'}\rangle^2 - \frac{1}{N'}\Big[\! \sum_{j=1}^m \vert\nabla h_j\vert^2\Big]^2\Big]\quad\meas\text{-a.e.}
	\end{align*}
\end{lemma}

\begin{proof} We define $\mu_2[f,g,h],\mu_3[h] \in \Meas_\fin^\pm(\mms)$ by
	\begin{align*}
		\mu_2[f,g,h] &:= \sum_{i=1}^n\sum_{j=1}^m\big[\langle\nabla f_i,\nabla h_j\rangle\,\langle\nabla g_i,\nabla h_j\rangle + g_i\,\rmH[f_i](h_j,h_j)\big]\,\meas\\
		&\qquad\qquad - \frac{1}{N'}\sum_{i=1}^n\sum_{j=1}^m \big[g_i\,\Delta f_i + \langle \nabla f_i,\nabla g_i\rangle\big]\,\vert\nabla h_j\vert^2\,\meas,\\
		\mu_3[h] &:= \Big[\!\sum_{j,j'=1}^m\langle\nabla h_j,\nabla h_{j'}\rangle^2 - \frac{1}{N'}\Big[\!\sum_{j=1}^m \vert\nabla h_j\vert^2\Big]^2\Big]\,\meas.
	\end{align*}
	Both claims readily follow from \autoref{Le:Measure lemma} as soon as $\lambda^2\,\mu_1[f,g] + 2\lambda\,\mu_2[f,g,h] + \mu_3[h] \geq 0$ for every $\lambda \in\R$, which is what we concentrate on in the sequel.
	
	Let $\lambda\in\R$ and pick $a,b\in\R^n$ as well as $c\in\R^m$. Define the function $\varphi\in\Cont^\infty(\R^{2n+m})$ through
	\begin{align*}
	\varphi(x,y,z) := \sum_{i=1}^n \big[\lambda\,x_i\,y_i + a_i\,x_i - b_i\,y_i\big] + \sum_{j=1}^m \big[(z_j-c_j)^2-c_j^2\big].
	\end{align*}
	For every $i\in\{1,\dots,n\}$ and every $j\in\{1,\dots,m\}$, those first and second partial derivatives of $\varphi$ which,  do not always vanish identically read
	\begin{align*}
	\varphi_i(x,y,z) &= \lambda\,y_i +a_i,\\
	\varphi_{n+i}(x,y,z) &= \lambda\,x_i - b_i,\\
	\varphi_{2n+j}(x,y,z) &= 2(z_j-c_j),\\
	\varphi_{i,n+i}(x,y,z) &= \lambda,\\
	\varphi_{n+i,i}(x,y,z) &=\lambda,\\
	\varphi_{2n+j,2n+j}(x,y,z) &=2.
	\end{align*}
	For convenience, we write 
\begin{align*}
\RMA^{2\kappa}(\lambda,a,b,c) &:= \RMA^{2\kappa}[\varphi\circ q],\\
\rmB(\lambda,a,b,c) &:= \rmB[\varphi\circ q],\\
\rmC(\lambda,a,b,c) &:= \rmC[\varphi\circ q],\\
\rmD(\lambda,a,b,c) &:= \rmD[\varphi\circ q],
\end{align*}	 
where the respective r.h.s.'s are defined as in \autoref{Le:Calculus rules} for $\alpha := 2n+m$ and $q:= (f,g,h)$. Using the same \autoref{Le:Calculus rules}, we compute
	\begin{align*}
	\RMA^{2\kappa}(\lambda,a,b,c) &= \sum_{i,i'=1}^n (\lambda\,\widetilde{g}_i + a_i)\,(\lambda\,\widetilde{g}_{i'} + a_{i'})\,\bdGamma_2^{2\kappa}(f_i,f_{i'}) + \text{other terms},\\
	\rmB(\lambda,a,b,c) &= 4\sum_{i,i'=1}^n (\lambda\,g_i +a_i)\,\lambda\,\rmH[f_i](f_{i'},g_{i'})\\
	&\qquad\qquad +4\sum_{i=1}^n\sum_{j=1}^m (\lambda\,g_i+a_i)\,\rmH[f_i](h_j,h_j) + \text{other terms},\\
	\rmC(\lambda,a,b,c) &= 2\sum_{i,i'=1}^n\lambda^2\,\big[\langle\nabla f_i,\nabla f_{i'}\rangle\,\langle\nabla g_i,\nabla g_{i'}\rangle + \langle\nabla f_i,\nabla g_{i'}\rangle\,\langle \nabla g_i,\nabla f_{i'}\rangle\big]\\
	&\qquad\qquad + 8\sum_{i=1}^n\sum_{j=1}^m \lambda\,\langle\nabla f_i,\nabla h_j\rangle\,\langle \nabla g_i,\nabla h_j\rangle\\
	&\qquad\qquad + 4\sum_{j,j'=1}^m\langle\nabla h_j,\nabla h_{j'}\rangle^2 + \text{other terms},\\
	\rmD(\lambda,a,b,c) &= \sum_{i,i'=1}^n (\lambda\,g_i + a_i)\,(\lambda\,g_{i'}+a_i)\,\Delta f_i\,\Delta f_{i'}\\
	&\qquad\qquad + 4\sum_{i,i'=1}^n\lambda\,(\lambda\,g_i+a_i)\,\Delta f_i\,\langle\nabla f_{i'},\nabla g_{i'}\rangle\\
	&\qquad\qquad + 4\sum_{i,i'=1}^n\lambda^2\,\langle\nabla f_i,\nabla g_i\rangle\,\langle\nabla f_{i'},\nabla g_{i'}\rangle\\
	&\qquad\qquad + 4\sum_{i=1}^n\sum_{j=1}^m (\lambda\,g_i+a_i)\,\Delta f_i\,\vert\nabla h_j\vert^2\\
	&\qquad\qquad + 8\sum_{i=1}^n\sum_{j=1}^m\lambda\,\langle\nabla f_i,\nabla g_i\rangle\,\vert\nabla h_j\vert^2\\
	&\qquad\qquad + 4\,\Big[\sum_{j=1}^m \vert\nabla h_j\vert^2\Big]^2 + \text{other terms}.
	\end{align*}
	Here, every ``other term'' contains at least one factor of the form $\smash{\lambda\,\widetilde{f}_i - b_i}$ or $\smash{\widetilde{h}_j - c_j}$ for some $i\in \{1,\dots,n\}$ and $j\in \{1,\dots,m\}$.
	
	By \autoref{Le:Calculus rules} and  \autoref{Pr:Bakry Emery measures} with the nonnegativity of $\rmD(\lambda,a,b,c)$ as well as the trivial inequality $1/N \geq 1/N'$,
	\begin{align*}
		\RMA^{2\kappa}(\lambda,a,b,c) + \Big[\rmB(\lambda,a,b,c) + \rmC(\lambda,a,b,c) -\frac{1}{N'}\,\rmD(\lambda,a,b,c)\Big]\,\meas \geq 0.
	\end{align*}
	By the arbitrariness of $a,b\in\R^n$ and $c\in\R^m$, for every Borel partition $(E_p)_{p\in\N}$ of $\mms$, every Borel set $F\subset\mms$ and all sequences $(a_k)_{k\in\N}$ and $(b_k)_{k\in\N}$ in $\R^n$ as well as $(c_k)_{k\in\N}$ in $\R^m$,
	\begin{align}\label{Eq:Intermediate}
	\begin{split}
		&\One_{F}\sum_{k\in\N} \One_{E_k}\,\Big[\RMA^{2\kappa}(\lambda,a_k,b_k,c_k) + \Big[\rmB(\lambda,a_k,b_k,c_k) + \rmC(\lambda,a_k,b_k,c_k)\\
		&\qquad \qquad -\frac{1}{N'}\,\rmD(\lambda,a_k,b_k,c_k)\Big]\,\meas\Big]\geq 0.
		\end{split}
	\end{align}
	
We now choose the involved quantities appropriately. Let $(F_k)_{k\in\N}$ be an $\Ch$-nest with the property that the restrictions of $\smash{\widetilde{f}}$, $\smash{\widetilde{g}}$ and $\smash{\widetilde{h}}$ to $F_k$ are continuous for every $k\in\N$, and set $\smash{F := \bigcup_{k\in\N} F_k}$. Since $F^\rmc$ is an $\Ch$-polar set and thus not seen by $\meas$ and $\smash{\bdGamma_2^{2\kappa}(f_i,f_{i'})}$, $i,i'\in\{1,\dots,n\}$, its contribution to the subsequent manipulations is ignored. For $l\in\N$ we now take a Borel partition $\smash{(E_k^l)_{k\in\N}}$ of $\mms$ and sequences $\smash{(a_k^l)_{k\in\N}}$ and $\smash{(b_k^l)_{k\in\N}}$ in $\R^n$ as well as $\smash{(c_k^l)_{k\in\N}}$ in $\R^m$ with
\begin{align*}
\sup_{k,l\in\N} \big[\vert a_k^l\vert + \vert b_k^l\vert +\vert c_k^l\vert\big] < \infty
\end{align*}
in such a way that
\begin{align*}
\lim_{l\to\infty} \sum_{k\in\N}\One_{E_k^l}\,a_k^l &= \lambda\,\widetilde{g},\\
\lim_{l\to\infty} \sum_{k\in\N} \One_{E_k^l}\,b_k^l &= \lambda\,\widetilde{f},\\
\lim_{l\to\infty} \sum_{k\in\N} \One_{E_k^l}\,c_k^l &= \widetilde{h}
\end{align*}
pointwise on $F$. Thus, the l.h.s.~of \eqref{Eq:Intermediate} with $\smash{(E_k)_{k\in\N}}$, $\smash{(a_k)_{k\in\N}}$, $\smash{(b_k)_{k\in\N}}$ and $\smash{(c_k)_{k\in\N}}$ replaced by $\smash{(E_k^l)_{k\in\N}}$, $\smash{(a_k^l)_{k\in\N}}$, $\smash{(b_k^l)_{k\in\N}}$ and $\smash{(c_k^l)_{k\in\N}}$, $l\in\N$, respectively, converges w.r.t.~$\Vert\cdot\Vert_\TV$ as $l\to\infty$. In fact, in the limit as $l\to\infty$ every ``other term'' above becomes zero, and the prefactors $\smash{\lambda\,\widetilde{g}_i+(a_k^l)_i}$ become $2\lambda\,\widetilde{g}_i$, $i\in \{1,\dots,n\}$. We finally obtain
	\begin{align*}
		&4\lambda^2\sum_{i,i'=1}^n \widetilde{g}_i\,\widetilde{g}_{i'}\,\bdGamma_2^{2\kappa}(f_i,f_{i'})\\
		&\qquad\qquad + 8\lambda^2\sum_{i,i'=1}^n g_i\,\rmH[f_i](f_{i'},g_{i'}) \,\meas+ 8\lambda\sum_{i=1}^n\sum_{j=1}^m g_i\,\rmH[f_i](h_j,h_j)\,\meas\\
		&\qquad\qquad + 2\lambda^2\sum_{i,i'=1}^n\big[\langle\nabla f_i,\nabla f_{i'}\rangle\,\langle\nabla g_i,\nabla g_{i'}\rangle + \langle\nabla f_i,\nabla g_{i'}\rangle\,\langle\nabla g_i,\nabla f_{i'}\rangle\big]\,\meas\\
		&\qquad\qquad\qquad\qquad + 8\lambda\sum_{i=1}^n\sum_{j=1}^m\langle\nabla f_i,\nabla h_j\rangle\,\langle\nabla g_i,\nabla h_j\rangle\,\meas\\
		&\qquad \qquad \qquad\qquad + 4\sum_{j,j'=1}^m\langle\nabla h_j,\nabla h_{j'}\rangle^2\,\meas\\
		&\qquad\qquad - \frac{4\lambda^2}{N'}\,\Big[\sum_{i=1}^n \big[g_i\,\Delta f_i + \langle \nabla f_i,\nabla g_i\rangle\big]\Big]^2\,\meas\\
		&\qquad\qquad\qquad\qquad - \frac{8\lambda}{N'}\sum_{i=1}^n\sum_{j=1}^m \big[g_i\,\Delta f_i + \langle\nabla f_i,\nabla g_i\rangle\big]\,\vert \nabla h_j\vert^2\,\meas\\
		&\qquad\qquad\qquad \qquad - \frac{4}{N'}\,\Big[\sum_{j=1}^m\vert\nabla h_j\vert^2\Big]^2 \geq 0.
	\end{align*}
	Dividing by $4$ and sorting terms by the order of $\lambda$ yields the claim.
\end{proof}

We note the following consequence of \autoref{Le:Extremely key lemma} that is used in \autoref{Th:Hess} below as well, but becomes especially important in \autoref{Sec:Structural consequences}.

\begin{remark}\label{Re:Trace} The nonnegativity of $\mu_3[h]$ from \autoref{Le:Extremely key lemma} can be translated into the following trace inequality, compare with \cite[Rem.~2.19]{braun2020} and the proof of \cite[Prop.~3.2]{han2018}. With the pointwise trace defined as in \eqref{Eq:Pointwise trace}, we have
\begin{align*}
\Big\vert\! \sum_{j=1}^m \nabla h_j\otimes\nabla h_j\Big\vert_\HS^2 \geq \frac{1}{N}\tr\!\Big[\sum_{j=1}^m \nabla h_j\otimes\nabla h_j\Big]^2\quad\meas\text{-a.e.}
\end{align*}
for every $m\in\N$ and every $h\in\Test(\mms)^m$.
\end{remark}

\begin{theorem}\label{Th:Hess} Every $f\in\Test(\mms)$ belongs to $\Dom(\Hess)$ and satisfies
\begin{align}\label{Eq:H = H}
\Hess f(\nabla g_1,\nabla g_2) = \rmH[f](g_1,g_2)\quad\meas\text{-a.e.}
\end{align}
for every $g_1,g_2\in\Test(\mms)$. Moreover, denoting by $\smash{\gamma_2^{2\kappa}(f)\in \Ell^1(\mms)}$ the density of the $\meas$-absolutely continuous part of $\bdGamma_2^{2\kappa}(f)$, we have
\begin{align}\label{Eq:Pointwise Bochner}
\big\vert\!\Hess f\big\vert_\HS^2 \leq \gamma_2^{2\kappa}(f)\quad\meas\text{-a.e.}
\end{align}
\end{theorem}

\begin{proof} Recall that indeed $\smash{\gamma_2^{2\kappa}(f)\in\Ell^1(\mms)}$ by \autoref{Pr:Fin tot var}. Let $g,h_1,\dots,h_m\in\Test(\mms)$, $m\in\N$. Applying \autoref{Le:Extremely key lemma} for $N' := \infty$ and $n := 1$ then entails
\begin{align*}
&\Big[\!\sum_{j=1}^m \big[\langle \nabla f,\nabla h_j\rangle\,\langle\nabla g,\nabla h_j\rangle + g\,\rmH[f](h_j,h_j)\big]\Big]^2\\
&\qquad\qquad \leq  \Big[g^2\,\gamma_2^{2\kappa}(f) + 2g\,\rmH[f](f,g) + \frac{1}{2}\,\vert\nabla f\vert^2\,\vert\nabla g\vert^2 + \frac{1}{2}\,\langle\nabla f,\nabla g\rangle^2\Big]\\
&\qquad\qquad\qquad\qquad\times \sum_{j,j'=1}^m \langle \nabla h_j,\nabla h_{j'}\rangle^2\\
&\qquad\qquad = \Big[g^2\,\gamma_2^{2\kappa}(f) + g\,\big\langle \nabla\vert\nabla f\vert^2,\nabla g \big\rangle + \frac{1}{2}\,\vert\nabla f\vert^2\,\vert\nabla g\vert^2 + \frac{1}{2}\,\langle\nabla f,\nabla g\rangle^2\Big]\\
&\qquad\qquad\qquad\qquad\times \sum_{j,j'=1}^m \langle \nabla h_j,\nabla h_{j'}\rangle^2\quad\meas\text{-a.e.}
\end{align*}
In the last identity, we used the definition \eqref{Eq:Pre-Hessian} of $\rmH[f](f,g)$. Using the first part of  \autoref{Le:Mollified heat flow} and possibly passing to subsequences, this $\meas$-a.e.~inequality extends to all $g\in \F\cap\Ell^\infty(\mms)$. Thus, successively setting $g := g_n$, $n\in\N$, where $(g_n)_{n\in\N}$ is the sequence provided by \autoref{Le:Approx to id}, together with the locality of $\nabla$ from \autoref{Cor:Calculus rules d}, and by the definition \eqref{Eq:Pointwise norm HS} of the pointwise Hilbert--Schmidt norm of $\Ell^2(T^{\otimes 2}\mms)$, we obtain
\begin{align}\label{Eq:Tttuuuuuuus}
\Big\vert\!\sum_{j=1}^m \rmH[f](h_j,h_j)\Big\vert &\leq \gamma_2^{2\kappa}(f)^{1/2}\, \Big\vert\!\sum_{j=1}^m \nabla h_j\otimes\nabla h_j\Big\vert_\HS\quad\meas\text{-a.e.}
\end{align}
This implies pointwise $\meas$-a.e.~off-diagonal estimates as follows. Given any $m'\in \N$ and $\smash{h_j, h_j'\in\Test(\mms)}$, $j\in \{1,\dots,m'\}$, since
\begin{align*}
\sum_{j=1}^{m'} \rmH[f](h_j,h_j') = \frac{1}{2}\sum_{j=1}^{m'} \big[\rmH[f](h_j+h_j', h_j+h_j') - \rmH[f](h_j,h_j) - \rmH[f](h_j',h_j')\big]
\end{align*}
holds $\meas$-a.e., applying \eqref{Eq:Tttuuuuuuus}, using that
\begin{align*}
&\frac{1}{2}\sum_{j=1}^{m'} \big[\nabla(h_j+h_j')\otimes \nabla (h_j+h_j') - \nabla h_j\otimes \nabla h_j - \nabla h_j' \otimes \nabla h_j'\big]\\
&\qquad\qquad = \frac{1}{2}\sum_{j=1}^{m'} \big[\nabla h_j\otimes \nabla h_j' + \nabla h_j'\otimes\nabla h_j\big]\\
&\qquad\qquad = \Big[\!\sum_{j=1}^{m'} \nabla h_j\otimes \nabla h_j'\Big]_\sym
\end{align*}
and finally employing that $\vert T_\sym\vert_\HS \leq \vert T\vert_\HS$ for every $T\in\Ell^2(T^{\otimes 2}\mms)$, we get
\begin{align*}
\Big\vert\!\sum_{j=1}^{m'} \rmH[f](h_j,h_j') \Big\vert
&\leq \gamma_2^{2\kappa}(f)^{1/2}\,\Big\vert \frac{1}{2}\sum_{j=1}^{m'} \big[\nabla h_j\otimes \nabla h_j' + \nabla h_j'\otimes \nabla h_j\big]\Big\vert_\HS\\
&\leq \gamma_2^{2\kappa}(f)^{1/2}\,\Big\vert\!\sum_{j=1}^{m'} \nabla h_j\otimes \nabla h_j'\Big\vert_\HS\quad\meas\text{-a.e.}
\end{align*}
We replace $h_j$ by $a_j\,h_j$, $j\in \{1,\dots,m'\}$, for arbitrary $a_1,\dots,a_{m'}\in \Q$. This gives
\begin{align}\label{Eq:abc}
\Big\vert\!\sum_{j=1}^{m'} a_j\,\rmH[f](h_j,h_j') \Big\vert
\leq \gamma_2^{2\kappa}(f)^{1/2}\,\Big\vert\!\sum_{j=1}^{m'}  a_j\,\nabla h_j\otimes \nabla h_j'\Big\vert_\HS\quad\meas\text{-a.e.}
\end{align}
In fact, since $\Q$ is countable, we find an $\meas$-negligible Borel set $B\subset \mms$ on whose complement \eqref{Eq:abc} holds pointwise for every $a_1,\dots,a_{m'}\in \Q$. Since both sides of \eqref{Eq:abc} are continuous in $a_1,\dots,a_{m'}$, by density of $\Q$ in $\R$ we deduce that \eqref{Eq:abc} holds pointwise on $\smash{B^\rmc}$ for every $a_1,\dots,a_{m'}\in \R$. Therefore, given any $g_1,\dots,g_{m'}\in\Test(\mms)$, up to possibly removing a further $\meas$-negligible Borel set $C\subset\mms$, for every $x\in (B\cup C)^\rmc$ we may replace $a_j$ by $g_j(x)$, $j\in\{1,\dots,m'\}$, in \eqref{Eq:abc}. This leads to
\begin{align}\label{Eq:Zu}
\Big\vert\!\sum_{j=1}^{m'} g_j\,\rmH[f](h_j,h_j')\Big\vert^2 \leq \gamma_2^{2\kappa}(f)^{1/2}\,\Big\vert\!\sum_{j=1}^{m'} g_j\,\nabla h_j\otimes\nabla h_j'\Big\vert_\HS\quad\meas\text{-a.e.}
\end{align}

We now define the operator $\Phi\colon \Test(T^{\otimes 2}\mms) \to \Ell^0(\mms)$ by
\begin{align}\label{Eq:The Phi def}
\Phi\,\sum_{j=1}^{m'} g_j\,g_j'\,\nabla h_j\otimes \nabla h_j' := \sum_{j=1}^{m'} g_j\,g_j'\,\rmH[f](h_j,h_j').
\end{align}
From \eqref{Eq:Zu} and the algebra property of $\Test(\mms)$, it follows that $\Phi$ is well-defined, i.e.~the value of $\Phi(T)$ does not depend on the specific way of representing a given element $T\in \Test(T^{\otimes 2}\mms)$. Moreover, the map $\Phi$ is clearly linear, and for every $g\in\Test(\mms)$ and every $T\in\Test(T^{\otimes 2}\mms)$, 
\begin{align}\label{Eq:tes linear}
\Phi(g\,T) = g\,\Phi(T).
\end{align}
Since the $\meas$-singular part $\bdGamma_2^{2\kappa}(f)_\perp$ of $\bdGamma_2^{2\kappa}(f)$ is nonnegative, by \eqref{Eq:Identities Gamma2 DELTA} we get
\begin{align}\label{Eq:Kappa comp}
\int_\mms \gamma_2^{2\kappa}(f)\d\meas \leq \bdGamma_2^{2\kappa}(f)[\mms] = \int_\mms (\Delta f)^2\d\meas - \big\langle \kappa \,\big\vert\,\vert\nabla f\vert^2\big\rangle.
\end{align}
After integrating \eqref{Eq:Zu} and employing Cauchy--Schwarz's inequality, 
\begin{align*}
\Vert \Phi(T)\Vert_{\Ell^1(\mms)} \leq \Big[\!\int_\mms (\Delta f)^2\d\meas - \big\langle\kappa\,\big\vert\,\vert\nabla f\vert^2\big\rangle\Big]^{1/2}\,\Vert T\Vert_{\Ell^2(T^{\otimes 2}\mms)}
\end{align*}
holds for every $T\in \Test(T^{\otimes 2}\mms)$. Thus, by density of $\Test(T^{\otimes 2}\mms)$ in $\Ell^2(T^{\otimes 2}\mms)$  and \eqref{Eq:tes linear}, $\Phi$ uniquely extends to a (non-relabeled) continuous, $\Ell^\infty$-linear map from $\Ell^2(T^{\otimes 2}\mms)$ into $\Ell^1(\mms)$, whence $\Phi\in \Ell^2((T^*)^{\otimes 2}\mms)$ by definition of the latter space.

To check that $f\in \Dom(\Hess)$ and $\Phi = \Hess f$, first note that by the continuity of $\Phi$ and \autoref{Le:Product convergence}, we can replace $\smash{g_j\,g_j'}$ by arbitrary $k_j\in\Test(\mms)$, $j\in \{1,\dots,m'\}$, still retaining the identity \eqref{Eq:The Phi def}. Therefore, slightly changing the notation in \eqref{Eq:The Phi def}, let $g_1,g_2,h \in \Test(\mms)$ and use \eqref{Eq:tes linear}, the definition \eqref{Eq:Pre-Hessian} of $\rmH[f]$ and  \autoref{Le:Div g nabla f} to derive that
\begin{align*}
&2\int_\mms h\,\Phi(\nabla g_1, \nabla g_2)\d\meas\\
&\qquad\qquad = \int_\mms h\,\big\langle\nabla g_1,\nabla \langle\nabla f,\nabla g_2\rangle\big\rangle\d\meas + \int_\mms h\,\big\langle\nabla g_2,\nabla \langle\nabla f,\nabla g_1\rangle\big\rangle\d\meas\\
&\qquad\qquad\qquad\qquad -\int_\mms h\,\big\langle\nabla f,\nabla \langle\nabla g_1,\nabla g_2\rangle\big\rangle\d\meas\\
&\qquad\qquad = -\int_\mms \langle\nabla f,\nabla g_2\rangle\div(h\,\nabla g_1)\d\meas - \int_\mms \langle\nabla f,\nabla g_1\rangle\div(h\,\nabla g_2)\d\meas\\
&\qquad\qquad\qquad\qquad -\int_\mms h\,\big\langle\nabla f,\nabla \langle\nabla g_1,\nabla g_2\rangle\big\rangle\d\meas,
\end{align*}
which is the desired assertion $f\in \Dom(\Hess)$ and $\Phi=\Hess f$.

The same argument gives \eqref{Eq:H = H}, while the inequality  \eqref{Eq:Pointwise Bochner} is due to \eqref{Eq:Zu}, the density of $\Test(T^{\otimes 2}\mms)$ in $\Ell^2(T^{\otimes 2}\mms)$ as well as the definition \eqref{Eq:Pointwise norm HS} of the pointwise Hilbert--Schmidt norm.
\end{proof}

\autoref{Th:Hess} implies the following qualitative result. A quantitative version of it, as directly deduced in \cite[Cor.~3.3.9]{gigli2018} from \cite[Thm.~3.3.8]{gigli2018}, is however not yet available only with the information collected so far. See \autoref{Re:rererer} below.

\begin{corollary}\label{Cor:Dom(Delta) subset W22} Every $f\in \Dom(\Delta)$ belongs to the closure of $\Test(\mms)$ in $\Dom(\Hess)$, and in particular to $\Dom(\Hess)$. More precisely, let $\rho'\in (0,1)$ and $\alpha'\in\R$ be as in \autoref{Le:Form boundedness} for $\mu := \kappa^-$. Then for every $f\in \Dom(\Delta)$, we have $f\in \Dom(\Hess)$ with
\begin{align*}
\int_\mms \big\vert \!\Hess f\big\vert_\HS^2\d\meas \leq \frac{1}{1-\rho'}\int_\mms (\Delta f)^2\d\meas + \frac{\alpha'}{1-\rho'}\int_\mms\vert\nabla f\vert^2\d\meas.
\end{align*}
\end{corollary}

\begin{proof} Since $(\mms,\Ch,\meas)$ satisfies $\BE_1(\kappa,\infty)$ by \cite[Thm.~6.9]{erbar2020}, it also trivially obeys $\BE_1(-\kappa^-,\infty)$, see also \cite[Prop.~6.7]{erbar2020}. As in the proof of \autoref{Pr:Fin tot var}, 
\begin{align*}
\Ch^{-\kappa^-}\!\big(\vert\nabla f\vert\big) \leq \int_\mms (\Delta f)^2\d\meas
\end{align*}
holds for every $f\in\Test(\mms)$. Hence, using \eqref{Eq::E^k identity} we estimate
\begin{align*}
\big\langle\kappa^-\,\big\vert\,\vert\nabla f\vert^2\big\rangle  &\leq \rho'\,\Ch\big(\vert\nabla f\vert\big) +\alpha'\int_\mms \vert\nabla f\vert^2\d\meas\\
&= \rho'\,\Ch^{-\kappa^-}\!\big(\vert\nabla f\vert\big) + \rho'\,\big\langle\kappa^-\,\big\vert\,\vert\nabla f\vert^2\big\rangle + \alpha'\int_\mms \vert\nabla f\vert^2\d\meas\\
&\leq \rho'\int_\mms (\Delta f)^2\d\meas + \rho'\,\big\langle\kappa^-\,\big\vert\,\vert\nabla f\vert^2\big\rangle + \alpha'\int_\mms \vert\nabla f\vert^2\d\meas.
\end{align*}
The claim for $f\in\Test(\mms)$ now follows easily. We already know from \autoref{Th:Hess} that $f\in \Dom(\Hess)$. Integrating \eqref{Eq:Pointwise Bochner} and using \eqref{Eq:Kappa comp} thus yields
\begin{align}\label{Eq:Write down}
\int_\mms \big\vert\!\Hess f\big\vert_\HS^2\d\meas &\leq \int_\mms (\Delta f)^2\d\meas - \big\langle \kappa\,\big\vert\,\vert\nabla f\vert^2\big\rangle\\
&\leq \int_\mms (\Delta f)^2\d\meas + \big\langle\kappa^-\,\big\vert\,\vert\nabla f\vert^2\big\rangle\nonumber\\
&\leq \frac{1}{1-\rho'}\int_\mms (\Delta f)^2\d\meas + \frac{\alpha'}{1-\rho'}\int_\mms \vert\nabla f\vert^2\d\meas.\nonumber
\end{align}

Finally, given $f\in\Dom(\Delta)$, let $f_n := \max\{\min\{f,n\},-n\} \in\Ell^2(\mms)\cap\Ell^\infty(\mms)$, $n\in\N$. Note that $\ChHeat_t f_n \in \Test(\mms)$ for every $t>0$ and every  $n\in\N$, and that $\ChHeat_t f_n \to \ChHeat_t f$ in $\F$ as well as, thanks to \autoref{Th:Heat flow properties}, $\Delta\ChHeat_t f_n \to \Delta\ChHeat_t f$ in $\Ell^2(\mms)$ as $n\to\infty$. Moreover $\ChHeat_t f\to f$ in $\F$ as well as $\Delta\ChHeat_t f = \ChHeat_t \Delta f \to \Delta f$ in $\Ell^2(\mms)$ as $t\to 0$. These observations imply that $f$ belongs to the closure of $\Test(\mms)$ in $\Dom(\Hess)$, whence $f\in \Dom(\Hess)$ by \autoref{Th:Hess properties}, and the claimed inequality, with unchanged constants, is clearly stable under this approximation procedure.
\end{proof}

\begin{remark}\label{Re:rererer} The subtle reason why we still cannot deduce \eqref{Eq:Write down} for general $f\in\Dom(\Delta)$ is that we neither know whether the r.h.s.~of \eqref{Eq:Write down} makes sense --- which essentially requires $\vert\nabla f\vert\in \F$ --- nor, in the notation of the previous proof, whether $\smash{\big\langle\kappa\,\big\vert\,\vert\nabla\ChHeat_tf_n\vert^2\big\rangle \to \big\langle\kappa\,\big\vert\,\vert\nabla f\vert^2\big\rangle}$ as $n\to\infty$ and $t\to 0$. (Neither we know if $\smash{\Ch\big(\vert\nabla \ChHeat_tf_n\vert\big) \to \Ch\big(\vert\nabla f\vert\big)}$ as $n\to\infty$ and $t\to 0$.) Both points are trivial in the more restrictive $\RCD(K,\infty)$ case from \cite[Cor.~3.3.9]{gigli2018}, $K\in\R$. In our setting, solely \autoref{Le:Form boundedness} does not seem sufficient to argue similarly. Instead, both points will follow from \autoref{Le:Kato inequality} and \autoref{Le:Pre.Bochner}, see \autoref{Cor:Kappa cont} and \autoref{Cor:Integr Bochner II}. 
\end{remark}

\subsection{Structural consequences of \autoref{Le:Extremely key lemma}}\label{Sec:Structural consequences} Solely in this section, we assume that $\BE_2(\kappa,N)$ holds for $N< \infty$.  In this case, we derive a nontrivial upper bound on the local dimension of $\Ell^2(T\mms)$ --- and hence of $\Ell^2(T^*\mms)$ by \autoref{Th:Riesz theorem modules} --- in \autoref{Pr:Upper bound local dimension}. Our proof follows \cite[Prop.~3.2]{han2018}.  Reminiscent of \autoref{Re:Hino} and \autoref{Cor:Hino}, as a byproduct we obtain an upper bound on the Hino index of $\mms$. The key point is the trace inequality derived in \autoref{Re:Trace}.

Let $\smash{(E_n)_{n\in\N\cup\{\infty\}}}$ be the dimensional decomposition of $\Ell^2(T\mms)$ as provided by \autoref{Th:Dimensional decomposition}. The \emph{maximal essential local dimension} of $\Ell^2(T\mms)$ is
\begin{align*}
\dim_{\Ell^\infty,\max}\Ell^2(T\mms) := \sup\!\big\lbrace n \in  \N\cup\{\infty\} : \meas[E_n]>0\big\rbrace.
\end{align*}

\begin{proposition}\label{Pr:Upper bound local dimension} We have
\begin{align*}
\dim_{\Ell^\infty,\max} \Ell^2(T\mms) \leq \lfloor N\rfloor.
\end{align*}
Moreover, if $N$ is an integer, then for every $f\in\Test(\mms)$,
\begin{align*}
\tr\Hess f = \Delta f\quad\meas\text{-a.e.}\quad\text{on }E_N.
\end{align*}
\end{proposition}

\begin{proof} Suppose to the contrapositive that $\meas[E_n] > 0$ for some $n\in\N\cup\{\infty\}$ with $n > N$. Let $m\in \N$ be a finite number satisfying $N < m \leq n$. Let $B\subset E_m$ be a given Borel set of finite, but positive $\meas$-measure. By \autoref{Th:Module structure} for $\mu := \meas$ and \cite[Thm.~1.4.11]{gigli2018}, there exist vectors $V_1,\dots,V_m\in \Ell^2(T\mms)$ such that $\One_{B^\rmc}\,V_i = 0$ and $\langle V_i,V_j\rangle = \delta_{ij}$ $\meas$-a.e.~in $B$ for every $i,j\in\{1,\dots,m\}$ which generate $\smash{\Ell^2(T\mms)}$ on $B$. Recall that by \autoref{Re:Trace}, for every $h\in\Test(\mms)^m$,
\begin{align}\label{Eq:xcv}
\Big\vert\!\sum_{j=1}^m \nabla h_j\otimes\nabla h_j\Big\vert_\HS^2 \geq \frac{1}{N}\tr\!\Big[\!\sum_{j=1}^m\nabla h_j\otimes\nabla h_j\Big]^2\quad\meas\text{-a.e.}
\end{align}
By a similar argument as for \autoref{Th:Hess}, we can replace $\nabla h_j$ by $f_j\,\nabla h_j$, $j\in \{1,\dots,m\}$, for arbitrary $f_1,\dots,f_m\in\Ell^\infty(\mms)$,  still retaining \eqref{Eq:xcv}. By \autoref{Pr:Generators cotangent module} and \autoref{Sub:Test objects}, the linear span of such vector fields generates $\Ell^2(T\mms)$ on $B$. We can thus further replace $f_j\,\nabla h_j$ by $1_B\,V_j$, $j\in \{1,\dots,m\}$, and \eqref{Eq:xcv} translates into
\begin{align}\label{Eq:Equality occurs}
m = \Big\vert\!\sum_{j=1}^m V_j\otimes V_j\Big\vert_\HS^2 \geq \frac{1}{N}\tr\!\Big[\!\sum_{j=1}^m V_j\otimes V_j\Big]^2 = \frac{m^2}{N}\quad\meas\text{-a.e.}\quad\text{on }B.
\end{align}
This is in contradiction with the assumption $N < m$.

The second claim is only nontrivial if $\meas[E_N]> 0$. In this case, retain the notation of the previous part and observe that under our given assumptions, equality occurs in \eqref{Eq:Equality occurs} for $m$ replaced by $N$. Using \autoref{Le:Extremely key lemma}, \eqref{Eq:H = H} and similar arguments as for \autoref{Th:Hess} to get rid of the term containing $g\in\Test(\mms)$ and from above to pass from $\nabla h_j$ to $V_j$, $j\in \{1,\dots,N\}$, we get
\begin{align*}
&\Big\vert\!\sum_{j=1}^N \Hess f(V_j,V_j) - (\Delta f)^2\Big\vert^2\\
&\qquad\qquad = \Big\vert\!\sum_{j=1}^N \Hess f(V_j,V_j) - \frac{1}{N}\,(\Delta f)^2\sum_{j=1}^N \vert V_j\vert^2\Big\vert^2  =0\quad\meas\text{-a.e.}\quad\text{on }B.
\end{align*}
This provides the assertion by the arbitrariness of $B$.
\end{proof}

\begin{corollary} For every $k\in \N$ with $k>\lfloor N\rfloor$,
\begin{align*}
\Ell^2(\Lambda^kT^*\mms) = \{0\}.
\end{align*}
\end{corollary}

\autoref{Pr:Upper bound local dimension} opens the door for considering an $N$-Ricci tensor on $\BE_2(\kappa,N)$ spaces with $N< \infty$, see \autoref{Sub:Dimensional Ricci tensor} below.

\begin{remark}\label{Re:Dim RCD} It is an interesting task to carry out a detailed study of  sufficient and necessary conditions for the constancy of the local dimension of $\Ell^2(T\mms)$ as well as its maximality. Natural questions in this respect are the following.
\begin{enumerate}[label=\alph*.]
\item\label{La:Ukyu} Under which hypotheses does there exist $d\in \{1,\dots,\lfloor N\rfloor\}$ such that
\begin{align}\label{Eq:Ed full measure}
\meas\big[E_d^\rmc\big] = 0?
\end{align}
\item\label{La:Bukyu} If \eqref{Eq:Ed full measure} holds for some $d\in \{1,\dots,\lfloor N\rfloor\}$, which conclusions can be drawn for the space $(\mms,\Ch,\meas)$? Does it satisfy $\BE_2(\kappa,d)$?
\item\label{La:Dukyu} What happens if $N$ is an integer and $d=N$ in \eqref{Eq:Ed full measure}?
\end{enumerate}

In \cite{honda2018}, a general class of examples which obey $\BE_2(K,N)$, $K\in\R$ and $N\in [1,\infty)$, but do not have constant local dimension has been pointed out. The latter already happens, for instance, for metric measure spaces obtained by gluing together two compact pointed Riemannian manifolds at their base points. The point is that such a space does not satisfy the Sobolev-to-Lipschitz property, hence cannot be $\RCD(K',\infty)$ for any $K'\in \R$ \cite[Thm.~6.2]{ambrosio2014a}.

On the other hand, some existing results in the framework of  $\RCD(K,N)$ spaces $(\mms,\met,\meas)$, $K\in\R$ and $N\in [1,\infty)$, are worth mentioning.

Originating in \cite{mondino2019},  \ref{La:Ukyu} has completely been solved in \cite[Thm.~0.1]{bruesemola2020}. (This result from \cite{bruesemola2020} uses optimal transport tools. See \cite{han2018} for the connections of the latter to \cite{gigli2018}.) However, it is still unknown what the corresponding value of $d$ really is, e.g.~whether it generally coincides with the Hausdorff dimension of $(\mms,\met)$. 

Questions \ref{La:Bukyu} and \ref{La:Dukyu} are still subject to high research interest. Results in this direction have been initiated in \cite{dephilippis2018}. Indeed, every $\RCD(K,N)$ space with integer $N$ and $\smash{\meas[E_N^\rmc]=0}$ is weakly noncollapsed \cite[Def.~1.10]{dephilippis2018} by \cite[Thm.~1.12, Rem.~1.13]{dephilippis2018}. In fact, it is conjectured in \cite[Rem.~1.11]{dephilippis2018} that every such weakly noncollapsed $\RCD(K,N)$ space is \emph{noncollapsed} \cite[Def.~1.1]{dephilippis2018}, i.e.~$\meas$ is a constant multiple of $\Haus^N$. This conjecture has first been solved if $\mms$ is compact \cite[Cor.~1.3]{honda2020}, and recently been shown in full generality in  \cite[Thm.~1.3]{brena}. 
\end{remark}

Since to our knowledge, the result from \cite{bruesemola2020} in \autoref{Re:Dim RCD}  has not yet been considered in the context of Hino indices, let us phrase it separately according to our compatibility result in \autoref{Cor:Hino} to bring it to a broader audience.

\begin{corollary} Let $(\mms,\Ch,\meas)$ be the Dirichlet space induced by an $\RCD(K,N)$  space $(\mms,\met,\meas)$, $K\in\R$ and $N\in [1,\infty)$. Then there exists $d\in \{1,\dots,\lfloor N\rfloor\}$ such that the pointwise Hino index of $\Ch$ is $\meas$-a.e.~constantly equal to $d$.
\end{corollary}

\subsection{Calculus rules}\label{Sub:Calc rules hess} This section contains calculus rules for the Hessian as well as preparatory material, such as different function spaces, required to develop them. We shortly comment on the two major challenges in establishing these.

First, the calculus rules in \autoref{Subsub:Hess calc rules} below  --- which easily hold for test functions --- do not transfer to arbitrary elements in $\Dom(\Hess)$ in general, at least not by approximation. In fact, in general $\Test(\mms)$ is even \emph{not} dense in $\Dom(\Hess)$: on a compact smooth Riemannian manifold $\mms$ with boundary, any nonconstant, affine $f\colon \mms \to \R$ belongs to $\Dom(\Hess)$, but as a possible limit of elements of $\Test(\mms)$ in $\Dom(\Hess)$  it would necessarily have $\surf$-a.e.~vanishing normal derivative at $\partial\mms$ by \autoref{Le:Div g nabla f},  \autoref{Ex:Mflds with boundary} and the trace theorem from  \autoref{Pr:Trace thm} for $k:= p:= 2$ and $E:=\mms\times\R$, see also \cite[p.~32]{schwarz1995}.  Hence, in \autoref{Def:H22} we consider the closure $\Dom_\reg(\Hess)$ of $\Test(\mms)$ in $\Dom(\Hess)$. Many calculus rules will ``only'' hold for this class.

Second, to prove a product rule for the Hessian, \autoref{Pr:Product rule fcts}, similarly to \autoref{Def:Hess} above one would try to define the space $\smash{W^{2,1}(\mms)}$ of all $f\in\Ell^1(\mms)$ having a gradient $\nabla f\in \Ell^1(T\mms)$ and a Hessian $\Hess f \in \Ell^1((T^*)^{\otimes 2}\mms)$.  However, we refrain from doing so, since in this case the term $\smash{\big\langle \nabla f, \nabla\langle\nabla g_1,\nabla g_2\rangle\big\rangle}$, $g_1,g_2\in\Test(\mms)$, appearing in the defining property of $\Hess f$ in \autoref{Def:Hess} cannot be guaranteed to have any integrability by the lack of $\Ell^\infty$-bounds on $\nabla\langle\nabla g_1,\nabla g_2\rangle$. (In view of \autoref{Pr:Compatibility} and \autoref{Le:Kato inequality}, this would amount to the strong requirement of bounded Hessians of $g_1$ and $g_2$. Compare with \cite[Rem.~4.10, Rem.~4.11, Rem.~4.13]{giglipasqu2020}.) A similar issue arises for the Hessian chain rule in \autoref{Pr:Hess chain rule}. Since the requirement $\nabla f\in\Ell^2(\mms)$ thus cannot be dropped, we are forced to work with the space $\smash{\Dom_{2,2,1}(\Hess)}$ introduced in \autoref{Def:W221}.

\subsubsection{First order spaces} 

\begin{definition}\label{Def:W11} We define the space $\calG$ to consist of all $f\in \Ell^1(\mms)$ for which there exists $\eta\in\Ell^1(T^*\mms)$ such that
\begin{align*}
\int_\mms \eta(\nabla g)\,h\d\meas = -\int_\mms f\div(h\,\nabla g)\d\meas
\end{align*}
for every $\smash{g\in\Test_{\Ell^\infty}(\mms)}$ and every $h\in\Test(\mms)$. If such an $\eta$ exists, it is unique, denoted by $\rmd_1 f$ and termed the \emph{differential} of $f$.
\end{definition}

The defining equality makes sense by \autoref{Le:Div g nabla f}. The uniqueness statement follows from weak$^*$ density of $\Test(T\mms)$ in $\Ell^\infty(T\mms)$ as discussed in \autoref{Sub:Test objects}. Therefore, $\rmd_1$ becomes a linear operator on $\calG$, which turns the latter into a real vector space. Since both sides of the defining property for $\rmd_1$ in \autoref{Def:W11} are strongly $\Ell^1$-continuous in $f$ and $\eta$, respectively, for fixed $\smash{g\in\Test_{\Ell^\infty}(\mms)}$ and $h\in\Test(\mms)$, $\calG$ is a Banach space w.r.t.~the norm $\Vert\cdot\Vert_{\calG}$ given by
\begin{align*}
\Vert f\Vert_{\calG} := \Vert f\Vert_{\Ell^1(\mms)} + \Vert\rmd_1f\Vert_{\Ell^1(T^*\mms)}.
\end{align*}

Products of test functions belong to $\calG$ by \autoref{Le:Div g nabla f}. In fact, by \autoref{Le:Product convergence}, $\calG\cap \F$ and $\calG \cap \Test(\mms)$ are dense in $\F$. The subsequent definition --- that we introduce since we do not know if $\calG\cap \Test(\mms)$ is dense in $\calG$ --- is thus non-void.

\begin{definition}\label{Def:H11 functions} We define the space $\calG_\reg\subset \calG$ by
\begin{align*}
\calG_\reg := \cl_{\Vert \cdot\Vert_{\calG}}\big[\calG\cap\Test(\mms)\big].
\end{align*}
\end{definition}

One  checks through \autoref{Le:Product convergence} and \autoref{Le:Mollified heat flow}  that $\calG_\reg$ is dense in $\Ell^1(\mms)$.

A priori, the differential from \autoref{Def:W11} could differ from the differential $\rmd$ on $\smash{\F_\rme}$ from \autoref{Def:Differential}.  However, these objects agree on the intersection $\smash{\calG\cap \F_\rme}$. In particular, in what follows we simply write $\rmd f$ in place of $\rmd_1 f$ for $f\in \calG$ although the axiomatic difference should always be kept in mind. 

\begin{lemma}\label{Le:d = d1} If $f\in \calG\cap \F_\rme$, then
\begin{align*}
\rmd_1 f = \rmd f.
\end{align*}
\end{lemma}

\begin{proof} Given any $n\in\N$, the function $f_n := \max\{\min\{f,n\},-n\}\in \F_\rme$ belongs to $\Ell^1(\mms)\cap\Ell^\infty(\mms)$ and thus to $\F$ by \autoref{Pr:Extended domain props}. Let $\smash{g\in\Test_{\Ell^\infty}(\mms)}$ and $h\in\Test(\mms)$ be arbitrary. Observe that $f_n\to f$ in $\Ell^1(\mms)$ and $\rmd f_n = \One_{\{\vert f\vert < n\}}\d f \to \rmd f$ in $\Ell^2(T^*\mms)$ as $n\to\infty$. Since $h\,\nabla g\in \Dom_\TV(\DIV)\cap \Dom(\div)$ with $\norm(h\,\nabla g) = 0$ and  $\div(h\,\nabla g) = \langle \nabla h,\nabla g\rangle + h\,\Delta g\in\Ell^\infty(\mms)$ by \autoref{Le:Div g nabla f},
\begin{align*}
\int_\mms \rmd_1f(\nabla g)\,h\d\meas &= -\int_\mms f\div(h\,\nabla g)\d\meas\\
&= -\lim_{n\to\infty} \int_\mms f_n\div(h\,\nabla g)\d\meas\\
&= \lim_{n\to\infty} \int_{\{\vert f\vert < n\}} \rmd f(\nabla g)\,h\d\meas\\
&= \int_\mms \rmd f(\nabla g)\,h\d\meas. 
\end{align*}
The statement follows from the weak$^*$ density of $\Test_{\Ell^\infty}(T\mms)$ in $\Ell^\infty(T\mms)$.
\end{proof}

Occasionally we adopt the dual perspective of $\rmd f$ for a given $f\in \calG$ similarly to \autoref{Def:Gradient} --- more precisely, under the compatibility granted by \autoref{Le:d = d1} and keeping in mind \autoref{Sub:Lebesgue sp},  define $\nabla f\in\Ell^1(T\mms)$ by
\begin{align*}
\nabla f := (\rmd f)^\sharp.
\end{align*}

\begin{lemma}\label{Le:IBP for W11} For every $f\in \calG_\reg$ and every $g\in\F$ such that $\rmd g\in\Ell^\infty(T^*\mms)$ and $\Delta g\in\Ell^\infty(\mms)$,  we have
\begin{align*}
\int_\mms f\,\Delta g \d\meas = -\int_\mms \rmd f(\nabla g)\d\meas.
\end{align*}
\end{lemma}

\begin{proof} Given a sequence $(f_n)_{n\in\N}$ in $\Test(\mms)\cap \calG$ converging to $f$ in $\calG$, since $f_n\in \F$ for every $n\in\N$, by \autoref{Le:d = d1} we have
\begin{align*}
\int_\mms f\,\Delta g \d\meas &= \lim_{n\to\infty}\int_\mms f_n\,\Delta g\d\meas\\
&= -\lim_{n\to\infty}\int_\mms \rmd f_n(\nabla g)\d\meas\\
&= -\int_\mms \rmd f(\nabla g)\d\meas.\qedhere
\end{align*} 
\end{proof}

\begin{remark}\label{Re:No W11 in general} It is unclear to us  whether the integration by parts formula from \autoref{Le:IBP for W11} holds if merely $f\in \calG$. The subtle point is that in general, we are not able to get rid of the zeroth order term $h$ in \autoref{Def:W11} --- \autoref{Le:Approx to id} does not give global first order controls. Of course, what still rescues this  important identity (see e.g.~\autoref{Pr:BE vector fields}  below) for $\calG_\reg$-functions is its validity for the approximating test functions, and \autoref{Le:d = d1}. 

By the same reason, in our setting we lack an analogue of \cite[Prop.~3.3.18]{gigli2018}, see also \cite[Lem.~6.2.26]{gigli2020}, which grants $\F$-regularity of $f\in \calG\cap \Ell^2(\mms)$ as soon as $\rmd f\in \Ell^2(T^*\mms)$.  Compare with \autoref{Re:Wd120 in W12} and also \autoref{Re:HSU forms?} below.

Both statements are clearly true if $\mms$ is intrinsically complete.
\end{remark}

In line with \autoref{Re:No W11 in general}, the following proposition readily follows from corresponding properties of test functions from \autoref{Cor:Calculus rules d}, \autoref{Le:d = d1} and an argument as for \cite[Thm.~2.2.6]{gigli2018}.

\begin{proposition}\label{Pr:H11 calculus rules} The following properties hold for the space $\calG_\reg$ and the differential $\rmd$.
\begin{enumerate}[label=\textnormal{\textcolor{black}{(}\roman*\textcolor{black}{)}}]
\item \emph{Locality.} For every $f,g\in \calG_\reg$,
\begin{align*}
\One_{\{f=g\}}\d f = \One_{\{f=g\}}\d g.
\end{align*}
\item \emph{Chain rule.} For every $f\in \calG_\reg$ and every $\Leb^1$-negligible Borel set $C\subset\R$, 
\begin{align*}
\One_{f^{-1}(C)}\d f =0.
\end{align*}
In particular, for every $\varphi\in\Lip(\R)$, we have $\varphi\circ f \in \calG_\reg$ with
\begin{align*}
\rmd(\varphi\circ f) = \big[\varphi'\circ f\big]\d f,
\end{align*}
where the derivative $\varphi'$ is defined arbitrarily on the intersection of the set of non-differentiability points of $\varphi$ with the image of $f$.
\item \emph{Leibniz rule.} For every $f,g\in \calG_\reg\cap\Ell^\infty(\mms)$, $f\,g\in \calG_\reg$ with
\begin{align*}
\rmd (f\,g) = f\d g+g\d f.
\end{align*}
\end{enumerate}
\end{proposition}

\begin{remark}
We do not know if \autoref{Pr:H11 calculus rules} holds for functions merely in $\calG$. E.g., unlike the identification result \autoref{Le:d = d1} it is unclear whether the gradient estimate from \eqref{Eq:1BE grad est}  holds for every $f\in \calG$ or whether $\ChHeat_t$ maps $\calG$ into $\calG$, $t>0$. 
\end{remark}

\begin{lemma}\label{Le:L1 to H11} If $f\in  \F_\rme\cap\Ell^1(\mms)$ with $\rmd f\in \Ell^1(T^*\mms)$, then $f\in \calG_\reg$.
\end{lemma}

\begin{proof} Define $f_n \in \F_\rme\cap \Ell^1(\mms)\cap\Ell^\infty(\mms)$ by $f_n := \max\{\min\{f,n\},-n\}$, $n\in\N$. By \autoref{Cor:Calculus rules d}, we have $\rmd f_n = \One_{\{\vert f\vert < n\}}\d f\in\Ell^1(T^*\mms)$, and therefore $\rmd f_n \to \rmd f$ in $\Ell^1(T^*\mms)$ as $n\to\infty$. 

We claim that $f_n \in \calG_\reg$ for every $n\in\N$, which then readily yields the conclusion of the lemma. Indeed, given $n\in\N$ and $t>0$, we have $\ChHeat_tf_n\in\Test(\mms)\cap \Ell^1(\mms)$, and in fact $\ChHeat_tf_n \in \calG$ by  \autoref{Le:Div g nabla f}. Since $\ChHeat_tf_n \to f_n$ in $\F$ as $t\to 0$, by Lebesgue's theorem and the assumption that $f\in \Ell^1(\mms)$ it follows that $\ChHeat_t f_n \to f_n$ in $\Ell^1(\mms)$ as $t\to 0$. Moreover, since $\rmd \ChHeat_t f_n \to \rmd f_n$ in $\Ell^2(T^*\mms)$ as $t\to 0$ and since $\smash{(\rmd\ChHeat_tf_n)_{t\in [0,1]}}$ is bounded in $\Ell^1(T^*\mms)$ thanks to \eqref{Eq:1BE grad est}  and exponential boundedness of $(\Schr{\kappa}_t)_{t\geq 0}$ in $\Ell^1(\mms)$ \cite[Rem.~2.14]{erbar2020}, applying Lebesgue's theorem again we obtain that $\rmd \ChHeat_tf_n \to \rmd f_n$ in $\Ell^1(T^*\mms)$ as $t\to 0$.
\end{proof}

\begin{lemma} For every $f,g\in \F$ we have $f\,g\in \calG_\reg$ with
\begin{align*}
\rmd (f\,g) = f\d g + g\d f.
\end{align*}
\end{lemma}

\begin{proof} Again we define $f_n,g_n \in \F\cap\Ell^\infty(\mms)$ by $f_n := \max\{\min\{f,n\},-n\}$ and $g_n := \max\{\min\{g,n\},-n\}$, $n\in\N$. By \autoref{Cor:Calculus rules d} and \autoref{Re:Min E dom}, we have $f_n\,g_n\in  \F_\rme\cap\Ell^1(\mms)$ with
\begin{align*}
\rmd(f_n\,g_n) = f_n\d g_n + g_n\d f_n = f_n\,\One_{\{\vert g\vert < n\}}\d g + g_n\,\One_{\{\vert f\vert < n\}}\d f
\end{align*}
whence $f_n\,g_n \in \calG_\reg$ for every $n\in\N$ by \autoref{Le:L1 to H11}. Since $f_n\,g_n \to f\,g$ in $\Ell^1(\mms)$ and $f_n\,\One_{\{\vert g\vert < n\}}\d g + g_n\,\One_{\{\vert f\vert < n\}}\d f \to f\d g + g\d f$ in $\Ell^1(T^*\mms)$ as $n\to\infty$, respectively, the conclusion follows.
\end{proof}

\subsubsection{Second order spaces}

\begin{definition}\label{Def:H22} We define the space $\Dom_\reg(\Hess)\subset \Dom(\Hess)$ as
\begin{align*}
\Dom_\reg(\Hess) := \cl_{\Vert\cdot\Vert_{\Dom(\Hess)}}\Test(\mms).
\end{align*}
\end{definition}

By \autoref{Cor:Dom(Delta) subset W22}, we have
\begin{align*}
\Dom_\reg(\Hess) = \cl_{\Vert\cdot\Vert_{\Dom(\Hess)}}\Dom(\Delta).
\end{align*}
This space plays an important role in \autoref{Pr:Product rule for gradients}, but also in later discussions on calculus rules for the covariant derivative, see \autoref{Sub:Calculus rules}.

As indicated in the beginning of \autoref{Sub:Calc rules hess},  we consider the space $\smash{\Dom_{2,2,1}(\Hess)}$ of $\F$-functions with an $\Ell^1$-Hessian, which seems to be the correct framework for \autoref{Pr:Product rule fcts} and \autoref{Pr:Hess chain rule} below. (The numbers in the subscript denote the degree of integrability of $f$, $\rmd f$ and $\Hess f$, respectively, $\smash{f\in \Dom_{2,2,1}(\Hess)}$.) As after \autoref{Ex:Hess}, one argues that the defining property is well-defined. 

\begin{definition}\label{Def:W221} We define the space $\Dom_{2,2,1}(\Hess)$ to consist of all $f\in \F$ for which there exists $A\in \Ell^1((T^*)^{\otimes 2}\mms)$ such that for every $g_1,g_2,h\in\Test(\mms)$,
\begin{align*}
&2\int_\mms h\,A(\nabla g_1,\nabla g_2)\d\meas\\
&\qquad\qquad =-\int_\mms \langle\nabla f,\nabla g_1\rangle\div(h\,\nabla g_2)\d\meas - \int_\mms \langle\nabla f,\nabla g_2\rangle\div(h\,\nabla g_1)\d\meas\\
&\qquad\qquad\qquad\qquad - \int_\mms h\,\big\langle\nabla f,\nabla\langle\nabla g_1,\nabla g_2\rangle\big\rangle\d\meas.
\end{align*}
In case of existence, $A$ is unique, denoted by $\Hess_1 f$ and termed the \emph{Hessian} of $f$.
\end{definition}

By the weak$^*$ density of $\Test(T^{\otimes 2}\mms)$ in $\Ell^\infty(T^{\otimes 2}\mms)$, the uniqueness statement in \autoref{Def:W221} is indeed true given that such an element $A$ exists. Furthermore, if $\smash{f\in \Dom_{2,2,1}(\Hess)\cap \Dom(\Hess)}$ then of course
\begin{align*}
\Hess_1 f = \Hess f.
\end{align*}
We shall thus simply write $\Hess$ in place of $\Hess_1$ also for functions in $\smash{\Dom_{2,2,1}(\Hess)}$ without further notice. $\smash{\Dom_{2,2,1}(\Hess)}$ is a real vector space, and $\Hess$ is a linear operator on it.

The space $\smash{\Dom_{2,2,1}(\Hess)}$ is endowed with the norm $\Vert\cdot\Vert_{\Dom_{2,2,1}(\Hess)}$ given by
\begin{align*}
\Vert f \Vert_{\Dom_{2,2,1}(\Hess)} := \Vert  f\Vert_{\Ell^2(\mms)} + \Vert \rmd f\Vert_{\Ell^2(T^*\mms)} + \Vert \!\Hess f\Vert_{\Ell^1((T^*)^{\otimes 2}\mms)}. 
\end{align*}
Since both sides of the defining property in \autoref{Def:W221} are continuous in $f$ and $A$ w.r.t.~convergence in $\F$ and $\Ell^1((T^*)^{\otimes 2}\mms)$, respectively, this norm turns $\smash{\Dom_{2,2,1}(\Hess)}$ into a Banach space. It is also separable, since the map $\Id\times \rmd \times\Hess \colon \smash{\Dom_{2,2,1}(\Hess)}\to \Ell^2(\mms)\times \Ell^2(T^*\mms)\times \Ell^1((T^*)^{\otimes 2}\mms)$ is an isometry onto its image, where the latter space is endowed with the usual product norm
\begin{align*}
\Vert (f,\omega,A) \Vert := \Vert  f\Vert_{\Ell^2(\mms)} + \Vert \omega\Vert_{\Ell^2(T^*\mms)} + \Vert A\Vert_{\Ell^1((T^*)^{\otimes 2}\mms)}
\end{align*}
which is separable by \autoref{Pr:Generators cotangent module} and the discussion from \autoref{Sub:Lebesgue sp}.

\subsubsection{Calculus rules for the Hessian}\label{Subsub:Hess calc rules} After having introduced the relevant spaces, we finally proceed with the calculus rules for functions in $\smash{\Dom_{2,2,1}(\Hess)}$.

The next lemma will be technically useful. For $h\in\Test(\mms)$, the asserted equality is precisely the defining property of $\Hess f$ stated in \autoref{Def:Hess}. The general case $h\in \F\cap\Ell^\infty(\mms)$ follows by replacing $h$ by $\ChHeat_th\in\Test(\mms)$, $t>0$, and letting $t\to 0$ with the aid of \autoref{Le:Div g nabla f}.

\begin{lemma}\label{Le:More h} For every $f\in \Dom(\Hess)$, every $g_1,g_2\in\Test(\mms)$ and every $h\in \F\cap\Ell^\infty(\mms)$,
\begin{align*}
&2\int_\mms h\Hess f(\nabla g_1,\nabla g_2)\d\meas\\
&\qquad\qquad =-\int_\mms \langle\nabla f,\nabla g_1\rangle\div(h\,\nabla g_2)\d\meas - \int_\mms \langle\nabla f,\nabla g_2\rangle\div(h\,\nabla g_1)\d\meas\\
&\qquad\qquad\qquad\qquad -\int_\mms h\,\big\langle \nabla f,\nabla\langle\nabla g_1,\nabla g_2\rangle\big\rangle\d\meas.
\end{align*}
\end{lemma}

\begin{proposition}[Product rule]\label{Pr:Product rule fcts} If $f,g\in \Dom(\Hess)\cap\Ell^\infty(\mms)$, we have $f\,g\in \Dom_{2,2,1}(\Hess)$ with
\begin{align*}
\Hess(f\,g) &= g\Hess f + f\Hess g + \rmd f\otimes \rmd g + \rmd g \otimes \rmd f.
\end{align*}
\end{proposition} 

\begin{proof} Note that $f\,g\in\F$, and that the r.h.s.~of the claimed identity for $\Hess(f\,g)$ defines an element in $\Ell^1((T^*)^{\otimes 2}\mms)$. Now, given any $g_1,g_2,h\in\Test(\mms)$, by \autoref{Cor:Calculus rules d} and \autoref{Le:Div g nabla f} it follows that
\begin{align*}
&-\big\langle\nabla (f\,g),\nabla g_1\big\rangle\div(h\,\nabla g_2)\\
&\qquad\qquad = -f\,\langle\nabla g,\nabla g_1\rangle\div(h\,\nabla g_2) - g\,\langle\nabla f,\nabla g_1\rangle\div(h\,\nabla g_2)\textcolor{white}{\big\langle}\\
&\qquad\qquad = - \langle\nabla g,\nabla g_1\rangle\div(f\,h\,\nabla g_2) + h\,\langle\nabla g,\nabla g_1\rangle\,\langle\nabla f,\nabla g_2\rangle\textcolor{white}{\big\langle}\\
&\qquad\qquad\qquad\qquad - \langle\nabla f,\nabla g_1\rangle\div(g\,h\,\nabla g_2) + h\,\langle\nabla f,\nabla g_1\rangle\,\langle\nabla g,\nabla g_2\rangle \quad\meas\text{-a.e.},
\end{align*}
and analogously
\begin{align*}
&-\big\langle\nabla (f\,g),\nabla g_2\big\rangle\div(h\,\nabla g_1)\\
&\qquad\qquad = - \langle\nabla g,\nabla g_2\rangle\div(f\,h\,\nabla g_1) + h\,\langle\nabla g,\nabla g_2\rangle\,\langle\nabla f,\nabla g_1\rangle\textcolor{white}{\big\langle}\\
&\qquad\qquad\qquad\qquad - \langle\nabla f,\nabla g_2\rangle\div(g\,h\,\nabla g_1) + h\,\langle\nabla f,\nabla g_2\rangle\,\langle\nabla g,\nabla g_1\rangle \quad\meas\text{-a.e.},
\end{align*}
while finally
\begin{align*}
&-h\,\big\langle \nabla (f\,g),\nabla\langle\nabla g_1,\nabla g_2\rangle\big\rangle\\
&\qquad\qquad =-h\,f\,\big\langle \nabla g,\nabla\langle\nabla g_1,\nabla g_2\rangle\big\rangle - h\,g\,\big\langle\nabla f,\nabla\langle\nabla g_1,\nabla g_2\rangle\big\rangle\quad\meas\text{-a.e.}
\end{align*}
Adding up these three identities and using \autoref{Le:More h} for $f$ with $g\,h$ in place of $h$ and for $g$ with $f\,h$ in place of $h$ and by \eqref{Eq:Tensor product pointwise sc prod}, the claim readily follows.
\end{proof}

\begin{proposition}[Chain rule]\label{Pr:Hess chain rule} Let $f\in \Dom(\Hess)$, and let $\varphi\in\Cont^1(\R)$ such that $\smash{\varphi'\in\Lip_\bounded(\R)}$. If $\meas[\mms]=\infty$, we also assume $\varphi(0)=0$. Then $\varphi\circ f\in \Dom_{2,2,1}(\Hess)$ as well as
\begin{align*}
\Hess(\varphi\circ f) = \big[\varphi''\circ f\big]\d f\otimes \rmd f + \big[\varphi'\circ f\big]\Hess f,
\end{align*}
where $\varphi''$ is defined arbitrarily on the intersection of the set of non-differentiability points of $\varphi'$ with the image of $f$.
\end{proposition}

\begin{proof} By \autoref{Cor:Calculus rules d} and the boundedness of $\varphi'$, we have $\varphi\circ f \in \F$. The r.h.s.~of the claimed identity defines an object in $\Ell^1(\mms)$. Now, given any $g_1,g_2,h\in\Test(\mms)$, as in the proof of \autoref{Pr:Product rule fcts} we have
\begin{align*}
&-\big\langle\nabla(\varphi\circ f),\nabla g_1\big\rangle\div(h\,\nabla g_2)\\
&\qquad\qquad = -\big[\varphi'\circ f\big]\,\langle\nabla f,\nabla g_1\rangle\div(h\,\nabla g_2)\\
&\qquad\qquad = - \langle\nabla f,\nabla g_1\rangle\div\!\big(\big[\varphi'\circ f\big]\,h\,\nabla g_2\big)\\
&\qquad\qquad\qquad\qquad + h\,\big[\varphi''\circ f\big]\,\langle\nabla f,\nabla g_1\rangle\,\langle\nabla f,\nabla g_2\rangle\quad\meas\text{-a.e.},
\end{align*} 
and analogously
\begin{align*}
&-\big\langle\nabla(\varphi\circ f),\nabla g_2\big\rangle\div(h\,\nabla g_1)\\
&\qquad\qquad = - \langle\nabla f,\nabla g_2\rangle\div\!\big(\big[\varphi'\circ f\big]\,h\,\nabla g_1\big)\\
&\qquad\qquad\qquad\qquad + h\,\big[\varphi''\circ f\big]\,\langle\nabla f,\nabla g_2\rangle\,\langle\nabla f,\nabla g_1\rangle\quad\meas\text{-a.e.},
\end{align*}
while finally
\begin{align*}
-h\,\big\langle\nabla(\varphi\circ f),\nabla\langle\nabla g_1,\nabla g_2\rangle\big\rangle = - h\,\big[\varphi'\circ f\big]\,\big\langle\nabla f,\nabla\langle\nabla g_1,\nabla g_2\rangle\big\rangle\quad\meas\text{-a.e.}
\end{align*}
Adding up these three identities and using \autoref{Le:More h} for $f$ with $\smash{h\,[\varphi'\circ f]}$ in place of $h$ and by \eqref{Eq:Tensor product pointwise sc prod}, the claim readily follows.
\end{proof}

\begin{proposition}[Product rule for gradients]\label{Pr:Product rule for gradients} The following properties hold for every $f\in \Dom(\Hess)$ and every $g\in \Dom_\reg(\Hess)$. 
\begin{enumerate}[label=\textnormal{\textcolor{black}{(}\roman*\textcolor{black}{)}}]
\item\label{La:Ras} We have $\langle\nabla f,\nabla g\rangle\in \calG$ with
\begin{align*}
\rmd\langle\nabla f,\nabla g\rangle = \Hess f(\nabla g,\cdot) + \Hess g(\nabla f,\cdot).
\end{align*}
\item\label{La:Dwa} For every $g_1,g_2\in \Dom_\reg(\Hess)$,
\begin{align*}
2\Hess f(\nabla g_1,\nabla g_2) &=  \big\langle\nabla  g_1,\nabla\langle\nabla f,\nabla g_2\rangle\big\rangle + \big\langle\nabla g_2, \nabla\langle\nabla f,\nabla g_1\rangle\big\rangle\\
&\qquad\qquad - \big\langle\nabla f,\nabla \langle\nabla g_1,\nabla g_2\rangle\big\rangle\quad\meas\text{-a.e.}
\end{align*}
\item\label{La:Tri} If $f\in \Dom_\reg(\Hess)$ as well, then $\langle\nabla f,\nabla g\rangle\in \calG_\reg$.
\end{enumerate}
\end{proposition}

\begin{proof} We first treat item \ref{La:Ras} under the additional assumption that $g\in\Test(\mms)$. Let $g'\in\Test_{\Ell^\infty}(\mms)$ and $h\in\Test(\mms)$. Let $(f_n)_{n\in\N}$ be a sequence in $\Test(\mms)$ which converges to $f$ in $\F$. Then by \autoref{Def:Hess},
\begin{align*}
&\int_\mms h\,\big[\!\Hess f(\nabla g,\nabla g') + \Hess g(\nabla f,\nabla g')\big]\d\meas\\
&\qquad\qquad = \lim_{n\to\infty} \int_\mms h\,\big[\!\Hess f(\nabla g,\nabla g') + \Hess g(\nabla f_n,\nabla g')\big]\d\meas\\
&\qquad\qquad = -\frac{1}{2}\int_\mms \langle\nabla f,\nabla g\rangle\div(h\,\nabla g')\d\meas - \frac{1}{2}\int_\mms \langle\nabla f,\nabla g'\rangle\div(h\,\nabla g)\d\meas\\
&\qquad\qquad\qquad\qquad\qquad\qquad - \frac{1}{2}\int_\mms h\,\big\langle\nabla f,\nabla\langle\nabla g,\nabla g'\rangle\big\rangle\d\meas\\
&\qquad\qquad\qquad\qquad - \frac{1}{2}\lim_{n\to\infty} \int_\mms \langle\nabla g,\nabla f_n\rangle\div(h\,\nabla g')\d\meas\\
&\qquad\qquad\qquad\qquad\qquad\qquad - \frac{1}{2}\lim_{n\to\infty}\int_\mms \langle\nabla g,\nabla g'\rangle\div(h\,\nabla f_n)\d\meas\\
&\qquad\qquad\qquad\qquad\qquad\qquad - \frac{1}{2}\lim_{n\to\infty} \int_\mms h\,\big\langle\nabla g,\nabla\langle\nabla f_n,\nabla g'\rangle\big\rangle\d\meas\\
&\qquad\qquad = -\int_\mms \langle\nabla f,\nabla g\rangle\div(h\,\nabla g')\d\meas - \frac{1}{2}\int_\mms \langle\nabla f,\nabla g'\rangle\div(h\,\nabla g)\d\meas\\
&\qquad\qquad\qquad\qquad\qquad\qquad + \frac{1}{2}\lim_{n\to\infty}\int_\mms \langle\nabla f_n,\nabla g'\rangle\div(h\,\nabla g)\d\meas\\
&\qquad\qquad\qquad\qquad - \frac{1}{2}\int_\mms h\,\big\langle\nabla f,\nabla\langle\nabla g,\nabla g'\rangle\big\rangle\d\meas\\
&\qquad\qquad\qquad\qquad\qquad\qquad + \frac{1}{2}\lim_{n\to\infty}\int_\mms h\,\big\langle\nabla \langle\nabla g,\nabla g'\rangle,\nabla f_n\big\rangle\d\meas\\
&\qquad\qquad = -\int_\mms \langle\nabla f,\nabla g\rangle\div(h\,\nabla g')\d\meas.
\end{align*}
In the second last step, we used \autoref{Le:Div g nabla f} and the fact that $\langle\nabla g,\nabla g'\rangle\in \F$ as well as $\langle\nabla f_n,\nabla g'\rangle\in \F$ for every $n\in\N$ provided by \autoref{Pr:Bakry Emery measures}. Since $\Hess f(\nabla g,\cdot) + \Hess g(\nabla f,\cdot)\in \Ell^1(T^*\mms)$, the claim follows from the definition of $\calG$ in \autoref{Def:W11}.

To cover the case of general $g\in \Dom_\reg(\Hess)$, simply observe that given any $f\in \Dom(\Hess)$, $g'\in \Test_{\Ell^\infty}(\mms)$ and $h\in \Test(\mms)$, both sides of the above computation are continuous in $g$ w.r.t.~$\Vert\cdot\Vert_{\Dom(\Hess)}$.

Point \ref{La:Dwa} easily follows by expressing each summand in the r.h.s.~of the claimed identity in terms of the formula from \ref{La:Ras}, and the symmetry of the Hessian known from \autoref{Th:Hess properties}.

We finally turn to \ref{La:Tri}. If $f,g\in\Test(\mms)$, then $\langle\nabla f,\nabla g\rangle\in \F\cap\Ell^1(\mms)$ thanks to \autoref{Pr:Bakry Emery measures}, while $\rmd \langle\nabla f,\nabla g\rangle \in  \Ell^1(T^*\mms)$ by \ref{La:Ras}. The $\calG_\reg$-regularity of $\langle\nabla f,\nabla g\rangle$  thus follows from \autoref{Le:L1 to H11}. For general $f,g\in \Dom_\reg(\Hess)$, let $(f_n)_{n\in\N}$ and $(g_n)_{n\in\N}$ be two sequences in $\Test(\mms)$ such that $f_n \to f$ and $g_n\to g$ in $\Dom(\Hess)$ as $n\to\infty$, respectively. Then clearly $\langle\nabla f_n,\nabla g_n\rangle\to \langle\nabla f,\nabla g\rangle$ in $\Ell^1(\mms)$ as $n\to\infty$, while $\Hess f_n(\nabla g_n,\cdot) + \Hess g_n(\nabla f_n,\cdot) \to \Hess f(\nabla g,\cdot) + \Hess g(\nabla f,\cdot)$ in $\Ell^1(T^*\mms)$ as $n\to\infty$. By the  $\calG_\reg$-regularity of $\langle\nabla f_n,\nabla g_n\rangle$ for every $n\in\N$ already shown above, the proof is terminated.
\end{proof}

Combining the last two items of \autoref{Pr:Product rule for gradients} with \autoref{Pr:H11 calculus rules} thus entails the following locality property.

\begin{lemma}[Locality of the Hessian] For every $f,g\in \Dom_\reg(\Hess)$,
\begin{align*}
\One_{\{f=g\}}\Hess f = \One_{\{f=g\}}\Hess g.
\end{align*}
\end{lemma}

\section{Covariant derivative}\label{Ch:Covariant der}

\subsection{The Sobolev space $W^{1,2}(T\mms)$}\label{Sub:Cov der 1} In a similar kind as in \autoref{Ex:Hess}, the smooth context motivates how we should define a nonsmooth covariant derivative acting on vector fields, having now the notion of Hessian at our disposal.

\begin{example}\label{Ex:Smooth cov der motiv} Let $\mms$ be a Riemannian manifold with boundary. The covariant derivative $\nabla X$ of $X\in \Gamma(T\mms)$ is uniquely defined by the pointwise relation
\begin{align}\label{Eq:Covariant derivative def}
\begin{split}
\big\langle\nabla_{\nabla g_1}X, \nabla g_2\big\rangle &=\nabla X :(\nabla g_1\otimes \nabla g_2)\\
&= \big\langle\nabla\langle X,\nabla g_2\rangle,\nabla g_1\big\rangle - \Hess g_2(\nabla g_1,X)
\end{split}
\end{align}
for every $g_1,g_2\in \Cont_\comp^\infty(\mms)$, see e.g.~\cite[Thm.~6.2.1]{gigli2020}. That is, $\nabla$ is tensorial in its first and derivative in its second component, torsion-free and metrically compatible if and only if the second equality in \eqref{Eq:Covariant derivative def} holds for every $g_1$ and $g_2$ as above. This yields an alternative definition of the Levi-Civita connection $\nabla$ in place of Koszul's formula, see e.g.~\cite[p.~25]{petersen2006}, which does not use Lie brackets. (Indeed, in our setting it is not even clear how to define the Lie bracket without covariant derivatives.)

As in \autoref{Ex:Hess}, $\nabla$ is still uniquely determined on $\Gamma(T\mms)$ by the validity of \eqref{Eq:Covariant derivative def} for every $\smash{g_1,g_2\in\Cont_\comp^\infty(\mms)}$ for which
\begin{align*}
\langle\nabla g_1,\sfn\rangle = \langle\nabla g_2,\sfn\rangle = 0\quad\text{on }\partial\mms.
\end{align*}
For such $g_1$ and $g_2$, we integrate \eqref{Eq:Covariant derivative def} against $h\,\vol$ for a given $h\in\Cont_\comp^\infty(\mms)$ to obtain, using again the regularity discussion around \eqref{Eq:Smooth div identities},
\begin{align}\label{Eq:nabla X motivation}
\begin{split}
&\int_\mms h\,\nabla X : (\nabla g_1\otimes \nabla g_2)\d\vol\\
&\qquad\qquad =-\int_\mms \langle X,\nabla g_2\rangle\div(h\,\nabla g_1)\d\vol - \int_\mms h\Hess g_2(X,\nabla g_1)\d\vol.
\end{split}
\end{align} 
This identity characterizes $\nabla X$ by the existence of a smooth extension to $\partial\mms$.
\end{example}

In our setting, the r.h.s.~of \eqref{Eq:nabla X motivation} --- with $\vol$ replaced by $\meas$ --- still makes sense for $g_1,g_2,h\in\Test(\mms)$ and is compatible with the smooth case (in the sense that no boundary terms show up in the nonsmooth terminology of \autoref{Def:Normal component}), which is argued as in the paragraph after \autoref{Ex:Hess}. Indeed, for $X\in \Ell^2(T\mms)$ we have $\langle X,\nabla g_2\rangle\in\Ell^2(\mms)$ as well as $\div(h\,\nabla g_1) = \langle\nabla h,\nabla g_1\rangle + h\,\Delta g_1\in\Ell^2(\mms)$ by \autoref{Le:Div g nabla f}, while the second integral on the r.h.s.~of \eqref{Eq:nabla X motivation} is trivially well-defined.

This motivates the subsequent definition.

\begin{definition}\label{Def:Cov der} The space $W^{1,2}(T\mms)$ is defined to consist of all $X\in \Ell^2(T\mms)$ for which there exists $T\in \Ell^2(T^{\otimes 2}\mms)$ such that for every $g_1,g_2,h\in\Test(\mms)$,
\begin{align*}
\begin{split}
&\int_\mms h\,T:(\nabla g_1\otimes \nabla g_2)\d\meas\\
&\qquad\qquad= -\int_\mms \langle X,\nabla g_2\rangle\div(h\,\nabla g_1)\d\meas  -\int_\mms h\Hess g_2(X,\nabla g_1)\d\meas.
\end{split}
\end{align*}
In case of existence, the element $T$ is unique, denoted by $\nabla X$ and termed the \emph{covariant derivative} of $X$.
\end{definition}

Arguing as for the Hessian after \autoref{Def:Hess}, the uniqueness statement in \autoref{Def:Cov der} is derived. In particular, $W^{1,2}(T\mms)$ constitutes a vector space and the covariant derivative $\nabla$ is a linear operator on it. Further properties can be consulted in \autoref{Th:Properties W12 TM} below.

The space $W^{1,2}(T\mms)$ is endowed with the norm $\Vert\cdot\Vert_{W^{1,2}(T\mms)}$ given by
\begin{align*}
\big\Vert X \big\Vert_{W^{1,2}(T\mms)}^2 := \big\Vert X\big\Vert_{\Ell^2(T\mms)}^2 + \big\Vert\nabla X\big\Vert_{\Ell^2(T^{\otimes 2}\mms)}^2.
\end{align*}
We also define the \emph{covariant} functional $\Ch_\cov \colon \Ell^2(T\mms)\to[0,\infty]$ by
\begin{align}\label{Eq:E cov W}
\Ch_\cov(X) := \begin{cases} \displaystyle\int_\mms \big\vert\nabla X\big\vert_\HS^2\d\meas & \text{if }X\in W^{1,2}(T\mms),\\
\infty & \text{otherwise}.
\end{cases}
\end{align}

\begin{theorem}\label{Th:Properties W12 TM} The space $W^{1,2}(T\mms)$, the covariant derivative $\nabla$ and the functional $\Ch_\cov$ possess the following properties.
\begin{enumerate}[label=\textnormal{\textcolor{black}{(}\roman*\textcolor{black}{)}}]
\item\label{La:Cov 1} $W^{1,2}(T\mms)$ is a separable Hilbert space w.r.t.~$\Vert\cdot\Vert_{W^{1,2}(T\mms)}$.
\item\label{La:Cov 2} The covariant derivative $\nabla$ is a closed operator. That is, the image of the map $\Id\times \nabla\colon W^{1,2}(T\mms)\to \Ell^2(T\mms)\times\Ell^2(T^{\otimes 2}\mms)$ is a closed subspace of $\Ell^2(T\mms)\times\Ell^2(T^{\otimes 2}\mms)$.
\item\label{La:Cov 5} For every $f\in \Test(\mms)$ and every $g\in \Test(\mms)\cup\R\,\One_\mms$, we have $g\,\nabla f\in W^{1,2}(T\mms)$ with
\begin{align*}
\nabla(g\,\nabla f) = \nabla g\otimes\nabla f + g\,(\Hess f)^\sharp,
\end{align*}
with the interpretation $\nabla \One_\mms := 0$. In particular $\Reg(T\mms)\subset W^{1,2}(T\mms)$, and $W^{1,2}(T\mms)$ is dense in $\Ell^2(T\mms)$.
\item\label{La:Cov 3} The functional $\Ch_\cov$ is lower semicontinuous, and for every $X\in \Ell^2(T\mms)$ we have the duality formula
\begin{align*}
\Ch_\cov(X) &= \sup\!\Big\lbrace\!-\!2\sum_{i=1}^n\int_\mms \langle X, Z_i\rangle\div Y_i\d\meas -2\sum_{i=1}^n\int_\mms  \nabla Z_i : (Y_i\otimes X)\d\meas\\
&\qquad\qquad - \int_\mms \Big\vert\!\sum_{i=1}^n Y_i\otimes Z_i\Big\vert^2\d\meas :  n\in \N,\ Y_i, Z_i\in \Test(T\mms)\Big\rbrace.
\end{align*}
\end{enumerate}
\end{theorem}

\begin{proof} Item \ref{La:Cov 2} is addressed by observing that given any $g_1,g_2,h\in\Test(\mms)$, both sides of the defining property of $\nabla$ in \autoref{Def:Cov der} are continuous in $X$ and $T$ w.r.t.~weak convergence in $\Ell^2(T\mms)$ and $\Ell^2(T^{\otimes 2}\mms)$, respectively.

For \ref{La:Cov 1}, it is clear from \ref{La:Cov 2} and the trivial parallelogram identity of $\smash{\Vert\cdot\Vert_{W^{1,2}(T\mms)}}$ that $W^{1,2}(T\mms)$ is a Hilbert space. To prove its separability, we endow $\Ell^2(T\mms)\times \Ell^2(T^{\otimes 2}\mms)$ with the product norm $\Vert\cdot\Vert$ given by
\begin{align*}
\Vert (X,T) \Vert^2 := \big\Vert X\big\Vert_{\Ell^2(T\mms)}^2+\big\Vert T\big\Vert_{\Ell^2(T^{\otimes 2}\mms)}^2.
\end{align*}
Recall from \autoref{Pr:Generators cotangent module}, \autoref{Th:Riesz theorem modules} and the discussions in \autoref{Sub:Tensor products} and \autoref{Sub:Test objects} that $\Ell^2(T\mms)$ and $\Ell^2(T^{\otimes 2}\mms)$ are separable Hilbert spaces, and so is their Cartesian product. As $\Id\times\nabla \colon W^{1,2}(T\mms) \to \Ell^2(T\mms)\times\smash{\Ell^2(T^{\otimes 2}\mms)}$ is a bijective isometry onto its image, the proof of \ref{La:Cov 1} is completed.

Item \ref{La:Cov 5} for $g:= \One_\mms$ follows from \autoref{Pr:Product rule for gradients}. Indeed, given any $g_1,g_2,h\in\Test(\mms)$, by definition of $\calG$ we have
\begin{align*}
&\int_\mms h\Hess f(\nabla g_1,\nabla g_2)\d\meas\\
&\qquad\qquad = \int_\mms h\,\big\langle \nabla\langle\nabla f,\nabla g_2\rangle,\nabla g_1 \big\rangle\d\meas - \int_\mms h\Hess g_2(\nabla f,\nabla g_2)\d\meas\\
&\qquad\qquad = - \int_\mms \langle\nabla f,\nabla g_2\rangle\div(h\,\nabla g_1)\d\meas - \int_\mms h\Hess g_2(\nabla f,\nabla g_2)\d\meas.
\end{align*}
The argument for $g\in\Test(\mms)$  follows similar lines.  To prove it, let $g_1,g_2,h\in\Test(\mms)$. Recall from \autoref{Pr:Bakry Emery measures} that $\langle\nabla f,\nabla g_2\rangle\in \F$. Since additionally $\langle\nabla f,\nabla g_2\rangle\in\Ell^\infty(\mms)$, we also have $g\,\langle\nabla f,\nabla g_2\rangle\in \F$ by the Leibniz rule for the gradient --- compare with \autoref{Cor:Calculus rules d} ---  which, also taking into account \autoref{Pr:Product rule for gradients} and \autoref{Le:d = d1}, yields
\begin{align*}
\rmd\big[g\,\langle\nabla f,\nabla g_2\rangle\big] = \langle\nabla f,\nabla g_2\rangle\,\rmd g + g\Hess f(\nabla g_2,\cdot) + g\Hess g_2(\nabla f,\cdot).
\end{align*}
This entails that
\begin{align*}
&-\int_\mms g\,\langle\nabla f,\nabla g_2\rangle\div(h\,\nabla g_1)\d\meas - \int_\mms h\,g\Hess g_2(\nabla f,\nabla g_1)\d\meas\\
&\qquad\qquad = \int_\mms \rmd\big[g\,\langle\nabla f,\nabla g_2\rangle\big](h\,\nabla g_1)\d\meas - \int_\mms h\,g\Hess g_2(\nabla f,\nabla g_1)\d\meas\\
&\qquad\qquad = \int_\mms h\,\big[\langle\nabla g,\nabla g_1\rangle\,\langle\nabla f,\nabla g_2\rangle + g\Hess f(\nabla g_1,\nabla g_2)\big]\d\meas,
\end{align*}
which is the claim by the definition \eqref{Eq:Tensor product pointwise sc prod} of the pointwise tensor product.

Concerning \ref{La:Cov 3}, the lower semicontinuity of $\Ch_\cov$ follows from the fact that bounded sets in the Hilbert space $\Ell^2(T^{\otimes 2}\mms)$ are weakly relatively compact, and from the closedness of $\nabla$ that has been shown in \ref{La:Cov 2}. 

The proof of the duality formula for $\Ch_\cov$ is analogous to the one for the duality formula for $\Ch_2$ in \autoref{Th:Hess properties}. Details are left to the reader.
\end{proof}

\begin{remark}\label{Re:Conf trafos II} Similarly to \autoref{Re:Conf trafos} and motivated by the $\RCD^*(K,N)$ result \cite[Prop.~3.12]{han2019}, $K\in\R$ and $N\in [1,\infty)$,  in practice the notion of covariant derivative from \autoref{Def:Cov der} should only depend on conformal transformations of $\langle\cdot,\cdot\rangle$, but be independent of drift transformations of $\meas$.
\end{remark}

\subsection{Calculus rules}\label{Sub:Calculus rules} In this section, we proceed in showing less elementary --- still expected --- calculus rules for the covariant derivative. Among these, we especially regard \autoref{Pr:Compatibility} and \autoref{Le:Kato inequality} as keys for the functionality of our second order axiomatization which also potentially involves ``boundary contributions''.

\subsubsection{Some auxiliary spaces of vector fields} As in \autoref{Sub:Calc rules hess}, we introduce two Sobolev spaces of vector fields we use in the sequel. In fact, the space $H^{1,2}(T\mms)$ which is introduced next plays a dominant role later, see e.g.~\autoref{Sub:Heat flow vector fields}.

\begin{definition}\label{Def:H12 vfs} We define the space $H^{1,2}(T\mms)\subset W^{1,2}(T\mms)$ as
\begin{align*}
H^{1,2}(T\mms) := \cl_{\Vert \cdot \Vert_{W^{1,2}(T\mms)}} \Reg(T\mms).
\end{align*}
\end{definition}

By \autoref{Th:Properties W12 TM}, we see that $\nabla\,\Dom_\reg(\Hess)\subset H^{1,2}(T\mms)$.

$H^{1,2}(T\mms)$ is in general a strict subset of $W^{1,2}(T\mms)$. For instance, on compact Riemannian manifolds with boundary, this follows as in the beginning of \autoref{Sub:Calc rules hess} by \eqref{Eq:div nabla = Delta}, \autoref{Le:Div g nabla f}, \autoref{Ex:Mflds with boundary} and \autoref{Pr:Trace thm} for $E:= T\mms$.

\begin{remark} For noncollapsed mGH-limits of  Riemannian manifolds without boundary under uniform Ricci and diameter bounds, it is proved in \cite[Prop.~4.5]{hondaspec} that $\smash{H^{1,2}(T\mms) = W^{1,2}(T\mms)}$. This does not conflict with our above argument involving the presence of a boundary, since spaces ``with boundary'' are known not to appear as noncollapsed Ricci limits  \cite[Thm.~6.1]{cheeger1997}.
\end{remark}

\begin{remark}\label{Re:Test closure} Besides technical reasons, \autoref{Def:H12 vfs} has the advantage that it includes both gradient vector fields of test functions as well as general elements of $\Test(T\mms)$ in the calculus rules below. See also \autoref{Sec:Curvature}. 

The ``$H^{1,2}$-space'' of vector fields introduced in  \cite[Def.~3.4.3]{gigli2018} is rather given by $\smash{\cl_{\Vert \cdot \Vert_{W^{1,2}(T\mms)}}\Test(T\mms)}$, which --- in our generality --- is a priori \emph{smaller} than $H^{1,2}(T\mms)$. The converse inclusion seems more subtle: as similarly encountered in \autoref{Re:No W11 in general}, the issue is that we do not really know how to pass from zeroth-order factors belonging to $\Test(\mms)$ to constant ones in a topology which also takes into account ``derivatives'' of vector fields in our generality. However, since ``$\supset$'' holds e.g.~if $\mms$ is intrinsically complete by \autoref{Th:Properties W12 TM}, \autoref{Def:H12 vfs} and the  approach from \cite[Def.~3.4.3]{gigli2018} give rise to the same space on $\RCD(K,\infty)$ spaces, $K\in\R$.
\end{remark}

As in \autoref{Def:W221}, in view of \autoref{Le:Leibniz rule W(21)} we introduce a space intermediate between $W^{1,2}(T\mms)$ and ``$W^{1,1}(T\mms)$'', a suitable space (that we do not define) of all $X\in \Ell^1(T\mms)$ with covariant derivative $\nabla X$ in $\Ell^1(T\mms)$. The problem arising in defining the latter, compare with \autoref{Def:Cov der}, is that we do not know if there exists a large enough class of $g_2\in\Test(\mms)$ such that $\Hess g_2\in \Ell^\infty((T^*)^{\otimes 2}\mms)$.

\begin{definition}\label{Def:W21} The space $W^{(2,1)}(T\mms)$ is defined to consist of all $X\in \Ell^2(T\mms)$ for which there exists $T\in \Ell^1(T^{\otimes 2}\mms)$ such that for every $g_1,g_2,h\in \Test(\mms)$,
\begin{align*}
\begin{split}
&\int_\mms h\,T:(\nabla g_1\otimes \nabla g_2)\d\meas\\
&\qquad\qquad = -\int_\mms \langle X,\nabla g_2\rangle\div(h\,\nabla g_1)\d\meas -\int_\mms h\Hess g_2(X,\nabla g_1)\d\meas. 
\end{split}
\end{align*}
In case of existence, the element $T$ is unique, denoted by $\nabla_1 X$ and called the \emph{covariant derivative} of $X$.
\end{definition}

Indeed, the uniqueness follows from weak$^*$ density of $\Test(T^{\otimes 2}\mms)$ in $\Ell^\infty(T^{\otimes 2}\mms)$. In particular, $\smash{W^{(2,1)}(T\mms)}$ is clearly a vector space, and $\nabla_1$ is a linear operator on it. By definition, for $X\in W^{(2,1)}(T\mms)\cap W^{1,2}(T\mms)$ we furthermore have
\begin{align*}
\nabla_1 X = \nabla X,
\end{align*}
whence we subsequently write $\nabla$ in place of $\nabla_1$ as long as confusion is excluded.

We endow $\smash{W^{(2,1)}(T\mms)}$ with the norm $\Vert\cdot\Vert_{W^{(2,1)}(T\mms)}$ given by
\begin{align*}
\Vert X \Vert_{W^{(2,1)}(T\mms)} := \Vert X\Vert_{\Ell^2(T\mms)} + \Vert \nabla X\Vert_{\Ell^1(T^{\otimes 2}\mms)}.
\end{align*}
Since both sides of the defining property in \autoref{Def:W21} are continuous w.r.t.~weak convergence in $X$ and $T$, respectively, this norm turns $\smash{W^{(2,1)}(T\mms)}$ into a Banach space. Moreover, since the map $\smash{\Id\times \nabla \colon W^{(2,1)}(T\mms)\to \Ell^2(T\mms)\times \Ell^1(T^{\otimes 2}\mms)}$ is a bijective isometry onto its image --- where $\smash{\Ell^2(T\mms)\times \Ell^1(T^{\otimes 2}\mms)}$ is endowed with the usual product norm ---  \autoref{Pr:Generators cotangent module}, \autoref{Th:Riesz theorem modules} and the discussion from \autoref{Sub:Leb sp tp} show that $\smash{W^{(2,1)}(T\mms)}$ is separable.

\subsubsection{Calculus rules for the covariant derivative}

\begin{lemma}[Leibniz rule]\label{Le:Leibniz rule W(21)} Let $X\in W^{1,2}(T\mms)$ and $f\in \F\cap\Ell^\infty(\mms)$. Then $f\,X\in W^{(2,1)}(T\mms)$ and
\begin{align*}
\nabla(f\,X) = \nabla f\otimes X + f\,\nabla X.
\end{align*}
\end{lemma}

\begin{proof} We first assume that $f\in \Test(\mms)$. Given any $g_1,g_2,h\in\Test(\mms)$, we have $h\,f\in\Test(\mms)$ since $\Test(\mms)$ is an algebra. By \autoref{Le:Div g nabla f} and the Leibniz rule for the gradient, \autoref{Cor:Calculus rules d}, we obtain
\begin{align*}
\div(h\,f\,\nabla g_1) = h\,\langle\nabla f,\nabla g_1\rangle + f\div(h\,\nabla g_1)\quad\meas\text{-a.e.}
\end{align*}
Hence, by definition of $\nabla X$ from \autoref{Def:Cov der} and \autoref{Th:Hess},
\begin{align*}
&\int_\mms h\,(f\,\nabla X) : (\nabla g_1\otimes\nabla g_2)\d\meas\\
&\qquad\qquad = -\int_\mms \langle X,\nabla g_2\rangle\div(h\,f\,\nabla g_1)\d\meas - \int_\mms h\,f\Hess g_2(X,\nabla g_1)\d\meas\\
&\qquad\qquad = -\int_\mms h\,\langle\nabla f,\nabla g_1\rangle\,\langle X,\nabla g_2\rangle\d\meas - \int_\mms \langle f\,X,\nabla g_2\rangle\div(h\,\nabla g_1)\d\meas\\
&\qquad\qquad\qquad\qquad -\int_\mms h\Hess g_2(f\,X,\nabla g_1)\d\meas.
\end{align*}
Rearranging terms according to \autoref{Def:W21} yields the claim.

For general $f\in \F\cap\Ell^\infty(\mms)$, note that on the one hand  $\ChHeat_tf\in \Test(\mms)$ and $\Vert\ChHeat_t f\Vert_{\Ell^\infty(\mms)} \leq \Vert f\Vert_{\Ell^\infty(\mms)}$ holds for every $t>0$, and on the other hand $\ChHeat_tf\to f$ in $\F$ as $t\to 0$. Thus, we easily see that $\ChHeat_tf\,X \to f\,X$ in $\Ell^2(T\mms)$ as well as $\nabla (\ChHeat_tf\,X) = \nabla \ChHeat_t f\otimes X + \ChHeat_t f\,\nabla X \to f\otimes X + f\,\nabla X$ in $\Ell^1(T^{\otimes 2}\mms)$ as $t \to 0$, and this suffices to conclude the proof.
\end{proof}

\begin{remark}\label{Re:fX in H12} Elementary approximation arguments, also using \autoref{Le:Mollified heat flow} for \ref{La:Item Zwei}, entail the following two $H^{1,2}$-variants of \autoref{Le:Leibniz rule W(21)}.
\begin{enumerate}[label=\textcolor{black}{(}\roman*\textcolor{black}{)}]
\item\label{La:lalalala} If $f\in \Test(\mms)$ and $X\in H^{1,2}(T\mms)$, then $f\,X \in H^{1,2}(T\mms)$ and
\begin{align*}
\nabla (f\,X) = \nabla f\otimes X + f\,\nabla X.
\end{align*}
\item\label{La:Item Zwei} The same conclusion as in \ref{La:lalalala} holds if merely $f\in \F\cap\Ell^\infty(\mms)$ and $X\in H^{1,2}(T\mms)\cap\Ell^\infty(T\mms)$.
\end{enumerate}
\end{remark}

In view of the important \emph{metric compatibility} of $\nabla$ discussed in \autoref{Pr:Compatibility} as well as  its consequences below, we shall first make sense of \emph{directional derivatives} of a given $X\in W^{1,2}(T\mms)$ in the direction of $Z\in \Ell^0(T\mms)$. There exists a unique vector field $\nabla_Z X\in \Ell^0(T\mms)$ which satisfies
\begin{align}\label{Eq:Dir der def}
\langle\nabla_ZX,Y\rangle = \nabla X : (Z\otimes Y)\quad\meas\text{-a.e.}
\end{align}
for every $Y\in \Ell^0(T\mms)$. Indeed, the r.h.s.~is a priori well-defined for every $Y,Z\in\Ell^0(T\mms)$ for which $Z\otimes Y\in \Ell^2(T^{\otimes 2}\mms)$, but this definition can and will be uniquely extended by continuity to a bilinear map $\nabla X : (\cdot\otimes\cdot)\colon \Ell^0(T\mms)^2\to\Ell^0(\mms)$, and the appropriate existence of $\nabla_ZX$ is then due to \autoref{Pr:Dual of L0}. In particular,
\begin{align}\label{Eq:Bound nabla Z X}
\vert\nabla_Z X\vert \leq \vert\nabla X\vert\,\vert Z\vert\quad\meas\text{-a.e.}
\end{align} 

\begin{proposition}[Metric compatibility]\label{Pr:Compatibility} Let $X\in W^{1,2}(T\mms)$ and $Y\in H^{1,2}(T\mms)$. Then $\langle X,Y\rangle\in \calG$, and for every $Z\in \Ell^0(T\mms)$,
\begin{align*}
\rmd\langle X,Y\rangle(Z) = \big\langle\nabla_ZX,Y\big\rangle + \big\langle X, \nabla_ZY\big\rangle\quad\meas\text{-a.e.}
\end{align*}
Moreover, if $X\in H^{1,2}(T\mms)$ as well, then $\langle X,Y\rangle\in \calG_\reg$.
\end{proposition}

\begin{proof} We first prove the claim for $Y = \nabla f$ for some $f\in\Test(\mms)$. In this case, given any $g,h\in\Test(\mms)$, by definition of $\nabla X$,
\begin{align*}
&\int_\mms h\,\nabla X : (\nabla g\otimes \nabla f)\d\meas\\
&\qquad\qquad = -\int_\mms \langle X,\nabla f\rangle\div(h\,\nabla g)\d\meas - \int_\mms h\Hess f(X,\nabla g)\d\meas.
\end{align*}
Rearranging terms and recalling \autoref{Def:W11} as well as  \autoref{Th:Properties W12 TM},  we see that $\langle X,\nabla f\rangle \in \Ell^1(\mms)$ belongs to $\calG$ and
\begin{align*}
\rmd\langle X,\nabla f\rangle = \nabla X:(\cdot\otimes \nabla f) + (\Hess f)^\sharp : (\cdot\otimes \nabla f).
\end{align*}
It follows from \eqref{Eq:Dir der def} that for every $Z\in \Test(T\mms)$,
\begin{align*}
\rmd\langle X,\nabla f\rangle(Z) = \big\langle\nabla_ZX,\nabla f\big\rangle + \big\langle \nabla_Z\nabla f,X\big\rangle\quad\meas\text{-a.e.}
\end{align*}
By $\Ell^\infty$-linear extension and the density of $\Test(T\mms)$ in $\Ell^2(T\mms)$, this identity holds in fact for every $Z\in\Ell^2(T\mms)$, and therefore extends to arbitrary $Z\in\Ell^0(T\mms)$ by construction of the latter space in \autoref{Sub:L0 modules}.

The case $Y := g\,\nabla f$, $f,g\in\Test(\mms)$, follows by the trivial $\meas$-a.e.~identity $\langle X, Y\rangle = \langle g\,X, \nabla f\rangle$, the previously derived identity and \autoref{Le:Leibniz rule W(21)}. 

By linearity of the covariant derivative and \autoref{Th:Properties W12 TM}, the foregoing discussion thus readily covers the case of general $Y\in\Reg(T\mms)$.

Now, given any $Y\in H^{1,2}(T\mms)$, a sequence $(Y_n)_{n\in\N}$ in $\Reg(T\mms)$ such that $Y_n \to Y$ in $W^{1,2}(T\mms)$ as $n\to\infty$ and $Z\in\Ell^2(T\mms)\cap \Ell^\infty(T\mms)$, we have $\langle X,Y_n\rangle \to \langle X,Y\rangle$ and $\langle\nabla_Z X,Y_n\rangle + \langle \nabla_ZY_n,X\rangle\to \langle\nabla_Z X,Y\rangle + \langle \nabla_ZY,X\rangle$ in $\Ell^1(\mms)$ as $n\to\infty$. Passing to the limit in the definition of $\calG$, it straightforwardly follows that $\langle X,Y\rangle\in \calG$. The claimed identity for $\rmd\langle X,Y\rangle(Z)$ follows after passing to suitable subsequences, and it extends to arbitrary $Z\in\Ell^0(T\mms)$ as in the first step of the current proof.

To prove the last claim, suppose first that $X,Y\in\Reg(\mms)$, in which case $\langle X,Y\rangle \in \calG\cap \F$ by \autoref{Th:Properties W12 TM}, what we already proved, \autoref{Pr:Bakry Emery measures} and \autoref{Cor:Calculus rules d}. By the duality between $\Ell^1(T\mms)$ and $\Ell^\infty(T^*\mms)$ as $\Ell^\infty$-modules, see \autoref{Sub:Lebesgue sp}, and \autoref{Re:CSU} below, we deduce that $\rmd\langle X,Y\rangle \in \Ell^1(T^*\mms)$, whence $\langle X,Y\rangle\in \calG_\reg$ thanks to \autoref{Le:d = d1} and \autoref{Le:L1 to H11}.

Lastly, given arbitrary $X,Y\in H^{1,2}(T\mms)$, let $(X_n)_{n\in\N}$ and $(Y_n)_{n\in\N}$ be sequences in $\Reg(T\mms)$ that converge to $X$ and $Y$ in $W^{1,2}(T\mms)$, respectively. Then clearly $\langle X_n,Y_n\rangle\to\langle X,Y\rangle$ in $\Ell^1(\mms)$ as $n\to\infty$, while \autoref{Re:CSU} below ensures that $\rmd \langle X_n,Y_n\rangle \to \rmd\langle X,Y\rangle$ in $\Ell^1(T^*\mms)$ as $n\to\infty$, which is the claim.
\end{proof}

\begin{remark}\label{Re:CSU} By duality and \eqref{Eq:Bound nabla Z X}, it follows in particular from \autoref{Pr:Compatibility} that for every $X\in W^{1,2}(T\mms)$ and every $Y\in H^{1,2}(T\mms)$,
\begin{align*}
\vert\rmd \langle X,Y\rangle\vert \leq \,\vert\nabla X\vert_\HS\,\vert Y \vert + \vert\nabla Y\vert_\HS\,\vert X\vert\quad\meas\text{-a.e.}
\end{align*}
\end{remark}

The following lemma is a version of what is known as \emph{Kato's inequality} (for the Bochner Laplacian) in the smooth case \cite[Ch.~2]{hess1980}. See also \cite[Lem.~3.5]{debin2021}.

\begin{lemma}[Kato's inequality]\label{Le:Kato inequality} For every $X\in H^{1,2}(T\mms)$,  $\vert X\vert\in \F$ and
	\begin{align*}
		\big\vert\nabla\vert X\vert\big\vert \leq \big\vert\nabla X\big\vert_\HS\quad\meas\text{-a.e.}	
	\end{align*}
	In particular, if $X\in H^{1,2}(T\mms)\cap\Ell^\infty(T\mms)$ then $\vert X\vert^2\in \F$.
\end{lemma}

\begin{proof} We initially prove the first claim for $X\in\Reg(T\mms)$. Define $\varphi_n\in\Lip([0,\infty))$ by $\smash{\varphi_n(t) := (t+1/n^2)^{1/2}-1/n}$, $n\in\N$. By  polarization, \autoref{Pr:Bakry Emery measures} and the Leibniz rule stated in \autoref{Cor:Calculus rules d}, we have $\vert X\vert^2\in \F$, and thus $\varphi_n\circ\vert X\vert^2\in \F$ for every $n\in\N$. Moreover, $\varphi_n\circ\vert X\vert^2\to \vert X\vert$ pointwise $\meas$-a.e.~and in $\Ell^2(\mms)$ as $n\to\infty$, as well as $\varphi_n'(t)^2\,t \leq 1/4$ for every $t\geq 0$ and every $n\in\N$. Hence by \autoref{Re:CSU} and --- since also $\vert X\vert^2\in \calG_\reg$ by \autoref{Pr:Compatibility} --- \autoref{Le:d = d1},
\begin{align}\label{Eq:99 Luftballons}
\begin{split}
\big\vert\nabla(\varphi_n\circ\vert X\vert^2)\big\vert^2 &= \big\vert\varphi_n'\circ\vert X\vert^2\big\vert^2\,\big\vert\nabla \vert X\vert^2\big\vert^2\\
&\leq 4\,\big\vert\varphi_n'\circ\vert X\vert^2\big\vert^2\,\vert X\vert^2\,\big\vert\nabla X\big\vert_\HS^2\\
&\leq \big\vert \nabla X\big\vert_\HS^2\quad\meas\text{-a.e.}
\end{split}
\end{align}
Therefore $\vert X\vert\in \F$ by \autoref{Pr:Extended domain props}, and $(\varphi_n\circ\vert X\vert^2)_{n\in\N}$ converges $\F$-weakly to $\vert X\vert$. By Mazur's lemma, suitable convex combinations of elements of $(\varphi_n\circ\vert X\vert^2)_{n\in\N}$ converge $\F$-strongly to $\vert X\vert$. The convexity of the carré du champ, see e.g.~\cite[p.~249]{ambrosio2015}, and \eqref{Eq:99 Luftballons} imply the claimed $\meas$-a.e.~upper bound on $\smash{\big\vert\nabla\vert X\vert\big\vert}$. 

For general $X\in H^{1,2}(T\mms)$, let $(X_n)_{n\in\N}$ be a sequence in $\Reg(T\mms)$ such that $X_n \to X$ in $H^{1,2}(T\mms)$ and $\vert X_n \vert \to \vert X\vert$ both pointwise $\meas$-a.e.~and in $\Ell^2(\mms)$ as $n\to\infty$. By what we proved above, $(\vert X_n\vert)_{n\in\N}$ is bounded in $\F$, whence $\vert X\vert\in \F$ again by  \autoref{Pr:Extended domain props}. Still by what we already proved, every $\Ell^2$-weak limit of subsequences of $(\big\vert\nabla \vert X_n\vert\big\vert)_{n\in\N}$ is clearly no larger than $\vert\nabla X\vert_\HS$ $\meas$-a.e., whence we obtain $\smash{\big\vert\nabla \vert X\vert\big\vert \leq \vert\nabla X\vert_\HS}$ $\meas$-a.e.~by Mazur's lemma. (See also (2.10) in \cite{ambrosio2015}.)

If $X\in H^{1,2}(T\mms)\cap\Ell^\infty(T\mms)$, the $\F$-regularity of $\vert X\vert^2$ follows from the one of $\vert X\vert$ and the chain rule in \autoref{Cor:Calculus rules d}.
\end{proof}

\autoref{Le:Kato inequality} has numerous important consequences. First, the $\F$-regularity asserted therein makes it possible in \autoref{Sec:Curvature} to pair the function $\smash{\vert X\vert\in \F}$, $X\in H^{1,2}(T\mms)$, with the given distribution $\smash{\kappa\in \F_\qloc^{-1}(\mms)}$. Second, it yields the improved semigroup comparison for the covariant heat flow in \autoref{Th:HSU Bochner}. Third, it cancels out the covariant term appearing in the definition of the Ricci curvature, \autoref{Th:Ricci measure}, leading to a vector $q$-Bochner inequality for $q\in [1,2]$,  \autoref{Th:Vector Bochner}. Under additional assumptions, the latter again implies improved semigroup comparison results in \autoref{Th:HSU forms}, this time for the contravariant heat flow. Fourth,  it is regarded as the key technical tool in showing that every $X\in H^{1,2}(T\mms)$ has a ``quasi-continuous representative'' similar to \cite[Thm.~3.14]{debin2021}. This latter topic is not addressed here, but the arguments of \cite{debin2021} do not seem hard to adapt to our setting.

\begin{corollary}\label{Cor:Kappa cont} The real-valued function $\smash{X \mapsto \langle \kappa\,\big\vert\,\vert X\vert^2\big\rangle}$ defined on $H^{1,2}(T\mms)$ is $H^{1,2}$-continuous. More precisely, let $\rho'\in (0,1)$ and $\alpha'\in\R$ satisfy \autoref{Le:Form boundedness} for every $\mu\in \{\kappa^+,\kappa^-,\vert\kappa\vert\}$. Then for every $X,Y\in H^{1,2}(T\mms)$,
\begin{align*}
\big\vert\big\langle \mu\,\big\vert\,\vert X\vert^2\big\rangle^{1/2} - \big\langle \mu\,\big\vert\,\vert Y\vert^2\big\rangle^{1/2}\big\vert^2   \leq \rho'\,\Ch_\cov(X-Y) + \alpha'\,\big\Vert X-Y\big\Vert_{\Ell^2(T\mms)}^2.
\end{align*}
\end{corollary}

\begin{proof} Since $\kappa=\kappa^+-\kappa^-$, it suffices to prove the last statement. Given $X,Y\in H^{1,2}(T\mms)$, by \autoref{Le:Kato inequality} we have $\vert X\vert, \vert Y\vert, \vert X - Y\vert\in \F$. In particular, the expressions $\smash{\big\langle\mu\,\big\vert\,\vert X\vert^2\big\rangle}$ and $\smash{\big\langle\mu\,\big\vert\,\vert Y\vert^2\big\rangle}$ make sense. Since
\begin{align*}
&\big\vert \big\langle\mu \,\big\vert\,\vert X\vert^2\big\rangle^{1/2} - \big\langle \mu\,\big\vert\,\vert Y\vert^2\big\rangle^{1/2}\big\vert^2\textcolor{white}{\big\Vert_{\Ell^2}^2}\\
&\qquad\qquad \leq \big\langle\mu\,\big\vert\,\big\vert\vert X\vert - \vert Y\vert\big\vert^2\big\rangle\textcolor{white}{\big\Vert_{\Ell^2}^2}\\
&\qquad\qquad \leq \big\langle \mu\,\big\vert\,\vert X-Y\vert^2\big\rangle\textcolor{white}{\big\Vert_{\Ell^2}^2}\\
&\qquad\qquad \leq \rho'\,\Ch\big(\vert X-Y\vert\big) + \alpha'\,\big\Vert X-Y\big\Vert_{\Ell^2(T\mms)}^2\\
&\qquad\qquad \leq \rho'\,\Ch_\cov(X-Y) +\alpha'\,\big\Vert X-Y\big\Vert_{\Ell^2(T\mms)}^2
\end{align*}
by \autoref{Le:Form boundedness} and again thanks to \autoref{Le:Kato inequality}, the claim readily follows.
\end{proof}

\begin{remark} We should not expect the function in \autoref{Cor:Kappa cont} to admit a continuous extension to $\Ell^2(T\mms)$ in general. The reason is once again the case of compact Riemannian manifolds $\mms$ with boundary and, say, with nonnegative Ricci curvature. In this case, $\kappa := \mathcall{l}\,\surf\in\Kato_0(\mms)$ by \cite[Lem.~2.33, Thm.~4.4]{erbar2020}, where $\mathcall{l}\colon\partial\mms \to \R$ designates the lowest eigenvalue function of the second fundamental form $\mathbb{I}$ at $\partial\mms$. The pairing $\smash{\big\langle \kappa\,\big\vert\,\vert X\vert^2\big\rangle}$ then simply does not make sense for general vector fields in $\Ell^2(T\mms)$, which are only defined up to $\vol$-negligible sets (unlike the $H^{1,2}$-case, where the well-definedness of $\smash{\big\langle \kappa\,\big\vert\,\vert X\vert^2\big\rangle}$ comes from the trace theorem).
\end{remark}

\begin{lemma}[Triviality of the torsion tensor]\label{Le:Torsion free} Suppose that $f\in \Dom_\reg(\Hess)$ and $X,Y\in W^{1,2}(T\mms)$. Then $\langle X,\nabla f\rangle,\langle Y,\nabla f\rangle\in \calG$ and
\begin{align*}
\big\langle X,\nabla \langle Y,\nabla f\rangle\big\rangle - \big\langle Y,\nabla \langle X,\nabla f\rangle\big\rangle = \rmd f(\nabla_X Y - \nabla_Y X)\quad\meas\text{-a.e.}
\end{align*}
\end{lemma}

\begin{proof} \autoref{Pr:Compatibility} shows that $\langle Y,\nabla f\rangle,\langle X,\nabla f\rangle \in \calG$ with
\begin{align*}
\big\langle X,\nabla \langle Y,\nabla f\rangle\big\rangle &= \nabla Y : (X\otimes \nabla f)+\Hess f(X,Y)\\
&= \rmd f(\nabla_XY) + \Hess f(X,Y)\quad\meas\text{-a.e.},\\
\big\langle Y,\nabla \langle X,\nabla f\rangle\big\rangle &= \nabla X : (Y\otimes \nabla f)+\Hess f(Y,X)\\
&= \rmd f(\nabla_YX) + \Hess f(X,Y)\quad\meas\text{-a.e.}
\end{align*}
In the last step, we used the symmetry of $\Hess f$ from \autoref{Th:Hess properties}. Subtracting the two previous identities gives the assertion.
\end{proof}

\begin{remark} In the setting of \autoref{Le:Torsion free}, a more familiar way --- compared to classical Riemannian geometry --- of writing the stated identity is
\begin{align*}
X(Y\,f) - Y(X\,f) = \rmd f(\nabla_X Y - \nabla_Y X)\quad\meas\text{-a.e.}
\end{align*}
when defining $Xf := \langle X,\nabla f\rangle$, $Y f := \langle Y,\nabla f\rangle$, and accordingly $X(Y\,f)$ and $Y(X\,f)$.
\end{remark}

Since $\rmd\,\F$ generates $\Ell^2(T^*\mms)$ by \autoref{Pr:Generators cotangent module}, it follows by \autoref{Le:Mollified heat flow} that $\smash{\rmd\,\Dom_\reg(\Hess)}$ generates $\smash{\Ell^2(T^*\mms)}$ as well, in the sense of $\Ell^\infty$-modules. Hence, by \autoref{Le:Torsion free}, given $X,Y\in W^{1,2}(T\mms)$, $\nabla_X Y - \nabla_YX$ is the unique vector field $Z\in\Ell^1(T\mms)$ such that
\begin{align*}
X(Y\,f) - Y(X\,f) = \rmd f(Z)\quad\meas\text{-a.e.}
\end{align*}
for every $f\in \Dom_\reg(\Hess)$. This motivates the following definition.

\begin{definition}\label{Def:Lie bracket} The \emph{Lie bracket} $[X,Y]\in \Ell^1(T\mms)$ of two given vector fields $X,Y\in W^{1,2}(T\mms)$ is defined by
\begin{align*}
[X,Y] := \nabla_X Y-\nabla_Y X.
\end{align*}
\end{definition}

We terminate with the following locality property. It directly follows from the second part of \autoref{Pr:Compatibility} as well as \autoref{Pr:H11 calculus rules}.

\begin{lemma}[Locality of the covariant derivative] If $X,Y\in H^{1,2}(T\mms)$, then
\begin{align*}
\One_{\{X=Y\}}\,\nabla X = \One_{\{X=Y\}}\,\nabla Y.
\end{align*}
\end{lemma}

\subsection{Heat flow on vector fields}\label{Sub:Heat flow vector fields} Now we introduce and study the canonical heat flow $(\CHeat_t)_{t\geq 0}$ on $\Ell^2$-vector fields. First, after defining its generator  in \autoref{Def:Bochner Laplacian}, following well-known lines \cite{brezis1973} we collect elementary properties of the flow $(\CHeat_t)_{t\geq 0}$ in \autoref{Th:CHeat properties}. Then we prove the important semigroup comparison  \autoref{Th:HSU Bochner} between $(\CHeat_t)_{t\geq 0}$ and $(\ChHeat_t)_{t\geq 0}$, using  \autoref{Le:Kato inequality}.

\subsubsection{Bochner Laplacian} In fact, we have a preliminary choice to make, i.e.~either to define the Bochner Laplacian $\Bochner$ on $W^{1,2}(T\mms)$ or on the \emph{strictly smaller} space $H^{1,2}(T\mms)$. We choose the latter one since the calculus rules from \autoref{Sub:Calculus rules} are more powerful, in particular in view of \autoref{Pr:BE vector fields} below. Also, no ambiguity occurs for the background boundary conditions, see \autoref{Ex:Bochner smooth}.

\begin{definition}\label{Def:Bochner Laplacian} We define $\Dom(\Bochner)$ to consist of all $X\in H^{1,2}(T\mms)$ for which there exists $Z\in \Ell^2(T\mms)$ such that for every $Y\in H^{1,2}(T\mms)$,
\begin{align*}
\int_\mms \langle Y,Z\rangle\d\meas = -\int_\mms \nabla Y:\nabla X\d\meas.
\end{align*}
In case of existence, $Z$ is uniquely determined, denoted by $\Bochner X$ and termed the \emph{Bochner Laplacian} \textnormal{(}or \emph{connection Laplacian} or \emph{horizontal Laplacian}\textnormal{)} of $X$.
\end{definition}

Observe that $\Dom(\Bochner)$ is a vector space, and that $\Bochner\colon \Dom(\Bochner)\to\Ell^2(T\mms)$ is a linear operator. Both are easy to see from the linearity of the covariant derivative.

We modify the functional from \eqref{Eq:E cov W} with domain $W^{1,2}(T\mms)$ by introducing the ``augmented'' covariant energy functional $\smash{\widetilde{\Ch}_\cov\colon \Ell^2(T\mms)\to [0,\infty]}$ with
\begin{align*}
\widetilde{\Ch}_\cov(X) := \begin{cases} \displaystyle \int_\mms \big\vert\nabla X\big\vert_\HS^2\d\meas & \text{if }X\in H^{1,2}(T\mms),\\
\infty & \text{otherwise}.
\end{cases}
\end{align*}
Clearly, its (non-relabeled) polarization $\smash{\widetilde{\Ch}_\cov\colon H^{1,2}(T\mms)^2\to \R}$ is a closed, symmetric form, and $\Bochner$ is the nonpositive, self-adjoint generator uniquely associated to it according to  \cite[Thm.~1.3.1]{fukushima2011}. 

A first elementary consequence of this discussion,  \autoref{Le:Kato inequality} and Rayleigh's theorem is the following inequality between the spectral bottoms of $\Delta$ and $\Bochner$. (Recall that, since $-\Delta$ is nonnegative and symmetric, the spectrum $\sigma(-\Delta)$ of $-\Delta$ is the set of all $\lambda\geq 0$ such that the operator $-\Delta-\lambda$ fails to be bijective; the spectrum $\sigma(-\Bochner)$ of $-\Bochner$ is defined analogously.)

\begin{corollary}\label{Cor:Spectra cov} We have
	\begin{align*}
		\inf \sigma(-\Delta) \leq \inf \sigma(-\Bochner).
	\end{align*}
\end{corollary}

\begin{example}\label{Ex:Bochner smooth} Let $\mms$ be a Riemannian manifold with boundary. Recall that every element of $H^{1,2}(T\mms)$ has $\surf$-a.e.~vanishing normal component at $\partial\mms$ by the local  version of \autoref{Pr:Trace thm}. In particular, $\Bochner$ coincides with the self-adjoint realization in $\Ell^2(T\mms)$ of the restriction of the usual (non-relabeled) Bochner Laplacian $\Bochner$ to the class of compactly supported elements $X\in\Gamma(T\mms)$ satisfying the following mixed boundary conditions on $\partial\mms$, see \eqref{Eq:Normal parts vfs smooth world}:
\begin{align}\label{Eq:Bdry cond Bochner}
\begin{split}
X^\perp &= 0,\\
(\nabla_\sfn X)^\Vert &=0.
\end{split}
\end{align}
Indeed, for \emph{any} compactly supported $X,Y\in \Gamma(T\mms)$, 
\begin{align*}
&\int_\mms \langle\Bochner X,Y\rangle\d\vol - \int_\mms \langle X,\Bochner Y\rangle\d\vol\\
&\qquad\qquad = \int_{\partial\mms}\langle\nabla_\sfn X,Y\rangle\d\surf - \int_{\partial\mms} \langle X,\nabla_\sfn Y\rangle\d\surf
\end{align*}
according to the computations carried out in \cite[Ch.~2]{charalambous2010}. The last two integrals vanish under \eqref{Eq:Bdry cond Bochner}. Moreover, as remarked in \cite{charalambous2010} this suffices to recover the defining integration by parts formula from \autoref{Def:Bochner Laplacian}.

Let us remark for completeness that in \cite{charalambous2010}, the boundary conditions that are \emph{dual} to \eqref{Eq:Bdry cond Bochner} have been considered. See also \cite[Prop.~1.2.6]{schwarz1995}.
\end{example}

\subsubsection{Heat flow and its elementary properties}\label{Sub:HF vector fields} Analogously to \autoref{Subsub:Neumann heat flow}, we may and will define the \emph{heat flow} on vector fields as the semigroup $(\CHeat_t)_{t\geq 0}$ of bounded, linear and self-adjoint operators on $\Ell^2(T\mms)$ by
\begin{align*}
\CHeat_t := \rme^{\Bochner t}.
\end{align*}

Following e.g.~\cite{brezis1973} or \cite[Subsec.~3.4.4]{gigli2018}, the subsequent elementary properties of $(\CHeat_t)_{t\geq 0}$ are readily established.

\begin{theorem}\label{Th:CHeat properties} The following properties of $(\CHeat_t)_{t\geq 0}$ hold for every $X\in \Ell^2(T\mms)$ and every $t>0$.
\begin{enumerate}[label=\textnormal{\textcolor{black}{(}\roman*\textcolor{black}{)}}]
\item The curve $t\mapsto\CHeat_t X$ belongs to $\Cont^1((0,\infty);\Ell^2(T\mms))$ with
\begin{align*}
\frac{\rmd}{\rmd t}\CHeat_tX = \Bochner\CHeat_tX.
\end{align*}
\item If $X\in \Dom(\Bochner)$, we have
\begin{align*}
\frac{\rmd}{\rmd t}\CHeat_tX = \CHeat_t\Bochner X.
\end{align*}
In particular, we have the identity 
\begin{align*}
\Bochner\,\CHeat_t = \CHeat_t\,\Bochner\quad\text{on }\Dom(\Bochner).
\end{align*}
\item For every $s\in [0,t]$,
\begin{align*}
\Vert \CHeat_t X\Vert_{\Ell^2(T\mms)}\leq \Vert\CHeat_s X\Vert_{\Ell^2(T\mms)}.
\end{align*}
\item The function $t\mapsto \widetilde{\Ch}_\cov(\CHeat_tX)$ belongs to $\Cont^1((0,\infty))$, is nonincreasing, and its derivative satisfies
\begin{align*}
\frac{\rmd}{\rmd t}\widetilde{\Ch}_\cov(\CHeat_tX) = -2\int_\mms\big\vert\Bochner \CHeat_t X\big\vert^2\d\meas.
\end{align*}
\item If $X\in H^{1,2}(T\mms)$, the map $t\mapsto \CHeat_tX$ is continuous on $[0,\infty)$ w.r.t.~strong convergence in $H^{1,2}(T\mms)$.
\item We have
\begin{align*}
\widetilde{\Ch}_\cov(\CHeat_tX) &\leq \frac{1}{2t}\,\big\Vert X\big\Vert_{\Ell^2(T\mms)}^2,\\
\big\Vert\Bochner\CHeat_tX\big\Vert_{\Ell^2(T\mms)}^2 &\leq \frac{1}{2t^2}\,\big\Vert X\big\Vert_{\Ell^2(T\mms)}^2.
\end{align*}
\end{enumerate}
\end{theorem}

\subsubsection{Functional inequalities and $\Ell^p$-properties} The calculus rules from \autoref{Sub:Calculus rules} allow us to derive useful functional inequalities of $(\CHeat_t)_{t\geq 0}$ w.r.t.~$(\ChHeat_t)_{t\geq 0}$. The main result, essentially coming from \autoref{Pr:Compatibility} and \autoref{Le:Kato inequality}, is the $\Ell^1$-estimate from \autoref{Th:HSU Bochner}.  $\Ell^p$-consequences of it, $p\in [1,\infty]$, are stated in \autoref{Cor:Lp props Bochner}.

In fact, the latter requires the following $\Ell^2$-version of it \emph{in advance} for technical reasons, see \autoref{Re:Technical remark} --- in particular, \autoref{Pr:BE vector fields} does \emph{not} follow from \autoref{Th:HSU Bochner} just by Jensen's inequality for $(\ChHeat_t)_{t\geq 0}$.

\begin{proposition}\label{Pr:BE vector fields} For every $X\in \Ell^2(T\mms)$ and every $t\geq 0$,
\begin{align*}
\vert\CHeat_tX\vert^2\leq \ChHeat_t\big(\vert X\vert^2\big)\quad\meas\text{-a.e.}
\end{align*}
\end{proposition}

\begin{proof} We only prove the nontrivial part in which $t>0$. Let $\phi\in\Test_{\Ell^\infty}(\mms)$ be nonnegative. Define the function $F\colon [0,t]\to \R$ by
\begin{align*}
F(s) := \int_\mms\phi\,\ChHeat_{t-s}\big(\vert\CHeat_s X\vert^2\big)\d\meas = \int_\mms \ChHeat_{t-s}\phi\,\vert\CHeat_sX\vert^2\d\meas.
\end{align*}
Of course, $F$ is well-defined. Note that for every $s,s'\in[0,t]$ with $s'<s$,
\begin{align}\label{Eq:NEUER BOUND}
\vert\ChHeat_{t-s}\phi - \ChHeat_{t-s'}\phi\vert = \Big\vert\!\int_{s'}^s \Delta\ChHeat_{t-r}\phi\d r\Big\vert \leq \Vert\Delta \phi\Vert_{\Ell^\infty(\mms)}\,(s-s')\quad\meas\text{-a.e.}
\end{align}
Furthermore, since $s\mapsto \CHeat_s X$ is continuous as a map from $[0,t]$ into $\Ell^2(T\mms)$ and locally absolutely continuous as a map from $(0,t]$ into $\Ell^2(T\mms)$, the $\Ell^1$-valued map $s\mapsto \vert\CHeat_s X\vert^2$ is continuous on $[0,t]$ and locally absolutely continuous on $(0,t]$. Combining this with \eqref{Eq:NEUER BOUND}, we obtain that $F$ is continuous on $[0,t]$ and locally absolutely continuous on $(0,t)$. By exchanging differentiation and integration, for $\Leb^1$-a.e.~$s\in (0,t)$ we thus get
\begin{align*}
F'(s) = -\int_\mms \Delta\ChHeat_{t-s}\phi\,\vert\CHeat_sX\vert^2 \d\meas+ 2\int_\mms\ChHeat_{t-s}f\,\langle \CHeat_sX,\Bochner\CHeat_sX\rangle\d\meas.
\end{align*}

Observe that $\ChHeat_{t-s}\phi\in\Test(\mms)$ with $\Delta\ChHeat_{t-s}\phi = \ChHeat_{t-s}\Delta\phi\in\Ell^\infty(\mms)$ as well as $\vert\CHeat_s X\vert^2\in \calG_\reg$ for every $s\in (0,t)$ by \autoref{Pr:Compatibility}. By \autoref{Le:IBP for W11} and \autoref{Re:fX in H12}, for $\Leb^1$-a.e.~$s\in (0,t)$ we have
\begin{align}\label{Eq:Ulk}
F'(s) &= \int_\mms \big\langle\nabla\ChHeat_{t-s}\phi,\nabla\vert\CHeat_sX\vert^2\big\rangle\d\meas - 2\int_\mms\big\langle\nabla (\ChHeat_{t-s}\phi\,\CHeat_sX),\nabla\CHeat_sX\big\rangle\d\meas\\
&= \int_\mms \big\langle\nabla\ChHeat_{t-s}\phi,\nabla\vert\CHeat_sX\vert^2\big\rangle\d\meas - 2\int_\mms[\nabla\ChHeat_{t-s}\phi\otimes \CHeat_sX] : \nabla\CHeat_sX \d\meas\nonumber\\
&\qquad\qquad - 2\int_\mms \ChHeat_{t-s}\phi\,\big\vert\nabla\CHeat_sX\big\vert_\HS^2\d\meas\nonumber\\
&\leq \int_\mms  \big\langle\nabla\ChHeat_{t-s}\phi,\nabla\vert\CHeat_sX\vert^2\big\rangle\d\meas - 2\int_\mms[\nabla\ChHeat_{t-s}\phi\otimes \CHeat_sX] : \nabla\CHeat_sX\d\meas = 0.\nonumber
\end{align}
In the last equality, we used \autoref{Pr:Compatibility}. Therefore,
\begin{align*}
\int_\mms \phi\,\vert\CHeat_tX\vert^2\d\meas = F(t) \leq F(0) = \int_\mms \phi\,\ChHeat_t\big(\vert X\vert^2\big)\d\meas.
\end{align*}
This proves the claim thanks to the arbitrariness of $\phi$ by \autoref{Le:Mollified heat flow}.
\end{proof}

In applications, the following corollary of \autoref{Pr:BE vector fields} could be useful. 

\begin{corollary}\label{Cor:Bounded Lapl density}
For every $X\in\Dom(\Bochner)$, there exists a sequence $(X_n)_{n\in\N}$ in $\Dom(\Bochner)\cap\Ell^\infty(T\mms)$ which converges to $X$ in $H^{1,2}(T\mms)$ such that in addition, $\Bochner X_n \to \Bochner X$ in $\Ell^2(T\mms)$ as $n\to\infty$. If $X\in\Ell^\infty(T\mms)$ in addition, this sequence can be constructed to be bounded in $\Ell^\infty(T\mms)$.
\end{corollary}

\begin{proof} Define $\smash{X_k := \One_{\{\vert X\vert \leq k\}}\,X}\in\Ell^2(T\mms)\cap\Ell^\infty(T\mms)$, $k\in\N$, and, given any $t > 0$, consider the element $X_{t,k} := \CHeat_tX_k$ which, thanks to \autoref{Pr:BE vector fields}, belongs to $\Dom(\Bochner)\cap\Ell^\infty(T\mms)$. By \autoref{Th:CHeat properties}, we have $X_{t,k}\to \CHeat_t X$ in $H^{1,2}(T\mms)$ and $\Bochner X_{t,k} \to \Bochner \CHeat_t X$ in $\Ell^2(T\mms)$ as $k\to\infty$ for every $t>0$. Furthermore, $\CHeat_t X\to X$ in $H^{1,2}(T\mms)$ and $\Bochner\CHeat_t X = \CHeat_t\Bochner X \to \Bochner X$ in $\Ell^2(T\mms)$ as $t\to 0$ again by \autoref{Th:CHeat properties}. The claim follows by a diagonal argument.
\end{proof}

The following improvement of \autoref{Pr:BE vector fields} is an instance of the correspondence between \emph{form domination} and \emph{semigroup domination}  \cite{hess1977,ouhabaz1999,shigekawa1997,simon1977}. It extends analogous results for Riemannian manifolds without boundary \cite{hess1977,hess1980}. 

\begin{theorem}\label{Th:HSU Bochner} For every $X\in \Ell^2(T\mms)$ and every $t\geq 0$,
\begin{align*}
\vert\CHeat_t X\vert \leq \ChHeat_t\vert X\vert\quad\meas\text{-a.e.}
\end{align*}
\end{theorem}

\begin{proof} Again, we restrict ourselves to $t>0$. By the $\Ell^2$-continuity of both sides of the claimed inequality in $X$, it is sufficient to prove the latter for $X\in \Test(\mms)$. Given any $\varepsilon > 0$, define the function $\varphi_\varepsilon\in \Cont^\infty([0,\infty))\cap\Lip([0,\infty))$ by $\varphi_\varepsilon(r) := (r+\varepsilon)^{1/2}-\varepsilon^{1/2}$. Moreover, let $\phi\in\Test(\mms)$ be nonnegative with $\Delta\phi\in\Ell^\infty(\mms)$. As in the proof of \autoref{Pr:BE vector fields}, one argues that the function $F_\varepsilon\colon [0,t]\to \R$ with
\begin{align*}
F_\varepsilon(s) := \int_\mms \phi\,\ChHeat_{t-s}\big(\varphi_\varepsilon\circ \vert \CHeat_s X\vert^2\big)\d\meas = \int_\mms \ChHeat_{t-s}\phi\,\big[\varphi_\varepsilon\circ \vert\CHeat_s X\vert^2\big]\d\meas
\end{align*}
is continuous on $[0,t]$, locally absolutely continuous on $(0,t)$, and in differentia\-ting it, integration and differentiation can be switched at $\Leb^1$-a.e.~$s\in (0,t)$, yielding
\begin{align*}
F_\varepsilon'(s) &= -\int_\mms \Delta\ChHeat_{t-s}\phi\,\big[\varphi_{\varepsilon}\circ \vert \CHeat_s X\vert^2\big]\d\meas\\
&\qquad\qquad + 2 \int_\mms\ChHeat_{t-s}\phi\,\big[\varphi_\varepsilon'\circ \vert\CHeat_sX\vert^2\big]\,\langle \CHeat_sX,\Bochner\CHeat_sX\rangle\d\meas.
\end{align*}

By \autoref{Pr:BE vector fields}, we have $\CHeat_sX\in \Ell^\infty(T\mms)$  and hence $\vert\CHeat_sX\vert^2\in \F$ for every $s\in (0,t)$ by \autoref{Le:Kato inequality}. In particular $\smash{\varphi_\varepsilon\circ \vert\CHeat_s X\vert^2,\varphi_\varepsilon'\circ\vert\CHeat_s X\vert^2\in \F} \cap\Ell^\infty(\mms)$, and hence by \autoref{Cor:Calculus rules d} and \autoref{Re:fX in H12},
\begin{align*}
F_\varepsilon'(s) &= \int_\mms \big\langle\Delta\ChHeat_{t-s}\phi, \nabla\big[\varphi_{\varepsilon}\circ\vert \CHeat_s X\vert^2\big]\big\rangle\d\meas\\
&\qquad\qquad - 2 \int_\mms \nabla\big[\ChHeat_{t-s}\phi\,\big[\varphi_\varepsilon'\circ\vert\CHeat_sX\vert^2\big]\,\CHeat_sX\big]:\nabla \CHeat_sX\d\meas\\
&= \int_\mms \big[\varphi_\varepsilon'\circ \vert\CHeat_sX\vert^2\big]\,\big\langle \nabla\ChHeat_{t-s}\phi,\nabla\vert\CHeat_sX\vert^2\big\rangle\d\meas\\
&\qquad\qquad -2\int_\mms \ChHeat_{t-s}\phi\,\big[\varphi_\varepsilon''\circ\vert\CHeat_s X\vert^2\big]\,\nabla \vert\CHeat_s X\vert^2 \otimes \CHeat_sX : \nabla\CHeat_sX\d\meas\\
&\qquad\qquad - 2\int_\mms \big[\varphi_\varepsilon'\circ\vert\CHeat_sX\vert^2\big]\,\nabla\ChHeat_{t-s}\phi\otimes\CHeat_sX : \nabla\CHeat_sX\d\meas\\
&\qquad\qquad - 2\int_\mms\ChHeat_{t-s}\phi\,\big[\varphi_\varepsilon'\circ\vert\CHeat_sX\vert^2\big]\,\big\vert\nabla\CHeat_sX\big\vert_\HS^2\d\meas\\
 &= -2\int_\mms \ChHeat_{t-s}\phi\,\big[\varphi_\varepsilon''\circ\vert\CHeat_s X\vert^2\big]\,\nabla \vert\CHeat_s X\vert^2 \otimes \CHeat_sX : \nabla\CHeat_sX\d\meas\\
&\qquad\qquad - 2\int_\mms\ChHeat_{t-s}\phi\,\big[\varphi_\varepsilon'\circ\vert\CHeat_sX\vert^2\big]\,\big\vert\nabla\CHeat_sX\big\vert_\HS^2\d\meas.
\end{align*}
In the last step, we used \autoref{Pr:Compatibility} to cancel out two integrals. Lastly, one easily verifies that $-2r\,\varphi_\varepsilon''(r) \leq \varphi_\varepsilon'(r)$ for every $r\geq 0$, and that $-\varphi_\varepsilon''$ is nonnegative. Taking \autoref{Le:Kato inequality} into account, we thus get
\begin{align*}
&-2\,\big[\varphi_\varepsilon''\circ\vert\CHeat_sX\vert^2\big]\,\nabla\vert\CHeat_sX\vert^2\otimes\CHeat_sX : \nabla\CHeat_sX\\
&\qquad\qquad = - 4\,\big[\varphi_\varepsilon''\circ\vert\CHeat_sX\vert^2\big]\,\vert\CHeat_s X\vert\,\nabla\vert \CHeat_sX\vert\otimes\CHeat_sX : \nabla \CHeat_sX\!\textcolor{white}{\big\vert^2}\\
&\qquad\qquad \leq - 4\,\big[\varphi_\varepsilon''\circ\vert\CHeat_sX\vert^2\big]\,\vert\CHeat_sX\vert^2\,\big\vert\nabla \vert\CHeat_sX\vert\big\vert\,\vert\nabla \CHeat_sX\vert_\HS\textcolor{white}{\big\vert^2}\\
&\qquad\qquad \leq 2\,\big[\varphi_\varepsilon'\circ\vert\CHeat_sX\vert^2\big]\,\big\vert\nabla\CHeat_sX\big\vert_\HS^2\quad\meas\text{-a.e.}
\end{align*}
This shows that $F'(s) \leq 0$ for $\Leb^1$-a.e.~$s\in (0,t)$, whence
\begin{align*}
\int_\mms\phi\,\big[\varphi_\varepsilon\circ \vert\CHeat_tX\vert^2\big]\d\meas = F_\varepsilon(t)\leq F_\varepsilon(0) = \int_\mms\phi\,\ChHeat_t(\varphi_\varepsilon\circ \vert X\vert^2)\d\meas
\end{align*}
for every $\varepsilon > 0$. Sending $\varepsilon \to 0$ with the aid of Lebesgue's theorem and using the arbitrariness of $\phi$ via \autoref{Le:Mollified heat flow} gives the desired assertion.
\end{proof}

Note that the only essential tool to prove \autoref{Pr:BE vector fields} and \autoref{Th:HSU Bochner} is the metric compatibility of $\nabla$ from \autoref{Pr:Compatibility}. In particular, no curvature shows up in both statements.

\begin{remark}\label{Re:Technical remark} In the notation of the proof of \autoref{Th:HSU Bochner},  \autoref{Le:Kato inequality} only guarantees that $\vert \CHeat_s X\vert \in \F$, $s\in (0,t)$. However, the required regularity $\vert \CHeat_s X\vert^2 \in \F$  is unclear without any a priori  information about $\Ell^\infty$-$\Ell^\infty$-regularizing properties of $(\CHeat_t)_{t\geq 0}$, which is precisely provided by \autoref{Pr:BE vector fields}. In turn, the proof of the latter only needs $\calG_\reg$-regularity of $\vert \CHeat_sX\vert^2$, which is true for any $X\in \Ell^2(T\mms)$ by \autoref{Pr:Compatibility}. To integrate by parts in \eqref{Eq:Ulk}, this missing $\F$-regularity is compensated by \autoref{Le:IBP for W11}, which is one key feature of the space $\calG_\reg$ (recall \autoref{Re:No W11 in general} as well).
\end{remark}

\begin{remark}[Heat kernel] Following the arguments for \cite[Thm.~6.5]{braun2020}, we deduce from \autoref{Th:HSU Bochner} that on any $\RCD(K,\infty)$ space, $K\in\R$ --- in fact, on any (tamed) Dirichlet space where $(\ChHeat_t)_{t\geq 0}$ admits a heat kernel with Gaussian upper bounds as in (4.1) in  \cite{tamanini2019}, see also \cite[Sec.~6.1]{ambrosio2014b} --- $(\CHeat_t)_{t\geq 0}$ has a heat kernel in the sense of \cite[Ch.~6]{braun2020}. The pointwise operator norm of the  latter is $\smash{\meas^{\otimes 2}}$-a.e.~no larger than the heat kernel of $(\ChHeat_t)_{t\geq 0}$, compare with  \cite[Thm.~6.7]{braun2020}.
\end{remark}

\begin{corollary}\label{Cor:Lp props Bochner}
	The heat flow $(\CHeat_t)_{t\geq 0}$ uniquely extends to a semigroup of bounded linear operators on $L^p(T\mms)$ for every $p \in [1,\infty]$ such that, for every $X\in\Ell^p(T\mms)$ and every $t\geq 0$,
	\begin{align*}
	\vert\CHeat_t X\vert^p \leq \ChHeat_t\big(\vert X\vert^p\big)\quad\meas\text{-a.e.},
	\end{align*}
	and in particular	
	\begin{align*}
		\Vert\CHeat_t\Vert_{\Ell^p(T\mms),\Ell^p(T\mms)} \leq 1.
	\end{align*}
	It is strongly continuous on $L^p(T\mms)$ if $p<\infty$ and weakly$^*$ continuous on $\Ell^\infty(T\mms)$.
\end{corollary}

\subsection{Bits of tensor calculus} In this section, we shortly outline basic elements of general nonsmooth tensor calculus. Fix $r,s\in\N_0$ throughout.

\subsubsection{Tensor fields}\label{Sub:Tensor fields} Define  the space of \emph{\textnormal{(}$\Ell^2$-\textnormal{)}tensor fields} of type $(r,s)$ over $\mms$ by\label{Not:L2 tfs rs}
\begin{align*}
\Ell^2(T_s^r\mms) := \Ell^2((T^*)^{\otimes r}\mms) \otimes \Ell^2(T^{\otimes s}\mms),
\end{align*}
where all tensor products are intended in the sense of \autoref{Sub:Tensor products}.  For $s=0$ or $r=0$, we employ the consistent interpretations
\begin{align*}
\Ell^2(T_0^r\mms) &:= \Ell^2((T^*)^{\otimes r}\mms),\\
\Ell^2(T_s^0\mms) &:= \Ell^2(T^{\otimes s}\mms),\\
\Ell^2(T_0^0\mms) &:= \Ell^2(\mms).
\end{align*}
The $\Ell^0$-normed module induced by $\Ell^2(T_s^r\mms)$  as in \autoref{Sub:L0 modules} is termed $\Ell^0(T_s^r\mms)$.  Given any $T\in \Ell^0(T_s^r\mms)$, $\omega_1,\dots,\omega_r\in\Ell^0(T^*\mms)$ as well as $X_1,\dots,X_s\in\Ell^0(T\mms)$, the pointwise scalar pro\-duct of $T$ and the element $\omega_1\otimes\dots\otimes\omega_r\otimes X_1\otimes\dots\otimes X_s\in \Ell^0(T_s^r\mms)$  is shortly written $T(\omega_1,\dots,\omega_r,X_1,\dots,X_s)\in\Ell^0(\mms)$. (This section is the only place in our work where this bracket notation, strictly speaking, does not mean pointwise duality pairings, but rather pointwise pairings of elements of the same vector  space.)

\subsubsection{Tensorial covariant derivative}
We first introduce the concept of covariant derivative of suitable $\smash{T\in\Ell^2(T_s^r\mms)}$. We start again with a motivating example. 

\begin{example}\label{Ex:Tfs} Suppose that $\mms$ is a Riemannian manifold with boundary, and let $\smash{T\in \Gamma((T^*)^{\otimes r}\mms\otimes T^{\otimes s}\mms)}$ be an $(r,s)$-tensor field over $\mms$. Then, see e.g.~\cite[Ch.~8]{lee2018}, $\smash{\nabla T\in \Gamma((T^*)^{\otimes r}\mms \otimes T^{\otimes(s+1)}\mms)}$ is the unique $(r,s+1)$-tensor field such that for every $\eta_1,\dots,\eta_r\in\Gamma_\rmc(T^*\mms)$ and every $Z,Y_1,\dots,Y_s\in\Gamma_\rmc(T\mms)$, 
\begin{align*}
\nabla T(\eta_1,\dots,\eta_r,Z,Y_1,\dots,Y_s) &= \rmd\big[T(\eta_1,\dots,\eta_r,Y_1,\dots,Y_s)\big](Z)\\
&\qquad\qquad - \sum_{i=1}^r T(\nabla_Z \eta_i) - \sum_{j=1}^s T(\nabla_ZY_j)
\end{align*}
on $\mms$. Here, we used the shorthand notations
\begin{align}\label{Eq:Short}
\begin{split}
T(\nabla_Z\eta_i) &:= T(\eta_1,\dots,\underbrace{(\nabla_Z\eta_i^\sharp)^\flat}_{i\text{-th slot}},\dots,\eta_r,Y_1,\dots,Y_s),\\
T(\nabla_ZY_j) &:= T(\eta_1,\dots,\eta_r,Y_1,\dots,\!\!\!\!\!\!\underbrace{\nabla_ZY_j}_{(r+j)\text{-th slot}}\!\!\!\!\!\!,\dots,Y_s)
\end{split}
\end{align}
for $i\in\{1,\dots,r\}$ and $j\in\{1,\dots,s\}$. As for the ordinary covariant derivative, see \autoref{Ex:Smooth cov der motiv} and \eqref{Eq:Dir der def} and also \autoref{Re:Dir der tf} below, one thinks of $\nabla T(\cdot, Z,\cdot)$ as directional derivative $\nabla_Z T(\cdot,\cdot)$.  

As in \autoref{Ex:Hess} and \autoref{Ex:Smooth cov der motiv},  $\nabla T$ is still uniquely determined by the above identity when requiring the latter only for $\eta_1,\dots,\eta_r$ and $Z, Y_1,\dots,Y_s$ with vanishing normal parts at $\partial\mms$. Hence integration and integration by parts give
\begin{align*}
&\int_\mms \nabla T(\eta_1,\dots,\eta_r, Z, Y_1,\dots,Y_s)\d\vol\\
&\qquad\qquad = - \int_\mms T(\eta_1,\dots,\eta_r,Y_1,\dots,Y_s)\div Z\d\vol\\
&\qquad\qquad \qquad\qquad -\int_\mms \sum_{i=1}^r T(\nabla_Z\eta_i)\d\vol - \int_\mms\sum_{j=1}^s T(\nabla_ZY_j)\d\vol
\end{align*}
by \autoref{Ex:Mflds with boundary}. Then $\nabla T$ is still uniquely defined by this identity.
\end{example}

Note that all relevant objects in \autoref{Ex:Tfs}, in particular the covariant derivative of \emph{vector fields} and directional derivatives, see  \autoref{Sub:Cov der 1} and \autoref{Sub:Calculus rules}, have already been made sense of in our nonsmooth framework. Hence they can be used to define the covariant derivative $\nabla T$ for appropriate $\smash{T\in \Ell^2(T_s^r\mms)}$. Indeed, the r.h.s.~of the above integral identity --- with $\vol$ replaced by $\meas$ --- makes sense for arbitrary such $T$, for $\eta_1,\dots,\eta_r\in \Test(T^*\mms)$ and $Z, Y_1,\dots,Y_s\in\Test(T\mms)$. Clearly $T(\eta_1,\dots,\eta_r,Y_1,\dots,Y_s)\in\Ell^2(\mms)$ and $Z\in\Dom_\TV(\DIV)\cap\Dom(\div)$ with $\div Z\in\Ell^2(\mms)$ and $\norm Z=0$, which shows the well-definedness of the first integral. For the second, note that $\smash{(\nabla_Z\eta_i^\sharp)^\flat\in\Ell^2(T^*\mms)}$, $i\in \{1,\dots,r\}$, by \eqref{Eq:Bound nabla Z X}, which directly yields $\smash{T(\nabla_Z\eta_i)\in\Ell^2(\mms)}$. The third integral is discussed analogously. 

This leads to the subsequent definition. In the sequel, we retain the shorthand notations from \eqref{Eq:Short}.

\begin{definition}\label{Def:Tensor cov der} We define the space $W^{1,2}(T_s^r\mms)$ to consist of all $T\in \Ell^2(T_s^r\mms)$ for which there exists $A\in \Ell^2(T_{s+1}^r\mms)$ such that for every $\eta_1,\dots,\eta_r\in\Test(T^*\mms)$ and every $Z,Y_1,\dots,Y_s\in\Test(T\mms)$,
\begin{align*}
&\int_\mms A(\eta_1,\dots,\eta_r,Z,Y_1,\dots,Y_s)\d\meas\\
&\qquad\qquad = -\int_\mms T(\eta_1,\dots,\eta_r,Y_1,\dots,Y_s)\div Z\d\meas\\
&\qquad\qquad\qquad\qquad - \int_\mms \sum_{i=1}^r T(\nabla_Z\eta_i)\d\meas - \int_\mms \sum_{j=1}^s T(\nabla_ZY_j)\d\meas.
\end{align*}
In case of existence, the element $A$ is unique, denoted by $\nabla T$ and termed the \emph{covariant derivative} of $T$.
\end{definition}

The uniqueness follows by density of an appropriate class in $\Ell^2(T_s^r\mms)$, see \eqref{Eq:Regular tf} below and \autoref{Sub:Tensor products} for details. Clearly, $\smash{W^{1,2}(T_s^r\mms)}$ is thus a vector space, and $\nabla$ is a linear operator on it. Further properties  of this covariant derivative are summarized in \autoref{Th:Tensor field props} below. Thanks to \autoref{Th:Properties W12 TM} and \eqref{Eq:Dir der def}, we have $\smash{W^{1,2}(T_1^0\mms)= W^{1,2}(T\mms)}$, and on these spaces, the notions of covariant  derivative from \autoref{Def:Cov der} and \autoref{Def:Tensor cov der} coincide. Moreover, $\smash{W^{1,2}(T_0^1\mms)}$ coincides with the image of $W^{1,2}(T\mms)$ under $\flat$. Lastly, we have $\smash{\F\subset W^{1,2}(T_0^0\mms)}$, but in general it seems hard to verify equality. See also \autoref{Re:Wd120 in W12} below.

We endow $W^{1,2}(T_s^r\mms)$ with the norm $\smash{\Vert\cdot\Vert_{W^{1,2}(T_s^r\mms)}}$ given by
\begin{align*}
\big\Vert T\big\Vert_{W^{1,2}(T_s^r\mms)}^2 := \big\Vert T\big\Vert_{\Ell^2(T_s^r\mms)}^2 + \big\Vert\nabla T\big\Vert_{\Ell^2(T_{s+1}^r\mms)}^2.
\end{align*}
We introduce the functional $\smash{\Ch_s^r\colon \Ell^2(T_s^r\mms)\to [0,\infty]}$ given by
\begin{align*}
\Ch_s^r(T) := \begin{cases} \displaystyle\int_\mms \vert \nabla T\vert^2\d\meas & \text{if }T\in W^{1,2}(T_s^r\mms),\\
\infty & \text{otherwise}.
\end{cases}
\end{align*}

The proof of the subsequent theorem follows completely similar lines as in \autoref{Th:Hess properties} and \autoref{Th:Properties W12 TM}. We leave the details to the reader.

\begin{theorem}\label{Th:Tensor field props} The space $\smash{W^{1,2}(T_s^r\mms)}$, the covariant derivative $\nabla$ and the functional $\smash{\Ch_s^r}$ have the following properties.
\begin{enumerate}[label=\textnormal{\textcolor{black}{(}\roman*\textcolor{black}{)}}]
\item\label{La:Item1} $\smash{W^{1,2}(T_s^r\mms)}$ is a separable Hilbert space w.r.t.~$\smash{\Vert\cdot\Vert_{W^{1,2}(T_s^r\mms)}}$.
\item\label{La:Item2} The covariant derivative $\nabla$ is a closed operator. That is, the image of the map $\smash{\Id\times \nabla\colon W^{1,2}(T_s^r\mms)\to\Ell^2(T_s^r\mms)\times \Ell^2(T_{s+1}^r\mms)}$ is a closed subspace of $\smash{\Ell^2(T_s^r\mms)\times \Ell^2(T_{s+1}^r\mms)}$.
\item\label{La:Item3} The functional $\smash{\Ch_s^r}$ is $\Ell^2$-lower semicontinuous, and every $T\in \Ell^2(T_s^r\mms)$ obeys the duality formula
\begin{align*}
\Ch_s^r(T) &= \sup\!\Big\lbrace\! -\! 2\sum_{k=1}^n\int_\mms T(\eta_1^k,\dots,\eta_r^k,Y_1^k,\dots,Y_s^k)\div Y_0^k\d\meas \\
&\qquad\qquad - 2\sum_{k=1}^n\int_\mms \sum_{i=1}^r T(\nabla_{Y_0^k}\eta_i^k)\d\meas - 2\sum_{k=1}^n\int_\mms \sum_{j=1}^s T(\nabla_{Y_0^k}Y_j^k)\d\meas\\
&\qquad\qquad - \int_\mms \Big\vert\!\sum_{k=1}^n \eta_1^k\otimes \dots\otimes\eta_r^k \otimes Y_0^k\otimes Y_1^k\otimes\dots\otimes Y_s^k\Big\vert^2\d\meas :\\
&\qquad\qquad\qquad\qquad n\in\N,\ \eta_i^k\in \Test(T^*\mms),\ Y_j^k\in \Test(T\mms)\Big\rbrace.\textcolor{white}{\sum_j^n}
\end{align*}
\end{enumerate}
\end{theorem}

\subsubsection{Tensor algebra and Leibniz rule} Yet, unless $r\in\{0,1\}$ and $s=0$ or $r = 0$ and $s\in\{0,1\}$ we do not know whether $W^{1,2}(T_s^r\mms)$ is nontrivial. As we show in \autoref{Le:Leibniz rule reg tf}, $\smash{W^{1,2}(T_s^r\mms)}$ is in fact dense in $\smash{\Ell^2(T_s^r\mms)}$ for arbitrary $r,s\in \N_0$, for it contains the space of \emph{regular $(r,s)$-tensor fields}  given by\label{Not:Reg tfields}
\begin{align}\label{Eq:Regular tf}
\begin{split}
\Reg(T_s^r\mms) &:= \Big\lbrace\! \sum_{k=1}^n \omega_1^k\otimes\dots\otimes \omega_r^k\otimes X_1^k\otimes\dots\otimes X_s^k :\\
&\qquad\qquad n\in\N,\ \omega_i^k\in \Reg(T^*\mms),\ X_j^k\in\Reg(T\mms)\Big\rbrace.
\end{split}
\end{align}
Since $\Test(\mms)$ is both an algebra and closed under multiplication with constant functions, we consistently set $\Reg(T_0^0\mms) := \Test(\mms)$.

As expected from the smooth setting \cite[Ch.~8]{lee2018}, the crucial tool to prove the inclusion outlined above is the \emph{Leibniz rule} --- on every element of $\Reg(T_s^r\mms)$, the covariant derivative should pass through every slot. Technically, this requires to deal with the  \emph{\textnormal{(}$\Ell^2$-\textnormal{)}tensor algebra} $
\rmT_{\Ell^2}(\mms)$, since every summand of \eqref{Eq:LR} below belongs to a different $\Ell^\infty$-tensor product.

To this aim, we consider the following sequence $\smash{(\calM_n)_{n\in\N_0}}$ of all possible finite tensor products of $\Ell^2(T^*\mms)$ and $\Ell^2(T\mms)$ as in  \autoref{Sub:Tensor products}. Set
\begin{align*}
\calM_0 &:= \Ell^2(\mms),\\
\calM_1 &:= \Ell^2(T^*\mms),\\
\calM_2 &:= \Ell^2(T\mms).
\end{align*}
Inductively, if $\smash{\calM_{2^k-1},\dots, \calM_{2(2^k-1)}}$ are  defined for a given $k\in \N$ then, for $i\in \{1,\dots,2^{k+1}\}$, we set
\begin{align*}
\calM_{2(2^k-1)+i} := \begin{cases}  \calM_{2^k-1+\lfloor i/2\rfloor}\otimes \Ell^2(T^*\mms) & \text{if } i \text{ is odd},\\
\calM_{2^k-1 + \lfloor (i-1)/2\rfloor}\otimes \Ell^2(T\mms) & \text{otherwise}.
\end{cases}
\end{align*}

\begin{definition} The \emph{\textnormal{(}$\Ell^2$\textnormal{)}-tensor algebra} over $\mms$ is defined as
\begin{align*}
\rmT_{\Ell^2}(\mms) := \bigoplus_{n\in\N_0}\calM_n.
\end{align*}
\end{definition}

All module operations, e.g.~taking pointwise norms or multiplication with $\Ell^\infty$-functions, can be made sense of componentwise for elements of $\rmT_{\Ell^2}(\mms)$.

\begin{lemma}\label{Le:Leibniz rule reg tf} We have the inclusion $\Reg(T_r^s)\subset W^{1,2}(T_s^r\mms)$. More precisely, for every $\omega_1,\dots,\omega_r\in\Reg(T^*\mms)$ and every $X_1,\dots,X_s\in\Reg(T\mms)$, we have $\omega_1\otimes\dots\otimes\omega_r\otimes X_1\otimes\dots\otimes X_s\in \smash{W^{1,2}(T_s^r\mms)}$ and, as an identity in $\rmT_{\Ell^2}(\mms)$,
\begin{align}\label{Eq:LR}
\begin{split}
&\nabla(\omega_1\otimes\dots\otimes\omega_r\otimes X_1\otimes\dots\otimes X_s)\\
&\qquad\qquad = \sum_{i=1}^r \omega_1\otimes\dots\otimes\underbrace{(\nabla \omega_i^\sharp)^\flat}_{i\text{-th slot}}\otimes\dots\otimes \omega_r\otimes X_1\otimes\dots\otimes X_s\\
&\qquad\qquad\qquad\qquad + \sum_{j=1}^s \omega_1\otimes\dots\otimes \omega_r \otimes X_1\otimes\dots\otimes\!\!\!\!\!\!\! \underbrace{\nabla X_j}_{(r+j)\text{-th slot}}\!\!\!\!\!\!\!\otimes \dots\otimes X_s.
\end{split}
\end{align}
\end{lemma}

\begin{proof} We first comment on the cases $r\in \{0,1\}$ and $s=0$ or $r=0$ and $s\in\{0,1\}$ in which no formula \eqref{Eq:LR} has to be shown. The case $r=s=0$ is straightforward from the identifications and inclusions $\smash{\Reg(T_0^0\mms) = \Test(\mms)\subset \F}\subset \smash{W^{1,2}(T_0^0\mms)}$ by \autoref{Le:Div g nabla f} and the \autoref{Def:L2 div} of the $\Ell^2$-divergence. In this case, the covariant derivative and the gradient from \autoref{Def:Gradient} agree. In the cases $r=0$ and $s=1$ or $r=1$ and $s=0$, the claimed  regularity follows from \autoref{Th:Properties W12 TM}. 

For $r=s=1$, let $\omega\in \Reg(T^*\mms)$ and $X\in \Reg(T\mms)$, and let $\eta\in\Test(T^*\mms)$ and $Z,Y\in \Test(T\mms)$ be fixed. Then by \eqref{Eq:Tensor product pointwise sc prod}, \autoref{Le:Div g nabla f}, \autoref{Pr:Compatibility},  \eqref{Eq:Dir der def} and finally \eqref{Eq:Transpose}, we infer that
\begin{align*}
&-\int_\mms \omega\otimes X(\eta, Y) \div Z\d\meas\\
&\qquad\qquad\qquad\qquad - \int_\mms \omega\otimes X((\nabla_Z \eta^\sharp)^\flat, Y) \d\meas -\int_\mms \omega\otimes X(\eta,\nabla_Z Y)\d\meas\\
&\qquad\qquad = \int_\mms \rmd\big[\langle \omega,\eta\rangle\,\langle X,Y\rangle\big](Z)\d\meas\\
&\qquad\qquad\qquad\qquad - \int_\mms \big\langle\omega,(\nabla_Z\eta^\sharp)^\flat\big\rangle\,\langle X,Y\rangle\d\meas  -\int_\mms \langle\omega,\eta\rangle\,\langle X,\nabla_ZY\rangle\d\meas\\
&\qquad\qquad = \int_\mms \big\langle (\nabla_Z\omega^\sharp)^\flat, \eta\big\rangle\,\langle X,Y\rangle\d\meas + \int_\mms \big\langle\omega,(\nabla_Z\eta^\sharp)^\flat\big\rangle\,\langle X,Y\rangle\d\meas\\
&\qquad\qquad\qquad\qquad + \int_\mms\langle\omega,\eta\rangle\,\langle\nabla_ZX,Y\rangle \d\meas + \int_\mms \langle\omega,\eta\rangle\,\langle X,\nabla_ZY \rangle\d\meas\\
&\qquad\qquad\qquad\qquad -\int_\mms \big\langle \omega,(\nabla_Z\eta^\sharp)^\flat\big\rangle\,\langle X,Y\rangle\d\meas - \int_\mms \langle\omega,\eta\rangle\,\langle X,\nabla_ZY\rangle\d\meas\\
&\qquad\qquad = \int_\mms (\nabla_Z\omega^\sharp)^\flat\otimes X(\eta,Y)\d\meas + \int_\mms \omega\otimes\nabla_Z X(\eta,Y)\d\meas.
\end{align*}
By \eqref{Eq:Dir der def} and identification of the r.h.s.~with the scalar product in $\smash{\Ell^2(T_1^1\mms)}$, it follows that $\smash{\omega\otimes X\in W^{1,2}(T_1^1\mms)}$. 

The case of general $r,s\in\N_0$ are now deduced similarly by induction over $r$ or $s$ while keeping the other variable fixed, respectively.
\end{proof}

In the next final proposition, let us fix $r',s'\in\N_0$. Given any $\smash{T\in\Ell^0(T_s^r\mms)}$ and $\smash{S\in \Ell^0(T_{s'}^{r'}\mms)}$, by $\smash{T\boxtimes S}$ we mean the unique element of $\smash{\Ell^0(T_{s+s'}^{r+r'}\mms)}$ such that for every $\eta_1,\dots,\eta_{r+r'}\in\Ell^0(T^*\mms)$ and every $Y_1,\dots,Y_{s+s'}\in\Ell^0(T\mms)$,
\begin{align*}
&T\boxtimes S(\eta_1,\dots,\eta_{r+r'},Y_1,\dots,Y_{s+s'})\\
&\qquad\qquad = T(\eta_1,\dots,\eta_r,Y_1,\dots,Y_s)\\
&\qquad\qquad\qquad\qquad\times S(\eta_{r+1},\dots,\eta_{r+r'},Y_{s+1},\dots,Y_{s+s'})\quad\meas\text{-a.e.}
\end{align*}

\begin{proposition}[Leibniz rule for tensor fields]\label{Pr:Leibniz tfs} Suppose that $T\in \smash{W^{1,2}(T_r^s\mms)}$ and $\smash{S\in\Reg(T_{r'}^{s'}\mms)}$. Then $\smash{T\boxtimes S\in W^{1,2}(T_{r+r'}^{s+s'}\mms)}$ and, as an identity in $\smash{\rmT_{\Ell^2}(\mms)}$,
\begin{align*}
\nabla(T\boxtimes S) = \nabla T \otimes S + T\otimes\nabla S.
\end{align*}
\end{proposition}

\begin{proof} We write $S$ in the form
\begin{align*}
S := \omega_{r+1}\otimes\dots\otimes \omega_{r+r'}\otimes Y_{s+1}\otimes\dots\otimes Y_{s+s'}
\end{align*}
for given $\omega_{r+1},\dots,\omega_{r+r'}\in\Reg(T^*\mms)$ and $Y_{s+1},\dots,Y_{s+s'}\in\Reg(T\mms)$. Given any $\eta_1,\dots,\eta_{r+r'}\in\Test(T^*\mms)$ and $Z,X_1,\dots, X_{s+s'}\in\Test(T\mms)$, we abbreviate
\begin{align*}
f &:= T(\eta_1,\dots,\eta_r,Y_1,\dots,Y_s),\\
g &:= S(\eta_{r+1},\dots,\eta_{r+r'},Y_{s+1},\dots,Y_{s+s'})
\end{align*}
and, for $i\in \{r+1,\dots,r+r'\}$ and $j\in \{s+1,\dots,s+s'\}$,
\begin{align*}
U_i &:= \omega_{r+1}\otimes\dots\otimes \underbrace{(\nabla_Z\omega_i^\sharp)^\flat}_{i\text{-th slot}}\otimes\dots\otimes \omega_{r+r'}\otimes Y_{s+1}\otimes\dots\otimes Y_{s+s'},\\
V_j &:= \omega_{r+1}\otimes\dots\otimes\omega_{r+r'}\otimes Y_{s+1}\otimes\dots\otimes \!\!\!\!\!\!\underbrace{\nabla_ZY_j}_{(r'+j)\text{-th slot}}\!\!\!\!\!\!\otimes\dots\otimes Y_{s+s'}.
\end{align*}
By \eqref{Eq:Tensor product pointwise sc prod}, \autoref{Pr:Compatibility} and the Leibniz rule from \autoref{Cor:Calculus rules d}, it follows that $g\in \F\cap\Ell^\infty(\mms)$. By \autoref{Le:Div g nabla f}, we deduce that $\smash{g\,Z\in\Dom_\TV(\DIV)\cap\Dom(\div)}$ with $\norm(g\,Z)=0$ and that for every sequence $(g_n)_{n\in\N}$ in $\Test(\mms)$ converging to $g$ in $\smash{\F}$, we have $\div(g_n\,Z)\to \div(g\,Z)$ in $\Ell^2(\mms)$ as $n\to\infty$. By \autoref{Le:Div g nabla f} again and the fact that $g_n\,Z\in\Test(T\mms)$ for every $n\in\N$,
\begin{align*}
&-\int_\mms T\boxtimes S(\eta_1,\dots,\eta_{r+r'},Y_1,\dots,Y_{s+s'})\div Z\d\meas\\
&\qquad\qquad\qquad\qquad -\int_\mms\sum_{i=1}^{r+r'} T\boxtimes S(\nabla_Z\eta_i)\d\meas - \int_\mms\sum_{j=1}^{s+s'}T\boxtimes S(\nabla_ZY_j)\d\meas\\
&\qquad\qquad = -\lim_{n\to\infty}\int_\mms f \div(g_n\,Z)\d\meas + \int_\mms f\d g(Z)\d\meas\\
&\qquad\qquad\qquad\qquad - \lim_{n\to\infty} \int_\mms g_n\sum_{i=1}^{r} T(\nabla_Z\eta_i)\d\meas - \lim_{n\to\infty}\int_\mms g_n\sum_{j=1}^{s} T(\nabla_ZY_j)\d\meas\\
&\qquad\qquad\qquad\qquad -\int_\mms f\!\sum_{i=r+1}^{r+r'}\!S(\nabla_Z\eta_i)\d\meas - \int_\mms f\!\sum_{j=s+1}^{s+s'}\!S(\nabla_Z Y_j)\d\meas.
\end{align*}
Applying \autoref{Pr:Compatibility} to the fifth last integral, then applying \autoref{Le:Leibniz rule reg tf} to the sum of those derivatives that fall on $S$, and finally using the definition of $\nabla T$, the above sum is equal to
\begin{align*}
&-\lim_{n\to\infty}\int_\mms f\div(g_n\,Z)\d\meas\\
&\qquad\qquad\qquad\qquad + \int_\mms f\!\sum_{i=r+1}^{r+r'} U_i(\eta_{r+1},\dots,\eta_{r+r'},Y_{s+1},\dots,Y_{s+s'})\d\meas\\
&\qquad\qquad\qquad\qquad + \int_\mms f\!\sum_{j=s+1}^{s+s'} V_j(\eta_{r+1},\dots,\eta_{r+r'},Y_{s+1},\dots,Y_{s+s'})\d\meas\\
&\qquad\qquad\qquad\qquad - \lim_{n\to\infty}\int_\mms g_n\sum_{i=1}^rT(\nabla_Z\eta_i)\d\meas -\lim_{n\to\infty}\int_\mms g_n\sum_{j=1}^sT(\nabla_ZY_j)\d\meas\\
&\qquad\qquad = \int_\mms g\,\nabla T(\eta_1,\dots,\eta_r,Z,Y_1,\dots,Y_s)\d\meas\\
&\qquad\qquad\qquad\qquad + \int_\mms f\!\sum_{i=r+1}^{r+r'} U_i(\eta_{r+1},\dots,\eta_{r+r'},Y_{s+1},\dots,Y_{s+s'})\d\meas\\
&\qquad\qquad\qquad\qquad + \int_\mms f\!\sum_{j=s+1}^{s+s'} V_j(\eta_{r+1},\dots,\eta_{r+r'},Y_{s+1},\dots,Y_{s+s'})\d\meas.
\end{align*}
The claimed identity in $\rmT_{\Ell^2}(\mms)$ readily follows.
\end{proof}

\begin{remark}\label{Re:Dir der tf} In a similar way as in \eqref{Eq:Dir der def}, one can define the \emph{directional derivative} $\smash{\nabla_Z T\in \Ell^0(T_s^r\mms)}$ of a given $\smash{T\in W^{1,2}(T_s^r\mms)}$ in the direction of $Z\in\Ell^0(T\mms)$. In the notation of \autoref{Pr:Leibniz tfs}, given such $Z$ the Leibniz rule becomes
\begin{align*}
\nabla_Z(T\boxtimes S) = \nabla_ZT \otimes S + T\otimes\nabla_ZS,
\end{align*}
in $\rmT_{\Ell^2}(\mms)$, 
and accordingly in the framework of \autoref{Le:Leibniz rule reg tf}.
\end{remark}

\begin{remark} A more general Leibniz rule seems hard to obtain by evident integrability issues. Compare with \autoref{Re:Integr issues} below. One framework in which one could instead work is an appropriate version $\smash{W^{1,1}(T_{s+s'}^{r+r'}\mms)}$ of \autoref{Def:Tensor cov der}. However, as in the motivating remarks before \autoref{Def:W221} and \autoref{Def:W21}, it is not clear if such a notion gives rise to nontrivial objects.
\end{remark}

\section{Exterior derivative}\label{Ch:Ext derivative}

Throughout this chapter, let us fix $k\in\N_0$.

\subsection{The Sobolev space $\smash{\Dom(\rmd^k)}$}  We now give a meaning to the exterior derivative acting on suitable $k$-forms, i.e.~elements of $\Ell^2(\Lambda^kT^*\mms)$ (recall \autoref{Sub:Exterior products}).

\subsubsection{Definition and basic properties} Before the motivating smooth \autoref{Ex:Ext der smooth}, a notational comment is in order. Given $\omega\in\Ell^0(\Lambda^kT^*\mms)$ and $X_0,\dots,X_k,Y\in \Ell^0(T\mms)$, we shall use the standard abbreviations
\begin{align*}
\omega(\widehat{X}_i) &:= \omega(X_0,\dots,\widehat{X}_i,\dots,X_k)\\
\textcolor{white}{\widehat{X}}&:= \omega(X_0 \wedge \dots\wedge X_{i-1}\wedge X_{i+1}\wedge\dots\wedge X_k),\\
\omega(Y,\widehat{X}_i,\widehat{X}_j) &:= \omega(Y,X_0,\dots,\widehat{X}_i,\dots,\widehat{X}_j,\dots, X_k)\\
\textcolor{white}{\widehat{X}_i}&:= \omega(Y\wedge X_0\wedge\dots\wedge X_{i-1}\wedge X_{i+1}\wedge\dots \wedge X_{j-1}\wedge X_{j+1}\wedge \dots\wedge X_k).
\end{align*}

\begin{example}\label{Ex:Ext der smooth} On a Riemannian manifold $\mms$ with boundary, the exterior derivative $\smash{\rmd \colon \Gamma(\Lambda^kT^*\mms)\to\Gamma(\Lambda^{k+1}T^*\mms)}$ is defined by three axioms \cite[Thm.~9.12]{lee2018}. It can be shown \cite[Sec.~A.2]{petersen2006} that the unique such $\rmd$ satisfies the following pointwise, chart-free representation for any $\omega\in \Gamma(\Lambda^kT^*\mms)$ and any  $X_0,\dots,X_k \in \Gamma_\rmc(T\mms)$:
\begin{align}\label{Eq:Ext chart free}
\begin{split}
\rmd\omega(X_0,\dots,X_k) &= \sum_{i=0}^k (-1)^i\,\rmd\big[\omega(\widehat{X}_i)\big](X_i)\\
&\qquad\qquad + \sum_{i=0}^k\sum_{j=i+1}^k (-1)^{i+j}\,\omega([X_i,X_j], \widehat{X}_i,\widehat{X}_j).
\end{split}
\end{align}

By the discussion from \autoref{Sub:Riem mflds}, the map $\rmd$ is still uniquely determined on $\Gamma(\Lambda^kT^*\mms)$ by this identity when restricting to those  $X_0,\dots,X_k$ for which
\begin{align*}
\langle X_0,\sfn\rangle = \dots = \langle X_k,\sfn\rangle = 0\quad\text{on }\partial\mms.
\end{align*}
In this case, integrating \eqref{Eq:Ext chart free} leads to
\begin{align*}
\int_\mms \rmd\omega(X_0,\dots,X_k)\d\vol &= \int_\mms\sum_{i=0}^k(-1)^i\,\omega(\widehat{X}_i)\div X_i\d\vol\\
&\qquad\qquad + \int_\mms \sum_{i=0}^k\sum_{j=i+1}^k (-1)^{i+j}\,\omega([X_i,X_j],\widehat{X}_i,\widehat{X}_j)\d\vol
\end{align*}
after integration by parts in conjunction with \autoref{Ex:Mflds with boundary}. Given this integral identity for every compactly supported $X_0,\dots,X_k\in\Gamma(\Lambda^kT^*\mms)$ with vanishing normal parts at $\partial\mms$ as above, the differential $\rmd\omega\in\Gamma(\Lambda^{k+1}T^*\mms)$ of $\omega\in \Gamma(\Lambda^kT^*\mms)$ is of course still uniquely determined.
\end{example}

Now note that the r.h.s.~of the last integral identity --- with $\vol$ replaced by $\meas$ --- is meaningful for arbitrary $\omega\in\Ell^2(\Lambda^kT^*\mms)$ and $X_0,\dots,X_k\in\Test(T\mms)$. Indeed, $\smash{\omega(\widehat{X}_i)\in\Ell^2(\mms)}$ since $X_0,\dots,X_k\in \Ell^\infty(T\mms)$, and $X_0,\dots,X_k\in \Dom_\TV(\DIV)\cap \Dom(\div)$ with $\div X_i\in \Ell^2(\mms)$ and $\norm X_i = 0$ by \autoref{Le:Div g nabla f}, $i\in \{0,\dots,k\}$. Moreover, by \eqref{Eq:Bound nabla Z X} the Lie bracket $[X_i,X_j]$ belongs to $\Ell^2(T\mms)$, whence $\smash{\omega([X_i,X_j],\widehat{X}_i,\widehat{X}_j)}\in \Ell^2(\mms)$, $i\in \{0,\dots,k\}$ and $j\in \{i+1,\dots,k\}$.

These considerations motivate the subsequent definition. (We only make explicit the degree $k$ in the name of the space, but not in the differential object itself.)

\begin{definition}\label{Def:Exterior derivative} We define $\smash{\Dom(\rmd^k)}$ to consist of all $\omega\in \Ell^2(\Lambda^kT^*\mms)$ for which there exists $\eta\in \Ell^2(\Lambda^{k+1}T^*\mms)$ such that for every $X_0,\dots,X_k\in \Test(T\mms)$,
\begin{align*}
\int_\mms \eta(X_0,\dots,X_k)\d\meas &=\int_\mms \sum_{i=0}^k(-1)^{i+1}\,\omega(\widehat{X}_i)\div X_i\d\meas\\
&\qquad\qquad +\int_\mms \sum_{i=0}^k\sum_{j=i+1}^k (-1)^{i+j}\, \omega([X_i,X_j],\widehat{X}_i,\widehat{X}_j)\d\meas.
\end{align*}
In case of existence, the element $\eta$ is unique, denoted by $\rmd \omega$ and termed the \emph{exterior derivative} \textnormal{(}or \emph{exterior differential}\textnormal{)} of $\omega$.
\end{definition}

The uniqueness follows by density of $\smash{\Test(\Lambda^{k+1}T^*\mms)}$ in $\smash{\Ell^2(\Lambda^{k+1}T^*\mms)}$ as discussed in \autoref{Sub:Test objects}. It is then clear that $\smash{\Dom(\rmd^k)}$ is a real vector space and that $\rmd$ is a linear operator on it.

We always endow $\smash{\Dom(\rmd^k)}$ with the norm $\smash{\Vert \cdot\Vert_{\Dom(\rmd^k)}}$ given by
\begin{align*}
\big\Vert \omega \big\Vert_{\Dom(\rmd^k)}^2 := \big\Vert\omega\big\Vert_{\Ell^2(\Lambda^kT^*\mms)}^2 + \big\Vert \rmd\omega\big\Vert_{\Ell^2(\Lambda^{k+1}T^*\mms)}^2.
\end{align*}
We introduce the functional $\smash{\Ch_\rmd\colon\Ell^2(\Lambda^kT^*\mms)\to [0,\infty]}$ with
\begin{align*}
\Ch_\rmd(\omega) :=\begin{cases} \displaystyle\int_\mms \vert\rmd\omega\vert^2\d\meas & \text{if }\omega\in \Dom(\rmd^k),\\
\infty & \text{otherwise}.
\end{cases}
\end{align*}
We do not make explicit the dependency of $\Ch_\rmd$ on the degree $k$. It will always be clear from the context which one is intended.

\begin{remark}\label{Re:Wd120 in W12} By \autoref{Le:Div g nabla f} it is easy to see that $\F$ is contained in $\smash{\Dom(\rmd^0)}$, and that $\rmd\omega$ is simply the exterior differential from \autoref{Def:Differential}, $\omega\in \F$. The reverse inclusion, however, seems more subtle, but at least  holds true if $\mms$ is intrinsically complete as in \autoref{Def:Intr compl}. Compare with \autoref{Re:No W11 in general}, \cite[p.~136]{gigli2018} and (the proof of) \cite[Prop.~3.3.13]{gigli2018}.
\end{remark}

\begin{remark} Similarly to \autoref{Re:Conf trafos} and \autoref{Re:Conf trafos II}, motivated by its axiomatization in Riemannian geometry we expect the differential $\rmd$ to neither depend on conformal transformations of $\langle\cdot,\cdot\rangle$, nor  on drift transformations of $\meas$. 
\end{remark}

The next theorem collects basic properties of the above notions. It is proven in a similar fashion as \autoref{Th:Hess properties} and \autoref{Th:Properties W12 TM}.

\begin{theorem}\label{Th:Wd12 properties} The space $\Dom(\rmd^k)$, the exterior derivative $\rmd$ and the functional $\Ch_\rmd$ satisfy the following properties.
\begin{enumerate}[label=\textnormal{\textcolor{black}{(}\roman*\textcolor{black}{)}}]
\item\label{La:I} $\smash{\Dom(\rmd^k)}$ is a separable Hilbert space w.r.t.~$\smash{\Vert\cdot\Vert_{\Dom(\rmd^k)}}$.
\item\label{La:II} The exterior differential is a closed operator. That is, the image of the map $\smash{\Id\times \rmd\colon \Dom(\rmd^k)\to \Ell^2(\Lambda^kT^*\mms)\times\Ell^2(\Lambda^{k+1}T^*\mms)}$ is a closed subspace of $\smash{\Ell^2(\Lambda^kT^*\mms)\times\Ell^2(\Lambda^{k+1}T^*\mms)}$.
\item\label{La:III} The functional $\Ch_\rmd$ is $\Ell^2$-lower semicontinuous, and for every $\smash{\omega\in\Ell^2(\Lambda^kT^*\mms)}$ we have the duality formula
\begin{align*}
\Ch_\rmd(\omega) &= \sup\!\Big\lbrace 2\sum_{l=1}^n\int_\mms \sum_{i=0}^k(-1)^{i+1}\,\omega(\widehat{X}^l_i)\div X_i^l\d\meas\\
&\qquad\qquad + 2\sum_{l=1}^n\int_\mms \sum_{i=0}^k\sum_{j=i+1}^k (-1)^{i+j}\,\omega([X_i^l,X_j^l], \widehat{X}_i^l,\widehat{X}_j^l)\d\meas\\
&\qquad\qquad - \int_\mms \Big\vert\!\sum_{j=1}^n X_0^l\wedge\dots\wedge X_k^l\Big\vert^2\d\meas:  n\in \N,\ X_i^l\in\Test(T\mms)\Big\rbrace.
\end{align*}
\item\label{La:IV} For every $f_0\in\Test(\mms)\cup \R\,\One_\mms$ and every $f_1,\dots,f_k\in \Test(\mms)$ we have $\smash{f_0\d f_1\wedge\dots\wedge\rmd f_k\in \Dom(\rmd^k)}$ with
\begin{align*}
\rmd(f_0\d f_1\wedge\dots\wedge\rmd f_k) &= \rmd f_0\wedge\rmd f_1\wedge\dots\wedge\rmd f_k,
\end{align*}
with the usual interpretation $\rmd\One_\mms := 0$. In particular $\smash{\Reg(\Lambda^kT^*\mms)}\subset \smash{\Dom(\rmd^k)}$, and $\smash{\Dom(\rmd^k)}$ is dense in $\Ell^2(\Lambda^kT^*\mms)$.
\end{enumerate}
\end{theorem}

\begin{proof} The items \ref{La:I}, \ref{La:II} and \ref{La:III} follow completely analogous lines as the proofs of corresponding statements in \autoref{Th:Hess properties} and \autoref{Th:Properties W12 TM}. We omit the details.

We turn to \ref{La:IV}. We concentrate on the proof of the claimed formula, from which the last two statements then readily follow by linearity of $\rmd$. First observe that the r.h.s.~of the claimed identity belongs to $\Ell^2(\Lambda^{k+1}T^*\mms)$. By definition \eqref{Eq:Tensor product pointwise sc prod} of the pointwise scalar product in $\Ell^2(\Lambda^kT^*\mms)$ and \autoref{Pr:BE vector fields}, we have $\omega(X_1,\dots,X_k) \in \F\cap\Ell^\infty(\mms)$, where $\omega := \rmd f_1 \wedge\dots\wedge \rmd f_k$. Direct computations using the  \autoref{Def:Lie bracket} of the Lie bracket and \autoref{Th:Properties W12 TM} yield
\begin{align*}
\sum_{i=0}^k (-1)^i\,\rmd\big[\omega(\widehat{X}_i)\big](X_i) = -\sum_{i=0}^k\sum_{j=i+1}^k (-1)^{i+j}\,\omega([X_i,X_j],\widehat{X}_i,\widehat{X}_j)\quad\meas\text{-a.e.}
\end{align*}
For $f_0\in\Test(\mms)$, it thus follows from \autoref{Le:Div g nabla f} that
\begin{align*}
&\int_\mms \sum_{i=0}^k (-1)^{i+1}\,f_0\,\omega(\widehat{X}_i)\div X_i\d\meas\\
&\qquad\qquad\qquad\qquad + \int_\mms \sum_{i=0}^k \sum_{j=i+1}^k (-1)^{i+j}\,f_0\,\omega([X_i,X_j],\widehat{X}_i,\widehat{X}_j)\d\meas\\
&\qquad\qquad = \int_\mms \sum_{i=0}^k (-1)^i\,\rmd f_0(X_i)\,\omega(\widehat{X}_i)\d\meas\\
&\qquad\qquad\qquad\qquad + \int_\mms \sum_{i=0}^k (-1)^i\, f_0\d\big[\omega(\widehat{X}_i)\big](X_i)\d\meas\\
&\qquad\qquad\qquad\qquad + \int_\mms \sum_{i=0}^k\sum_{j=i+1}^k (-1)^{i+j}\,f_0\,\omega([X_i,X_j],\widehat{X}_i,\widehat{X}_j)\d\meas\\
&\qquad\qquad = \int_\mms (\rmd f_0\wedge \omega)(X_0,\dots,X_k)\d\meas,
\end{align*}
which shows the first claimed identity. The same computation can be done for $f_0 \in\R\,\One_\mms$ with the formal interpretation $\rmd f_0 := 0$.
\end{proof}

\begin{remark}\label{Re:Kato role} For arbitrary, not necessarily tamed Dirichlet spaces, certainly the spaces $\Ell^2(\Lambda^kT^*\mms)$ and $\Ell^2(\Lambda^{k+1}T^*\mms)$ from \autoref{Sub:Exterior products} make sense. One is then tempted to define the exterior derivative of $f_0\,\rmd f_1\wedge\dots\wedge\rmd f_k$ for appropriate $f_0,\dots,f_k\in \F_\rme$ simply as $\rmd f_0 \wedge\rmd f_1\wedge\dots\wedge\rmd f_k\in\smash{\Ell^2(\Lambda^{k+1}T^*\mms)}$. However, it is in general not clear if $\rmd$ defined in that way is closable. In our approach, this is clear from \autoref{Def:Exterior derivative} by integration by parts, for which it has been crucial to know the existence of a large class of vector fields whose Lie bracket is well-defined. In our approach, this is precisely $\Test(T\mms)$, whose nontriviality --- in fact, density in $\Ell^2(T\mms)$ --- is a consequence of the (extended Kato condition on the) lower Ricci bound $\kappa$, see \autoref{Sub:Test fcts} and \autoref{Sub:Test objects}.
\end{remark}

\subsubsection{Calculus rules} We proceed with further calculus rules for $\rmd$. In view of \autoref{Pr:Leibniz rule ext der} below, the following preliminary lemma is required.

\begin{lemma}\label{Le:f omega} Suppose that $\omega\in\smash{\Dom(\rmd^k)}$, and that $f\in \F\cap\Ell^\infty(\mms)$. Then for every $X_0,\dots,X_k\in \Test(T\mms)$,
\begin{align*}
&\int_\mms f\d\omega(X_0,\dots,X_k)\d\meas\\
&\qquad\qquad= \int_\mms \sum_{i=0}^k (-1)^{i+1}\,\omega(\widehat{X}_i)\div(f\,X_i)\d\meas\\
&\qquad\qquad\qquad\qquad + \int_\mms \sum_{i=0}^k\sum_{j=i+1}^k(-1)^{i+j}\,f\,\omega([X_i,X_j],\widehat{X}_i,\widehat{X}_j)\d\meas.
\end{align*}
\end{lemma}

\begin{proof} We first prove the claim for $f\in\Test(\mms)$. As $f\,X_0\in\Test(T\mms)$, by definition of the Lie bracket  and \autoref{Le:Leibniz rule W(21)} we have
\begin{align*}
[f\,X_0, X_j] &= \nabla_{f\,X_0}X_j - \nabla_{X_j}(f\,X_0)\\
&= f\,\nabla_{X_0}X_j - \rmd  f(X_j)\,X_0 - f\,\nabla_{X_j}X_0\\
&= f\,[X_0,X_j] - \rmd f(X_j)\,X_0
\end{align*}
for every $j\in \{1,\dots,k\}$. Hence, by \autoref{Le:Div identities},
\begin{align*}
&\int_\mms f\d\omega(X_0,\dots,X_k)\d\meas\\
&\qquad\qquad =\int_\mms \rmd\omega(f\,X_0,X_1,\dots,X_k)\d\meas\\
&\qquad\qquad = -\int_\mms \omega(\widehat{X}_0)\d f(X_0)\d\meas + \int_\mms \sum_{i=0}^k (-1)^{i+1}f\,\omega(\widehat{X}_i)\div X_i\d \meas \\
&\qquad\qquad\qquad\qquad +\int_\mms \sum_{i=0}^k\sum_{j=i+1}^k f\,\omega([X_i,X_j],\widehat{X}_i,\widehat{X}_j)\d\meas\\
&\qquad\qquad\qquad\qquad + \int_\mms\sum_{j=1}^k (-1)^{j+1}\,\omega(\widehat{X}_j)\d f(X_j)\d\meas.
\end{align*}
Since $\div(f\,X_i) = \rmd f(X_i) + f\div X_i$ $\meas$-a.e.~by \autoref{Le:Div identities}, we are done.

The claim for general $f\in \F\cap \Ell^\infty(\mms)$ follows by the approximation result from \autoref{Le:Mollified heat flow} together with \autoref{Le:Div identities}.
\end{proof}

\begin{proposition}[Leibniz rule]\label{Pr:Leibniz rule ext der} Let $\smash{\omega\in \Dom(\rmd^k)}$ and, for some $k'\in\N_0$, suppose that $\smash{\omega'\in \Reg(\Lambda^{k'}T^*\mms)}$. Then $\omega\wedge\omega'\in \smash{\Dom(\rmd^{k+k'})}$ with
\begin{align*}
\rmd(\omega\wedge\omega')= \rmd\omega \wedge\omega' + (-1)^k\,\omega\wedge\rmd\omega'.
\end{align*}
\end{proposition}

\begin{proof} We proceed by induction on $k'$ and start with $k'=0$. In this case, $\omega'$ is simply an element $f\in\Test(\mms)$. Given any $X_0,\dots,X_k\in \Test(T\mms)$, by \autoref{Le:Div identities} and \autoref{Le:f omega} we obtain
\begin{align*}
&\int_\mms\sum_{i=0}^k (-1)^{i+1}\,f\,\omega(\widehat{X}_i)\div X_i\d\meas\\
&\qquad\qquad\qquad\qquad + \int_\mms \sum_{i=0}^k \sum_{j=i+1}^k (-1)^{i+j}\,f\,\omega([X_i,X_j],\widehat{X}_i,\widehat{X}_j)\d\meas\\
&\qquad\qquad = \int_\mms \sum_{i=0}^k (-1)^{i+1}\,\omega(\widehat{X}_i)\div(f\,X_i)\d\meas + \int_\mms \sum_{i=0}^k (-1)^i\,\omega(\widehat{X}_i)\d f(X_i)\d\meas\\
&\qquad\qquad\qquad\qquad + \int_\mms \sum_{i=0}^k \sum_{j=i+1}^k (-1)^{i+j}\,f\,\omega([X_i,X_j],\widehat{X}_i,\widehat{X}_j)\d\meas\\
&\qquad\qquad = \int_\mms f\d\omega(X_0,\dots,X_k)\d\meas + \int_\mms (\rmd f\wedge\omega)(X_0,\dots,X_k)\d\meas.
\end{align*}
In the last equality, we used the definition \eqref{Eq:Ptw scalar prod ext} of the pointwise scalar product in $\Ell^2(\Lambda^kT^*\mms)$. Therefore, we obtain that $f\,\omega\in\smash{\Dom(\rmd^k)}$ with
\begin{align*}
\rmd(f\,\omega) = f\d \omega + \rmd f \wedge\omega = f\d\omega +(-1)^k\,\omega\wedge\rmd f,
\end{align*}
which is precisely the claim for $k'=0$.

Before we proceed with the induction step, we show the claim under the assumption that $\omega' := \rmd f$ for some $f\in\Test(\mms)$, in which case we more precisely claim that $\omega\wedge\rmd f\in \smash{\Dom(\rmd^{k+1})}$ with
\begin{align}\label{Eq:k'=1 case}
\rmd(\omega\wedge \rmd f) = \rmd\omega\wedge\rmd f,
\end{align}
keeping in mind that $\rmd(\rmd f)=0$ by \autoref{Th:Wd12 properties}. To this aim, let $X_0,\dots,X_{k+1}\in\Test(T\mms)$. By definition \eqref{Eq:Ptw scalar prod ext} of the pointwise scalar product and \autoref{Le:f omega} --- which can be applied since $\rmd f(X_i)\in \F\cap\Ell^\infty(\mms)$ by \autoref{Pr:Compatibility} --- we get
\begin{align*}
&\int_\mms (\rmd\omega\wedge \rmd f)(X_0,\dots,X_{k+1})\d \meas\\
&\qquad\qquad = \int_\mms\sum_{i=0}^{k+1} (-1)^{i+k+1}\d\omega(\widehat{X}_i)\d f(X_i)\d\meas\\
&\qquad\qquad = \int_\mms \sum_{i=0}^{k+1}\sum_{\substack{j=0,\\j\neq i}}^{k+1} a_{ij}\,\omega(\widehat{X}_i,\widehat{X}_j)\div\!\big[\rmd f(X_i)\,X_j\big]\d\meas\\
&\qquad\qquad\qquad\qquad + \int_\mms \sum_{i=0}^{k+1}\sum_{\substack{j=0,\\ j\neq i}}^{k+1}\sum_{\substack{j'=j+1,\\j'\neq i}}^{k+1} b_{ijj'}\,\omega([X_j,X_{j'}], \widehat{X}_i,\widehat{X}_j,\widehat{X}_{j'})\d f(X_i)\d\meas,
\end{align*}
where, for $i,j,j'\in \{0,\dots,k+1\}$,
\begin{align*}
a_{ij} &:= \begin{cases} (-1)^{i+j+k+1} & \text{if } j \leq i,\\
(-1)^{i+j+k} & \text{otherwise},
\end{cases}\\
b_{ijj'} &:= \begin{cases} (-1)^{i+j+j'+k} & \text{if } j < i < j',\\
(-1)^{i+j+j'+k+1} & \text{otherwise}.
\end{cases}
\end{align*}
Then \eqref{Eq:k'=1 case} directly follows since, by \autoref{Le:Torsion free},
\begin{align*}
&\div\!\big[\rmd f(X_i)\,X_j\big] -\div\!\big[\rmd f(X_j)\,X_i\big]\\
&\qquad\qquad = \rmd f(X_i)\div X_j - \rmd f(X_j)\div X_i - \rmd f([X_i,X_j])\quad\meas\text{-a.e.}
\end{align*}

Now we are ready to perform the induction step. Given the assertion for $k'-1$ with $k'\in \N$, by linearity it suffices to consider the case $\omega' := f_0\d f_1\wedge\dots\wedge\rmd f_{k'}$, where $f_0\in\Test(\mms)\cup\R\,\One_\mms$ and $f_1,\dots,f_{k'}\in\Test(\mms)$. By \autoref{Th:Wd12 properties} we have $f_0\,\omega\in \smash{\Dom(\rmd^k)}$. Writing $\omega'' := \rmd f_2\wedge\dots\wedge \rmd f_{k'}$ thus yields
\begin{align*}
\rmd(\omega\wedge\omega') &= \rmd\big[(\omega\wedge f_0\d f_1)\wedge\omega''\big]\\
&= \rmd\big[(f_0\,\omega)\wedge\rmd f_1\big]\wedge\omega''\\
&= \big[f_0\,\rmd \omega + \rmd f_0\wedge\omega\big]\wedge \rmd f_1\wedge\omega''\\
&= \rmd\omega\wedge\omega' + (-1)^k\,\omega\wedge\rmd\omega'
\end{align*}
where we used the induction hypothesis in the second identity and \autoref{Th:Wd12 properties} in the last two equalities.
\end{proof}

\begin{remark}\label{Re:Integr issues} In \autoref{Pr:Leibniz rule ext der}, by evident integrability issues we cannot go really beyond the assumption $\omega\in\smash{\Reg(\Lambda^{k'}T^*\mms)}$, not even to the space $\smash{\Dom_\reg(\rmd^k)}$ introduced in the subsequent \autoref{Def:Hd12 forms}.
\end{remark}

\begin{definition}\label{Def:Hd12 forms} We define the space $\smash{\Dom_\reg(\rmd^k)\subset \Dom(\rmd^k)}$ by
\begin{align*}
\Dom_\reg(\rmd^k) := \cl_{\Vert \cdot\Vert_{\Dom(\rmd)}} \Reg(\Lambda^kT^*\mms).
\end{align*}
\end{definition}

This definition is non-void thanks to \autoref{Th:Wd12 properties} --- in fact, $\smash{\Dom_\reg(\rmd^k)}$ is a dense subspace of $\Ell^2(\Lambda^kT^*\mms)$. As in \autoref{Re:Test closure} we see that $\smash{\Dom_\reg(\rmd^k)}$ coincides with the closure of $\Test(\Lambda^kT^*\mms)$ in $\smash{\Dom(\rmd^k)}$ if $\mms$ is intrinsically complete, but might be larger in general. In line with this observation, we also do not know if $\smash{\Dom(\rmd^0) = \F}$ unless constant functions belong to $\F$. 

\begin{proposition}\label{Pr:d^2 =0} For every $\omega\in \Dom_\reg(\rmd^k)$, we have $\rmd\omega\in \Dom_\reg(\rmd^{k+1})$ with
\begin{align*}
\rmd(\rmd\omega)=0.
\end{align*}
\end{proposition}

\begin{proof} The statement for $\omega\in\Reg(\Lambda^kT^*\mms)$ follows from \autoref{Th:Wd12 properties} above. By definition of $\smash{\Dom(\rmd^k)}$ and the closedness of $\rmd$ again by \autoref{Th:Wd12 properties}, the claim extends to arbitrary $\smash{\omega\in \Dom(\rmd^k)}$.
\end{proof}

\begin{remark} A locality property for $\rmd$ such as
\begin{align*}
\One_{\{\omega = 0\}}\d \omega = 0
\end{align*}
for general $\smash{\omega\in \Dom(\rmd^k)}$ seems hard to obtain from our axiomatization. Compare with a similar remark at \cite[p.~140]{gigli2018}.
\end{remark}

\subsection{Nonsmooth de Rham cohomology and Hodge theorem}\label{Sub:Hodge thm}

Motivated by \autoref{Pr:d^2 =0}, the goal of this section is to make sense of a nonsmooth de Rham complex. The link with the smooth setting is outlined in \autoref{Re:Compat Hodge}. 

Given $k\in \N_0$ we exceptionally designate by $\smash{\rmd^k}$ the exterior differential defined on $\smash{\Dom(\rmd^k)}$, which is well-defined by  \autoref{Pr:d^2 =0}. Define the spaces $\rmC_k(\mms)$ and $\rmE_k(\mms)$ of \emph{closed} and \emph{exact} $k$-forms by
\begin{align*}
\rmC_k(\mms) &:= \mathrm{ker}\,\rmd^k \\
&\textcolor{white}{:}= \big\lbrace \omega \in \Dom(\rmd^k) : \rmd\omega = 0 \big\rbrace,\\
\rmE_k(\mms) &:= \mathrm{im}\,\rmd^{k-1}\\
&\textcolor{white}{:}= \big\lbrace \omega \in \Dom(\rmd^k) : \omega = \rmd\omega' \text{ for some }\omega'\in \Dom(\rmd^{k-1})\big\rbrace.
\end{align*}
By \autoref{Th:Wd12 properties}, we know that $\rmC_k(\mms)$ is a closed subspace of $\Ell^2(\Lambda^kT^*\mms)$, but not if the same is true for $\rmE_k(\mms)$. Since $\rmE_k(\mms) \subset \rmC_k(\mms)$ and $\rmC_k(\mms)$ is $\Ell^2$-closed, the $\Ell^2$-closure of $\rmE_k(\mms)$ is contained in $\rmC_k(\mms)$ as well, hence the following definition is meaningful and non-void.

\begin{definition} The \emph{$k$-th de Rham cohomology group} of $\mms$ is defined by
\begin{align*}
H_{\dR}^k(\mms) := \rmC_k(\mms)\,\big/\,\cl_{\Vert\cdot\Vert_{\Ell^2(\Lambda^kT^*\mms)}} \rmE_k(\mms).
\end{align*}
\end{definition}

In view of  the Hodge \autoref{Th:Hodge thm}, we first need to make sense of the Hodge Laplacian and of harmonic $k$-forms. We start with the following. (Again, we only make explicit the degree $k$ in the denotation of the space, but not of the differential object itself.)

\begin{definition}\label{Def:Codifferential} Given any $k\geq 1$, the space $\smash{\Dom(\delta^k)}$ is defined to consist of all $\smash{\omega\in \Ell^2(\Lambda^kT^*\mms)}$ for which there exists $\smash{\rho\in \Ell^2(\Lambda^{k-1}T^*\mms)}$ such that for every $\eta\in \Test(\Lambda^{k-1}T^*\mms)$, we have
\begin{align*}
\int_\mms \langle\rho,\eta\rangle\d\meas = \int_\mms \langle\omega,\rmd\eta\rangle\d\meas.
\end{align*}
If it exists, $\rho$ is unique, denoted by $\delta\omega$ and called the \emph{codifferential} of $\omega$. We simply define $\smash{\Dom(\delta^0) := \Ell^2(\mms)}$ and $\delta := 0$ on this space.
\end{definition}

By the  density of $\Test(\Lambda^{k-1}T^*\mms)$ in $\Ell^2(\Lambda^{k-1}T^*\mms)$, the uniqueness statement is indeed true. Furthermore, $\delta$ is a closed operator, i.e.~the image of the assignment $\smash{\Id\times \delta\colon \Dom(\delta^k)\to \Ell^2(\Lambda^kT^*\mms)\times\Ell^2(\Lambda^{k-1}T^*\mms)}$ is closed in $\Ell^2(\Lambda^kT^*\mms)\times\Ell^2(\Lambda^{k-1}T^*\mms)$. Lastly, by comparison of \autoref{Def:Codifferential} with  \autoref{Def:L2 div} we have $\smash{\Dom(\delta^1) = \Dom(\div)^\flat}$, and for every $\smash{\omega\in \Dom(\delta^1)}$,
\begin{align}\label{Eq:delta = -div}
\delta\omega = -\div\omega^\sharp\quad\meas\text{-a.e.}
\end{align}

The next result shows that $\smash{\Dom(\delta^k)}$ is nonempty --- in fact, it is dense in $\Ell^2(\Lambda^kT^*\mms)$. There and in the sequel, for appropriate $f_1,\dots,f_k\in\F_\rme$ and $i,j\in \{1,\dots,k\}$ with $i<j$, we use the abbreviations
\begin{align*}
\{\widehat{\rmd f}_i\} &:= \rmd f_1\wedge\dots\wedge \widehat{\rmd f}_i\wedge\dots\wedge\rmd f_k\\
\{\widehat{\rmd f}_i,\widehat{\rmd f}_j\} &:= \rmd f_1\wedge\dots\wedge\widehat{\rmd f}_i\wedge\dots\wedge\widehat{\rmd f}_j\wedge\dots \wedge\rmd f_k.
\end{align*}

\begin{lemma}\label{Le:delta of d} For every $f_0\in \Test(\mms)\cup\R\,\One_\mms$ and every $f_1,\dots,f_k\in \Test(\mms)$, we have $f_0\d f_1\wedge\dots\wedge\rmd f_k\in \smash{\Dom(\delta^k)}$ with
\begin{align*}
\delta(f_0\d f_1\wedge\dots\wedge\rmd f_k)  &= \sum_{i=1}^k (-1)^i\,\big[f_0\,\Delta f_i + \langle \rmd f_0,\rmd f_i\rangle\big]\,\{\widehat{\rmd f}_i\}\\
&\qquad\quad  + \sum_{i=1}^k \sum_{j=i+1}^k (-1)^{i+j}\,f_0\,[\nabla f_i,\nabla f_j]^\flat\wedge \{\widehat{\rmd f}_i,\widehat{\rmd f}_j\},
\end{align*}
with the usual interpretation $\rmd\One_\mms := 0$.
\end{lemma}

\begin{proof} We abbreviate $\omega := f_0\d f_1\wedge\dots\wedge \rmd f_k$. By linearity, it clearly suffices to consider the defining property from \autoref{Def:Codifferential} for $\eta := g_1\d g_2\wedge\dots\wedge \rmd g_k$, $g_1,\dots,g_k\in\Test(\mms)$. Write $\mathfrak{S}_k$ for the set of permutations of $\{1,\dots,k\}$. Then by \eqref{Eq:Ptw scalar prod ext} and the Leibniz formula,
\begin{align*}
\int_\mms \langle\omega,\rmd\eta\rangle\d\meas &= \int_\mms \sum_{\sigma\in\mathfrak{S}_k} \sgn \sigma\,f_0\,\big\langle\nabla g_1,\nabla f_{\sigma(1)}\big\rangle\,\prod_{i=2}^k \big\langle\nabla g_i,\nabla f_{\sigma(i)}\big\rangle\d\meas.
\end{align*}
Since $\nabla g_1\in\Dom(\div)$, regardless of whether $f_0\in \Test(\mms)$ or $f_0\in\R\,\One_\mms$ --- and with appropriate interpretation $\nabla f_0 := 0$ in the latter case --- integration by parts and then using \autoref{Le:Div g nabla f} and \autoref{Pr:Compatibility} yields
\begin{align*}
\int_\mms \langle\omega,\rmd\eta\rangle\d\meas &= -\int_\mms \sum_{\sigma\in\mathfrak{S}_k}\sgn \sigma\,g_1\div\!\Big[f_0\,\nabla f_{\sigma(1)}\,\prod_{i=2}^k \big\langle\nabla g_i,\nabla f_{\sigma(i)}\big\rangle\Big]\d\meas\\
&= -\int_\mms \sum_{\sigma\in\mathfrak{S}_k}\sgn \sigma\,g_1\,\big\langle \nabla f_0,\nabla f_{\sigma(1)}\big\rangle\,\prod_{i=2}^k \big\langle\nabla g_i,\nabla f_{\sigma(i)}\big\rangle\d\meas\\
&\qquad\qquad - \int_\mms \sum_{\sigma\in\mathfrak{S}_k}\sgn\sigma\,g_1\,f_0\,\Delta f_{\sigma(1)}\,\prod_{i=2}^k \big\langle\nabla g_i,\nabla f_{\sigma(i)}\big\rangle\d\meas\\
&\qquad\qquad - \int_\mms \sum_{\sigma\in\mathfrak{S}_k} \sgn\sigma\, g_1\,f_0\,\sum_{i=2}^k \Big[\!\Hess g_i(\nabla f_{\sigma(1)},\nabla f_{\sigma(i)})\\
&\qquad\qquad\qquad\qquad\times \prod_{\substack{j=2,\\j\neq i}}^k \big\langle\nabla g_j,\nabla f_{\sigma(j)}\big\rangle\Big]\d\meas\\
&\qquad\qquad -\int_\mms \sum_{\sigma\in\mathfrak{S}_k}\sgn\sigma\, g_1\,f_0\sum_{i=2}^k\Big[\!\Hess f_{\sigma(i)}(\nabla f_{\sigma(1)},\nabla g_i)\\
&\qquad\qquad\qquad\qquad \times\prod_{\substack{j=2,\\j\neq i}}^k \big\langle\nabla g_j,\nabla f_{\sigma(j)}\big\rangle\Big]\d\meas.
\end{align*}
The second last integral vanishes identically, which follows by symmetry of the Hessian and by comparing a given $\sigma\in\mathfrak{S}_k$ with the permutation that swaps $\sigma(1)$ and $\sigma(i)$ in $\sigma$, $i\in \{2,\dots,k\}$. The claim  follows from the combinatorial formulas
\begin{align*}
\sum_{\substack{\sigma\in\mathfrak{S}_k,\\\sigma(1) = I}}\sgn\sigma \prod_{j=1}^k \big\langle\nabla g_j,\nabla f_{\sigma(j)}\big\rangle &= (-1)^{I+1}\,\big\langle\rmd g_2\wedge\dots\wedge\rmd g_k, \{\widehat{\rmd f}_I\}\big\rangle\\
\sum_{\substack{\sigma\in\mathfrak{S}_k,\\
\sigma(1)=I,\\\sigma(J) = K\\}} \sgn\sigma\prod_{\substack{j=1,\\j\neq J}}^k \big\langle\nabla g_i,\nabla f_{\sigma(i)}\big\rangle &= a_{IJK}\,\big\langle \{\widehat{\rmd g}_J\}, \{\widehat{\rmd f}_I,\widehat{\rmd f}_K\}\big\rangle
\end{align*}
for every $I,K\in \{1,\dots,k\}$ with $I\neq K$ and every $J\in \{2,\dots,k\}$, where
\begin{align*}
a_{IJK} := \begin{cases} (-1)^{1+I+J+K} & \text{if }I < K,\\
(-1)^{I+J+K} & \text{otherwise},
\end{cases}
\end{align*}
from the identity
\begin{align*}
\Hess f_K(\nabla f_I,\cdot) - \Hess f_I(\nabla f_K,\cdot) = [\nabla f_I,\nabla f_K]^\flat
\end{align*}
stemming from the definition of the Lie bracket and \autoref{Th:Hess properties}, and
\begin{align*}
&\sum_{J=2}^k (-1)^J\,\big\langle[\nabla f_I,\nabla f_K],\nabla g_J\big\rangle\,\big\langle \{\widehat{\rmd g}_J\}, \{\widehat{\rmd f}_I,\widehat{\rmd f}_K\}\big\rangle\\
&\qquad\qquad = \big\langle \rmd g_2\wedge\dots\wedge\rmd g_k, [\nabla f_I,\nabla f_K]^\flat \wedge \{\widehat{\rmd f}_I,\widehat{\rmd f}_K\}\big\rangle.\qedhere
\end{align*}
\end{proof}

\begin{remark} By approximation, using \autoref{Th:Wd12 properties} and \autoref{Le:delta of d}, one readily proves  the following. Let $f\in \F_\eb$ and $\omega\in H^{1,2}(T^*\mms)$. Assume furthermore that $\rmd f\in\Ell^\infty(T^*\mms)$ or that $\omega\in\Ell^\infty(T^*\mms)$. Then $f\,\omega \in H^{1,2}(T^*\mms)$ with
		\begin{align*}
		\rmd(f\,\omega) &= f\,\rmd\omega + \rmd f\wedge\omega,\\
		\delta(f\,\omega) &= f\,\delta\omega - \langle\rmd f,\omega\rangle.
		\end{align*}
\end{remark}

\begin{definition}\label{Def:W12 forms space} We define the space $W^{1,2}(\Lambda^kT^*\mms)$ by
\begin{align*}
W^{1,2}(\Lambda^kT^*\mms) &:= \Dom(\rmd^k)\cap \Dom(\delta^k).
\end{align*}
\end{definition}

By \autoref{Th:Wd12 properties} and \autoref{Le:delta of d}, we already know that $W^{1,2}(\Lambda^kT^*\mms)$ is a dense subspace of $\Ell^2(\Lambda^kT^*\mms)$. 

We endow $W^{1,2}(\Lambda^kT^*\mms)$ with the norm $\Vert\cdot\Vert_{W^{1,2}(\Lambda^kT^*\mms)}$ given by
\begin{align*}
\Vert\omega\Vert_{W^{1,2}(\Lambda^kT^*\mms)}^2 := \Vert\omega\Vert_{\Ell^2(\Lambda^kT^*\mms)}^2 + \Vert \rmd\omega\Vert_{\Ell^2(\Lambda^{k+1}T^*\mms)}^2 + \Vert\delta\omega\Vert_{\Ell^2(\Lambda^{k-1}T^*\mms)}^2
\end{align*}
and we define the \emph{contravariant} functional $\Ch_\con\colon\Ell^2(\Lambda^kT^*\mms)\to [0,\infty]$ by
\begin{align*}
\Ch_\con(\omega) := \begin{cases}\displaystyle \int_\mms \big[\vert\rmd\omega\vert^2 + \vert\delta\omega\vert^2\big]\d\meas & \text{if }\omega\in W^{1,2}(\Lambda^kT^*\mms),\\
\infty & \text{otherwise}.
\end{cases}
\end{align*}
Arguing as for \autoref{Th:Hess properties}, \autoref{Th:Properties W12 TM} and \autoref{Th:Wd12 properties}, $\smash{W^{1,2}(\Lambda^kT^*\mms)}$ becomes a separable Hilbert space w.r.t.~$\smash{\Vert\cdot\Vert_{W^{1,2}(\Lambda^kT^*\mms)}}$. Moreover, the functional $\Ch_\con$ is clearly $\Ell^2$-lower semicontinuous.

Again by \autoref{Th:Wd12 properties} and \autoref{Le:delta of d}, we have $\smash{\Reg(\Lambda^kT^*\mms)\subset W^{1,2}(\Lambda^kT^*\mms)}$, so that the following definition  makes sense.

\begin{definition}\label{Def:H Hodge space} The space $H^{1,2}(\Lambda^kT^*\mms)\subset W^{1,2}(\Lambda^kT^*\mms)$ is defined by
\begin{align*}
H^{1,2}(\Lambda^kT^*\mms) := \cl_{\Vert \cdot\Vert_{W^{1,2}(\Lambda^kT^*\mms)}}\Reg(\Lambda^kT^*\mms).
\end{align*}
\end{definition}

\begin{remark}[Absolute boundary conditions]\label{Re:Abs bdry cond} We adopt the interpretation that $\omega$ and $\rmd\omega$ have ``vanishing normal components'' for any given $\smash{\omega\in H^{1,2}(\Lambda^kT^*\mms)}$. For instance, on a compact Riemannian manifold $\mms$ with boundary, by \eqref{Eq:Normal parts forms smooth world} every $\omega\in \Reg(\Lambda^kT^*\mms)$ satisfies the \emph{absolute boundary conditions} \cite[Sec.~2.6]{schwarz1995}
\begin{align}\label{Eq:Abs bdry cond}
\begin{split}
\rmn\,\omega &= 0,\\
\rmn\d\omega &= 0
\end{split}
\end{align}
$\surf$-a.e.~at $\partial\mms$. By Gaffney's inequality --- see \autoref{Re:Gaffney} below --- \autoref{Pr:d^2 =0} and \autoref{Le:delta of d}, both $\omega$ and $\rmd \omega$  belong to the corresponding $W^{1,2}$-Sobolev spaces over $\boldsymbol{\mathrm{F}} := \Lambda^kT^*\mms$ and $\boldsymbol{\mathrm{F}} := \Lambda^{k+1}T^*\mms$ induced by the Levi-Civita connection $\nabla$ as introduced in \autoref{Sub:Riem mflds}, respectively, and these inclusions are continuous. Hence, the trace theorem can be applied, and \eqref{Eq:Abs bdry cond} passes to the limit in the definition of $H^{1,2}(\Lambda^kT^*\mms)$. In particular, \eqref{Eq:Abs bdry cond} holds $\surf$-a.e.~for every $\omega \in H^{1,2}(\Lambda^kT^*\mms)$.
\end{remark}

\begin{remark}[Gaffney's inequality]\label{Re:Gaffney} In general, $H^{1,2}(\Lambda^kT^*\mms)$ does not coincide with $W^{1,2}(\Lambda^kT^*\mms)$. On a compact Riemannian manifold $\mms$ with boundary, our usual argument using \autoref{Pr:Trace thm} gives the claim, up to an important technical detail to be fixed before. (For notational simplicity, we restrict ourselves to the case $k=1$. In the general case, $H^{1,2}(T\mms)$ has to be replaced by $W^{1,2}(\boldsymbol{\mathrm{F}})$ defined according to \autoref{Sub:Riem mflds} with $\boldsymbol{\mathrm{F}} := \Lambda^kT^*\mms$.) 

In general, $\rmd$ and $\delta$ are continuous w.r.t.~convergence in $\smash{H^{1,2}(T\mms)^\flat}$ [sic] by \cite[p.~62]{schwarz1995}, whence $\smash{H^{1,2}(T^*\mms)\subset H^{1,2}(T\mms)^\flat}$ with continuous inclusion. However, to apply the trace theorem to infer that all elements of $H^{1,2}(T^*\mms)$ have vanishing normal component at $\partial\mms$, the reverse inclusion is required (compare with the foregoing \autoref{Re:Abs bdry cond} and with \autoref{Le:Inclusion} below). The latter is a classical result by Gaffney, see \cite[Cor.~2.1.6]{schwarz1995} for a proof: there exists a finite constant $C>0$ such that for every $\smash{\omega\in H^{1,2}(T\mms)^\flat}$ with $\rmn\,\omega = 0$ $\surf$-a.e.~on $\partial\mms$, we have
\begin{align*}
\big\Vert \omega^\sharp\big\Vert_{W^{1,2}(T\mms)}^2 \leq C\,\big\Vert \omega\big\Vert_{W^{1,2}(T^*\mms)}^2.
\end{align*}
\end{remark}

\begin{definition}\label{Def:Hodge Lapl} The space $\Dom(\Hodge_k)$ is defined to consist of all $\smash{\omega\in H^{1,2}(\Lambda^kT^*\mms)}$ for which there exists $\alpha\in \Ell^2(\Lambda^kT^*\mms)$ such that for every $\eta\in H^{1,2}(\Lambda^kT^*\mms)$,
\begin{align*}
\int_\mms \langle\alpha,\eta\rangle\d\meas = \int_\mms \big[\langle\rmd \omega,\rmd\eta\rangle + \langle\delta\omega,\delta\eta\rangle\big]\d\meas.
\end{align*}
In case of existence, the element $\alpha$ is unique, denoted by $\Hodge_k \omega$ and termed the \emph{Hodge Laplacian} of $\omega$. Moreover, the space $\Harm(\Lambda^kT^*\mms)$ of \emph{harmonic $k$-forms} is defined as the space of all $\omega\in H^{1,2}(\Lambda^kT^*\mms)$ with $\rmd\omega = 0$ and $\delta\omega = 0$.
\end{definition}

For the most important case $k=1$, we shall write $\Hodge$ instead of $\Hodge_k$. By \eqref{Eq:IBP Laplacian},
\begin{align*}
\Hodge_0 = -\Delta.
\end{align*}
Moreover, the Hodge Laplacian $\Hodge_k$ is a closed operator, which can be e.g.~seen by identifying $\Hodge_k\omega$, $\omega\in \Dom(\Hodge_k)$, with the only element in the subdifferential of $\smash{\widetilde{\Ch}_\con}(\omega)$, where the functional $\smash{\widetilde{\Ch}}_\con$ is defined in \eqref{Eq:C con} below \cite[p.~145]{gigli2018}. In particular, $\smash{\Harm(\Lambda^kT^*\mms)}$ is a closed subspace of $\Ell^2(\Lambda^kT^*\mms)$.

\begin{remark} Our definition of harmonic $k$-forms follows \cite[Def.~2.2.1]{schwarz1995}. Of course, in the framework of \autoref{Def:Hodge Lapl}, $\omega\in\Harm(\Lambda^kT^*\mms)$ if and only if $\omega\in\Dom(\Hodge_k)$ with $\Hodge_k\omega=0$, see e.g.~\cite[p.~145]{gigli2018}. This, however, is a somewhat implicit consequence of the interpretation of any $\omega\in \Harm(\Lambda^kT^*\mms)$ as obeying absolute boundary conditions (recall \autoref{Re:Abs bdry cond}). Compare with \cite[Prop.~1.2.6, Prop.~2.1.2, Cor.~2.1.4]{schwarz1995}. In general, the vanishing of $\Hodge_k\omega$ is a weaker condition than asking for $\rmd\omega$ and $\delta\omega$ to vanish identically for appropriate  $\omega\in\Ell^2(\Lambda^kT^*\mms)$ \cite[p.~68]{schwarz1995}.
\end{remark}

We can now state and prove the following variant of Hodge's theorem.

\begin{theorem}[Hodge theorem]\label{Th:Hodge thm} The map $\omega\mapsto [\omega]$ from $\Harm(\Lambda^kT^*\mms)$ into $H_\dR^k(\mms)$ is an isomorphism of Hilbert spaces.
\end{theorem}

\begin{proof} Set $H := \rmC_k(\mms)$ and $V := \rmE_k(\mms)$. Thanks to \autoref{Th:Wd12 properties}, $H$ is a Hilbert space w.r.t.~$\smash{\Vert \cdot\Vert_H := \Vert \cdot\Vert_{\Ell^2(\Lambda^kT^*\mms)}}$.  Moreover, $V$ is a subspace of $H$. Lastly, by \autoref{Def:Codifferential} it is elementary to see that
\begin{align*}
V^\perp = \Harm(\Lambda^kT^*\mms),
\end{align*}
where we intend the orthogonal complement w.r.t.~the usual scalar product in $\Ell^2(\Lambda^kT^*\mms)$. Therefore, \autoref{Th:Hodge thm} follows since by basic Hilbert space theory \cite[Thm.~2.2.3]{kadison1983}, the map sending any $\smash{\omega\in V^\perp}$ to $\smash{\omega + \cl_{\Vert\cdot\Vert_H}V \in H/\cl_{\Vert\cdot\Vert_H}V}$ is a Hilbert space isomorphism.
\end{proof}

\begin{remark} It is unclear if $\smash{W^{1,2}(\Lambda^kT^*\mms)\setminus H^{1,2}(\Lambda^kT^*\mms)}$ contains elements with vanishing differential and codifferential, hence our choice of the domain of definition of $\Hodge_k$ and of harmonic $k$-forms. Compare with \cite[Rem.~3.5.16]{gigli2018}.
\end{remark}

\begin{remark}\label{Re:Compat Hodge} \autoref{Th:Hodge thm} is a variant of the Hodge theorem on compact manifolds $\mms$ with boundary \cite[Thm.~2.6.1]{schwarz1995}. Quite interestingly, in contrast to our setting --- where boundary conditions somewhat come as a byproduct of our class of ``smooth $k$-forms'' --- no boundary conditions are needed to build the corresponding de Rham cohomology group $H^k_\dR(\mms)$ \cite[p.~103]{schwarz1995}. \emph{Still}, the latter is isomorphic to the space of harmonic \emph{Neumann fields} \cite[Def.~2.2.1]{schwarz1995} which by definition satisfy absolute boundary conditions. This follows from the \emph{Hodge--Morrey decomposition} \cite[Thm.~2.4.2]{schwarz1995} combined with the \emph{Friedrichs decomposition} \cite[Thm.~2.4.8]{schwarz1995}.

It is worth mentioning that in this setting, $\dim\Harm(\Lambda^kT^*\mms)$ coincides with the $k$-th Betti number of $\mms$ \cite[p.~68]{schwarz1995}.
\end{remark}

\section{Curvature measures} \label{Sec:Curvature}

We are now in a position to introduce various concepts of \emph{curvature}. In \autoref{Sub:Ricci measure}, we prove our main result,  \autoref{Th:Ricci measure}, where the \emph{$\kappa$-Ricci measure} $\RIC^\kappa$ is made sense of, among others in terms of  $\DELTA^{2\kappa}$. In \autoref{Sub:Curv tensors from RIC}, we  separate $\kappa$ from $\DELTA^{2\kappa}$, which induces the \emph{measure-valued Laplacian} $\DELTA$ fully compatible with the definition of the measure-valued divergence $\DIV$, see \autoref{Def:Meas val Lapl}. In turn, this allows us to define a Ricci curvature and an intrinsic second fundamental form on $\mms$. 

\subsection{$\kappa$-Ricci measure}\label{Sub:Ricci measure}  Our main  \autoref{Th:Ricci measure} requires some technical preliminary ingredients that are subsequently discussed.

\subsubsection{Preliminary preparations} 

\begin{lemma}\label{Le:Hodge test} For every $g \in\Test(\mms)\cup\R\,\One_\mms$  and every $f\in\Test(\mms)$, we have $g\d f\in \Dom(\Hodge)$ with
\begin{align*}
\Hodge(g\d f)= - g\d \Delta f -\Delta g\d f- 2\Hess f(\nabla g,\cdot),
\end{align*}
with the usual interpretations $\nabla\One_\mms := 0$ and $\Delta \One_\mms := 0$. More generally, for every $X\in\Reg(T\mms)$ and every $h\in\Test(\mms)\cup \R\,\One_\mms$, we have $\smash{h\,X^\flat\in\Dom(\Hodge)}$ with
\begin{align*}
\Hodge(h\,X) = h\,\Hodge X - \Delta h\,X-2\,\nabla_{\nabla h}X.
\end{align*}
\end{lemma}

\begin{proof} The respective r.h.s.'s of the claimed identities belong to $\Ell^2(T^*\mms)$, which grants their  meaningfulness.

To prove the first, we claim that for every $f',g'\in\Test(\mms)$,
\begin{align*}
&\int_\mms \big\langle\rmd(g\d f), \rmd(g'\d f')\big\rangle\d\meas + \int_\mms\delta(g\d f)\,\delta(g'\d f')\d\meas\\
&\qquad\qquad = -\int_\mms \big\langle g\d \Delta f+ \Delta g\d f + 2\Hess f(\nabla g,\cdot), g' \d f'\big\rangle\d\meas. 
\end{align*}
By linearity of both sides in $g'\d f'$ and density, the first identity for $\Hodge(g\d f)$ then readily follows. Indeed, since $\delta(g\d f) = -\langle\nabla g,\nabla f\rangle - g\,\Delta f$ $\meas$-a.e.~by \eqref{Eq:delta = -div} --- with appropriate interpretation if $g\in\R\,\One_\mms$ according to \eqref{Eq:div nabla = Delta} --- $\delta (g\d f)$ belongs to $\F$ by \autoref{Pr:Compatibility}, and we have
\begin{align*}
&\int_\mms \delta(g\d f)\,\delta(g'\d f')\d\meas\\
&\qquad\qquad = -\int_\mms\big\langle\rmd\big[\langle\nabla g,\nabla f\rangle + g\,\Delta f\big],g'\d f'\big\rangle\d\meas\\
&\qquad\qquad = -\int_\mms \big\langle \!\Hess g(\nabla f,\cdot) + \Hess f(\nabla g,\cdot) + \rmd g\,\Delta f + g\,\rmd \Delta f,g'\d f'\big\rangle\d\meas,
\end{align*}
with $\Hess g := 0$ whenever $g\in \R\,\One_\mms$. On the other hand, by \autoref{Th:Wd12 properties} we have $\rmd(g\d f)=\rmd g\wedge\rmd f$, which  belongs to $\smash{\Dom(\delta^2)}$ by \autoref{Le:delta of d}  with
\begin{align*}
&\int_\mms \big\langle \rmd(g\d f),\rmd(g'\d f')\big\rangle\d\meas\\
&\qquad\qquad = \int_\mms \big\langle\delta(\rmd g\wedge\rmd f),g'\d f'\big\rangle\d\meas\\
&\qquad\qquad = \int_\mms \big\langle \Delta f\d g - \Delta g\d f - [\nabla g,\nabla f]^\flat,g'\d f'\big\rangle\d\meas.
\end{align*}
Adding up these two identities yields the claim since
\begin{align*}
[\nabla g,\nabla f]^\flat = \Hess f(\nabla g,\cdot) - \Hess g(\nabla f,\cdot).
\end{align*}

Concerning the second claim, from what we already proved, the linearity of $\Hodge$ and since $\Test(\mms)\cup \R\,\One_\mms$ is an algebra, it follows that $h\,X^\flat\in\Dom(\Hodge)$. Again by linearity, it thus suffices to consider the case $X := g\,\nabla f$, $g\in\Test(\mms)\cup\R\,\One_\mms$ and $f\in\Test(\mms)$. Indeed,
\begin{align*}
\Hodge(h\,g\d f) &= -h\,g\d f - \Delta(h\,g)\d f - 2\Hess f(\nabla(h\,g),\cdot)\\
&= -h\,g\d \Delta f - g\,\Delta h \d f - 2 \langle\nabla h,\nabla g\rangle\d f - h\,\Delta g\d f\\
&\qquad\qquad - 2\,g\Hess f(\nabla h,\cdot) - 2\,h\Hess f(\nabla g,\cdot)\\
&= h\,\Hodge(g\d f) - \Delta h\,(g\d f) -\smash{2\,\big[\nabla_{\nabla h}(g\,\nabla f)\big]^\sharp.}
\end{align*}
In the last identity, we used the identity
\begin{align*}
\nabla(g\,\nabla f) = \nabla g\otimes \nabla f + g\,(\Hess f)^\sharp
\end{align*}
inherited from \autoref{Th:Hess properties}.
\end{proof}

Recall from \autoref{Pr:Generators cotangent module} that $\nabla\,\Test(\mms)$, a set consisting of $\meas$-essentially bounded elements, generates $\smash{\Ell^2(T\mms)}$ in the sense of $\Ell^\infty$-modules. Hence, for every $A\in\Ell^2(T^{\otimes 2}\mms)$, by \eqref{Eq:Duality formula symm part I} its symmetric part obeys the duality formula
\begin{align}\label{Eq:Duality formula symm part}
\begin{split}
\big\vert A_\sym\big\vert_\HS^2 &= \esssup\!\Big\lbrace 2 A : \sum_{j=1}^m\nabla h_j\otimes \nabla h_j - \Big\vert\!\sum_{j=1}^m \nabla h_j\otimes\nabla h_j\Big\vert_\HS^2 : \\
&\qquad\qquad m\in\N,\ h_1,\dots,h_m\in\Test(\mms)\Big\rbrace.
\end{split}
\end{align}
This is crucial in the next \autoref{Le:Pre.Bochner}. Therein,  the divergence of $X\in\Reg(T\mms)$ is understood in the sense of \autoref{Def:L2 div}. Recall that if $X\in\Test(T\mms)$, this is the same as interpreting it according to \autoref{Def:Measure-valued divergence} by \autoref{Pr:Div comp}.

\begin{lemma}\label{Le:Pre.Bochner} For every $X\in\Reg(T\mms)$, we have $\vert X\vert^2\in\Dom(\DELTA^{2\kappa})$ with
\begin{align*}
\DELTA^{2\kappa}\frac{\vert X\vert^2}{2} \geq \Big[\big\vert \nabla X\big\vert_\HS^2 -\big\langle X,(\Hodge X^\flat)^\sharp\big\rangle\Big]\,\meas.
\end{align*}
\end{lemma}

\begin{proof} The Leibniz rule for $\smash{\DELTA^{2\kappa}}$ \cite[Cor.~6.3]{erbar2020} together with polarization and the linearity of $\smash{\DELTA^{2\kappa}}$ on $\F$, recall in particular \autoref{Pr:Bakry Emery measures}, ensure that $\vert X\vert^2\in\F$, and in fact $\smash{\vert X\vert^2\in\Dom(\DELTA^{2\kappa})}$.

Write $X := g_1\,\nabla f_1 + \dots + g_n\,\nabla f_n$ for certain $g_i \in \Test(\mms)\cup\R\,\One_\mms$ and $f_i\in \Test(\mms)$, $i\in \{1,\dots,n\}$. Retaining the notation from \autoref{Le:Extremely key lemma} for $N':= \infty$, we first claim the identity
\begin{align}\label{Eq:rho1 claim}
\rho_1[f,g] = \DELTA^{2\kappa}\frac{\vert X\vert^2}{2} + \Big[\big\langle X,(\Hodge X^\flat)^\sharp\big\rangle - \big\vert(\nabla X)_\asym\big\vert_\HS^2 \Big]\,\meas.
\end{align}
Again by the Leibniz rule for $\DELTA^{2\kappa}$ from \cite[Cor.~6.3]{erbar2020} and \autoref{Pr:Product rule for gradients},
\begin{align*}
\DELTA^{2\kappa}\frac{\vert X\vert^2}{2} &= \frac{1}{2}\sum_{i,i'=1}^n \widetilde{g_i}\,\widetilde{g}_{i'}\,\DELTA^{2\kappa}\langle\nabla f_i,\nabla f_{i'}\rangle + \sum_{i,i'=1}^n \big\langle\nabla [g_i\,g_{i'}],\nabla\langle\nabla f_i,\nabla f_{i'}\rangle\big\rangle\,\meas\\
&\qquad\qquad + \frac{1}{2}\sum_{i,i'=1}^n \langle\nabla f_i,\nabla f_{i'}\rangle\,\Delta[g_i\,g_{i'}]\,\meas\\
&= \frac{1}{2}\sum_{i,i'=1}^n \widetilde{g_i}\,\widetilde{g}_{i'}\,\DELTA^{2\kappa}\langle\nabla f_i,\nabla f_{i'}\rangle\\
&\qquad\qquad +\sum_{i,i'=1}^n \big[\!\Hess f_i(\nabla f_{i'},\nabla [g_i\,g_{i'}]) + \Hess f_{i'}(\nabla f_i,\nabla [g_i\,g_{i'}])\big]\,\meas\\
&\qquad\qquad + \sum_{i,i'=1}^n \big[g_{i'}\,\Delta g_i\,\langle\nabla f_i,\nabla f_{i'}\rangle + \langle\nabla g_i,\nabla g_{i'}\rangle\,\langle\nabla f_i,\nabla f_{i'}\rangle\big]\,\meas\\
&= \frac{1}{2}\sum_{i,i'=1}^n \widetilde{g_i}\,\widetilde{g}_{i'}\,\DELTA^{2\kappa}\langle\nabla f_i,\nabla f_{i'}\rangle\\
&\qquad\qquad +2\sum_{i,i'=1}^n \big[g_i\Hess f_i(\nabla f_{i'},\nabla g_{i'}) + g_i\Hess f_{i'}(\nabla f_i,\nabla g_{i'})\big]\,\meas\\
&\qquad\qquad + \sum_{i,i'=1}^n \big[g_{i'}\,\Delta g_i\,\langle\nabla f_i,\nabla f_{i'}\rangle + \langle\nabla g_i,\nabla g_{i'}\rangle\,\langle\nabla f_i,\nabla f_{i'}\rangle\big]\,\meas.
\end{align*}
Furthermore, by \autoref{Le:Hodge test} we obtain
\begin{align*}
\Hodge X^\flat = -\sum_{i=1}^n \big[g_i\d\Delta f_i + \Delta g_i\d f_i + 2\Hess f_i(\nabla g_i,\cdot)\big],
\end{align*}
which entails
\begin{align*}
\big\langle X,(\Hodge X^\flat)^\sharp\big\rangle &= -\sum_{i,{i'}=1}^n \big[g_{i'}\,\Delta g_i\,\langle\nabla f_i,\nabla f_{i'}\rangle + 2\,g_{i'}\Hess f_i(\nabla g_i,\nabla f_{i'})\big]\\
&\qquad\qquad -\sum_{i,i'=1}^n \big[g_i\,g_{i'}\,\langle\nabla \Delta f_i,\nabla f_{i'}\rangle\big]\quad\meas\text{-a.e.}
\end{align*}
Next, by item \ref{La:Cov 5} in \autoref{Th:Properties W12 TM}, the symmetry of the Hessian by \autoref{Th:Hess properties}, and the definition of the anti-symmetric part of an element in $\Ell^2(T^{\otimes 2}\mms)$,
\begin{align*}
(\nabla X)_\asym &= \frac{1}{2}\sum_{i=1}^n \big[\nabla g_i\otimes\nabla f_i - \nabla f_i\otimes \nabla g_i\big]
\end{align*}
from which we get
\begin{align*}
\big\vert (\nabla X)_\asym\big\vert_\HS^2 &= \frac{1}{2}\sum_{i,i'=1}^n \big[\langle \nabla f_i,\nabla f_{i'}\rangle\,\langle \nabla g_i,\nabla g_{i'}\rangle - \langle\nabla f_i,\nabla g_{i'}\rangle\,\langle\nabla g_i,\nabla f_{i'}\rangle\big]\quad\meas\text{-a.e.}
\end{align*}
Lastly, the definition \eqref{Eq:Gamma_2 opertor} of the measure-valued  $\Gamma_2$-operator yields
\begin{align*}
\bdGamma_2^{2\kappa}(f_i, f_{i'}) =  \frac{1}{2}\,\DELTA^{2\kappa}\langle\nabla f_i,\nabla f_{i'}\rangle - \frac{1}{2}\big[\langle\nabla \Delta f_i,\nabla f_{i'}\rangle + \langle\nabla \Delta f_{i'},\nabla f_i\rangle\big]\,\meas
\end{align*} 
for every $i,i'\in\{1,\dots,n\}$. Patching terms together straightforwardly leads to \eqref{Eq:rho1 claim}.

Therefore, w.r.t.~$\meas$ we have
\begin{align*}
\Big[\DELTA^{2\kappa}\frac{\vert X\vert^2}{2}\Big]_\perp = \rho_1[f,g]_\perp \geq 0
\end{align*}
thanks to \autoref{Le:Extremely key lemma}. To finally prove the claimed inequality, it thus suffices to show that, setting $\delta^{2\kappa}\vert X\vert^2/2 := \rmd(\DELTA^{2\kappa}\vert X\vert^2/2)_\ll/\rmd \meas$,
\begin{align}\label{Eq:12345}
\delta^{2\kappa}\frac{\vert X\vert^2}{2} \geq \big\vert\nabla X\big\vert_\HS^2 - \big\langle X, (\Hodge X^\flat)^\sharp\big\rangle\quad\meas\text{-a.e.}
\end{align}
Indeed, having \eqref{Eq:rho1 claim} at our disposal, \autoref{Le:Extremely key lemma} implies  that for every $m\in\N$ and for every $h_1,\dots,h_m\in\Test(\mms)$,
\begin{align*}
\Big\vert \nabla X : \sum_{j=1}^m\nabla h_j\otimes\nabla h_j \Big\vert &\leq \Big[\delta^{2\kappa}\frac{\vert X\vert^2}{2} + \big\langle X,(\Hodge X^\flat)^\sharp\big\rangle - \big\vert(\nabla X)_\asym\big\vert_\HS^2\Big]^{1/2}\\
&\qquad\qquad \times \Big\vert\! \sum_{j=1}^m \nabla h_j\otimes\nabla h_j\Big\vert_\HS\quad\meas\text{-a.e.}
\end{align*}
Applying Young's inequality at the r.h.s.~and optimizing over $h_1,\dots,h_m\in\Test(\mms)$ according to \eqref{Eq:Duality formula symm part} yields
\begin{align*}
\big\vert(\nabla X)_\sym\big\vert_\HS^2 \leq \delta^{2\kappa}\frac{\vert X\vert^2}{2} + \big\langle X,(\Hodge X^\flat)^\sharp\big\rangle - \big\vert(\nabla X)_\asym\big\vert_\HS^2\quad\meas\text{-a.e.}
\end{align*}
which is the remaining claim \eqref{Eq:12345} by the decomposition \eqref{Eq:sym plus asym}.
\end{proof}

The following consequence of \autoref{Le:Pre.Bochner} is not strictly needed in \autoref{Th:Ricci measure}, but gives an idea about the reasoning for \autoref{Le:Inclusion} below.

\begin{corollary}\label{Cor:Integr Bochner II} For every $f\in \Dom(\Delta)$, we have $\vert\nabla f\vert\in\F$ and
\begin{align*}
\Ch_2(f) \leq \int_\mms (\Delta f)^2\d\meas - \big\langle \kappa \,\big\vert\,\vert\nabla f\vert^2\big\rangle.
\end{align*}
\end{corollary}

\begin{proof} Since $\Dom(\Delta)\subset \Dom_\reg(\Hess)$ by \autoref{Cor:Dom(Delta) subset W22} and $\nabla\,\Dom_\reg(\Hess)\subset H^{1,2}(T\mms)$ by \autoref{Th:Properties W12 TM}, \autoref{Le:Kato inequality} implies that $\vert\nabla f\vert\in\F$ whenever $f\in\Dom(\Delta)$. In particular, $\Ch_2(f)$ is finite and the pairing $\smash{\big\langle \kappa\,\big\vert\,\vert\nabla f\vert^2\big\rangle}$ is well-defined for every such $f$.

The claimed estimate for $f\in\Test(\mms)$ follows by integrating  \autoref{Le:Pre.Bochner} for $X := \nabla f\in\Reg(T\mms)$ over all of $\mms$, and then using \eqref{Eq:div nabla = Delta} and \eqref{Eq:Identities Gamma2 DELTA}. In the more general case $f\in\Dom(\Delta)$, let $(f_k)_{k\in\N}$ be a sequence in $\Test(\mms)$ as constructed in the proof of  \autoref{Cor:Dom(Delta) subset W22}, i.e.~which satisfies $\Ch_2(f_k)\to \Ch_2(f)$, $\Delta f_k \to \Delta f$ in $\Ell^2(\mms)$ and $\nabla f_k \to \nabla f$ in $\Ell^2(T\mms)$ as $k\to\infty$. The first and the third convergence,  with \autoref{Th:Properties W12 TM}, imply that $\nabla f_k\to \nabla f$ in $H^{1,2}(T\mms)$ as $k\to\infty$. Hence
\begin{align*}
\lim_{k\to\infty}\big\langle\kappa\,\big\vert\,\vert\nabla f_k\vert\big\rangle = \big\langle\kappa\,\big\vert\,\vert\nabla f\vert\big\rangle
\end{align*}
by \autoref{Cor:Kappa cont}. The conclusion follows easily.
\end{proof}

For \autoref{Le:Inclusion}, but also in \autoref{Sub:Dim-free Ric} and \autoref{Th:Vector Bochner} below, we shall need the subsequent \autoref{Le:Epsilon lemma}. (Recall \autoref{Sub:Tamed spaces} for the well-definedness of $\Ch^{q\kappa}$ for $\kappa\in \Kato_{1-}(\mms)$ and $q\in [1,2]$.)  \autoref{Cor:Kappa X^2 identity} is then deduced along the same lines as  \autoref{Pr:Fin tot var} after setting $q= 1$  and letting $\varepsilon \to 0$ in \autoref{Le:Epsilon lemma}. 

\begin{lemma}\label{Le:Epsilon lemma} Let $X\in\Reg(T\mms)$,  and let $\phi\in\Dom(\Delta^{q\kappa})\cap\Ell^\infty(\mms)$ be nonnegative with $\Delta^{q\kappa}\phi\in\Ell^\infty(\mms)$. Given any $\varepsilon > 0$, we define $\varphi_\varepsilon\in\Cont^\infty([0,\infty))$ by $\varphi_\varepsilon(r) := (r+\varepsilon)^{q/2}-\varepsilon^{q/2}$. Then for every $q\in[1,2]$,
\begin{align*}
&\int_\mms \big[\varphi_\varepsilon\circ \vert X\vert^2\big]\,\Delta^{q\kappa}\phi\d\meas\\
&\qquad\qquad \geq -2\int_\mms \phi\,\big[\varphi_\varepsilon'\circ\vert X\vert^2\big]\,\big\langle X,(\Hodge X^\flat)^\sharp\big\rangle\d\meas\\
&\qquad\qquad\qquad\qquad  + 2\,\big\langle \kappa \,\big\vert\,\phi\,\vert X\vert^2\,\varphi_\varepsilon'\circ\vert X\vert^2\big\rangle - q\,\big\langle\kappa\,\big\vert\,\phi\,\varphi_\varepsilon\circ\vert X\vert^2\big\rangle.\!\!\!\textcolor{white}{\int_\mms}
\end{align*}
\end{lemma}

\begin{proof} According to our choice of $q$, we have $2r\,\varphi_\varepsilon''(r) \geq -\varphi_\varepsilon'(r)$ for every $r\geq 0$. Recall from \autoref{Le:Kato inequality} and \autoref{Th:Properties energy measure} that $\phi\,\varphi_\varepsilon'\circ\vert X \vert^2\in \F_\bounded$. Therefore, by \eqref{Eq::E^k identity}, \autoref{Th:Properties energy measure},  \autoref{Le:Pre.Bochner} and \autoref{Le:Kato inequality} we get
\begin{align*}
&\frac{1}{2}\int_\mms \big[\varphi_\varepsilon\circ \vert X\vert^2\big]\,\Delta^{q\kappa}\phi\d\meas - \big\langle \kappa\,\big\vert\, \phi\,\vert X\vert^2\,\varphi_\varepsilon'\circ\vert X\vert^2\big\rangle\\
&\qquad\qquad= - \frac{1}{2}\int_\mms \big[\varphi_\varepsilon'\circ\vert X\vert^2\big]\,\big\langle\nabla \phi,\nabla \vert X\vert^2\big\rangle\d\meas\\
&\qquad\qquad\qquad\qquad - \big\langle \kappa\,\big\vert\,\phi\,\big[q(\varphi_\varepsilon\circ \vert X\vert^2)/2 + \vert X\vert^2\, \varphi_\varepsilon'\circ\vert X\vert^2\big]\big\rangle\textcolor{white}{\int_\mms}\\
&\qquad\qquad = -\frac{1}{2}\,\Ch^{2\kappa}\big(\phi\,\varphi_\varepsilon'\circ \vert X\vert^2, \vert X\vert^2\big) + \frac{1}{2}\int_\mms \phi\,\big[\varphi_\varepsilon''\circ\vert X\vert^2\big]\,\big\vert\nabla \vert X\vert^2\big\vert^2\d\meas\\
&\qquad\qquad\qquad\qquad - \frac{q}{2}\,\big\langle \kappa\,\big\vert\,\phi\,\varphi_\varepsilon\circ \vert X\vert^2\big\rangle\textcolor{white}{\int_\mms}\\
&\qquad\qquad\geq \int_\mms \phi\,\big[\varphi_\varepsilon'\circ\vert X\vert^2\big]\,\big\vert \nabla X\big\vert_\HS^2\d\meas - \int_\mms \phi\,\big[\varphi_\varepsilon'\circ\vert X\vert^2\big]\,\big\langle X, (\Hodge X^\flat)^\sharp\big\rangle\d\meas\\
&\qquad\qquad\qquad\qquad + 2\int_\mms \phi\,\big[\varphi''\circ\vert X\vert^2\big]\,\vert X\vert^2\,\big\vert\nabla \vert X\vert\big\vert^2\d\meas - \frac{q}{2}\,\big\langle\kappa\,\big\vert\,\phi\,\varphi_\varepsilon\circ \vert X\vert^2\big\rangle\\
&\qquad\qquad \geq - \int_\mms \phi\,\big[\varphi_\varepsilon'\circ\vert X\vert^2\big]\,\big\langle X, (\Hodge X^\flat)^\sharp\big\rangle\d\meas  - \frac{q}{2}\,\big\langle\kappa\,\big\vert\,\phi\,\varphi_\varepsilon\circ \vert X\vert^2\big\rangle.
\end{align*}
Multiplying this inequality by $2$ terminates the proof.
\end{proof}

\begin{remark}\label{Re:Epsilon Remark} Let \autoref{As:Restr} below hold. Then, after partial integration of the Hodge Laplacian term and approximation, the conclusion of \autoref{Le:Epsilon lemma}, with $\kappa$ replaced by $\kappa_n$, $n\in\N$, holds under the  more general hypothesis $q = 2$, $\varepsilon = 0$, $\smash{X\in \Dom(\Hodge)^\sharp}$ and $\phi\in\Test_{\Ell^\infty}(\mms)$. See also \autoref{Cor:GrnsSchnsII} below.
\end{remark}

\begin{corollary}\label{Cor:Kappa X^2 identity} For every $X\in\Reg(T\mms)$,
\begin{align*}
\DELTA^{2\kappa}\vert X\vert^2[\mms] = -2\,\big\langle\kappa\,\big\vert\,\vert X\vert^2\big\rangle.
\end{align*}
\end{corollary}

We will have to identify $1$-forms in $\smash{H^{1,2}(T^*\mms)}$ with their vector field counterparts while additionally retaining their respective first order regularities in \autoref{Th:Ricci measure}. This is discussed now. In some sense, \autoref{Le:Inclusion} can be seen as an analogue of Gaffney's inequality in \autoref{Re:Gaffney} under curvature lower bounds. 

\begin{definition}\label{Def:Hsharp} We define the space $\smash{H_\sharp^{1,2}(T\mms)}$ as the image of $H^{1,2}(T^*\mms)$ under the map $\sharp$, endowed with the norm
\begin{align*}
\Vert X \Vert_{H_\sharp^{1,2}(T\mms)} := \Vert X^\flat\Vert_{H^{1,2}(T^*\mms)}.
\end{align*}
\end{definition}

\begin{lemma}\label{Le:Inclusion} $\smash{H_\sharp^{1,2}(T\mms)}$ is a subspace of $H^{1,2}(T\mms)$. The aforementioned natural inclusion is continuous. Additionally, for every $\smash{X\in H_\sharp^{1,2}(T\mms)}$,
\begin{align*}
\Ch_\cov(X) \leq \Ch_\con(X^\flat) - \big\langle \kappa\,\big\vert\,\vert X\vert^2\big\rangle.
\end{align*}
\end{lemma}

\begin{proof}  Let $\rho'\in [0,1)$ and $\alpha'\in\R$ be as in \autoref{Le:Form boundedness} for $\mu:=  \kappa^-$. 

Clearly $\Reg(T\mms) = \Reg(T^*\mms)^\sharp \subset \smash{H_\sharp^{1,2}(T\mms)}$ as well as $\Reg(T\mms)\subset H^{1,2}(T\mms)$ by definition of the respective spaces. Moreover, the claimed inequality for $X\in \Reg(T\mms)$ follows after evaluating the inequality in \autoref{Le:Pre.Bochner} at $\mms$ and using \autoref{Le:Hodge test} and \autoref{Cor:Kappa X^2 identity}.

To prove that $\smash{H_\sharp^{1,2}(T^*\mms)\subset H^{1,2}(T\mms)}$ with continuous inclusion, we again first study what happens for $X\in \Reg(T\mms)$. By the previous step,  \autoref{Le:Form boundedness} for $\mu := \kappa^-$ together with  \autoref{Le:Kato inequality}, and \eqref{Eq:delta = -div},
\begin{align}\label{Eq:Concl}
\Ch_\cov(X) &\leq -\big\langle\kappa\,\big\vert\,\vert X\vert^2\big\rangle  + \Ch_\con(X^\flat) \\
&\leq \big\langle\kappa^-\,\big\vert\,\vert X\vert^2\big\rangle +\Ch_\con(X^\flat)\nonumber\\
&\leq \rho'\,\Ch_\cov(X) + \alpha'\,\big\Vert X\big\Vert_{\Ell^2(T\mms)}^2 + \Ch_\con(X^\flat).\nonumber
\end{align}
Rearranging yields
\begin{align}\label{Eq:Equats}
\Ch_\cov(X) &\leq \frac{\alpha'}{1-\rho'}\,\big\Vert X\big\Vert_{\Ell^2(T\mms)}^2   + \frac{1}{1-\rho'}\,\Ch_\con(X^\flat).
\end{align}
This inequality extends to arbitrary $X \in \smash{H_\sharp^{1,2}(T^*\mms)}$ by applying it to all members of a sequence $(X_n)_{n\in\N}$ in $\Reg(T\mms)$ such that $X_n \to X$ in $\smash{H_\sharp^{1,2}(T\mms)}$ as $n\to\infty$ and using the $\Ell^2$-lower semicontinuity of $\Ch_\cov$ from \autoref{Th:Properties W12 TM}. In particular, as the r.h.s.~of \eqref{Eq:Equats} is continuous in $\smash{H_\sharp^{1,2}(T\mms)}$, both $\smash{H_\sharp^{1,2}(T\mms)\subset H^{1,2}(T\mms)}$ and the continuity of this inclusion follow. In particular, by \autoref{Cor:Kappa cont} the inequality \eqref{Eq:Concl} is stable under this limit procedure.
\end{proof}

\subsubsection{Main result} We now come to the main result of this article.

\begin{theorem}\label{Th:Ricci measure} There exists a unique continuous mapping $\smash{\RIC^\kappa\colon H_\sharp^{1,2}(T\mms)^2} \to \smash{\Meas_\fin^\pm(\mms)_\Ch}$ satisfying the identity
\begin{align}\label{Eq:Ric on test vfs}
\begin{split}
\RIC^\kappa(X,Y) &= \DELTA^{2\kappa}\frac{\langle X,Y\rangle}{2} + \Big[\frac{1}{2} \big\langle X,(\Hodge Y^\flat)^\sharp\big\rangle\\
&\qquad\qquad + \frac{1}{2} \big\langle Y,(\Hodge X^\flat)^\sharp\big\rangle - \nabla X:\nabla Y\Big]\,\meas
\end{split}
\end{align}
for every $X,Y\in\Reg(T\mms)$. The map $\smash{\RIC^\kappa}$ is symmetric and $\R$-bilinear. Furthermore, for every $\smash{X,Y\in H_\sharp^{1,2}(T\mms)}$, it obeys
\begin{align}\label{Eq:Ric kappa ids}
\begin{split}
\RIC^\kappa(X,X) &\geq 0,\\
\RIC^\kappa(X,Y)[\mms] &=\int_\mms \big[\langle\rmd X^\flat,\rmd Y^\flat\rangle +\delta X^\flat\,\delta Y^\flat\big]\d\meas\\
&\qquad\qquad  -\int_\mms \nabla X:\nabla Y \d\meas  -  \big\langle\kappa\,\big\vert\, \langle X,Y\rangle\big\rangle,\\
\big\Vert\RIC^\kappa(X,Y)\big\Vert_\TV^2 &\leq \,\big[\Ch_\con(X^\flat) + \big\langle \kappa\,\big\vert\,\vert X\vert^2\big\rangle\big]\, \big[\Ch_\con(Y^\flat) +\big\langle \kappa\,\big\vert\,\vert Y\vert^2\big\rangle\big].
\end{split}
\end{align}
\end{theorem}

\begin{proof} Given any $X,Y\in\Reg(T\mms)$, we define $\smash{\RIC^\kappa(X,Y)\in \Meas_\fin^\pm(\mms)_\Ch}$ by \eqref{Eq:Ric on test vfs}. This assignment is well-defined since $\langle X,Y\rangle \in \Dom(\DELTA^{2\kappa})$ by \autoref{Le:Pre.Bochner} and gives a nonnegative element. The map $\RIC^\kappa\colon \Reg(T\mms)^2\to \Meas_\fin^\pm(\mms)_\Ch$ defined in that way is clearly symmetric and $\R$-bilinear. 

Let $X$ and $Y$ as above. Then by \autoref{Le:Hodge test} and \autoref{Cor:Kappa X^2 identity},
\begin{align*}
\Vert\RIC^{\kappa}(X,X)\Vert_\TV &= \RIC^{\kappa}(X,X)[\mms]\\
&=\Ch_\con(X^\flat) - \Ch_\cov(X) - \big\langle \kappa\,\,\big\vert\, \vert X\vert^2\big\rangle\\
&\leq \Ch_\con(X^\flat) - \big\langle \kappa\,\big\vert\,\vert X\vert^2\big\rangle.
\end{align*}
By symmetry and $\R$-bilinearity of $\RIC^\kappa$, for every $\lambda\in\R$ we obtain
\begin{align*}
&\lambda^2\,\RIC^\kappa(X,X) + 2\lambda\,\RIC^\kappa(X,Y) + \RIC^\kappa(Y,Y)\\
&\qquad\qquad = \RIC^\kappa(\lambda\, X+Y,\lambda\,X+Y)\geq 0.
\end{align*}
\autoref{Le:Measure lemma} thus implies that
\begin{align*}
\Vert \RIC^\kappa(X,Y)\Vert_\TV &\leq \big[\Ch_\con(X^\flat) - \big\langle \kappa\,\big\vert\,\vert X\vert^2\big\rangle\big]^{1/2}\,\big[\Ch_\con(Y^\flat) - \big\langle\kappa\,\big\vert\,\vert Y\vert^2\big\rangle\big]^{1/2}.
\end{align*}
Since the r.h.s.~is jointly continuous in $X$ and $Y$ w.r.t.~convergence in $\smash{H_\sharp^{1,2}(T\mms)}$ by \autoref{Le:Inclusion} and \autoref{Cor:Kappa cont}, the above map $\RIC^\kappa$ extends continuously and uniquely to a (non-relabeled) map $\smash{\RIC^\kappa\colon H_\sharp^{1,2}(T\mms)^2\to \Meas_{\fin}^\pm(\mms)}$. 

In particular, the last inequality of \eqref{Eq:Ric kappa ids} directly comes as a byproduct of the previous argument, while the first inequality of \eqref{Eq:Ric kappa ids} follows by continuously extending \autoref{Le:Pre.Bochner} to any $\smash{X\in H_\sharp^{1,2}(T\mms)}$. The second identity for $X\in \Reg(T\mms)$ follows from \autoref{Le:Hodge test} and \autoref{Cor:Kappa X^2 identity} after integration by parts. Since both sides are jointly $\smash{H_\sharp^{1,2}}$-continuous in $X$ and $Y$ by  \autoref{Le:Inclusion} and \autoref{Cor:Kappa cont}, this identity easily extends to all $X,Y\in \smash{H_\sharp^{1,2}(T\mms)}$.
\end{proof}

\begin{remark} In the abstract setting of diffusion operators,  Sturm \cite{sturm2018} introduced a ``pointwise'', possibly dimension-dependent Ricci tensor. Although the smoothness assumptions are not really justified in our setting \cite[Rem.~3.6.6]{gigli2018}, it would be interesting to study the behavior of $\RIC^\kappa$ under conformal and drift transformations of $\langle\cdot,\cdot\rangle$ and $\meas$, respectively, as done in the abstract framework of  \cite{sturm2018}.
\end{remark}

\subsection{Curvature tensors from the $\kappa$-Ricci measure}\label{Sub:Curv tensors from RIC} Next, we separate  $\kappa$ from $\RIC^\kappa$ in \autoref{Th:Ricci measure} to define the \emph{Ricci curvature} and the \emph{second fundamental form} intrinsically defined by the data $(\mms,\Ch,\meas)$.

\subsubsection{Measure-valued Laplacian}\label{Sub:Measure valued Lapl} 

\begin{definition}\label{Def:Meas val Lapl} We define $\Dom(\DELTA)$ to consist of all $f\in \F$ for which $\nabla f\in \Dom_{L^2}(\DIV)$, in which case we define the\label{Not:Meas val Lapl} \emph{measure-valued Laplacian} of $f$ by
\begin{align*}
\DELTA f := \DIV\nabla f.
\end{align*}
\end{definition}

\begin{lemma}\label{Le:Fin tot var} The signed Borel measure $\smash{\widetilde{u}^2\,\kappa}$ has finite total variation for every $u\in\F$. In particular, if $u^2\in\Dom(\DELTA^{2\kappa})$, we have $u^2\in\Dom(\DELTA)$ with
\begin{align*}
\DELTA(u^2) = \DELTA^{2\kappa}(u^2) + 2\,\widetilde{u}^2\,\kappa.
\end{align*}
In particular $\DELTA\vert\nabla f\vert^2$ has finite total variation for every $f\in \Test(\mms)$.
\end{lemma}

\begin{proof} The first claim follows by observing that $\kappa\in\Kato_{1-}(\mms)$ if and only if $\vert\kappa\vert\in\Kato_{1-}(\mms)$, whence $\smash{\widetilde{u}^2\,\vert\kappa\vert}$ is a finite measure by \cite[Cor.~2.25]{erbar2020}. 

The remaining claims follow by direct computations using \autoref{Def:Measure valued Schr}, \eqref{Eq::E^k identity} and \autoref{Pr:Fin tot var}.
\end{proof}

\subsubsection{Dimension-free Ricci curvature}\label{Sub:Dim-free Ric} Keeping in mind \autoref{Th:Ricci measure}, we define the map $\smash{\RIC\colon H_\sharp^{1,2}(T\mms)^2 \to \Meas_{\fin}^\pm(\mms)_\Ch}$ by
\begin{align*}
\RIC(X,Y) := \RIC^\kappa(X,Y) + \langle X,Y\rangle_\sim\,\kappa.
\end{align*}
This map  is well-defined, symmetric and $\R$-bilinear --- indeed, by polarization, \autoref{Le:Inclusion} and \autoref{Le:Kato inequality},  $\langle X,Y\rangle_\sim\,\kappa\in \smash{\Meas_\fin^\pm(\mms)_\Ch}$ for every $X,Y\in\smash{H_\sharp^{1,2}(T\mms)}$. $\RIC$ is also jointly continuous, which follows from polarization, \autoref{Cor:Kappa cont} and \autoref{Le:Inclusion} again. Lastly, by \autoref{Th:Ricci measure},  \autoref{Le:Pre.Bochner} and \autoref{Le:Fin tot var}, for every $X,Y\in\Reg(T\mms)$,
\begin{align}\label{Eq:Ric IID}
\begin{split}
\RIC(X,Y)& = \DELTA \frac{\langle X,Y\rangle}{2} + \Big[\frac{1}{2}\big\langle X, (\Hodge Y^\flat)^\sharp\big\rangle\\
&\qquad\qquad + \frac{1}{2}\big\langle Y, (\Hodge X^\flat)^\sharp\big\rangle - \nabla X:\nabla Y\Big]\,\meas.
\end{split}
\end{align}
Of course, \eqref{Eq:Ric IID} is defined to recover the familiar \emph{vector Bochner formula}
\begin{align*}
\DELTA\frac{\vert X\vert^2}{2} + \big\langle X,(\Hodge X^\flat)^\sharp\big\rangle\,\meas = \RIC(X,X) + \big\vert\nabla X\big\vert_\HS^2\,\meas
\end{align*}
for $X\in\Reg(T\mms)$, which, setting $X := \nabla f$, $f\in\Test(\mms)$ and using \autoref{Le:Hodge test} as well as \autoref{Th:Properties W12 TM}, in turn reduces to the \emph{Bochner identity}
\begin{align*}
\DELTA\frac{\vert \nabla f\vert^2}{2} - \big\langle\nabla f,\nabla \Delta f\big\rangle\,\meas = \RIC(\nabla f,\nabla f) + \big\vert\!\Hess f\big\vert_\HS^2\,\meas.
\end{align*}

In general, $\RIC(X,X)$ is no longer nonnegative, but it is bounded from below by $\kappa$ in the following sense, which is a direct consequence of \eqref{Eq:Ric IID}, \eqref{Eq:Ric kappa ids} as well as the respective $\smash{H_\sharp^{1,2}}$-continuity of both sides of the resulting inequality.

\begin{proposition}\label{Pr:kappa lower bd} For every $\smash{X\in H_\sharp^{1,2}(T\mms)}$, the map $\RIC$ defined above obeys
\begin{align*}
\RIC(X,X) \geq \vert X\vert^2_\sim\,\kappa.
\end{align*}
\end{proposition}

It is unclear whether \eqref{Eq:Ric IID} or \eqref{Eq:Ric on test vfs} hold for general $\smash{X,Y\in H_\sharp^{1,2}(T\mms)}$, since we do not know whether $\langle X,Y\rangle\in \Dom(\DELTA)$ or $\langle X,Y\rangle\in \Dom(\DELTA^{2\kappa})$ in the respective situation. Still,  \eqref{Eq:Ric IID} makes sense in the subsequent weak form. See also  \autoref{Re:Weitzenböck} below.

\begin{lemma}\label{Le:Weak form} For every $\smash{X,Y\in H_\sharp^{1,2}(T\mms)}$ and every $f\in \Test(\mms)$,
\begin{align*}
\int_\mms \widetilde{f}\d\RIC(X,Y) &= \int_\mms \big[\big\langle\rmd X^\flat, \rmd(f\,Y^\flat)\big\rangle + \delta X^\flat\, \delta(f\,Y^\flat) - \nabla X : \nabla(f\,Y)\big]\d\meas\\
&= \int_\mms \big[\!\Hess f(X,Y) + \rmd f(X\div Y + Y\div X)\big]\d\meas\\
&\qquad\qquad + \int_\mms f\,\big[\big\langle\rmd X^\flat,\rmd Y^\flat\big\rangle + \delta X^\flat\,\delta Y^\flat - \nabla X:\nabla Y\big]\d\meas.
\end{align*}
\end{lemma}

\begin{proof} For a given $f\in\Test(\mms)$, all terms are continuous in $X$ and $Y$ w.r.t.~convergence in $\smash{H_\sharp^{1,2}(T\mms)}$. Hence, without restriction, we may and will assume in the sequel that $X,Y\in\Reg(T\mms)$. 

Under these assumptions, we first claim that
\begin{align}\label{Eq:Claim weak form}
\begin{split}
&\int_\mms \rmd X^\flat(\nabla f,Y)\d\meas\\
&\qquad\qquad = \int_\mms \big[\!-\!\langle X,Y\rangle \,\Delta f + \langle X,\nabla f\rangle\div Y- \big\langle X,[\nabla f,Y]\big\rangle\big]\d\meas.
\end{split}
\end{align}
If $Y\in \Test(T\mms)$ and $f$ is the product of two elements of $\Test(\mms)$, by \autoref{Th:Properties energy measure}, we have $\nabla f\in \Test(T\mms)$, and hence \eqref{Eq:Claim weak form} simply follows from the \autoref{Def:Exterior derivative} of the exterior derivative $\rmd$. To relax the assumption on $Y$, note that since $\langle X,\nabla f\rangle\in \F$ by \autoref{Le:Kato inequality}, by definition of the Lie bracket, \autoref{Th:Properties W12 TM}, \autoref{Pr:Compatibility} and \eqref{Eq:div nabla = Delta},
\begin{align*}
&\int_\mms \rmd X^\flat(\nabla f,Y)\d\meas\\
&\qquad\qquad = \int_\mms \big[\!-\!\langle X,Y\rangle\,\Delta f - \rmd\langle X,\nabla f\rangle(Y) - \nabla Y : (\nabla f\otimes X)\big] \d\meas\\
&\qquad\qquad\qquad\qquad + \int_\mms \Hess f(Y,X)\d\meas\\
&\qquad\qquad = \int_\mms \big[\!-\!\langle X,Y\rangle\,\Delta f - \rmd\langle X,\nabla f\rangle(Y) - \rmd\langle Y,X\rangle(\nabla f)\big]\d\meas\\
&\qquad\qquad\qquad\qquad + \int_\mms \big[\nabla X : (\nabla f\otimes Y) + \Hess f(Y,X)\big]\d\meas\\
&\qquad\qquad = \int_\mms \big[\!-\!\rmd\langle X,\nabla f\rangle(Y) + \nabla X : (\nabla f\otimes Y) + \Hess f(Y,X)\big]\d\meas.
\end{align*}
Both extremal sides of this identity are $\Ell^2$-continuous in $Y$, whence it extends to any $Y\in \Reg(T\mms)$. For such $Y$, the second and the third identity still hold --- compare with the above mentioned results --- and \eqref{Eq:Claim weak form} follows since $Y\in \Dom(\div)$ by \eqref{Eq:div nabla = Delta}, \eqref{Eq:Div lbnz rle} and \autoref{Le:Div g nabla f}, whence by \autoref{Def:L2 div},
\begin{align*}
-\int_\mms \rmd\langle X,\nabla f\rangle(Y)\d\meas = \int_\mms \langle X,\nabla f\rangle\div Y\d\meas.
\end{align*}
Similarly, the assumption on $f$ is relaxed now. Indeed, by \eqref{Eq:Claim weak form},
\begin{align*}
&\int_\mms \rmd X^\flat(\nabla f,Y)\d\meas\\
&\qquad\qquad = \int_\mms \big[\rmd \langle X,Y\rangle(\nabla f) - \langle X,\nabla f\rangle\div Y - \nabla Y : (\nabla f\otimes X)\big]\d\meas\\
&\qquad\qquad\qquad\qquad + \int_\mms \Hess f(Y,X)\d\meas\\
&\qquad\qquad = \int_\mms \big[\rmd \langle X,Y\rangle(\nabla f) - \langle X,\nabla f\rangle\div Y - \nabla Y : (\nabla f\otimes X)\big]\d\meas\\
&\qquad\qquad\qquad\qquad + \int_\mms \big[\rmd \langle \nabla f,X\rangle(Y) - \nabla X : (Y\otimes \nabla f)\big]\d\meas\\
&\qquad\qquad = \int_\mms \big[\rmd \langle X,Y\rangle(\nabla f) - \langle X,\nabla f\rangle\div Y - \nabla Y : (\nabla f\otimes X)\big]\d\meas\\
&\qquad\qquad\qquad\qquad - \int_\mms \big[\langle \nabla f, X\rangle\div Y + \nabla X : (Y\otimes \nabla f)\big]\d\meas.
\end{align*}
Here, we successively used that $\langle X,Y\rangle, \langle\nabla f, X\rangle\in \F$ by \autoref{Le:Kato inequality}, the definition of the Lie bracket, \autoref{Th:Properties W12 TM}, \autoref{Pr:Compatibility}, and that $Y\in \Dom(\div)$ by \eqref{Eq:div nabla = Delta}, \eqref{Eq:Div lbnz rle} and \autoref{Le:Div g nabla f}. Using \autoref{Le:Product convergence}, both extremal sides of this identity thus extend to general $f\in\Test(\mms)$. As above, the second and the third equality still hold under this assumption. Thus,  \eqref{Eq:Claim weak form} is proven in the desired generality.

We turn to the claim of the lemma and initially prove the second equality. First, using \autoref{Th:Wd12 properties},  \eqref{Eq:Claim weak form}, the definition of the Lie bracket and \autoref{Th:Properties W12 TM},
\begin{align*}
&\int_\mms \big\langle\rmd X^\flat, \rmd(f\,Y^\flat)\big\rangle\d\meas\\
&\qquad\qquad= \int_\mms \big[\rmd X^\flat(\nabla f, Y) + f\,\big\langle\rmd X^\flat,\rmd Y^\flat\big\rangle\big]\d\meas\\
&\qquad\qquad= \int_\mms \big[\!-\!\langle X,Y\rangle \,\Delta f + \langle X,\nabla f\rangle\div Y- \big\langle X,[\nabla f,Y]\big\rangle\big]\d\meas\\
&\qquad\qquad\qquad\qquad + \int_\mms f\,\big\langle\rmd X^\flat,\rmd Y^\flat\big\rangle\d\meas\\
&\qquad\qquad= \int_\mms \big[\!-\!\langle X,Y\rangle\,\Delta f + \langle X,\nabla f\rangle\div Y - \nabla Y : (\nabla f\otimes X)\big]\d\meas\\
&\qquad\qquad\qquad\qquad + \int_\mms \big[\!\Hess f(X,Y) + f\,\big\langle\rmd X^\flat, \rmd Y^\flat\big\rangle\big]\d\meas.
\end{align*}
Second, by \eqref{Eq:delta = -div}, \autoref{Le:Div identities} and \autoref{Le:Div g nabla f},
\begin{align*}
\int_\mms \delta X^\flat\,\delta(f\,Y^\flat)\d\meas &= \int_\mms \div X\, \div(f\,Y)\d\meas\\
&= \int_\mms \big[\!\div X\,\langle \nabla f,Y\rangle + f\div X\,\div Y\big]\d\meas.
\end{align*}
Third, by \autoref{Le:Leibniz rule W(21)},
\begin{align*}
-\int_\mms \nabla X: \nabla(f\,Y)\d\meas = -\int_\mms \big[\nabla X : (\nabla f\otimes Y) - f\,\nabla X:\nabla Y\big]\d\meas.
\end{align*}
Adding up these three identities and employing that
\begin{align*}
\int_\mms \big[\!- \!\langle X,Y\rangle\,\Delta f-\nabla Y : (\nabla f\otimes X) - \nabla X : (\nabla f\otimes Y)\big]\d\meas = 0
\end{align*}
thanks to \eqref{Eq:div nabla = Delta}, \autoref{Le:Kato inequality} and \autoref{Pr:Compatibility}, the second claimed equality in the lemma is shown.

To prove the first identity, denote the r.h.s.~of it by $\rmA(X,Y)$. By what we have proved above, it follows that $\rmA(X,Y) = \rmA(Y,X)$. Since $\smash{f\,X, f\,Y\in H_\sharp^{1,2}(T\mms)}$ by \autoref{Th:Wd12 properties}, by \autoref{Le:Fin tot var} and \eqref{Eq:Ric IID} we get
\begin{align*}
\int_\mms \widetilde{f}\d\RIC(X,Y) &= -\frac{1}{2}\int_\mms \big[\big\langle\nabla f,\nabla \langle X,Y\rangle\big\rangle + 2f\,\nabla X:\nabla Y\big]\d\meas\\
&\qquad\qquad + \frac{1}{2}\int_\mms \big[\big\langle \rmd(f\,X^\flat),\rmd Y^\flat\big\rangle + \delta(f\,X^\flat)\,\delta Y^\flat\big]\d\meas\\
&\qquad\qquad + \frac{1}{2}\int_\mms \big[\big\langle\rmd X^\flat,\rmd(f\,Y^\flat)\big\rangle + \delta X^\flat\,\delta(f\,Y^\flat)\big]\d\meas\\
&= - \frac{1}{2}\int_\mms \big[\nabla (f\,X) : \nabla Y + \nabla X : \nabla (f\,Y)\big]\d\meas\\
&\qquad\qquad + \frac{1}{2}\int_\mms \big[\big\langle \rmd(f\,X^\flat),\rmd Y^\flat\big\rangle + \delta(f\,X^\flat)\,\delta Y^\flat\big]\d\meas\\
&\qquad\qquad + \frac{1}{2}\int_\mms \big[\big\langle\rmd X^\flat,\rmd(f\,Y^\flat)\big\rangle + \delta X^\flat\,\delta(f\,Y^\flat)\big]\d\meas\\
&= \frac{1}{2}\,\big[\rmA(X,Y) + \rmA(Y,X)\big] = \rmA(X,Y).
\end{align*}
In the second step, we used \autoref{Pr:Compatibility} and then \autoref{Le:Leibniz rule W(21)}.
\end{proof}

\begin{lemma}\label{Le:Rausziehen} For every $X,Y\in H_\sharp^{1,2}(T\mms)$ and every $f\in\Test(\mms)$,
\begin{align*}
\RIC(f\,X,Y) = \widetilde{f}\,\RIC(X,Y).
\end{align*}
\end{lemma}

\begin{proof} By the $\smash{H_\sharp^{1,2}}$-continuity of both sides in $X$ and $Y$ (recall \autoref{Th:Wd12 properties} and \autoref{Le:delta of d}), we may and will assume without restriction that $X,Y\in\Reg(T\mms)$. Owing to \autoref{Le:Mollified heat flow}, it moreover suffices to prove that for every $g\in\Test(\mms)$,
\begin{align}\label{Eq:Rausziehen}
\int_\mms \widetilde{g}\d\RIC(f\,X,Y) = \int_\mms \widetilde{g}\,\widetilde{f}\d\RIC(X,Y).
\end{align}

As in the proof of \autoref{Le:Weak form} above, by \eqref{Eq:Ric IID} and \autoref{Le:Hodge test} we have
\begin{align*}
\int_\mms \widetilde{g}\d\RIC(f\,X,Y) &= \frac{1}{2}\int_\mms \big[\Delta g\, f\,\langle X,Y\rangle + g\,f\,\big\langle X, (\Hodge Y^\flat)^\sharp\big\rangle\big]\d\meas\\
&\qquad\qquad + \frac{1}{2}\int_\mms \big[g\,\big\langle Y, (\Hodge (f\,Y^\flat))^\sharp\big\rangle - 2g\,\nabla (f\,X):\nabla Y\big]\d\meas,\\
\int_\mms \widetilde{g}\,\widetilde{f}\d\RIC(X,Y) &= \frac{1}{2}\int_\mms \big[\Delta(g\,f)\,\langle X,Y\rangle + g\,f\,\big\langle X, (\Hodge Y^\flat)^\sharp\big\rangle\big]\d\meas\\
&\qquad\qquad +\frac{1}{2}\int_\mms \big[g\,f\,\big\langle Y, (\Hodge X^\flat)^\sharp\big\rangle - 2g\,f\,\nabla X:\nabla Y\big]\d\meas.
\end{align*}
Using \autoref{Le:Delta Leibniz rule} and \autoref{Le:Leibniz rule W(21)} yields
\begin{alignat*}{3}
\Delta(g\,f) &= g\,\Delta f + 2\,\langle\nabla g,\nabla f\rangle + f\,\Delta g & & \meas\text{-a.e.},\\
\nabla(f\,X) : \nabla Y &=f\,\nabla X : \nabla Y + (\nabla f\otimes X):\nabla Y &\quad & \meas\text{-a.e.},
\end{alignat*}
while \autoref{Le:Hodge test} ensures that
\begin{align*}
(\Hodge(f\,X^\flat))^\sharp =f\,(\Hodge X^\flat)^\sharp - \Delta f\,X - 2 \,\nabla_{\nabla f}X.
\end{align*}
Lastly, since $\smash{\langle X,Y\rangle\in \F_\bounded}$ by \autoref{Le:Kato inequality}, $\langle X,Y\rangle\,\nabla f \in\Dom_\TV(\DIV)\cap \Dom(\div)$ with $\norm(\langle X,Y\rangle\,\nabla f)=0$ by \autoref{Le:Div g nabla f}. Together with \autoref{Pr:Compatibility}, this yields
\begin{align*}
\int_\mms \langle\nabla g,\nabla f\rangle\,\langle X,Y\rangle\d\meas &= -\int_\mms g\div\!\big(\langle X,Y\rangle\,\nabla f\big)\d\meas\\
&= -\int_\mms \big[g\,\Delta f\,\langle X,Y\rangle + \nabla X : (\nabla f\otimes Y)\big]\d\meas\\
&\qquad\qquad +\int_\mms \nabla Y :(\nabla f\otimes X)\d\meas.
\end{align*}
Using the last four identities, a term-by-term comparison in the above identities for both sides of \eqref{Eq:Rausziehen} precisely yields \eqref{Eq:Rausziehen}.
\end{proof}

Following  \autoref{Subsub:Measure spaces}, denote by $\smash{\RIC_\ll\colon H_\sharp^{1,2}(T\mms)^2\to \Meas_\fin^\pm(\mms)_\Ch}$ the continuous map --- recall \eqref{Eq:TV under singularity} --- which assigns to $\RIC(X,Y)$ its $\meas$-absolutely continuous part $\RIC_\ll(X,Y) := \RIC(X,Y)_\ll$, $\smash{X,Y\in H_\sharp^{1,2}(T\mms)}$.

\begin{definition}\label{Def:Ric} Either of the maps $\RIC_\ll$ or $\ric\colon \smash{H_\sharp^{1,2}(T\mms)^2\to \Ell^1(\mms)}$ given by 
\begin{align*}
\ric(X,Y) := \frac{\rmd \RIC_\ll(X,Y)}{\rmd\meas}
\end{align*}
is called \emph{Ricci curvature} of $(\mms,\Ch,\meas)$.
\end{definition}

The advantage of this definition is threefold. First, the map $\ric$ extends to all of $\Ell^2(T\mms)^2$ in a sense made precise in \autoref{Pr:Ric extension} below in the finite-dimensional framework. Second, this separation of $\RIC$ into $\meas$-absolutely continuous and $\meas$-singular parts allows us to define the second fundamental form from the latter in \autoref{Def:Second fund form} below. Third, it makes it possible to say that the negative part of the ``lowest eigenvalue'' of $\ric$ satisfies the extended Kato condition according to  \autoref{Pr:Kato prop}. To this aim, define $\ric_*\colon \mms\to [-\infty,\infty]$ by
\begin{align*}
\ric_* := \essinf\!\big\lbrace \vert X\vert^{-2}\,\ric(X,X) : X\in H_\sharp^{1,2}(T\mms),\ \vert X\vert \leq 1\ \meas\text{-a.e.}\big\rbrace.
\end{align*}

\begin{proposition}\label{Pr:Kato prop} The measure $\ric_*^-\,\meas$ belongs to $\Kato_{1-}(\mms)$.
\end{proposition}

\begin{proof} By \cite[Rem.~2.8]{erbar2020}, $\ric_*^-\,\meas\in \Kato_{1-}(\mms)$ if and only if the function $\ric_*^-$ belongs to the extended functional Kato class of $\mms$ \cite[Def.~2.20]{erbar2020}. The proof of the latter will in particular yield $\ric_*^- \in \Ell_\loc^1(\mms)$ and $\ric_*^-\,\meas\in\Meas_\sigmafin^+(\mms)_\Ch$ as a byproduct.

Let $\kappa = \kappa^+-\kappa^-$ be the Jordan decomposition of $\kappa$ with $\kappa^+,\kappa^-\in \smash{\Meas_\sigmafin^+(\mms)}$. Since $\smash{\sfa^{\vert\kappa\vert}_t \geq \sfa^{\kappa^-}_t}$ for every $t\geq 0$, we have $\kappa^-\in\Kato_{1-}(\mms)$ and $\rmd\kappa^-/\rmd\meas\,\meas\in\Kato_{1-}(\mms)$ by \eqref{Eq:TV under singularity}. Hence, to prove the claim it is sufficient to prove that $\ric_*^-\leq \rmd\kappa^-/\rmd\meas$ $\meas$-a.e. Indeed, given any $X\in\smash{H_\sharp^{1,2}(T\mms)}$, by definition of $\ric(X,X)$ and \autoref{Th:Ricci measure},
\begin{align*}
\ric(X,X) \geq \vert X\vert^2\,\frac{\rmd \kappa}{\rmd\meas} \geq -\vert X\vert^2\,\frac{\rmd\kappa^-}{\rmd \meas}\quad\meas\text{-a.e.}
\end{align*}
In particular $\ric_*^- < \infty$ $\meas$-a.e., and thus
\begin{align*}
\ric_*^- &\leq \esssup\!\big\lbrace \vert X\vert^{-2}\,\ric(X,X)^- : X\in H_\sharp^{1,2}(T\mms),\ \vert X\vert \leq 1 \ \meas\text{-a.e.}\big\rbrace\\
&\leq \frac{\rmd\kappa^-}{\rmd\meas}\quad\meas\text{-a.e.}\qedhere
\end{align*}
\end{proof}

\subsubsection{Dimension-dependent Ricci tensor}\label{Sub:Dimensional Ricci tensor} Following the $\RCD^*(K,N)$-treatise \cite{han2018}, $K\in\R$, and motivated by \cite{bakry1985b}, we shortly outline the definition of an $N$-Ricci tensor on $\BE_2(\kappa,N)$ spaces for $N \in [1,\infty)$. The latter condition is assumed to hold for $(\mms,\Ch,\meas)$ throughout this subsection. Details are left to the reader.

Keeping in mind \autoref{Pr:Upper bound local dimension}, let $\smash{(E_n)_{n\in\N\cup\{\infty\}}}$ be the dimensional decomposition of $\Ell^2(T\mms)$ and define the function $\dim_\loc \colon\mms\to \{1,\dots,\lfloor N\rfloor\}$ by
\begin{align*}
\dim_\loc := \One_{E_1} + 2\,\One_{E_2} + \dots + \lfloor N\rfloor\,\One_{\lfloor N\rfloor}.
\end{align*}
Arguing as for \cite[Prop.~4.1]{han2018}, we see that for every $\smash{X\in H_\sharp^{1,2}(T\mms)}$, if $N\in\N$ then
\begin{align*}
\tr \nabla X = \div X\quad\meas\text{-a.e.}\quad\text{on }E_N
\end{align*}
according to the definition \eqref{Eq:Pointwise trace} of the trace of a generic $A\in\Ell^2(T^{\otimes 2}\mms)$. The function $\smash{\rmR_N\colon H_\sharp^{1,2}(T\mms)^2\to \Ell^1(\mms)}$ defined by
\begin{align*}
\rmR_N(X,Y) := \begin{cases} \displaystyle\frac{\big[\!\tr\nabla X - \div X\big]\,\big[\!\tr\nabla Y - \div Y\big]}{N-\dim_\loc} & \text{if } \dim_\loc < N,\\
0 & \text{otherwise}
\end{cases}
\end{align*}
is thus well-defined. In fact, the assignment $(X,Y) \mapsto \rmR_N(X,Y)\,\meas$ is continuous as a map from $\smash{H_\sharp^{1,2}(T\mms)^2}$ into $\smash{\Meas_\fin^\pm(\mms)_\Ch}$ thanks to \autoref{Le:Inclusion}.

\begin{definition} The map $\smash{\RIC_N\colon H_\sharp^{1,2}(T\mms)^2 \to \Meas_\fin^\pm(\mms)_\Ch}$ given by
\begin{align*}
\RIC_N(X,Y) := \RIC(X,Y) - \rmR_N(X,Y)\,\meas 
\end{align*}
is henceforth called \emph{$N$-Ricci tensor} of $(\mms,\Ch,\meas)$.
\end{definition}

\begin{theorem}\label{Th:BEkappa n ric bounds} For every $X\in H_\sharp^{1,2}(T\mms)$,
\begin{align*}
\RIC_N(X,X) &\geq \vert X\vert_\sim^2\,\kappa,\\
\DELTA\frac{\vert X\vert^2}{2} + \big\langle X,(\Hodge X^\flat)^\sharp\big\rangle &\geq \RIC_N(X,X) + \frac{1}{N}\,\vert\!\div X\vert^2\,\meas.
\end{align*}
\end{theorem}

\begin{proof} We only outline the main differences to the proof of \cite[Thm.~4.3]{han2018}. Set
\begin{align*}
\RIC_N^\kappa(X,Y) := \RIC^\kappa(X,Y) -\rmR_N(X,Y)\,\meas
\end{align*}
for $\smash{X,Y\in H_\sharp^{1,2}(T\mms)}$. By \autoref{Le:Pre.Bochner}, we see that $\smash{\RIC_N^\kappa(X,X)_\perp \geq 0}$, understood w.r.t.~$\meas$, for every such $X$. The nonnegativity of $\smash{\RIC_N^\kappa(X,X)_\ll}$ for $X := \nabla f$, $f\in\Test(\mms)$, is argued similarly to \cite[Thm.~3.3]{han2018} up to replacing $\bdGamma_2(f)$ therein by $\bdGamma_2^{2\kappa}(f)$, compare with \autoref{Le:Extremely key lemma}. By \autoref{Le:Rausziehen} and the definition of $\RIC$, it follows that $\RIC^\kappa$ is both $\R$- and $\Test$-bilinear. The same is true for $\rmR_N$ by \autoref{Le:Leibniz rule W(21)} and the definition \eqref{Eq:Pointwise trace} of the trace. Proceeding now as in the proof of \cite[Thm.~4.3]{han2018} implies the nonnegativity of $\smash{\RIC_N^\kappa(X,X)}$ for every $X\in\Reg(T\mms)$, and hence for every $\smash{X\in H_\sharp^{1,2}(T\mms)}$ by continuity. We conclude the first inequality from the definition of $\RIC$ again. The argument for the second is the same as in  \cite{han2018}.
\end{proof}

We finally turn to the existence proof of an appropriate extension of $\ric$ to $\Ell^2(T\mms)^2$ announced after \autoref{Def:Ric}. By $\Test$-bilinearity of $\RIC_N$ encountered in the above proof of \autoref{Th:BEkappa n ric bounds}, a similar statement holds for the map induced by $\rmd(\RIC_N)_\ll/\rmd\meas$ but is not written down for notational convenience.

The quite technical proof of the following preparatory result is the same as for \cite[Lem.~A.1, Thm.~A.2]{giglipasqu2020} --- up to replacing the norm $\smash{\Vert\cdot\Vert_{W^{1,2}(T\mms)}}$ by $\smash{\Vert\cdot\Vert_{H_\sharp^{1,2}(T\mms)}}$ --- and thus omitted.

\begin{lemma}\label{Le:Regular vfs generation} Denote by $(E_n)_{n\in\N\cup\{\infty\}}$ the dimensional decomposition of $\Ell^2(T\mms)$ according to \autoref{Th:Dimensional decomposition}. Then there exist $V_1,\dots,V_{\lfloor N\rfloor}\in \smash{H_\sharp^{1,2}(T\mms)}$ such that $\{V_1,\dots,V_n\}$ is a local basis of $\Ell^2(T\mms)$ on $E_n$ for every $n\in\{1,\dots,\lfloor N\rfloor\}$.
\end{lemma}

\begin{theorem}\label{Pr:Ric extension} There exists a unique symmetric and $\Ell^\infty$-bilinear assignment $\mathfrak{ric}\colon \Ell^2(T\mms)^2\to\Ell_\loc^1(\mms)$ whose restriction to $\smash{H_\sharp^{1,2}(T\mms)^2}$ coincides with $\ric$.
\end{theorem}

\begin{proof} We first define $\mathfrak{ric}(f\,X,Y)$ for given $\smash{X,Y\in H_\sharp^{1,2}(T\mms)}$ and $f\in\Ell^\infty(\mms)$. Let $(f_n)_{n\in\N}$ be a sequence in $\Test(\mms)$ such that $f_n \to f$ pointwise $\meas$-a.e.~as $n\to\infty$ as well as $\sup_{n\in\N}\Vert f_n\Vert_{\Ell^\infty}<\infty$. Since
\begin{align*}
&\int_\mms \big\vert \ric(f_n\,X,Y) - \ric(f_m\,X,Y)\big\vert\d\meas  = \int_\mms \vert f_n -f_m\vert\,\big\vert\ric(X,Y)\big\vert\d\meas
\end{align*}
by \autoref{Le:Rausziehen} for every $n,m\in\N$, $(\ric(f_n\,X,Y))_{n\in\N}$ is a Cauchy sequence in $\Ell^1(\mms)$ and hence converges to a limit denoted by $\mathfrak{ric}(f\,X,Y)\in\Ell^1(\mms)$. By an analogous argument, we see that this procedure does not depend on the chosen sequence in $\Test(\mms)$ with the above properties, and moreover $\mathfrak{ric}(f\,X,Y) = \ric(f\,X,Y)$ if $f\in\Test(\mms)$. The symmetry and bilinearity of $\ric$ thus straightforwardly induces a symmetric and $\Ell^\infty$-bilinear map $\smash{\mathfrak{ric}\colon \calW^2\to\Ell^1(\mms)}$. Here, $\calW\subset \Ell^2(T\mms)$ is the linear span of all elements of the form $f\,X$, $f\in \Ell^\infty(\mms)$ and $\smash{X\in H_\sharp^{1,2}(T\mms)}$.

Now we define $\mathfrak{ric}(X,Y)$ for general $X,Y\in\Ell^2(T\mms)$. Retaining the notation of \autoref{Le:Regular vfs generation}, let $n\in\N$, and let $C\in \Borel(\mms)$ be some subset of $E_n$ such that for some $f_1,\dots,f_n, g_1,\dots,g_n\in\Ell^\infty(\mms)$ vanishing $\meas$-a.e.~outside $C$,
\begin{align}\label{Eq:Repr X Y}
\begin{split}
\One_C\,X &= f_1\,V_1 + \dots + f_n\,V_n,\\
\One_C\,Y &= g_1\,V_1 + \dots + g_n\,V_n.
\end{split}
\end{align}
Locally, we then define
\begin{align*}
\One_{C}\,\mathfrak{ric}(X,Y) := \sum_{i,j=1}^n f_i\,g_j\,\mathfrak{ric}(V_i,V_j).
\end{align*}
This is a good definition that does not depend on the representations \eqref{Eq:Repr X Y} of $X$ and $Y$ on $C$ and $D$, respectively, and is easily seen to give rise to a (non-relabeled) map $\smash{\mathfrak{ric}\colon \Ell^2(T\mms)^2\to\Ell_\loc^1(\mms)}$ which has all desired properties.
\end{proof}

\begin{remark} By the local definition of $\mathfrak{ric}$ in \autoref{Pr:Ric extension}, we neither know if the integrability of $\mathfrak{ric}(X,Y)$ for given $X,Y\in \Ell^2(T\mms)$ can be improved, nor if $\mathfrak{ric}$ is actually continuous in an appropriate target topology.
\end{remark}

\subsubsection{Second fundamental form}\label{Sub:Second fund form} Similarly as for \autoref{Def:Ric}, let us denote by $\RIC_\perp\colon \smash{H_\sharp^{1,2}(T\mms)^2\to \Meas_\fin^\pm(\mms)_\Ch}$ the continuous map --- recall  \eqref{Eq:TV under singularity} --- assigning to $\RIC(X,Y)$ its $\meas$-singular part $\RIC_\perp(X,Y) :=\RIC(X,Y)_\perp$, $\smash{X,Y\in H_\sharp^{1,2}(T\mms)}$.

\begin{definition}\label{Def:Second fund form} The map $\smash{\II\colon H_\sharp^{1,2}(T\mms)^2\to \Meas_\fin^{\pm}(\mms)_\Ch}$ given by
\begin{align*}
\II(X,Y) := \RIC_\perp(X,Y)
\end{align*}
is called \emph{second fundamental form} of $(\mms,\Ch,\meas)$.
\end{definition}

In particular, if $X\in \Reg(T\mms)$, then by \eqref{Eq:Ric IID}, \autoref{Def:Meas val Lapl} and \autoref{Def:Normal component},
\begin{align}\label{Eq:Sec fund form normal cpt}
\II(X,X) = \DELTA_\perp\frac{\vert X\vert^2}{2} = \DIV_\perp \nabla \frac{\vert X\vert^2}{2} = -\norm \nabla\frac{\vert X\vert^2}{2}.
\end{align}
This observation complements the nonsmooth ``boundary'' discussion pursued so far. In particular, \eqref{Eq:Sec fund form normal cpt} is fully justified in the smooth context, as noted in the next example (and it only depends on the ``tangential parts'' of $X$ and $Y$ by definition of $\smash{H_\sharp^{1,2}(T\mms)}$, recall \autoref{Le:Div g nabla f}). See \cite[Ch.~2]{han2020} for similar computations.

\begin{example}\label{Ex:Sec fund form smooth} Let $\mms$ be a Riemannian manifold with boundary $\partial\mms$. The second fundamental form of $\partial\mms$ is the map $\bbI$ defined by
\begin{align*}
\bbI(X^\Vert,Y^\Vert) := \langle\nabla_{X^\Vert} \sfn,Y^\Vert\rangle_\jmath
\end{align*}
for $\smash{X^\Vert,Y^\Vert\in \Gamma(T\partial\mms)}$. According to \autoref{Sub:Riem mflds}, such an $\smash{X^\Vert}$  can be uniquely identified with $\smash{X\in \Gamma(T\mms)\big\vert_{\partial\mms}}$ such that $\langle X,\sfn\rangle=0$ on $\partial\mms$. We extend $X$ and $\sfn$ to  (non-relabeled) smooth vector fields defined on an open neighborhood of $\partial\mms$ \cite[Lem.~8.6]{lee2018}. Then by metric compatibility of the Levi-Civita connection on $\mms$, 
\begin{align*}
\bbI(X^\Vert,X^\Vert) &= -\langle\nabla_X X,\sfn\rangle + \big\langle \nabla X,\nabla \langle X,\sfn\rangle\big\rangle  = -\frac{1}{2}\big\langle\nabla\vert X\vert^2,\sfn\big\rangle\quad\text{on }\partial\mms.
\end{align*}
With \autoref{Ex:Mflds with boundary}, this shows that for every smooth, compactly supported and purely tangential $X\in \Gamma(T\mms)$ to which $\smash{X^\Vert\in \Gamma(T\partial\mms)}$ is uniquely associated,
\begin{align*}
\II(X,X) = \bbI(X^\Vert, X^\Vert)\,\surf.
\end{align*}
\end{example}

In line with \autoref{Ex:Sec fund form smooth}, we make the following bibliographical remark which, in fact, partly motivated \autoref{Def:Second fund form}.

\begin{remark} \autoref{Def:Ric} and \autoref{Def:Second fund form} together yield the identity
\begin{align*}
\RIC(X,X) = \ric(X,X)\,\meas + \II(X,X)
\end{align*}
for every $\smash{X\in H_\sharp^{1,2}(T\mms)}$. This identity --- and hence our definitions --- have a smooth evidence by \cite[Thm.~2.4]{han2020}. There it has been shown that on any smooth, connected Riemannian manifold $\mms$ with boundary  (with measure $\smash{\meas := \rme^{-2w}\,\vol}$, $w\in\Cont^2(\mms)$, so that $\smash{\surf = \rme^{-2w}\,\Haus^{d-1}\big\vert_{\partial\mms}}$), with $\RIC$ the  Ricci measure in the sense of \cite[Thm.~3.6.7]{gigli2018} --- and hence of \autoref{Sub:Dim-free Ric} --- we have
\begin{align*}
\RIC(\nabla f,\nabla f) = \Ric(\nabla f,\nabla f)\,\meas + \bbI(\nabla f^\Vert,\nabla f^\Vert)\,\surf
\end{align*}
for every $\smash{f\in\Cont_\comp^\infty(\mms)}$ with $\rmd f(\sfn) = 0$ on $\partial\mms$.
\end{remark}

\begin{remark}[Convexity of $\RCD$ spaces]\label{Re:Conv RCD c} In the novel  interpretation proposed by \autoref{Def:Second fund form}, on an $\RCD(K,\infty)$ space, $K\in\R$, \cite[Thm.~3.6.7]{gigli2018} implies that every such space is \emph{intrinsically convex} in the sense that
\begin{align}\label{Eq:Convexity}
\II(X,X) \geq 0
\end{align}
for every $\smash{X\in H_\sharp^{1,2}(T\mms)}$. In other words, every such space is necessarily \emph{convex} in the sense that \eqref{Eq:Convexity} holds for, say, every $X\in\Reg(T\mms)$. (This notion of convexity is frequently used in the smooth setting, see e.g.~\cite[Def.~1.2.2]{wang2014}.)

This should not be surprising from various perspectives. 

First, we know from \cite[Thm.~6.18]{ambrosio2014b} that \emph{geodesic convexity} of a subset $Y$ of $\mms$ is a \emph{sufficient} condition for it to naturally become again $\RCD(K,\infty)$ as soon as $\meas[\partial Y] = 0$ and $\meas[Y] > 0$. Through \eqref{Eq:Convexity} and \cite{han2020} we have thus provided a nonsmooth analogue of the fact that every, say, compact,  geodesically convex Riemannian manifold with boundary has nonnegative second fundamental form \cite[Lem.~61]{petersen2006}. (The converse, of course, does not hold in general, e.g.~for disks on the cylinder $\smash{\S^1 \times \R}$ with diameter larger than $\pi$.)

Second, recent results \cite{sturm2020} and examples \cite{wang2014} show that on nonconvex domains --- even if the boundary has arbitrarily small concavity --- in general, one cannot expect uniform lower Ricci bounds solely  described by relative entropies.
\end{remark}

\begin{remark}\label{Re:Also conc on interior sing} Our terminology of ``second fundamental form'' is of course leaned on the smooth case, where the singular part of $\RIC$ w.r.t.~the given volume measure is concentrated on the \emph{boundary} of $\mms$. However, it is worth pointing out that in general, $\II$ may be supported on \emph{interior} singularities --- and may not admit any boundary contribution at all --- as well. 

For example, let us consider the doubling $\smash{\mms := \hat{\mms}^+  \sqcup \hat{\mms}^- \sqcup \partial\hat{\mms}}$ of a (say compact)  Riemannian surface $\smash{\hat{\mms}}$ with boundary such that the curvature  of $\smash{\partial\hat{\mms}}$ is bounded from below by $1$, glued along $\smash{\partial\hat{\mms}}$. Here $\smash{\hat{\mms}^+}$ and $\smash{\hat{\mms}^-}$ are two copies of the interior of $\smash{\hat{\mms}}$. This space canonically becomes tamed  \cite[Thm.~7.17]{erbar2020}. The induced second fundamental form $\II$ is concentrated on $\smash{\partial\hat{\mms}}$ and satisfies $\smash{\II(X,X) \geq \vert X\vert^2_\sim\,\surf}$ for every $\smash{X\in H_\sharp^{1,2}(T\mms)}$, yet $\smash{\partial \mms =\emptyset}$.
\end{remark}

The second fundamental form really matters when one is tempted to derive a Weitzenböck formula from \autoref{Le:Weak form}. In other words, the latter should really be treated as an \emph{interior} identity, away from singularities  \cite[Cor.~21]{petersen2006}.

\begin{remark}[Weitzenböck identity]\label{Re:Weitzenböck}  If $X \in \Dom(\Hodge)^\sharp\cap \Dom(\Bochner)$ in \autoref{Le:Weak form}, from the latter we could deduce that
\begin{align*}
\int_\mms \widetilde{f}\d\RIC(X,Y) = \int_\mms f\,\big\langle Y, (\Hodge X^\flat)^\sharp + \Bochner X\big\rangle\d\meas
\end{align*}
for every $Y\in \smash{H_\sharp^{1,2}(T\mms)}$, which is strongly reminiscent of the Weitzenböck formula \cite[Cor.~21]{petersen2006}. In other words, $\RIC(X,Y) \ll \meas$ and
\begin{align}\label{Eq:ric abs cts}
\ric(X,Y) = \big\langle Y, (\Hodge X^\flat)^\sharp + \Bochner X\big\rangle\quad\meas\text{-a.e.},
\end{align}
which, if $Y\in\smash{\Dom(\Hodge)^\sharp\cap\Dom(\Bochner)}$ as well, especially implies the \emph{pointwise} symmetry
\begin{align*}
\big\langle Y, (\Hodge X^\flat)^\sharp + \Bochner X\big\rangle = \big\langle (\Hodge Y^\flat)^\sharp + \Bochner Y, X\big\rangle\quad\meas\text{-a.e.}
\end{align*}
The identity \eqref{Eq:ric abs cts} plays a crucial role in deriving the Feynman--Kac formula for the semigroup on differential $1$-forms on Riemannian manifolds, with or without boundary, as well as Bismut--Elworthy--Li formulas for $\rmd\ChHeat_tf$, $f\in\Ell^2(\mms)\cap\Ell^\infty(\mms)$ and $t>0$ \cite{bismut1984, elworthy1994, hsu2002, wang2014}. Still, we do not know whether $\Dom(\Hodge)^\sharp \cap \Dom(\Bochner) \neq \{0\}$.

Let us remark that for $\smash{X\in H_\sharp^{1,2}(T\mms)}$ to belong to both $\Dom(\Hodge)^\sharp$ and $\Dom(\Bochner)$, from \eqref{Eq:ric abs cts} one would necessarily have $\II(X,\cdot) = 0$. On a compact Riemannian manifold $\mms$ with boundary, this is underlined by comparison of the boundary conditions for $\Hodge$ and $\Bochner$. Indeed, let $X\in \Gamma(T\mms)$. By \eqref{Eq:Bdry cond Bochner}, recall that $X\in \Dom(\Bochner)$ means that
\begin{align*}
X^\perp &= 0,\\
(\nabla_\sfn X)^\Vert &=0
\end{align*}
on $\partial\mms$ according to \eqref{Eq:Normal parts vfs smooth world}. On the other hand, $X^\flat\in\Dom(\Hodge)$ entails absolute boundary conditions as in \autoref{Re:Abs bdry cond} for $X^\flat$ which, by \cite[Lem.~4.1]{hsu2002}, are equivalent to
\begin{align*}
X^\perp &= 0,\\
(\nabla_\sfn X)^\Vert - \mathbb{I}(X^\Vert,\cdot) &= 0
\end{align*}
at $\partial\mms$. Hence, we must have $\mathbb{I}(X^\Vert,\cdot) = 0$ on $\partial\mms$.
\end{remark}

\subsection{Vector Bochner inequality} The subsequent vector $q$-Bochner inequality is a direct consequence from  \autoref{Le:Epsilon lemma}, $q\in [1,2]$. For $\RCD(K,\infty)$ or $\RCD^*(K,N)$ spaces, $K\in\R$ and $N\in [1,\infty)$, it is due to \cite[Thm.~3.13]{braun2020} for $q=1$. Note that the assumption of \autoref{Th:Vector Bochner} is satisfied if $X\in\Test(T\mms)$ by \autoref{Le:Hodge test}.

Let $\Dom(\DELTA^{q\kappa})$ be defined w.r.t.~the closed form $\Ch^{q\kappa}$ as in \autoref{Def:Measure valued Schr}, $q\in [1,2]$.\label{Not:Meas valued Schr Ii}

\begin{theorem}\label{Th:Vector Bochner} Suppose that $X\in\Reg(T\mms)$ satisfies $\Hodge X^\flat \in \Ell^1(T^*\mms)$. Then for every $q\in [1,2]$, we have $\vert X\vert^q\in \Dom(\DELTA^{q\kappa})$ and
\begin{align*}
&\DELTA^{q\kappa}\frac{\vert X\vert^q}{q} + \vert X\vert^{q-2}\,\big\langle X, (\Hodge X^\flat)^\sharp\big\rangle\,\meas  \geq 0.
\end{align*}
\end{theorem}

\begin{proof} Letting $\varepsilon \to 0$ in \autoref{Le:Epsilon lemma} with Lebesgue's theorem yields
\begin{align*}
&\int_\mms \frac{\vert X\vert^q}{q}\,\Delta^{q\kappa}\phi\d\meas  \geq - \int_\mms \phi\,\vert X\vert^{q-2}\,\big\langle X,(\Hodge X^\flat)^\sharp\big\rangle \d\meas
\end{align*}
for every $\phi\in\Dom(\Delta^{q\kappa})\cap\Ell^\infty(\mms)$ with $\Delta^{q\kappa}\phi\in\Ell^\infty(\mms)$. Since the function on the r.h.s.~which involves $X$ belongs to $\Ell^1(\mms)\cap\Ell^2(\mms)$ --- and here is where we use that $\smash{\Hodge X^\flat\in\Ell^1(T^*\mms)}$ --- a variant of \cite[Lem.~6.2]{erbar2020} for $\Ch^{q\kappa}$ in place of $\Ch^{2\kappa}$ implies that $\vert X\vert^q\in\Dom(\DELTA^{q\kappa})$ with the desired inequality.
\end{proof}

\begin{remark} If $\kappa\in\Kato_0(\mms)$, the quadratic form $\Ch^{q\kappa}$ is well-defined and closed even for $q\in [2,\infty)$, and \autoref{Th:Vector Bochner} can be deduced along the same lines for this range of $q$ even without the assumption that $\smash{\Hodge X^\flat\in\Ell^1(T^*\mms)}$. Compare e.g.~with the functional treatise \cite[Ch.~3]{braun2021}.
\end{remark}

\subsection{Heat flow on 1-forms} A slightly more restrictive variant of \autoref{Th:Vector Bochner} yields functional inequalities for the heat flow $(\HHeat_t)_{t\geq 0}$ on $1$-forms, see \autoref{Th:HSU forms}. The latter is shortly introduced before, along with its basic properties. A thorough study of $(\HHeat_t)_{t\geq 0}$ on $\RCD(K,\infty)$ spaces, $K\in\R$, has been pursued in \cite{braun2020}.

\subsubsection{Heat flow and its elementary properties} Analogously to \autoref{Subsub:Neumann heat flow} and  \autoref{Sub:HF vector fields}, we define the \emph{heat flow} on $1$-forms as the semigroup $(\HHeat_t)_{t\geq 0}$ of bounded, linear and self-adjoint operators on $\Ell^2(T^*\mms)$ by
\begin{align}\label{Eq:Def HHt}
\HHeat_t := \rme^{-\Hodge t}.
\end{align}
It is associated \cite[Thm.~1.3.1]{fukushima2011} to the  functional $\smash{\widetilde{\Ch}_\con\colon \Ell^2(T^*\mms)\to [0,\infty]}$ with
\begin{align}\label{Eq:C con}
\widetilde{\Ch}_\con(\omega) := \begin{cases}
\displaystyle \int_\mms \big[\vert\rmd \omega\vert^2 + \vert\delta\omega\vert^2\big]\d\meas & \text{if }\omega\in H^{1,2}(T^*\mms),\\
\infty & \text{otherwise}.
\end{cases}
\end{align}

\begin{theorem} The subsequent properties of $(\HHeat_t)_{t\geq 0}$ hold for every $\omega\in\Ell^2(T^*\mms)$ and every $t>0$.
\begin{enumerate}[label=\textnormal{\textcolor{black}{(}\roman*\textcolor{black}{)}}]
\item The curve $t\mapsto \HHeat_t\omega$ belongs to $\Cont^1((0,\infty);\Ell^2(T^*\mms))$ with
\begin{align*}
\frac{\rmd}{\rmd t}\HHeat_t\omega = -\Hodge\HHeat_t\omega.
\end{align*}
\item If $\omega\in\Dom(\Hodge)$, we have
\begin{align*}
\frac{\rmd}{\rmd t}\HHeat_t\omega = -\HHeat_t\Hodge\omega.
\end{align*}
In particular, we have the identity
\begin{align*}
\Hodge\,\HHeat_t = \HHeat_t\,\Hodge\quad\text{on }\Dom(\Hodge).
\end{align*}
\item For every $s\in [0,t]$,
\begin{align*}
\Vert \HHeat_t\omega\Vert_{\Ell^2(T^*\mms)} \leq \Vert\HHeat_s\omega\Vert_{\Ell^2(T^*\mms)}.
\end{align*}
\item The function $t\mapsto\smash{\widetilde{\Ch}_\con(\HHeat_t\omega)}$ belongs to $\Cont^1((0,\infty))$, is nonincreasing, and its derivative satisfies
\begin{align*}
\frac{\rmd}{\rmd t}\widetilde{\Ch}_\con(\HHeat_t\omega) = - 2\int_\mms \big\vert\Hodge\HHeat_t\omega\big\vert^2\d\meas.
\end{align*}
\item If $\omega\in H^{1,2}(T^*\mms)$, the map $t\mapsto\HHeat_t\omega$ is continuous on $[0,\infty)$ w.r.t.~strong convergence in $H^{1,2}(T^*\mms)$.
\item We have
\begin{align*}
\widetilde{\Ch}_\con(\HHeat_t\omega) &\leq\frac{1}{2t}\,\big\Vert\omega\big\Vert_{\Ell^2(T^*\mms)}^2,\\
\big\Vert \Hodge\HHeat_t\omega\big\Vert_{\Ell^2(T^*\mms)}^2 &\leq \frac{1}{2t^2}\,\big\Vert \omega\big\Vert_{\Ell^2(T^*\mms)}^2.
\end{align*}
\end{enumerate}
\end{theorem}

Via the closedness of $\rmd$ from \autoref{Le:d closed} together with \autoref{Le:Hodge test}, the following \autoref{Le:Commutation relation Ht Pt} is verified. The spectral bottom inequality from \autoref{Cor:Spec bds} follows from the second identity of \eqref{Eq:Ric kappa ids}, \autoref{Le:Kato inequality}, \eqref{Eq::E^k identity} and Rayleigh's theorem.

\begin{lemma}\label{Le:Commutation relation Ht Pt} For every $f\in\F$ and every $t> 0$, $\rmd\ChHeat_tf \in\Dom(\Hodge)$ and
\begin{align*}
\HHeat_t\rmd f = \rmd\ChHeat_tf.
\end{align*}
\end{lemma}

\begin{remark} It is part of the statement of \autoref{Le:Commutation relation Ht Pt} that $\rmd\ChHeat_tf\in H^{1,2}(T^*\mms)$.  Indeed, if $f\in\F_\bounded$, we even have $\smash{\rmd\ChHeat_t f\in\Reg(T^*\mms)}$. Using that by \autoref{Th:Wd12 properties} and \autoref{Le:delta of d}, $\rmd(\rmd\ChHeat_t f)=0$  and $\delta(\rmd\ChHeat_tf) = -\Delta\ChHeat_t f$ for such $f$, the claim for general elements of $\F$ easily follows by truncation and \autoref{Th:Heat flow properties}.
\end{remark}

\begin{remark} Analogously to \eqref{Eq:Def HHt}, it is possible to define the heat flow $\smash{(\HHeat_t^k)_{t\geq 0}}$ in $\Ell^2(\Lambda^kT^*\mms)$ with generator $\smash{-\Hodge_k}$ for any $k\in\N$. However, it is not clear if the commutation relation from \autoref{Le:Commutation relation Ht Pt} holds  between $\smash{(\HHeat_t^k)_{t\geq 0}}$ and $\smash{(\HHeat_t^{k-1})_{t\geq 0}}$ for $k\geq 2$. Compare with \cite[Rem.~3.4]{braun2020}.
\end{remark}

	\begin{corollary} If $\omega\in\Dom(\delta)$ and $t> 0$, then $\HHeat_t\omega\in \Dom(\delta)$ with
		\begin{align*}
		\delta\HHeat_t\omega = \ChHeat_t\delta\omega.
		\end{align*}
	\end{corollary}

\begin{corollary}\label{Cor:Spec bds} We have
\begin{align*}
\inf\sigma(-\Delta^\kappa) \leq \inf\sigma(\Hodge).
\end{align*}
\end{corollary}

\subsubsection{Functional inequalities and $\Ell^p$-properties}\label{Sub:Funct inequ hf 1-forms} Unlike the results for $(\HHeat_t)_{t\geq 0}$ from  \cite{braun2020} for $\RCD(K,\infty)$  spaces, $K\in\R$, a collateral effect of the singular potential $\kappa$ is that we do not know how the domains $\Dom(\Delta)$ and $\Dom(\Delta^{2\kappa})$ are related. Compare with \autoref{Re:HSU forms?} below. We thus restrict ourselves to the following assumption throughout this subsection.

\begin{assumption}\label{As:Restr} In the framework of \autoref{As:Bakry Emery}, there exists $\mathcall{k}\in \Ell_\loc^1(\mms)$ in the functional extended Kato class of $\mms$ \cite[Def.~2.20]{erbar2020} which is uniformly bounded from below by some $K\in\R$, such that
\begin{align*}
\kappa = \mathcall{k}\,\meas.
\end{align*}
\end{assumption}

We define the sequence $(\kappa_n)_{n\in\N}$ in $\Kato_{1-}(\mms)$ by
\begin{align*}
\kappa_n := \mathcall{k}_n\,\meas
\end{align*}
with $\mathcall{k}_n := \min\{n,\mathcall{k}\}\in\Ell^\infty(\mms)$, $n\in\N$. Observe that $(\mms,\Ch,\meas)$ obeys $\BE_2(\kappa_n,N)$ for every $n\in\N$. A priori, the Schrödinger operator $\Delta^{2\kappa_n}$ is the \emph{form sum} \cite[p.~19]{faris1975} of $\Delta$ and $-2\mathcall{k}_n$, the latter being viewed as self-adjoint \cite[Thm.~1.7]{faris1975}  multiplication operator on $\Ell^2(\mms)$ with domain $\Dom(-2\kappa_n) := \{f\in \Ell^2(\mms) : \mathcall{k}_n\,f\in \Ell^2(\mms)\} = \Ell^2(\mms)$, $n\in\N$. In fact \cite[Prop.~3.1]{faris1975}, $\smash{\Delta^{2\kappa_n}}$ is an \emph{operator sum}, i.e.~$f\in \Dom(\Delta)$ if and only if $f\in\Dom(\Delta^{2\kappa_n})$ for every $n\in\N$, and for such $f$,
\begin{align}\label{Eq:Delta 2k Delta corresp}
\Delta^{2\kappa_n}f = \Delta f -2\mathcall{k}_n\,f\quad\meas\text{-a.e.}
\end{align}

\begin{proposition}\label{Pr:BE 1-forms} For every $\omega\in\Ell^2(T^*\mms)$ and every $t\geq 0$,
\begin{align*}
\vert\HHeat_t\omega\vert^2 \leq \Schr{2\kappa}_t\big(\vert\omega\vert^2\big)\quad\meas\text{-a.e.}
\end{align*}
\end{proposition}

\begin{proof} We only concentrate on the nontrivial part $t>0$. Let $\phi\in\Test_{\Ell^\infty}(\mms)$ be nonnegative, and define $F\colon [0,t]\to \R$ by
\begin{align*}
F(s) := \int_\mms\phi\,\Schr{2\kappa_n}_{t-s}\big(\vert\HHeat_s\omega\vert^2\big)\d\meas = \int_\mms \Schr{2\kappa_n}_{t-s}\phi\,\vert\HHeat_s\omega\vert^2\d\meas,
\end{align*}
where $n\in\N$. As in the proof of \autoref{Pr:BE vector fields}, we argue that  $F$ is locally absolutely continuous on $(0,t)$, and that for $\Leb^1$-a.e.~$s\in (0,t)$,
\begin{align*}
F'(s) = -\int_\mms \Delta^{2\kappa_n}\Schr{2\kappa_n}_{t-s}\phi\,\vert\HHeat_s\omega\vert^2\d\meas - 2\int_\mms\Schr{2\kappa_n}_{t-s}\phi\,\big\langle\HHeat_s\omega,\Hodge\HHeat_s\omega\big\rangle\d\meas.
\end{align*}

Given such an $s\in(0,t)$, using a mollified version of $(\ChHeat_t)_{t\geq 0}$ \cite[p.~1648]{savare2014} we construct a sequence $(f_i)_{i\in\N}$ of nonnegative functions in $\Test_{\Ell^\infty}(\mms)$ such that $(f_i)_{i\in\N}$ and $(\Delta f_i)_{i\in\N}$ are bounded in $\Ell^\infty(\mms)$, and $\smash{f_i\to \Schr{2\kappa_n}_{t-s}\phi}$ as well as $\smash{\Delta f_i \to \Delta \Schr{2\kappa_n}_{t-s}\phi}$ pointwise $\meas$-a.e.~as $i\to\infty$. By Lebesgue's theorem, \eqref{Eq:Delta 2k Delta corresp}, \autoref{Re:Epsilon Remark}  and \eqref{Eq:delta = -div},
\begin{align*}
F'(s) &= -\lim_{i\to\infty} \int_\mms \Delta^{2\kappa_n}f_i\,\vert\HHeat_s\omega\vert^2\d\meas \\
&\qquad\qquad -\lim_{i\to\infty} 2\int_\mms f_i\,\big\langle\HHeat_s\omega,\Hodge\HHeat_s\omega\big\rangle\d\meas\leq 0.
\end{align*}
Integrating this inequality from $0$ to $t$, employing the arbitrariness of $\phi$ and letting $n\to\infty$ via Levi's theorem readily provides the claimed inequality.
\end{proof}

As for \autoref{Cor:Bounded Lapl density}, we have the following consequence of \autoref{Pr:BE 1-forms}. A similar argument as for \autoref{Cor:GrnsSchnsII}, providing an extension of \autoref{Le:Epsilon lemma} beyond regular vector fields which is needed for the proof of \autoref{Th:HSU forms}, is due to \cite{braun2020}.

\begin{corollary}\label{Cor:GrnsSchns} For every $\omega\in \Dom(\Hodge)$, there exists a sequence $(\omega_n)_{n\in\N}$ in $\Dom(\Hodge)\cap\Ell^\infty(T^*\mms)$ which converges to $\omega$ in $H^{1,2}(T^*\mms)$ such that in addition, $\Hodge \omega_n \to\Hodge\omega$ in $\Ell^2(T^*\mms)$ as $n\to\infty$. If $\omega\in\Ell^\infty(T^*\mms)$, this sequence can be constructed to be uniformly bounded in $\Ell^\infty(T^*\mms)$.
\end{corollary}

\begin{lemma}\label{Cor:GrnsSchnsII} For every $n\in\N$, the conclusion from \autoref{Le:Epsilon lemma} holds for $\kappa$ replaced by $\kappa_n$ as well as $q=1$, $\varepsilon > 0$, $X\in \smash{\Dom(\Hodge)^\sharp\cap\Ell^\infty(T\mms)}$ and $\phi\in\Dom(\Delta)\cap \Ell^\infty(\mms)$.
\end{lemma}

\begin{proof} We shortly outline the argument. Let $\psi\in \F_\bounded$ be nonnegative, and let $(X_i)_{i\in\N}$ and $(\psi_j)_{j\in\N}$ be sequences in $\Reg(T\mms)$ and $\Test(\mms)$ converging to $X$ and $\psi$ in $\smash{H_\sharp^{1,2}(T\mms)}$ and $\F$, respectively. By \autoref{Le:Mollified heat flow}, we may and will assume that $\psi_j$ is nonnegative for every $j\in\N$. Integrating  \autoref{Le:Pre.Bochner}, for $\kappa$ replaced by $\kappa_n$,  $n\in\N$, against $\psi_j$ and using that $\vert X_i\vert^2\in\F_\bounded$ by \autoref{Le:Kato inequality}, for every $i,j\in\N$,
\begin{align*} 
&-\frac{1}{2}\int_\mms \big\langle\nabla \psi_j,\nabla \vert X_i\vert^2\big\rangle\d\meas - \int_\mms \mathcall{k}_n\,\psi_j\,\frac{\vert X_i\vert^2}{2}\d\meas\\
&\qquad\qquad \geq \int_\mms \psi_j\,\big\vert\nabla X_i\big\vert_\HS^2\d\meas - \int_\mms\psi_j\,\big\langle X_i, (\Hodge X_i^\flat)^\sharp\big\rangle\d\meas.
\end{align*}
Integrating by parts the last term, using \autoref{Pr:Compatibility} for the first, \autoref{Le:Inclusion} for the third and \autoref{Th:Wd12 properties} and \autoref{Le:delta of d} for the last term, and finally integrating by parts back the last term we send $i\to\infty$. This yields the previous inequality for $X_i$ replaced by $X$. Employing \autoref{Le:Inclusion} and  \autoref{Pr:Compatibility} again for the first term together with $X\in\Ell^\infty(T\mms)$ and $\nabla \psi_j\to\nabla\psi$ in $\Ell^2(T\mms)$ as $j\to\infty$, the above estimate still holds for $\psi_j$ replaced by $\psi$. Lastly, we insert $\psi := \phi\,[\varphi_\varepsilon'\circ\vert X\vert^2]$, $\varepsilon > 0$, where $\varphi_\varepsilon$ is defined as in \autoref{Le:Epsilon lemma} for $q=1$. The term containing $\varphi_\varepsilon''\circ\vert X\vert^2$ coming from the Leibniz rule in the first integral cancels out with the third integral thanks to \autoref{Le:Kato inequality}, and elementary further computations entail the claim.
\end{proof}

\autoref{Th:HSU forms} is known as \emph{Hess--Schrader--Uhlenbrock inequality} \cite{hess1977, hess1980,simon1977} in the case when $\mms$ is a compact Riemannian manifold without boundary. A similar, analytic access to the latter on such $\mms$ with \emph{convex} boundary is due to \cite{ouhabaz1999, shigekawa1997, shigekawa2000}. On general compact $\mms$ with boundary, it has been derived in \cite{hsu2002}  using probabilistic methods. In the noncompact case without boundary, one can  appeal to both analytic \cite{gueneysu2017} or stochastic \cite{driver2001,li1992} methods. Moreover, recently, $\Ell^p$-properties of $(\HHeat_t)_{t\geq 0}$ and related heat kernel estimates on Riemannian manifolds have been studied in \cite{magniez2020} under Kato curvature conditions.

\begin{theorem}[Hess--Schrader--Uhlenbrock inequality]\label{Th:HSU forms} For every $\omega\in \Ell^2(T^*\mms)$ and every $t\geq 0$,
\begin{align*}
\vert\HHeat_t\omega\vert\leq \Schr{\kappa}_t\vert\omega\vert\quad\meas\text{-a.e.}
\end{align*}
\end{theorem}

\begin{proof} Let $(\omega_l)_{l\in\N}$ be a sequence in $\Ell^1(T^*\mms) \cap\Ell^\infty(T^*\mms)$ which is obtained by appropriately cutting off and truncating the given $\omega$. (In this case, truncation means multiplication with an indicator function of $\{\vert \omega\vert \leq R\}$, $R>0$.) Moreover, given any $\varepsilon > 0$, define $\varphi_\varepsilon\in\Cont^\infty([0,\infty))$ by $\smash{\varphi_\varepsilon(r) := (r+\varepsilon)^{1/2}-\varepsilon^{1/2}}$. For a nonnegative  $\phi\in\Test_{\Ell^\infty}(\mms)$, given any $l\in\N$, consider the function $F_\varepsilon\colon [0,t]\to \R$ with
\begin{align*}
F_\varepsilon(s) := \int_\mms\phi\,\Schr{\kappa_n}_{t-s}\big(\varphi_\varepsilon\circ\vert\HHeat_s\omega_l\vert^2\big)\d\meas = \int_\mms\Schr{\kappa_n}_{t-s}\phi\,\big[\varphi_\varepsilon\circ\vert\HHeat_s\omega_l\vert^2\big]\d\meas.
\end{align*}
As for \autoref{Pr:BE 1-forms}, the function $F_\varepsilon$ is readily verified to be continuous on $[0,t]$, locally absolutely continuous on $(0,t)$, and integration and differentiation can be swapped in computing its derivative $F'_\varepsilon(s)$ at $\Leb^1$-a.e.~$s\in (0,t)$. 

Given such an $s\in (0,t)$, consider a sequence $(f_i)_{i\in\N}$  of nonnegative functions in $\Test_{\Ell^\infty}(\mms)$ associated to $\smash{\Schr{\kappa_n}_{t-s}\phi}$ as in the proof of \autoref{Pr:BE 1-forms}. Then, according to  \autoref{Cor:GrnsSchnsII} --- since $\HHeat_s\omega_l\in\Ell^\infty(T^*\mms)$ thanks to \autoref{Pr:BE 1-forms} ---   \eqref{Eq:div nabla = Delta} and Lebesgue's theorem,
\begin{align*}
F_\varepsilon'(s) &= -\int_\mms \Delta^{\kappa_n}\Schr{\kappa_n}_{t-s}\phi\,\big[\varphi_\varepsilon\circ \vert\HHeat_s\omega_l\vert^2\big]\d\meas\\
&\qquad\qquad -2\int_\mms \Schr{\kappa_n}_{t-s}\phi\,\big[\varphi_\varepsilon'\circ\vert\HHeat_s\omega_l\vert^2\big]\,\big\langle \HHeat_s\omega_l,\Hodge\HHeat_s\omega_l\big\rangle\d\meas\\
&= -\lim_{i\to\infty} \int_\mms \Delta^{\kappa_n}f_i\,\big[\varphi_\varepsilon\circ \vert\HHeat_s\omega_l\vert^2\big]\d\meas\\
&\qquad\qquad -\lim_{i\to\infty}2\int_\mms f_i\,\big[\varphi_\varepsilon'\circ\vert\HHeat_s\omega_l\vert^2\big]\,\big\langle \HHeat_s\omega_l,\Hodge\HHeat_s\omega_l\big\rangle\d\meas\\
&\leq  - \lim_{i\to\infty} 2\,\big\langle \kappa_n\,\big\vert\,f_i\,\vert\HHeat_s\omega_l\vert^2\,\varphi_\varepsilon\circ\vert\HHeat_s\omega_l\vert^2\big\rangle + \lim_{i\to\infty} \big\langle\kappa_n\,\big\vert\,f_i\,\varphi_\varepsilon\circ\vert\HHeat_s\omega_l\vert^2\big\rangle\!\!\!\textcolor{white}{\int}\\
&=  - 2\,\big\langle\kappa_n\,\big\vert\, \Schr{\kappa_n}_{t-s}\phi\,\vert\HHeat_s\omega_l\vert^2\,\varphi_\varepsilon'\circ\vert\HHeat_s\omega_l\vert^2\big\rangle + \big\langle\kappa_n\,\big\vert\,\Schr{\kappa_n}_{t-s}\phi\,\varphi_\varepsilon\circ\vert\HHeat_s\omega_l\vert^2\big\rangle.\!\!\!\textcolor{white}{\int}
\end{align*}
Integrating this inequality from $0$ to $t$, sending $\varepsilon\to 0$ with the aid of Lebesgue's theorem and employing the arbitrariness of $\phi$ imply that, for every $l,n\in\N$,
\begin{align*}
\vert\HHeat_t\omega_l\vert\leq\Schr{\kappa_n}_t\vert\omega_l\vert  \quad\meas\text{-a.e.}
\end{align*}
Sending $l\to\infty$ and $n\to\infty$ using Levi's theorem  terminates the proof.
\end{proof}

\begin{remark}\label{Re:HSU forms?} Technical issues in general  prevent us from proving \autoref{Pr:BE 1-forms} or \autoref{Th:HSU forms} beyond \autoref{As:Restr}. The key reason is that integrated versions or inequalities derived from \eqref{Eq:Ric on test vfs} and \eqref{Eq:Ric kappa ids} are hard to obtain beyond $X\in\Reg(T\mms)$ or integrands both belonging to $\Test(\mms)$ and $\smash{\Dom(\Delta^{2\kappa})}$. 

It is outlined in \autoref{Re:Epsilon Remark} above how to obtain  more general versions under \autoref{As:Restr}. The key obstacle, however, lies in dealing with the behavior of the term in \eqref{Eq:Ric on test vfs} containing the Hodge Laplacian or, in other words, to obtain an analogue to  \autoref{Cor:GrnsSchnsII}. In general, $\Ell^\infty$-bounds for derivatives of $\smash{\Schr{2\kappa}_{t-s}\phi}$ or $\smash{\Schr{\kappa}_{t-s}\phi}$, $s\in (0,t)$, lack for sufficiently many nonnegative $\phi\in\Ell^2(\mms)$, but being able to integrate by parts this term essentially requires  e.g.~$\smash{\Schr{2\kappa}_{t-s}\phi \in \Ell^\infty(\mms)}$ and $\smash{\rmd \Schr{2\kappa}_{t-s}\phi\in\Ell^\infty(T^*\mms)}$. (A related question is whether and when not only $\smash{\Schr{2\kappa}_{t-s}\phi, \Delta^{2\kappa}\Schr{2\kappa}_{t-s}\phi\in\Ell^\infty(\mms)}$ --- which can always be achieved by  \cite[Sec.~6.1]{erbar2020} --- but also $\smash{\rmd\Schr{2\kappa}_{t-s}\phi\in \Ell^\infty(T^*\mms)}$ holds.) Compare with \autoref{Pr:Leibniz rule ext der}, \autoref{Re:Integr issues}, \eqref{Eq:delta = -div} and \autoref{Le:Div identities}. We also cannot leave the Hodge Laplacian term as it is because we do not know if $\Reg(T^*\mms)$ is dense in $\Dom(\Hodge)$ w.r.t.~the induced graph norm. Under \autoref{As:Restr}, these deductions could still be done thanks to the explicit relation \eqref{Eq:Delta 2k Delta corresp} between $\smash{\Dom(\Delta^{2\kappa})}$ and $\Dom(\Delta)$.
\end{remark}

\begin{remark} If we know that, given $\omega\in\Ell^2(T^*\mms)$, there exists a sequence $(\omega_n)_{n\in\N}$ in $\Ell^2(T^*\mms)$ that $\Ell^2$-converges to $\omega$ such that $\HHeat_t\omega_n\in\Ell^\infty(T^*\mms)$ for every $t>0$ and every $n\in\N$, then \autoref{Th:HSU forms} can be deduced by the same arguments as above for more general $\smash{\kappa\in\Kato_{1-}(\mms)}$. In particular, since $\HHeat_t\rmd f_n \in\Ell^\infty(T^*\mms)$ for every $t>0$ and every $n\in\N$,  $f_n := \max\{\min\{f,n\},-n\}$, by \autoref{Le:Commutation relation Ht Pt}, \autoref{Th:HSU forms} recovers the gradient estimate \cite[Thm.~6.9]{erbar2020} for any $f\in\F$ by \eqref{Eq:delta = -div}. 

On  Riemannian manifolds with not necessarily convex boundary, such a sequence can be constructed under further geometric assumptions \cite[Thm.~5.2]{arnaudon2017}. (The latter result, in fact, implies  \autoref{Th:HSU forms} by Gronwall's inequality.)
\end{remark}

\appendix

\section{Extrinsic approaches}\label{Ch:Appendix}

Lastly, we compare some recent \emph{extrinsic} approaches \cite{brue2019,buffa2019, sturm2020} to boundary objects on $\RCD$ spaces with our intrinsic notions. More precisely, we outline some links to our notions of divergences and normal components.

\subsection{Sets of finite perimeter} Let $(\mms,\met,\meas)$ be a locally compact $\RCD(K,\infty)$ space, $K\in\R$, with induced Dirichlet space $(\mms,\Ch,\meas)$, cf.~\autoref{Ex:mms}. 

\subsubsection{Identification of the measure-valued divergence} Following \cite[Def.~4.1]{buffa2019}, let $\mathcall{D}\mathcall{M}^p(\mms)$, $p\in[1,\infty]$, be the space of all $X\in \Ell^p(T\mms)$ such that there exists $\smash{\textit{div}\,X\in \Meas_\fin^\pm(\mms)_\Ch}$ such that for every $h\in\Lip_\bs(\mms)$,
\begin{align*}
-\int_\mms h\d\textit{div}\,X = \int_\mms \rmd h(X)\d\meas.
\end{align*}
The density of $\Lip_\bs(\mms)$ in $W^{1,2}(\mms)$ \cite{ambrosio2014a} and \cite[Prop.~4.6]{buffa2019} yield the following.

\begin{lemma}\label{Le:Div identification BCM} Every $X\in \mathcall{D}\mathcall{M}^2(\mms)$ belongs to $\Dom(\DIV)$, and
\begin{align*}
\textit{div}\,X = \DIV X.
\end{align*}
\end{lemma}

\subsubsection{Gauß--Green formula}\label{Sub:GG formula} For boundary objects to really appear, we use the Gauß--Green formulas for appropriate subsets $E\subset\mms$ obtained in \cite{buffa2019}.

We say that a Borel set $E\subset\mms$ has \emph{finite perimeter} \cite[Def.~3.3]{buffa2019} if $\One_E\in \BV(\mms)$. It is associated with a Radon measure $\smash{\vert\rmD\One_E\vert\in \Meas_\finR^+(\mms)}$ \cite[Thm.~3.4]{buffa2019} which is supported on $\partial E$ and, if $\meas[\partial E] = 0$, in particular singular to $\meas$ \cite[Rem.~3.5]{buffa2019}. Here, the class of functions of \emph{bounded variation} $\BV(\mms)\subset\Ell^1(\mms)$ can be defined in various ways \cite{ambrosio2014, buffa2019, miranda2003} which all lead to the same spaces and objects in a large generality \cite[Thm.~1.1]{ambrosio2014}. (In particular, we require sets of finite perimeter to have finite $\meas$-measure, although this is not strictly needed \cite[Def.~1.1, Def.~1.2]{brue2019}.)

We now make the following assumptions on $E$. 
\begin{enumerate}[label=\alph*.]
\item\label{La:A} $E$ satisfies the obstructions from \autoref{Re:Subsets}.
\item\label{La:B} The inclusion $\smash{W^{1,2}(\mms)\big\vert_E \subset W^{1,2}(E)}$ from \eqref{Eq:W12 inclusions subsets} is dense.
\item\label{La:C} $E$ or $E^\rmc$ is a set of finite perimeter.
\end{enumerate}
Item \ref{La:A}~guarantees that $(E,\met_E,\meas_E)$ induces a quasi-regular, strongly local Dirichlet space $(E,\Ch_E,\meas_E)$, hence a tangent module $\Ell^2(TE)$ w.r.t.~$\meas_E$. By \eqref{Eq:W12 inclusions subsets}, we can identify $\smash{\Ell^2(T\mms)\big\vert_E}$ with $\Ell^2(TE)$. Given any $X\in\Ell^2(T\mms)$, denote by $X_E\in\Ell^2(TE)$ the image of $\One_E\,X$ under this identification. Of course, \ref{La:B} is satisfied if $E$ has the extension property $\smash{W^{1,2}(\mms)\big\vert_E = W^{1,2}(E)}$, and if $\met_E\leq C\,\met$ on $E^2$ for some finite $C>1$. For a  different variant of this condition \ref{La:B}, see \autoref{Sub:Nonsmooth example} below.

\begin{proposition}\label{Pr:GG sets finite per} For every $X\in\Test(T\mms)$, there exists a unique $\langle X,v_E\rangle_{\partial E}\in \Ell^\infty(\partial E,\vert\rmD\One_E\vert)$ such that for every $h\in W^{1,2}(E)$,
\begin{align}\label{Eq:Zvbd}
-\int_E\rmd h(X) \d\meas = \int_E h\,\frac{\rmd\textit{div}\,X}{\rmd \meas}\d\meas + \int_{\partial E}\widetilde{h}\,\big\langle X,v_E\big\rangle_{\partial E}\d\vert\rmD\One_E\vert.
\end{align}
In particular, we have $X_E\in\Dom_\TV(\DIV_E)$ with
\begin{align*}
\div_E X_E &= \frac{\rmd\textit{div}\,X}{\rmd\meas}\quad\meas_E\text{-a.e.},\\
\norm_E X_E &= - \big\langle X,v_E\big\rangle_{\partial E}\,\vert\rmD\One_E\vert.
\end{align*}
\end{proposition}

\begin{proof} The last statement follows from \eqref{Eq:Zvbd}, whence we concentrate on \eqref{Eq:Zvbd}. By \autoref{Le:Div g nabla f} and \autoref{Le:Div identification BCM}, we have $\vert\textit{div}\,X\vert\ll\meas$. (And furthermore, $\textit{div}\,X$ has finite total variation.) Hence, under \ref{La:C}, \cite[Prop.~6.11, Thm.~6.13]{buffa2019} implies that there exists a unique function $\langle X,v_E\rangle_{\partial E}\in\Ell^\infty(\partial E,\vert\rmD\One_E\vert)$ such that \eqref{Eq:Zvbd} holds for every $h\in\Lip_\bs(\mms)$. By \ref{La:B}, $\Ch$-quasi-uniform approximation \cite[Thm.~1.3.3]{chen2012} and \cite[Prop.~4.6]{buffa2019}, the latter extends to arbitrary $h\in W^{1,2}(E)$.
\end{proof}

\begin{remark}\label{Re:SFF} It would naturally follow from \autoref{Pr:GG sets finite per} that the second fundamental form of $E$ at $X_E$, $X\in\Test(T\mms)$, according to \autoref{Def:Second fund form} --- once $(E,\Ch_E,\meas_E)$ is  tamed by an appropriate $\kappa\in\Kato_{1-}(\mms)$  --- is
\begin{align*}
\II_E(X_E,X_E)  = \frac{1}{2}\big\langle\nabla \vert X\vert^2, v_E\big\rangle_{\partial E}\,\vert\rmD\One_E\vert
\end{align*}
\emph{provided} $\nabla \vert X\vert^2\in \Ell^\infty(T\mms)$. This latter assumption, however, seems to be quite restrictive. (Compare with \cite[Rem.~4.10, Rem.~4.11, Rem.~4.13]{giglipasqu2020}.) Unfortunately, boundedness of the vector field in \eqref{Eq:Zvbd} seems to be essential in \cite{buffa2019}. 
\end{remark}

\begin{remark} The notation $\langle X,v_E\rangle_{\partial E}$, $X\in\Test(T\mms)$, in \autoref{Pr:GG sets finite per} is purely formal, in the sense that the authors of \cite{buffa2019} neither consider any ``tangent module'' with scalar product $\langle \cdot,\cdot\rangle_{\partial E}$ over $\partial E$, nor define a unit normal vector field $v_E$.
\end{remark}

\begin{example}\label{Ex:SFP II} Another version of the Gauß--Green formula on $\RCD(K,N)$ spaces $(\mms,\met,\meas)$, $K\in\R$ and $N\in [1,\infty)$, has been obtained in \cite[Thm.~2.2]{brue2019}. Retain the assumptions \ref{La:A}, \ref{La:B} and \ref{La:C} on $E\subset\mms$. Then there exists a unique $v_E\in \smash{\Ell^2_E(T\mms)}$, the tangent module over $\partial E$ \cite[Thm.~2.1]{brue2019}, with $\vert v_E\vert=1$ $\vert \rmD \One_E\vert$-a.e.~on $\partial E$ such that for every $X\in H^{1,2}(T\mms)\cap \Dom(\div)\cap\Ell^\infty(T\mms)$,
\begin{align*}
\int_E \div X\d\meas = -\int_{\partial E} \big\langle\!\tr_E(X), v_E\big\rangle\d\vert\rmD\One_E\vert.
\end{align*}
Here $\smash{\tr_E\colon H^{1,2}(T\mms)\cap \Ell^\infty(T\mms)\to \Ell_E^2(T\mms)}$ is the \emph{trace operator} over $\partial E$.

Replacing $X$ by $h\,X$, where $h\in\W^{1,2}(\mms)\cap\Ell^\infty(\mms)$ has bounded support --- recall \autoref{Le:Div identities} and \autoref{Re:fX in H12} --- and using \eqref{Eq:Div lbnz rle} as well as the arbitrariness of $h$ we obtain that $X_E\in \Dom(\DIV_E)$ with
\begin{align*}
\div_EX_E &= \div X\quad\meas_E\text{-a.e.},\\
\norm_EX_E &= -\big\langle X, v_E\big\rangle\,\vert\rmd\One_E\vert.
\end{align*}
\end{example}

\subsection{Regular semiconvex subsets}\label{Sub:Nonsmooth example} Consider the canonical Dirichlet space induced by an $\RCD(\mathcall{k},N)$ metric measure space $(\mms,\met,\meas)$, see \autoref{Ex:mms}, where $\mathcall{k}\colon \mms \to\R$ is continuous and lower bounded as well as $N\in [2,\infty)$ \cite[Def.~3.1, Def.~3.3, Thm.~3.4]{sturm2020}. Let $E\subset\mms$ be as in \autoref{Re:Subsets} with $\meas[E] < \infty$.

For a function $f$ on $\mms$ or $E$, denote by $f_n$ its truncation $\max\{\min\{f,n\},-n\}$  at the levels $n$ and $-n$, $n\in\N$. Following \cite[Def.~2.1]{sturm2020} we set
\begin{align*}
W^{1,1+}(\mms) &:= \Big\lbrace f\in \Ell^1(\mms) : f_n\in \F\text{ for every }n\in\N,\\ 
&\qquad\qquad \sup_{n\in\N} \big\Vert \vert f_n\vert + \vert\rmd f_n\vert\big\Vert_{\Ell^1(\mms)} < \infty \Big\rbrace.
\end{align*}
Let $W^{1,1+}(E)$ be defined analogously w.r.t.~$W^{1,2}(E)$ and $\vert\rmd \cdot\vert_E$. We assume that $E$ has \emph{regular boundary} \cite[p.~1702]{sturm2020}, i.e.~$v\in \Dom(\Delta)$ with $v, \Delta  v\in \Cont(\mms)\cap \Ell^\infty(\mms)$, where $v:= \met(\cdot,E) - \met(\cdot, E^\rmc)$ is the signed distance function from $\partial E$, and
\begin{align*}
 W^{1,1+}(\mms)\big\vert_E = W^{1,1+}(E).
\end{align*}
Then thanks to \cite[Lem.~6.10]{sturm2020}, there exists a nonnegative $\smash{\sigma\in \Meas_\fin^+(\mms)_\Ch}$ supported on $\partial E$ such that for every $\smash{h\in W^{1,2}_\bounded(\mms)}$,
\begin{align}\label{Eq:84949}
\int_{\partial E} \widetilde{h}\d\sigma &= \int_E \rmd  v(\nabla h)\d\meas + \int_\mms\Delta v\,h\d\meas.
\end{align}

\begin{lemma}\label{Le:Nl cpnt} In the notation of \autoref{Sub:GG formula}, the vector field $(\nabla v)_E\in\Ell^2(TE)$ belongs to $\Dom_{\Ell^2}(\DIV_E)$ with
\begin{align*}
\div_E(\nabla v)_E &= \Delta v\quad\meas_E\text{-a.e.},\\
\norm_E(\nabla v)_E &= \sigma.
\end{align*}
\end{lemma}

\begin{proof} Since $\meas[E]< \infty$, any given $\smash{h\in W_\bc^{1,2}(E)}$ belongs to $W^{1,1+}(E)$, and hence to $\smash{W^{1,1+}(\mms)\big\vert_E}$ by regularity of $\partial E$. Thus, there exists $\smash{\overline{h} \in W^{1,1+}(\mms)}$ such that $\overline{h} = h$ $\meas$-a.e.~on $E$. In particular, $\overline{h}_n\in W_\bounded^{1,2}(\mms)$ for every $n\in\N$. Since $\smash{\overline{h}_n = h}$ $\meas$-a.e.~on $E$ for large enough $n\in\N$, the claim follows from  \eqref{Eq:84949}.
\end{proof}

\begin{remark} If the integration by parts formula as in  \cite[Lem.~6.11]{sturm2020} holds --- which, in fact,  uniquely characterizes $\sigma$ --- a similar argument as for \autoref{Le:Nl cpnt} yields that for every $f\in\Dom(\Delta)$ with $\nabla f\in\Ell^\infty(T\mms)$, we have $(\nabla f)_E\in\Dom(\DIV_E)$ with
\begin{align*}
\div_E(\nabla f)_E &= \Delta f\quad\meas_E\text{-a.e.},\\
\norm_E(\nabla f)_E &= \langle\nabla f,\nabla v\rangle_\sim\,\sigma.
\end{align*}
The latter is well-defined thanks to \autoref{Le:Kato inequality} and \autoref{Cor:Dom(Delta) subset W22}.
\end{remark}

\begin{proposition}\label{Pr:...} Given any $X\in\Reg(T\mms)$, define $X^\perp\in \Ell^2(T\mms)$ by
\begin{align*}
X^\perp := \langle X,\nabla v\rangle\,\nabla v.
\end{align*}
Then $\smash{X_E^\perp \in \Dom(\DIV_E)}$ with
\begin{align*}
\div_E X^\perp_E &= \langle X,\nabla v\rangle\,\Delta v + \nabla X : (\nabla v\otimes\nabla v)\quad\meas_E\text{-a.e.},\\
\norm_E X^\perp_E &= \langle X,\nabla v\rangle_\sim\,\sigma.
\end{align*}
\end{proposition}

\begin{proof} Observe that $\smash{\langle X,\nabla v\rangle\in W_\bounded^{1,2}(\mms)}$, so that both the statements make sense. The claimed formulas follow from \autoref{Le:Nl cpnt} as well as \autoref{Le:Div identities}  while noting that by \autoref{Pr:Compatibility} and since $\vert\nabla v\vert = 1$ $\meas$-a.e.,
\begin{align*}
\big\langle \nabla \langle X,\nabla v\rangle,\nabla v\big\rangle &= \nabla X : (\nabla v\otimes\nabla v) + \Hess v(X,\nabla v)\\
&= \nabla X : (\nabla v\otimes \nabla v) + \rmd\vert\nabla v\vert^2(X)/2\\
&= \nabla X: (\nabla v \otimes \nabla v)\quad\meas\text{-a.e.}\qedhere
\end{align*}
\end{proof}

\begin{remark} We do not know if the pointwise defined second fundamental form from \cite[Rem.~5.13]{sturm2020} is related to $\II_E$. Similarly to \autoref{Re:SFF}, by \autoref{Pr:...} these notions coincide if $X\in\Reg(T\mms)$ obeys $\vert X\vert^2\in\Dom(\Hess)$, but we do not know if many of such $X$ can be found in general. 

At least, this time the taming condition for $(E,\Ch_E,\meas_E)$ is already provided once one can verify that the taming distribution
\begin{align*}
\kappa := -\mathcall{k}^-\,\meas_E - \mathcall{l}^-\,\sigma,
\end{align*}
for some appropriate $\mathcall{l}\in\Cont(\mms)$, according to \cite[Thm.~6.14]{sturm2020} and \cite[Prop.~2.16]{erbar2020} belongs to $\Kato_{1-}(E)$. See \cite{braunrigoni2021,erbar2020} for examples in this direction.
\end{remark}

\end{document}